\documentclass[reqno]{amsart}

\usepackage{amsmath, amsthm, amssymb, amscd, mathtools}
\usepackage{enumitem}
\usepackage{tikz}
\usepackage{tikz-cd}
\usepackage{subcaption}
\usepackage{hyperref}
\usepackage[capitalise]{cleveref}
\hypersetup{colorlinks=true, linkcolor={blue!50!black}, citecolor={blue!50!black}, urlcolor={blue!60!black}}
\usepackage{comment}
\usepackage{marginnote}
\usepackage{a4wide}

\newtheorem{theorem}{Theorem}[section]
\newtheorem{lemma}[theorem]{Lemma}
\newtheorem{proposition}[theorem]{Proposition}
\newtheorem{corollary}[theorem]{Corollary}

\theoremstyle{definition}
\newtheorem{definition}[theorem]{Definition}
\newtheorem{remark}[theorem]{Remark}
\newtheorem{question}[theorem]{Question}
\newtheorem{example}[theorem]{Example}
\newtheorem{assumption}[theorem]{Assumption}

\AddToHook{env/lemma/begin}{\crefalias{theorem}{lemma}}
\AddToHook{env/proposition/begin}{\crefalias{theorem}{proposition}}
\AddToHook{env/corollary/begin}{\crefalias{theorem}{corollary}}
\AddToHook{env/conjecture/begin}{\crefalias{theorem}{conjecture}}
\AddToHook{env/definition/begin}{\crefalias{theorem}{definition}}
\AddToHook{env/remark/begin}{\crefalias{theorem}{remark}}
\AddToHook{env/question/begin}{\crefalias{theorem}{question}}
\AddToHook{env/example/begin}{\crefalias{theorem}{example}}
\AddToHook{env/assumption/begin}{\crefalias{theorem}{assumption}}

\crefname{theorem}{Theorem}{Theorems}
\crefname{lemma}{Lemma}{Lemmas}
\crefname{proposition}{Proposition}{Propositions}
\crefname{corollary}{Corollary}{Corollaries} 
\crefname{conjecture}{Conjecture}{Conjectures}
\crefname{definition}{Definition}{Definitions}
\crefname{remark}{Remark}{Remarks}
\crefname{question}{Question}{Questions}
\crefname{example}{Example}{Examples}
\crefname{assumption}{Assumption}{Assumptions}
\crefname{figure}{Figure}{Figures}
\crefformat{equation}{(#2#1#3)}
\Crefformat{equation}{Equation #2(#1)#3}

\newcommand\bfx{\mathbf{x}}
\newcommand\bfy{\mathbf{y}}

\newcommand\bfz{\mathbf{z}}

\newcommand\cG{\mathcal{G}}
\newcommand\cH{\mathcal{H}}
\newcommand\cS{\mathcal{S}}

\newcommand\frG{\mathfrak{G}}
\newcommand\frT{\mathfrak{T}}

\newcommand\sfF{\mathsf{F}}
\newcommand\sfG{\mathsf{G}}
\newcommand\sfS{\mathsf{S}}
\newcommand\sfT{\mathsf{T}}

\newcommand\sfH{\mathsf{H}}
\newcommand\sfY{\mathsf{Y}}

\newcommand\sfb{\mathsf{b}}
\newcommand\sfc{\mathsf{c}}
\newcommand\sfe{\mathsf{e}}
\newcommand\sff{\mathsf{f}}
\newcommand\sfs{\mathsf{s}}
\newcommand\sfu{\mathsf{u}}
\newcommand\sfv{\mathsf{v}}
\newcommand\sfw{\mathsf{w}}

\newcommand{\PB}{\mathbb{P}}
\newcommand{\BG}{\mathbb{B}}
\newcommand{\im}{\operatorname{im}}
\newcommand{\Z}{\mathbb{Z}}
\newcommand{\R}{\mathbb{R}}
\newcommand{\RP}{\mathbb{RP}}
\newcommand{\co}[1]{\mathopen{\langle}#1\mathclose{\rangle}} %
\newcommand{\base}{\bfx^0}
\newcommand{\Lollipop}{\mathsf{L}}

\renewcommand{\bar}{\overline}
\renewcommand{\tilde}{\widetilde}

\DeclareMathOperator{\lk}{lk}
\DeclareMathOperator{\val}{val}
\DeclareMathOperator{\proj}{Proj}

\numberwithin{equation}{section}

\title[Extending Goldberg's Exact Sequence]{Extending Goldberg's Exact Sequence \\
to Braid Groups of Graphs and Simplicial Complexes}
\author{Byung Hee An}
\address{Department of Mathematics Education, Kyungpook National University, Daegu, South Korea 41566}
\email{anbyhee@knu.ac.kr}
\urladdr{https://orcid.org/0000-0002-9232-6985}
\date{\today}

\begin{document}

\begin{abstract}
\noindent For a finite connected simplicial complex \(X\), the
strand map \(\iota_\ast\), from the pure braid group
\(\PB_n(X)\) to
\(\prod_{i=1}^n\pi_1(X,x_i^0)\), sends a pure braid to the
homotopy classes of its strands. A
theorem of Goldberg (1973) computes its kernel when \(X\) is a
closed surface other than \(S^2\) and \(\RP^2\): the kernel is
the normal closure of the pure braids supported in an embedded
disc. We extend this picture to arbitrary finite connected
simplicial complexes. Call \(X\) \emph{weakly Goldberg} if some
contractible subcomplex \(X_0\subseteq X\) realises Goldberg's
description,
\(\ker\iota_\ast=\co{\im(\PB_n(X_0)\to\PB_n(X))}\), and
\emph{Goldberg} if \(X_0\) can moreover be chosen so that
\(\PB_n(X_0)\to\PB_n(X)\) is injective. For \(n\geq2\), we prove that the strand map is surjective if
and only if \(X\not\cong S^1\). Apart from the interval, the
circle and the two exceptional closed surfaces, \(X\) is weakly
Goldberg if and only if its free part is a forest, and Goldberg
if and only if it admits an
\emph{admissible tree} -- a maximal tree of a scaffold,
compatible with the boundary and interior types of the
attachments of the free part to the thick components. We also
classify the complexes for which the kernel is trivial, settle
the exceptional surfaces \(S^2\) and \(\RP^2\), and obtain
complete answers for manifolds and for graphs. The main tools
are a graph-of-spaces decomposition of the configuration space
at a point of \(X\) and a resolution procedure reducing an
arbitrary complex to a simple model.
\end{abstract}

\maketitle

{\enlargethispage{3pt}\tableofcontents}

\section{Introduction}

Throughout this paper \(X\) denotes a finite connected simplicial
complex of dimension at least \(1\), and \(n\geq 1\) a fixed integer.
The \emph{ordered configuration space} of \(n\) distinct points in
\(X\) is the subset
\[
F_n (X) = \{\,(x_1,\dots,x_n)\in X^n : x_i\neq x_j
\text{ for all } i\neq j\,\} \subseteq X^n.
\]
Fix a basepoint
\(\base = (x_1^0,\dots,x_n^0)\in F_n (X)\), i.e.\ an \(n\)-tuple of
pairwise distinct points of \(X\). The \emph{pure \(n\)-braid group
of \(X\) (based at \(\base\))} is
\[
\PB_n(X,\base) \coloneqq \pi_1(F_n (X), \base).
\]
When there is no risk of ambiguity we suppress the basepoint and
write simply \(\PB_n(X)\).

In existence statements involving a witness \(X_0\subseteq X\),
we choose the base configuration in \(F_n(X_0)\) after selecting
\(X_0\). Once chosen, this configuration is used for all maps
associated with that pair.

The natural inclusion \(F_n (X)\hookrightarrow X^n\) induces
\[
\iota_\ast \colon \PB_n(X,\base)\longrightarrow
\prod_{i=1}^n \pi_1(X, x_i^0),
\]
which sends a pure braid to the \(n\)-tuple of loops traced out by
its individual strands; we call this the \emph{strand map}. The
central theme of this paper is a detailed description of
\(\ker\iota_\ast\) in terms of the topology of \(X\), generalising a
classical theorem of Goldberg on surfaces which we now recall.

\subsection{Goldberg's theorem}

In this subsection we specialise to the case
\(X = M\), a compact surface, and recall Goldberg's 1973
\cite{Goldberg1973} theorem, which is the model our kernel
description seeks to extend. Fix a basepoint
\(\base = (x_1^0,\dots,x_n^0)\in F_n (M)\) with all
\(x_i^0\) contained in a single open disc \(D^2\subset M\). The
inclusion of the disc \(D^2\hookrightarrow M\) induces
\[
j_\ast \colon \PB_n(D^2,\base)\longrightarrow\PB_n(M,\base).
\]
The group \(\PB_n(D^2,\base)\) is the classical Artin pure braid group on
\(n\) strands.

\begin{theorem}[Goldberg, \protect{\cite[Thm.~1]{Goldberg1973}}]
\label{thm:goldberg}
If \(M\) is a compact surface other than \(S^2\) and \(\RP^2\), then
\[
1 \longrightarrow \PB_n(D^2,\base) \xrightarrow{j_\ast}
\PB_n(M,\base) \xrightarrow{\iota_\ast}
\prod_{i=1}^n \pi_1(M,x^0_i) \longrightarrow 1
\]
is exact in the non-abelian sense, i.e.\ \(j_\ast\) is injective,
\(\iota_\ast\) is surjective, and \(\ker \iota_\ast\) equals the normal
closure \(\co{\im j_\ast}\) in \(\PB_n(M,\base)\).
\end{theorem}

\begin{remark}\label{rem:goldberg-compact}
Goldberg states \cref{thm:goldberg} for \emph{closed}
surfaces \cite[Thm.~1]{Goldberg1973}; the version above, for
compact surfaces possibly with boundary, is already contained in
his argument. The injectivity of \(j_\ast\) is proved in
\cite[Lemma~3]{Goldberg1973} for \emph{any compact surface}
other than \(S^2\) and \(\RP^2\), and the kernel containment
\(\ker\iota_\ast\subseteq\co{\im j_\ast}\) of
\cite[Lemma~5]{Goldberg1973} uses only that the universal cover
of \(M\) embeds in \(\R^2\) -- which holds for every surface
\(M\not\cong S^2,\RP^2\), closed or with boundary, since a simply
connected surface other than \(S^2\) is planar.
\end{remark}

Intuitively, the theorem decomposes a surface braid into
(a) ``classical Artin braiding'' visible inside a single disc, and
(b) ``wandering of strands'' detected only by the fundamental group
of \(M\).

\begin{remark}\label{rem:s2-rp2}
For \(n\geq2\), when \(M = S^2\) or \(M = \RP^2\), the kernel description of
\cref{thm:goldberg} still holds, namely
\[
\ker\iota_\ast = \co{\im j_\ast}^{\PB_n(M,\base)},
\]
and \(\iota_\ast\) is still surjective; what fails is the
\emph{injectivity} of \(j_\ast\). For \(M=S^2\) and \(n=2\),
\(\PB_2(S^2)=1\), so the full twist generating
\(\PB_2(D^2)\cong\Z\) dies. For \(M=S^2\) with \(n\geq3\),
or \(M=\RP^2\) with \(n\geq2\), the infinite-order full twist
of the disc has non-trivial finite-order image
\cite{VanBuskirk1966}. Thus in every case \(j_\ast\) has a
non-trivial kernel, and the leftmost arrow
``\(1\to\PB_n(D^2)\)'' of the exact sequence is invalid.
\end{remark}

\begin{remark}\label{rem:birman-high-dim}
Goldberg's theorem is a two-dimensional phenomenon: in dimension
three and above there is no braiding to detect. By Birman's
theorem \cite[Thm.~1]{Birman1969}, for a manifold \(M\) with
\(\dim M\geq 3\) the inclusion \(F_n M\hookrightarrow M^n\) induces
an isomorphism
\[
\iota_\ast\colon\PB_n(M,\base)\xrightarrow{\ \cong\ }
\prod_{i=1}^n\pi_1(M,x_i^0).
\]
Birman states this for closed smooth manifolds. The same
general-position argument applies to PL manifolds: the pairwise
diagonals in \(M^n\) have codimension \(\dim M\geq3\), so paths
and homotopies can avoid them. Closedness is unnecessary, and
for a manifold with boundary a collar isotopy of embeddings
first pushes configurations into the interior. We use this
extension for the PL regular neighbourhoods below. In particular
\(\iota_\ast\) is injective, \(\ker\iota_\ast = 1\), and so
\emph{every manifold of dimension at least three is
Goldberg-trivial} in the terminology of
\cref{ssec:kernel-question}.
\end{remark}

\subsection{Generalising Goldberg's theorem}
\label{ssec:kernel-question}

For a connected subspace \(X_0\subseteq X\) with \(\base\in F_n(X_0)\),
the inclusion induces
\(j_\ast\colon\PB_n(X_0,\base)\to\PB_n(X,\base)\).

Goldberg's theorem (\cref{thm:goldberg}) asserts that, for a
closed surface \(X = M\neq S^2,\RP^2\) and \(X_0 = D^2\), the sequence
\[
1 \longrightarrow \PB_n(X_0,\base) \xrightarrow{j_\ast}
\PB_n(X,\base) \xrightarrow{\iota_\ast}
\prod_{i=1}^n \pi_1(X,x^i_0) \longrightarrow 1
\]
is exact.

Early work on graph braid groups includes Abrams's
thesis~\cite{Abrams2000}; their connection with robotic motion
planning is described by Abrams--Ghrist~\cite{AbramsGhrist2002}.
Subsequent developments for graphs and higher-dimensional complexes
include discrete Morse-theoretic presentations by
Farley--Sabalka~\cite{FarleySabalka2005}, structural and
embedding results by Ko--Park~\cite{KoPark2012} and
Kim--Ko--Park~\cite{KimKoPark2014}, subdivisional models and
homological computations by
An--Drummond-Cole--Knudsen~\cite{ADK}, and the reduction theory
to simple complexes by An--Park~\cite{AnPark2017}, on which the
present paper relies throughout. The strand map and the kernel
description above, however, do not seem to have been studied
systematically beyond the surface case.

For a general pair \((X, X_0)\), exactness of this sequence
decouples into three independent questions, one for exactness at
each position:

\begin{question}[Goldberg condition]\label{q:goldberg-cond}
\label{q:main}%
For which spaces \(X\) does there exist a contractible subspace
\(X_0\subseteq X\) with \(\base\in F_n(X_0)\) such that
\[
\ker\iota_\ast = \co{\im j_\ast}^{\PB_n(X,\base)}?
\]
\end{question}

\begin{remark}\label{rem:Q1.3-split}
The condition of \cref{q:goldberg-cond} splits into two inclusions.
For a contractible subspace \(X_0\), every strand of a braid
supported in \(X_0\) is null-homotopic in \(X\). Thus
\(\im j_\ast\subseteq\ker\iota_\ast\), and normality gives
\(\co{\im j_\ast}^{\PB_n(X,\base)}\subseteq\ker\iota_\ast\)
automatically. The question is therefore to find a
\emph{contractible} subspace \(X_0\subseteq X\), containing
the base configuration, for which the reverse inclusion holds:
\[
\ker\iota_\ast\subseteq
\co{\im j_\ast}^{\PB_n(X,\base)}.
\]
More generally, the automatic inclusion holds whenever
\(X_0\hookrightarrow X\) induces the zero map on fundamental
groups, but this observation does not remove the
contractibility requirement in the question.
\end{remark}

\begin{question}[Embedding condition]\label{q:embedding-cond}
For which spaces \(X\) can such an \(X_0\) be chosen so that, in
addition, \(j_\ast\colon\PB_n(X_0,\base)\to\PB_n(X,\base)\) is an embedding?
\end{question}

\begin{remark}\label{rem:embedding-standalone}
Independently of the Goldberg condition, one can also ask the
embedding question for an arbitrary pair \((X, X_0)\) with
\(X_0\subseteq X\) and \(\base\in F_n(X_0)\): for which such pairs is
\(j_\ast\colon \PB_n(X_0,\base)\to \PB_n(X,\base)\) injective? This
stand-alone embedding problem is interesting in its own right,
but we do not pursue it separately in this paper: the
injectivity of \(j_\ast\) enters only in tandem with the kernel
condition, through the notion of Goldbergness.
\end{remark}

\begin{question}[Surjectivity condition]\label{q:surj-cond}
For which spaces \(X\) is the strand map
\(\iota_\ast\colon \PB_n(X,\base)\to\prod_{i=1}^n\pi_1(X,x^0_i)\) surjective?
\end{question}

For brevity we record the following terminology, formalised in
\cref{def:goldberg-strengths} of
\cref{sec:isolate}: for a fixed \(n\ge 1\), a space \(X\) is \emph{weakly Goldberg}
when some contractible subspace \(X_0\subseteq X\) with \(\base\in F_n(X_0)\)
satisfies \(\co{\im j_\ast} = \ker \iota_\ast\), and \emph{Goldberg} when in addition \(j_\ast\) can be chosen to be an
embedding; and \(X\) is \emph{Goldberg-trivial} (or
\emph{trivially Goldberg}) when \(\ker\iota_\ast\) is trivial.

\begin{remark}\label{rem:gt-implies-goldberg}
Every Goldberg-trivial space is Goldberg; see
\cref{lem:gt-implies-goldberg} in \cref{sec:isolate}.
\end{remark}

In this language, Goldberg's \cref{thm:goldberg} says that for any \(n\ge 1\),
every closed surface \(M\neq S^2,\RP^2\) is Goldberg -- witnessed by any embedded disc \(D^2\subseteq M\) -- and its strand map is surjective.

\subsection{Trivial cases and the standing assumption}
\label{ssec:dichotomy}

\cref{q:main} concerns only the kernel of \(\iota_\ast\), not
its image, but it is still useful to clear away two degenerate
cases, \(n=1\) and \(X\cong S^1\), and to record separately the
special features of the simply connected case. Removing the two
degenerate cases leaves only the genuinely interesting situations
and (after a general proposition) ensures \(\iota_\ast\) is
automatically surjective.

\subsubsection{Trivial case~1: \(n = 1\).}\label{subsubsection:n-equals-one}

For \(n = 1\) we have \(F_1 (X) = X\) and
\(\PB_1(X,\base) = \pi_1(X,x_1^0)\); hence
\(X\) is Goldberg-trivial. We
therefore assume \(n\geq 2\) from now on.

\subsubsection{The simply connected case.}\label{subsubsection:simply-connected}

If \(X\) is simply connected, then \(\pi_1(X,x_i^0) = 1\) for every
\(i\), so the target of \(\iota_\ast\) is the trivial group; thus
\(\iota_\ast\) is the zero map, trivially surjective, and
\(\ker\iota_\ast = \PB_n(X,\base)\). Whether \(X\) is Goldberg then
amounts to whether \(\PB_n(X,\base)\) is realised by a
\emph{contractible} witness \(X_0\), in the sense of
\cref{def:goldberg-strengths}. This is immediate when
\(X\) is itself contractible -- one may take \(X_0 = X\) -- but it is
not a degenerate case in general: a simply connected complex that
is not contractible, such as \(X = S^2\), admits no contractible
witness equal to the whole space, and indeed need not be Goldberg
at all (\cref{thm:s2-rp2-exceptional}).

One simply connected complex, however, is degenerate enough to
be set aside once and for all: the interval \(\mathsf{I}\). Its
ordered configuration space \(F_n(\mathsf{I})\) consists of
\(n!\) contractible components, so \(\PB_n(\mathsf{I},\base)=1\)
for every \(n\geq1\) and \(\base\in F_n(\mathsf{I})\); in
particular \(\ker\iota_\ast=1\), and the interval is
Goldberg-trivial. We therefore exclude the interval -- but
retain all other simply connected complexes -- from the scope
of this paper.

\begin{remark}\label{rem:simpler-X0}
The contractibility requirement on the witness \(X_0\) in
\cref{def:goldberg-strengths} is what keeps the Goldberg
condition non-degenerate here. Were \(X_0\) allowed to be an
arbitrary connected -- or merely \(\pi_1\)-trivial -- subspace, every
simply connected \(X\) would be Goldberg via the degenerate
choice \(X_0 = X\), for which \(j_\ast\) is the identity and
\(\im j_\ast = \PB_n(X,\base) = \ker\iota_\ast\). Demanding that
\(X_0\) be contractible rules this out: a non-contractible simply
connected complex need not admit a Goldberg witness: \(S^2\)
with \(n\geq3\) is an example
(\cref{thm:s2-rp2-exceptional}). The classical witness of
Goldberg's theorem -- an embedded disc \(D^2\) in a surface -- is
contractible, so that theorem is unaffected by the requirement.
\end{remark}

\subsubsection{Trivial case~2: \(X \cong S^1\).}

For non-simply-connected \(X\) the kernel question becomes
non-trivial, with one further degenerate situation -- the case
\(X\cong S^1\), in which strands on a circle cannot move
independently.

\begin{example}\label{ex:s1-not-surj}
For \(X = S^1\) the pure braid group
\(\PB_n(S^1,\base)\cong\Z\) is infinite cyclic for every \(n\ge1\) and \(\base\in F_n(S^1)\), generated by the
element that simultaneously rotates all \(n\) points once around
the circle. Geometrically, the \(n\) strands wind around the circle
in lockstep, so the strand map \(\iota_\ast\) is \emph{not necessarily}
surjective: its image is the diagonal \(\Z\hookrightarrow\Z^n\)
(see \cref{prop:s1-case} below).
\end{example}

\begin{proposition}\label{prop:s1-case}
For \(X = S^1\), any \(n\geq 1\) and \(\base\in F_n(S^1)\),
\[
\iota_\ast \colon \PB_n(S^1,\base)\cong\Z\hookrightarrow \Z^{n} = \prod_{i=1}^n\pi_1(S^1,x_i^0)
\]
is the diagonal embedding \(k\mapsto (k,k,\dots,k)\). In particular
\(\iota_\ast\) is injective and \(\ker\iota_\ast = 1\). Choosing
\(\mathsf{I}\subseteq S^1\) to be any embedded arc with
\(\base\in F_n(\mathsf{I})\) gives \(\PB_n(\mathsf{I},\base)=1\),
as computed in \cref{subsubsection:simply-connected}, so
\(S^1\) is Goldberg-trivial.
\end{proposition}

\begin{assumption}[Standing assumption]\label{rem:standing}
Throughout the remainder of this paper we restrict
attention to the integer \(n\geq 2\) and finite connected
simplicial complexes \(X\) of dimension at least \(1\) with
\(X\) homeomorphic to none of the interval \(\mathsf{I}\),
\(S^1\), \(S^2\) and \(\RP^2\): the interval and the circle are
Goldberg-trivial (\cref{subsubsection:simply-connected},
\cref{prop:s1-case}), and the two exceptional surfaces are
settled completely by \cref{thm:s2-rp2-exceptional}. Simply
connected complexes other than the interval are \emph{not}
excluded: the contractibility requirement on the witness \(X_0\)
in \cref{def:goldberg-strengths} keeps the Goldberg
condition substantive for them
(cf.\ \cref{subsubsection:simply-connected}).
\end{assumption}

\subsection{Main results}
\label{ssec:main-results}

This paper answers
\cref{q:goldberg-cond,q:embedding-cond,q:surj-cond} completely.
The answer to the surjectivity question is a dichotomy
(\cref{prop:iota-surj}): for \(n\geq2\), the strand map
\(\iota_\ast\) is surjective if and only if
\(X\not\cong S^1\). The answers to the other two questions take
the form of three nested classifications,
\[
\{\text{Goldberg-trivial}\}\subseteq\{\text{Goldberg}\}
\subseteq\{\text{weakly Goldberg}\},
\]
each phrased in terms of elementary combinatorial invariants of
\(X\), which we now describe informally; the formal definitions
appear in \cref{ssec:free-parts}.

A \emph{free edge} of \(X\) is a \(1\)-simplex contained in
no simplex of dimension at least \(2\). The
\emph{free part} \(\mathsf{F}_X\subseteq X\) is the
at most one-dimensional subcomplex consisting of all free edges
together with all vertices of valency\footnote{The number of
connected components of the link \(\lk(\sfv)\).} at least \(2\)
whose link has positive dimension
(\cref{def:free-part}). The \emph{thick components}
\(\Sigma\) of \(X\) are the closures of the connected components
of \(X\setminus\mathsf{F}_X\); the free part is attached to them
at \emph{joints}. An attachment -- a pair \((\sfv,L)\) of a joint
and a positive-dimensional component of its link -- is of
\emph{type 1} if it is an attachment at a boundary point of a
surface piece, and of \emph{type 2} otherwise
(\cref{def:admissible-tree}). Finally, a
\emph{weakly admissible tree} is a maximal tree \(\sfT\) of a \emph{scaffold} -- the free part together with a system of arcs
connecting the joints through each thick component -- containing
all of \(\mathsf{F}_X\) (\cref{def:scaffold}); it is
\emph{admissible} if, on every thick component carrying
a type-2 attachment, \(\sfT\) connects each attachment of that
component to a type-2 one.

\begin{theorem}[Weak Goldbergness;
\cref{thm:weak-goldberg-char}]\label{thm:main-intro}
Let \(X\) be as in the standing assumption
(\cref{rem:standing}). Then \(X\) is weakly Goldberg if
and only if every connected component of the free part
\(\mathsf{F}_X\) is simply connected -- equivalently, if and
only if \(X\) admits a weakly admissible tree.
\end{theorem}

\begin{theorem}[Goldbergness;
\cref{thm:goldberg-char-trees-general}]\label{thm:main-strong-intro}
Let \(X\) be as above. Then
\(X\) is Goldberg if and only if \(X\) admits an admissible tree.
\end{theorem}

The gap between the two notions is genuine, and is witnessed
already by the simplest of examples: a disc \(D\) together with
a single free edge joining an interior point of \(D\) to a
boundary point is weakly Goldberg but not Goldberg
(\cref{ex:disc-chord}) -- weakly admissible trees exist, but
the cycle closed by the free edge and the disc prevents any of
them from being admissible.

\begin{theorem}[Goldberg-triviality;
\cref{thm:gt-classification}]\label{thm:main-gt-intro}
Let \(X\) be as above. Then
\(X\) is Goldberg-trivial if and only if \(X\) is
simple and: no point of \(X\) has valency at least \(3\); \(X\)
is obtained from a single thick component \(\Sigma\) by attaching
finitely many intervals at one endpoint each; and \(\Sigma\)
either is not a \(2\)-manifold, or carries one of these intervals
at an interior point.
\end{theorem}

Informally, \cref{thm:main-gt-intro} says that \(X\) is
Goldberg-trivial precisely when it is \emph{secretly
three-dimensional}: by a theorem of
Birman~\cite[Thm.~1]{Birman1969}, manifolds of dimension at
least \(3\) have \(\ker\iota_\ast=1\), and the conditions above
characterise the complexes whose braids enjoy the same freedom.

The bridge between the weak and the strong classification is a
local-to-global principle, which reduces the injectivity of
\(j_\ast\) -- a global condition -- to the geometry of the
witness inside the closed regular neighbourhoods \(N(\bar\Sigma)\) of
the thick components of the simple model \(\bar X\):

\begin{theorem}[Local-to-global;
\cref{thm:witness-local-injectivity} and
\cref{cor:goldberg-char-local}]\label{thm:main-local-intro}
The simple model \(\bar X\) is Goldberg if and only if it admits
a weak Goldberg witness \(\bar X_0\) such that
\(\PB_n(\bar\Sigma_0^i)\to\PB_n(N(\bar\Sigma))\) is injective
for every thick component \(\bar\Sigma\) and every connected
component \(\bar\Sigma_0^i\) of \(\bar X_0\cap N(\bar\Sigma)\).
\end{theorem}

The two closed surfaces excluded by the standing assumption
are settled separately, with one small surprise at \(n=2\):

\begin{theorem}[Exceptional surfaces;
\cref{thm:s2-rp2-exceptional}]\label{thm:main-s2rp2-intro}
For every \(n\geq2\), both \(S^2\) and \(\RP^2\) are weakly
Goldberg; they are Goldberg only in the single case
\(X\cong S^2\), \(n=2\), where \(\PB_2(S^2,\base)\) is trivial
and \(X\) is even Goldberg-trivial.
\end{theorem}

For the special case of manifolds the picture is complete:

\begin{theorem}[Manifold case]\label{thm:manifold-summary}
Let \(X = M\) be a connected manifold of dimension at least \(1\),
and fix \(n\geq 2\). Then:
\begin{enumerate}[label=(\roman*),leftmargin=2em]
    \item \(M\) is always weakly Goldberg;
    \item \(M\) is Goldberg if and only if
          \(M\not\cong S^2,\RP^2\), or \(M\cong S^2\) and
          \(n=2\);
    \item the strand map of \(M\) is surjective if and only if
          \(M\not\cong S^1\);
    \item \(M\) is Goldberg-trivial if and only if
          \(M\cong\mathsf{I}\) (an interval) or \(M\cong S^1\),
          or \(\dim M\geq 3\), or \(M\cong S^2\) and \(n=2\).
\end{enumerate}
\end{theorem}

\begin{proof}[Proof (assuming the results quoted)]
Statement~(iii) is the surjectivity dichotomy of
\cref{prop:iota-surj}. A compact connected \(1\)-manifold is an
interval or a circle, and both are Goldberg-trivial
(\cref{subsubsection:simply-connected,prop:s1-case}); a
manifold of dimension at least \(3\) is Goldberg-trivial by
Birman's Theorem~1~\cite[Thm.~1]{Birman1969}
(\cref{rem:birman-high-dim}). Since Goldberg-trivial complexes
are Goldberg, and Goldberg complexes are weakly Goldberg
(\cref{lem:gt-implies-goldberg}), this proves (i), (ii) and (iv) in
all dimensions other than \(2\).

Let \(\dim M=2\), so that \(M\) is a compact surface, possibly
with boundary. Then \(M\) has no free edges, so its free part is
empty, in particular a forest; when \(M\not\cong S^2,\RP^2\),
the standing assumption applies and \(M\) is weakly Goldberg by
\cref{thm:main-intro}, while \(S^2\) and \(\RP^2\) are weakly
Goldberg by \cref{thm:main-s2rp2-intro}. This proves (i). For
(ii), a compact surface other than \(S^2\) and \(\RP^2\) is
Goldberg by Goldberg's theorem (\cref{thm:goldberg}, in the
compact form of \cref{rem:goldberg-compact}), and among
\(S^2,\RP^2\) the only Goldberg case is \(M\cong S^2\) with
\(n=2\) (\cref{thm:main-s2rp2-intro}). For (iv), no surface is
Goldberg-trivial except \(S^2\) with \(n=2\): for
\(M\not\cong S^2,\RP^2\) the kernel of the strand map contains
the injectively embedded disc braids
\(\PB_n(D^2)\neq1\) (\cref{thm:goldberg,rem:goldberg-compact});
for \(M\cong\RP^2\) it contains the non-trivial image of the
full twist; and for \(M\cong S^2\) it is all of
\(\PB_n(S^2)\), which is non-trivial precisely when \(n\geq3\)
(see \cref{thm:s2-rp2-exceptional} and its proof).
\end{proof}

A particularly sharp consequence of \cref{thm:main-intro}
concerns the purely \(1\)-dimensional case:

\begin{theorem}[Graph case]\label{thm:graph-intro}
Let \(\Gamma\) be a finite connected graph (a simplicial complex
of dimension \(1\)), and fix \(n\ge 2\). Then the following are
equivalent:
\begin{enumerate}[label=\textup{(\roman*)},leftmargin=2em]
    \item \(\Gamma\) is weakly Goldberg;
    \item \(\Gamma\) is Goldberg;
    \item \(\Gamma\) is either a tree or homeomorphic to
          \(S^1\).
\end{enumerate}
\end{theorem}

\begin{proof}[Proof (assuming \cref{thm:main-intro})]
(ii)\(\Rightarrow\)(i) holds by definition. For
(iii)\(\Rightarrow\)(ii): if \(\Gamma\) is a tree, it is
contractible, so \(X_0=\Gamma\) is a Goldberg witness --
\(j_\ast\) is the identity, hence injective, and
\(\im j_\ast=\PB_n(\Gamma,\base)=\ker\iota_\ast\), since
\(\iota_\ast\) is the zero map
(\cref{subsubsection:simply-connected}); if
\(\Gamma\cong S^1\), then \(\Gamma\) is Goldberg-trivial, hence
Goldberg, by \cref{prop:s1-case}. For
(i)\(\Rightarrow\)(iii), suppose \(\Gamma\) is neither a tree
nor homeomorphic to \(S^1\); then \(\Gamma\) is not simply
connected and not an interval, so the standing assumption
applies to \(\Gamma\). Its free part is \(\Gamma\) itself
(every \(1\)-cell of \(\Gamma\) is free, since \(\Gamma\) has
no \(2\)-cells), so the unique connected component of
\(\mathsf{F}_\Gamma\) is \(\Gamma\), which is not simply
connected. By \cref{thm:main-intro}, \(\Gamma\) is not weakly
Goldberg.
\end{proof}

\subsection{Techniques and organisation}
\label{ssec:techniques}
The engine of the paper is a graph-of-spaces decomposition of
the configuration space at a point of \(X\)
(\cref{prop:valency-decomp}), whose collapse homomorphisms onto
the fundamental groups of the underlying \emph{distribution
graphs} serve as computable obstructions throughout; its first
fruits are free-product presentations of \(\PB_n(X)\) at free
edges and an explicit rank formula
(\cref{cor:free-edge-splitting,cor:free-edge-loop-rank}). The
basic element of \(\ker\iota_\ast\) is the \emph{shuffle braid}
supported in an embedded copy of the letter \(\sfH\)
(\cref{lem:H-shuffle}). Goldbergness in all its strengths is
invariant under An--Park resolutions
(\cref{thm:resolution-goldberg-equivalence}), which reduces
every question to the \emph{simple model} \(\bar X\). On the
necessity side, a battery of forcing results shows that a weak
Goldberg witness must contain every essential free edge, enter
every point of valency at least \(3\) along all its link
directions, and enter the thick part at every joint having no
interval side
(\cref{sec:lower-bound}). To exclude interval pieces of a
witness, we compare two strands' crossing records at the
thin edges on either side of a piece. Their difference in
the central coordinate of the integer Heisenberg group
vanishes on the witness normal closure but detects an
ambient shuffle (\cref{thm:no-interval-pieces}). On the sufficiency side,
weak witnesses are constructed explicitly from weakly admissible trees
(\cref{sec:simple-model-necessity}), and the local-to-global
principle is proved by unfolding distribution graphs into
finite covers via Stallings completions
(\cref{ssec:local-to-global}). \Cref{sec:goldberg-trivial}
classifies the Goldberg-trivial complexes, and
\cref{sec:goldberg-classification} assembles the
characterisation of Goldbergness by admissible trees.

\subsection*{Acknowledgements}
The author is grateful to Prof.\ Sunghwan Byun of the
Department of STEM Education at North Carolina State University
for his warm hospitality during the author's sabbatical year.
%
This work was supported by Samsung Science
and Technology Foundation under Project Number SSTF-BA2202-03.

\section{Preliminaries}

\subsection{Configuration spaces of cell complexes}

Let \(X\) be a finite connected simplicial complex of
dimension at least \(1\), and \(n\geq 1\) an integer. The
\emph{ordered configuration space} of \(n\) points on \(X\) is
\[
F_n (X) = \{\,(x_1, \dots, x_n) \in X^n : x_i \neq x_j
\text{ for } i \neq j\,\} \subseteq X^n.
\]
The symmetric group \(\mathbb{S}_n\) acts freely on \(F_n (X)\) by permutation
of entries, and the orbit space
\[
B_n (X) = F_n (X) / \mathbb{S}_n
\]
is called the \emph{unordered configuration space} of \(n\) points on
\(X\). Fix basepoints \(\base = (x_1^0,\dots,x_n^0)\in F_n (X)\) and
\([\base]\in B_n (X)\).

When \(X\) is connected, \(B_n(X)\) is always connected but \(F_n (X)\) need not be connected: for example, \(F_2(\mathsf{I})\)
consists of two components (distinguished by the order of the two
points on the interval), and the same phenomenon occurs more
generally whenever the induced permutation homomorphism \(\rho\)
below fails to be surjective. To work consistently with
fundamental groups, we therefore restrict attention to the
component of the basepoint: set
\[
F_n(X,\base) \subseteq F_n (X),
\qquad
B_n(X,[\base]) \subseteq B_n (X)
\]
to be the path components of \(F_n (X)\) and \(B_n (X)\) containing
\(\base\) and \([\base]\), respectively.

\begin{remark}[Standing convention on subspaces]\label{conv:subcomplex}
Unless explicitly described as open or punctured, every subspace
of \(X\) considered in this paper -- in particular any
\emph{witness} \(X_0\subseteq X\) (\cref{def:goldberg-strengths}) -- is
a finite simplicial subcomplex of \(X\); equivalently, after passing to
a subdivision of \(X\) it is a genuine subcomplex. We subdivide \(X\)
freely and without further comment whenever convenient. When
constructing a witness, we choose the base configuration inside
it; subdivision then makes the chosen points vertices of the
subcomplex.

Since \(X_0\) is a subcomplex of \(X\), links are respected by the
simplicial inclusion: for every vertex \(x\in X_0\) one has
\(\lk_{X_0}(x)\subseteq\lk_X(x)\).

\emph{Compact models for cuts.} When a finite complex is punctured
at an interior point of a free edge, we may replace each resulting
half-open end by a compact truncation. These replacements are
understood whenever a statement or induction for finite complexes
is applied to a punctured space. Indeed, on a half-collar with
coordinate \(r\in(0,2]\), the embeddings
\[
c_t(r)=t\varepsilon+(1-t\varepsilon/2)r,
\qquad 0\leq t\leq1,\quad 0<\varepsilon<1,
\]
fix \(r=2\), extend by the identity elsewhere, and at \(t=1\)
map into the truncation \(r\geq\varepsilon\). Their restrictions
to the truncation stay within it. Thus inclusion of the compact
model has an inverse up to isotopies of embeddings, and induces
homotopy equivalences on all ordered configuration spaces.
For finitely many cuts use disjoint collars. When a subcomplex
contains the cut germs, use the same compressions for it and the
ambient space; the comparisons then commute with inclusions and
with the strand maps, with base configurations chosen outside the
shortened collars. The new leaf tips in a compact cut model refer
to these truncations, not to points of the punctured space itself.
\end{remark}

The \emph{pure \(n\)-braid group} and the
(\emph{full}) \(n\)-braid group of \(X\) are, respectively,
\begin{align*}
\PB_n(X,\base) &= \pi_1(F_n(X,\base),\base)\cong \pi_1(F_n(X),\base),\\
\BG_n(X,[\base]) &= \pi_1(B_n(X,[\base]),[\base])\cong \pi_1(B_n(X),[\base]).
\end{align*}
The covering \(F_n (X)\to B_n(X)\) induces an exact sequence
\[
1 \longrightarrow \PB_n(X,\base)
\longrightarrow \BG_n(X,[\base])
\xrightarrow{\rho} \mathbb{S}_n,
\]
where the homomorphism \(\rho\) records the permutation of strands
induced by a braid; we call \(\rho\) the \emph{induced permutation
homomorphism}. The pure braid group \(\PB_n(X,\base)\) is therefore a
normal subgroup of \(\BG_n(X,[\base])\) with quotient
\(\im(\rho)\subseteq\mathbb{S}_n\). 
\begin{example}[Star graph \(\sfS_3\)]\label{ex:S3}
Let \(X = \sfS_3\) be the star graph with three leaves. For
each \(n\geq 1\) there is a short exact sequence
\[
1 \longrightarrow \PB_n(\sfS_3,\base)
\longrightarrow \BG_n(\sfS_3,[\base])
\xrightarrow{\rho} \mathbb{S}_n
\longrightarrow 1
\]
for any \(\base\in F_n(\sfS_3)\),
i.e.\ the induced permutation homomorphism \(\rho\) is surjective.
For two particles initially on distinct arms, move the first
through the centre to the empty third arm, move the second to
the first particle's initial position, and then move the first
to the second particle's initial position. This collision-free
path is the \emph{hexagonal exchange}
\(\tau\in\BG_2(\sfS_3)\); it exchanges the two particles.
Its square \(\tau^2\) generates \(\PB_2(\sfS_3)\cong\Z\).
Under an embedding in a disc, \(\tau\) maps to a half-twist
and \(\tau^2\) to the corresponding full twist, with sign
determined by the choice of exchange direction.
\end{example}

A connected simplicial complex \(X\) of dimension at least \(1\)
contains an embedded star graph \(\sfS_3\) if and only if
\(X\) does not embed in \(S^1\). Indeed, a connected \(1\)-dimensional
simplicial complex in which every vertex has valency \(\leq 2\) is
homeomorphic to an interval \(\mathsf{I}\) or to the circle \(S^1\),
both of which embed in \(S^1\); conversely, \(S^1\) contains no
embedded \(\sfS_3\).

In general \(\rho\) need
\emph{not} be surjective; the next lemma gives a clean geometric
condition that forces surjectivity.

\begin{lemma}\label{lem:rho-surj}
The induced permutation homomorphism
\(\rho\colon\BG_n(X,[\base])\to\mathbb{S}_n\) is surjective -- equivalently,
\(F_n(X)\) is connected -- if and only if at least one of the
following holds:
\begin{enumerate}[label=(\arabic*),leftmargin=2em]
    \item \(n = 1\);
    \item \(X\cong S^1\) and \(n = 2\);
    \item \(\sfS_3\) embeds in \(X\) and \(n\geq 1\).
\end{enumerate}
In each of these cases \(\PB_n(X,\base)\) has index \(n!\) in \(\BG_n(X,[\base])\).
\end{lemma}

\begin{proof}
Since \(\PB_n(X,\base)=\ker\rho\), the index statement is
immediate from the surjectivity: the index equals
\(|\im\rho|=n!\).

Case (1) is trivial, \(\mathbb{S}_1\) being the trivial group.
In case (2), place the two points antipodally by a change of
base configuration. A simultaneous half-turn exchanges them,
so the induced permutation is the transposition and \(\rho\) is
onto \(\mathbb{S}_2\). In case (3), fix an embedding
\(\sfS_3\subseteq X\) and a base configuration
\(\bfy^0\in F_n(\sfS_3)\). By \cref{ex:S3} the permutation
homomorphism of \(\sfS_3\) is surjective, and it factors as the
composition
\[
\BG_n(\sfS_3,[\bfy^0])\longrightarrow\BG_n(X,[\bfy^0])
\xrightarrow{\rho}\mathbb{S}_n
\]
of the inclusion-induced map with \(\rho\); transporting the
basepoint from \([\bfy^0]\) to \([\base]\) conjugates \(\rho\)
by a fixed permutation and does not affect surjectivity. Hence
\(\rho\) is surjective.

Conversely, suppose \(\rho\) is surjective and (3) fails. Since
\(\sfS_3\) does not embed in \(X\), the complex \(X\) embeds in
\(S^1\), as noted above, and is therefore homeomorphic to an
interval or to \(S^1\). If \(X\cong\mathsf{I}\), then
\(\BG_n(\mathsf{I},[\base])\) is trivial -- the unordered
configuration space of an interval is contractible -- so
\(\im\rho=1\) and surjectivity forces \(n=1\), which is
case~(1). If \(X\cong S^1\), then \(\BG_n(S^1,[\base])\) is
infinite cyclic, generated by the simultaneous rotation, whose
permutation is the \(n\)-cycle \((1\,2\,\cdots\,n)\); hence
\(\im\rho\) is the cyclic group generated by this \(n\)-cycle,
and surjectivity forces \(n!=n\), that is, \(n\leq2\) -- which
is case~(1) or case~(2).
\end{proof}

The cases excluded by \cref{lem:rho-surj} are easy to
catalogue: a connected simplicial complex of dimension \(\geq 1\)
that embeds in \(S^1\) is either an interval \(\mathsf{I}\) or the
circle \(S^1\) itself.

\begin{enumerate}
    \item If \(X = \mathsf{I}\), then
    \(\BG_n(\mathsf{I},[\base])\) is trivial for any
    \(\base\in F_n(\mathsf{I})\) -- the unordered configuration
    space of an interval is contractible -- and
    \(\PB_n(\mathsf{I},\base)=1\) as well
    (\cref{subsubsection:simply-connected}). The image of
    \(\rho\) is therefore trivial, and \(\rho\) is a surjection
    onto \(\mathbb{S}_n\) only when \(n = 1\).
    \item If \(X = S^1\), then both \(\PB_n(S^1,\base)\) and \(\BG_n(S^1,[\base])\)
    are infinite cyclic for any \(\base\in F_n(S^1)\); \(\PB_n(S^1,\base)\) sits in \(\BG_n(S^1,[\base])\) as a
    subgroup of index \(n\), and \(\im(\rho)\) is the cyclic
    subgroup of \(\mathbb{S}_n\) of order \(n\) -- consisting precisely
    of the cyclic permutations -- generated by the \(n\)-cycle
    \((1\,2\,\cdots\,n)\). Hence \(\rho\) is a surjection onto
    \(\mathbb{S}_n\) only when \(n\leq 2\) (since \(n! = n\) exactly in
    that range).
\end{enumerate}

\begin{example}[Classical braid group]\label{ex:disc}
For \(X = D^2\), the group \(\BG_n(D^2)\) is the classical \(n\)-strand
braid group of Artin \cite{Artin1947}. Attaching an edge to the
boundary of the disc does not change the braid group: writing
\(D^2\cup \mathsf{e}\) for \(D^2\) with an edge \(\mathsf{e}\) glued
to a boundary point (\cref{fig:disc-boundary-edge}), the inclusion
\(D^2\hookrightarrow D^2\cup \mathsf{e}\) induces an isomorphism
\[
\BG_n(D^2) \cong \BG_n(D^2\cup \mathsf{e}),
\]
since there is an embedding
\(D^2\cup\mathsf{e}\hookrightarrow D^2\), absorbing the edge
into a collar of the boundary, which is inverse to the
inclusion \(D^2\hookrightarrow D^2\cup\mathsf{e}\) up to
isotopy; isotopic embeddings induce the same homomorphism of
braid groups, so the two induced homomorphisms are mutually
inverse isomorphisms. (A deformation retraction
\(D^2\cup\mathsf{e}\to D^2\) alone would not suffice: not
being injective, it induces no map of configuration spaces.)
\end{example}

\begin{example}[Interior edge raises the effective
dimension]\label{ex:disc-interior}
Let \(D^2\cup_p \mathsf{e}\) denote \(D^2\) with a single edge
\(\mathsf{e}\) attached at an \emph{interior} point \(p\) of \(D^2\)
(\cref{fig:disc-interior-edge}). Then
\[
\BG_n(D^2\cup_p \mathsf{e}) \cong \mathbb{S}_n,
\]
and in particular the pure braid group
\(\PB_n(D^2\cup_p \mathsf{e})\) is trivial. Moreover, any inclusion \(D^2\cup_p \mathsf{e}\to D^3\) induces an isomorphism
\[
\BG_n(D^2\cup_p \mathsf{e}) \cong \BG_n(D^3)
\cong \mathbb{S}_n.
\]
Therefore attaching an interior edge effectively promotes the local
dimension to \(3\), which by Birman's Theorem~1~\cite[Thm.~1]{Birman1969}
gives the full braid group as \(\mathbb{S}_n\) and
trivialises the pure braid group.
\end{example}

\begin{figure}[ht]
\centering
\begin{subfigure}[b]{0.48\textwidth}
\centering
\begin{tikzpicture}[line width=0.8pt,scale=1]
  \draw[fill=gray!15] (0,0) ellipse (1.2 and 0.5);
  \node at (0,0) {\(D^2\)};
  \filldraw (1.2,0) circle (1.6pt);
  \draw (1.2,0) -- (2.4,0);
  \filldraw (2.4,0) circle (1.3pt);
  \node at (1.8,0.22) {\(\mathsf{e}\)};
  \path (0,1.35);
\end{tikzpicture}
\caption{The edge attached at a boundary point:
\(D^2\cup\mathsf{e}\).}
\label{fig:disc-boundary-edge}
\end{subfigure}
\hfill
\begin{subfigure}[b]{0.48\textwidth}
\centering
\begin{tikzpicture}[line width=0.8pt,scale=1]
  \draw[fill=gray!15] (0,0) ellipse (1.2 and 0.5);
  \node at (-0.55,-0.08) {\(D^2\)};
  \filldraw (0.35,-0.02) circle (1.6pt);
  \node[below] at (0.35,-0.10) {\(p\)};
  \draw (0.35,-0.02) -- (0.35,1.15);
  \filldraw (0.35,1.15) circle (1.3pt);
  \node at (0.62,0.62) {\(\mathsf{e}\)};
\end{tikzpicture}
\caption{The edge attached at an interior point:
\(D^2\cup_p\mathsf{e}\).}
\label{fig:disc-interior-edge}
\end{subfigure}
\caption{The complexes of \cref{ex:disc,ex:disc-interior}.}
\label{fig:disc-edges}
\end{figure}

\begin{example}[Graphs with one essential vertex]\label{ex:one-essential}
For integers \(\ell, m\geq 0\) with \(2\ell+m\ge 3\), let \(\Gamma_{\ell, m}\) be the graph
with a single essential vertex of valence \(2\ell + m\) carrying
\(\ell\) loops and \(m\) pendant leaves. Then \(\BG_n(\Gamma_{\ell, m})\)
is always a free group; by Ko--Park~\cite[Thm.~3.16]{KoPark2012}
and An--Kim~\cite[Prop.~3.8]{AnKim2026}, its rank is
\[
\mathrm{rank}\,\BG_n(\Gamma_{\ell, m})
=
(2\ell + m - 2)\binom{n + \ell + m - 2}{\ell + m - 1}
- \binom{n + \ell + m - 2}{\ell + m - 2}
+ 1.
\]

In the special case \(\ell = m = 1\) the graph \(\Gamma_{1,1}\) is called the
\emph{lollipop graph} \(\Lollipop\): a single circle \(\mathsf{C}\) with one additional edge
\(\mathsf{e}\) attached at a single vertex (\cref{fig:lollipop}). Substituting
\(\ell = m = 1\) above gives \(1\cdot\binom{n}{1} - \binom{n}{0} + 1= n\), so
\[
\BG_n(\Lollipop) \text{ is free of rank } n.
\]
Applying Schreier's formula to the subgroup
\(\PB_n(\Lollipop)\subseteq\BG_n(\Lollipop)\), whose index is
\(n!\) by the surjection
\(\rho:\BG_n(\Lollipop)\rightarrow \mathbb{S}_n\)
(\cref{lem:rho-surj}), then gives
\[
\PB_n(\Lollipop) \text{ is free of rank }
(n-1)\,n! + 1.
\]
\end{example}

Recall the standing assumption. For a nonsimply connected \(X\), the boundary between \(X\cong S^1\) and \(X\not\cong S^1\) is captured
neatly by the lollipop graph \(\Lollipop\). The ``parking edge'' \(\mathsf{e}\)
already gives enough room for the strand-independence phenomenon
responsible for surjectivity of \(\iota_\ast\), and this is in fact
the local model for every \(X\not\cong S^1\).

\begin{figure}[ht]
\centering
\begin{tikzpicture}[line width=0.8pt,scale=1]
  \draw (0,0) circle (1);
  \filldraw (1,0) circle (1.6pt);
  \draw (1,0) -- (2.6,0);
  \filldraw (2.6,0) circle (1.6pt);
  \node at (-1.3,0) {\(\mathsf{C}\)};
  \node at (1.8,0.28) {\(\mathsf{e}\)};
\end{tikzpicture}
\caption{The lollipop graph
\(\Lollipop = \mathsf{C}\cup \mathsf{e}\): a circle \(\mathsf{C}\)
together with a single edge \(\mathsf{e}\) attached at one vertex.}
\label{fig:lollipop}
\end{figure}

\begin{lemma}[Lollipop embedding]\label{lem:lollipop-embed}
Let \(X\) be a finite connected non-simply-connected simplicial
complex of dimension at least \(1\). Then
\[
X \not\cong S^1 \quad\Longleftrightarrow\quad
\text{there exists an embedding } \Lollipop
\hookrightarrow X
\]
of the lollipop graph
\(\Lollipop = \mathsf{C}\cup \mathsf{e}\) (\cref{fig:lollipop})
into \(X\). Moreover, when these equivalent conditions hold, for
every embedded cycle \(\mathsf{c}\subseteq X\) there
exists an embedding \(\phi\colon \Lollipop \hookrightarrow X\) whose
cycle \(\mathsf{C}\subseteq \Lollipop\) satisfies
\(\phi(\mathsf{C}) = \mathsf{c}\).
\end{lemma}

\begin{proof}
(\(\Leftarrow\)) If \(\Lollipop\) embeds in \(X\) then \(X\) contains a
circle together with a non-loop edge attached to it; in particular
\(X\) properly contains an \(S^1\), so \(X\not\cong S^1\).

(\(\Rightarrow\)) Assume \(X\not\cong S^1\)
and
fix an embedded cycle \(\mathsf{c}\subseteq X\).
We exhibit an embedding \(\phi\colon \Lollipop \hookrightarrow X\)
with \(\phi(\mathsf{C}) = \mathsf{c}\). Since \(X\not\cong S^1\), \(X\)
is connected and \(X\setminus \mathsf{c}\neq\emptyset\); choose a
point \(p\in X\setminus \mathsf{c}\) and a simple path \(\gamma\) in
\(X\) from \(p\) to a vertex \(\sfv\in \mathsf{c}\), where \(\gamma\) meets
\(\mathsf{c}\) only at \(\sfv\). The union \(\mathsf{c}\cup\gamma\) is then
an embedded copy of \(\Lollipop\) with cycle \(\mathsf{c}\) and stick
\(\gamma\), providing the desired embedding \(\phi\).

The ``moreover'' clause is the construction itself: it produces
an embedding sending \(\mathsf{C}\) onto the prescribed \(\mathsf{c}\).
\end{proof}

\subsection{Simple complexes via An--Park}
\label{ssec:simple-an-park}

To determine whether \(X\) enjoys the Goldberg property (\cref{def:goldberg-strengths}), it is convenient to
first replace \(X\) by a simpler model. We record here a uniform
structural reduction due to An--Park: every finite simplicial
complex is braid-equivalent to a simple complex (of dimension at
most \(2\)), and so we may always assume \(X\) is such.

Before stating the An--Park notion of simplicity, we fix some
local terminology. For a vertex \(\sfv\) of a finite simplicial
complex \(X\), write \(\lk_X(\sfv)=\bigcup_{i=1}^k L_i\) for the connected
components \(L_i\) of the link, and call \(k\) the
\emph{valency} of \(\sfv\), denoted \(\val_X(\sfv)\). More
generally, for any point \(x\in X\) (not necessarily a vertex),
the link components and valency of \(x\) are defined by first
subdividing \(X\), if necessary, so that \(x\) becomes a vertex, and
then taking the link components and valency of \(x\) in the
resulting subdivision; the homeomorphism type of the link and of
its components does not depend on the chosen subdivision, so the
definition is unambiguous.

\begin{definition}[{An--Park~\cite[Def.~2.7]{AnPark2017}}]
\label{def:simple-complex}
A vertex \(\sfv\) of a finite simplicial complex \(X\) is \emph{simple}
if its link \(\lk_X(\sfv)\) is one of the following:
\begin{enumerate}[label=(\roman*), leftmargin=2em]
    \item connected;
    \item a disjoint union of a connected complex and a single
          point;
    \item \(0\)-dimensional.
\end{enumerate}
The complex \(X\) is \emph{simple} if every vertex of \(X\) is simple.
\end{definition}

\Cref{fig:simple-vertices} shows a simple vertex for each of
the three cases, together with a non-simple vertex.

\begin{figure}[ht]
\centering
\begin{subfigure}[b]{0.24\textwidth}
\centering
\begin{tikzpicture}[line width=0.8pt,scale=0.85]
  \fill[gray!10] (-1.2,0) .. controls (-0.7,-1.15) and (0.7,-1.15) .. (1.2,0) -- cycle;
  \draw[red] (-1.2,0) .. controls (-0.7,-1.15) and (0.7,-1.15) .. (1.2,0);
  \fill[gray!20] (-1.2,0) .. controls (-0.7,-0.55) and (0.7,-0.55) .. (1.2,0) -- cycle;
  \draw[red] (-1.2,0) .. controls (-0.7,-0.55) and (0.7,-0.55) .. (1.2,0);
  \fill[gray!15] (-1.2,0) .. controls (-0.7,1.1) and (0.7,1.1) .. (1.2,0) -- cycle;
  \draw[red] (-1.2,0) .. controls (-0.7,1.1) and (0.7,1.1) .. (1.2,0);
  \draw[line width=1pt] (-1.2,0) -- (1.2,0);
  \filldraw[red] (-1.2,0) circle (1.3pt);
  \filldraw[red] (1.2,0) circle (1.3pt);
  \filldraw (0,0) circle (1.6pt);
  \node[below right=-2pt] at (0.02,-0.02) {\(\sfv\)};
  \path (0,-1.35) (0,1.35);
\end{tikzpicture}
\caption{\(\lk_X(\sfv)\cong\mathsf{K}_{2,3}\)}
\label{fig:simple-theta}
\end{subfigure}
\hfill
\begin{subfigure}[b]{0.24\textwidth}
\centering
\begin{tikzpicture}[line width=0.8pt,scale=0.85]
  \draw[color=red,fill=gray!15] (0,0) ellipse (1.2 and 0.5);
  \filldraw (0,0) circle (1.6pt);
  \node[below right=-1pt] at (0.04,-0.04) {\(\sfv\)};
  \draw (0,0) -- (0,1.1);
  \filldraw[red] (0,1.1) circle (1.3pt);
  \path (0,-1.35) (0,1.35);
\end{tikzpicture}
\caption{\(\lk_X(\sfv)\cong S^1\sqcup\{\mathrm{pt}\}\)}
\label{fig:simple-pendant}
\end{subfigure}
\hfill
\begin{subfigure}[b]{0.24\textwidth}
\centering
\begin{tikzpicture}[line width=0.8pt,scale=0.85]
  \draw (0,0) -- (0,1.2);
  \draw (0,0) -- (-1.05,-0.6);
  \draw (0,0) -- (1.05,-0.6);
  \filldraw[red] (0,1.2) circle (1.3pt);
  \filldraw[red] (-1.05,-0.6) circle (1.3pt);
  \filldraw[red] (1.05,-0.6) circle (1.3pt);
  \filldraw (0,0) circle (1.6pt);
  \node[below=2pt] at (0,0) {\(\sfv\)};
  \path (0,-1.35) (0,1.35);
\end{tikzpicture}
\caption{\(\lk_X(\sfv)=\{\text{3pts}\}\)}
\label{fig:simple-graph}
\end{subfigure}
\hfill
\begin{subfigure}[b]{0.24\textwidth}
\centering
\begin{tikzpicture}[line width=0.8pt,scale=0.85]
  \draw (0,0) -- (0,-1.05);
  \filldraw[red] (0,-1.05) circle (1.3pt);
  \fill[fill=gray!30,opacity=0.5] (0,0) ellipse (1.2 and 0.5);
  \draw[red] (0,0) ellipse (1.2 and 0.5);
\fill[fill=gray!30,opacity=0.5] (0,0.6) circle (0.6);
\draw (0.6,0.6) arc (0:-180:0.6);
\draw[red] (0.6,0.6) arc (0:180:0.6);
  \filldraw (0,0) circle (1.6pt);
  \node[right=2pt] at (0.02,-0.12) {\(\sfv\)};
  \path (0,-1.35) (0,1.35);
\end{tikzpicture}
\caption{\(\lk_X(\sfv)\cong S^1\sqcup\mathsf{I}\sqcup\{\mathrm{pt}\}\).}
\label{fig:non-simple-vertex}
\end{subfigure}
\caption{(\subref{fig:simple-theta})--(\subref{fig:simple-graph})
Simple vertices, one for each case of
\cref{def:simple-complex}.
(\subref{fig:non-simple-vertex}) 
A non-simple vertex.
In each picture the link
\(\lk_X(\sfv)\) is drawn in red.}
\label{fig:simple-vertices}
\end{figure}

\begin{definition}[{An--Park~\cite[Def.~2.8]{AnPark2017}}]
\label{def:braid-equiv}
An embedding \(f\colon X\hookrightarrow Y\) is a \emph{braid
equivalence} if the induced map
\(f_\ast\colon\BG_n(X)\to\BG_n(Y)\) is an isomorphism for every
\(n\geq 1\). Two complexes \(X, Y\) are \emph{braid equivalent},
written \(X\equiv_B Y\), if they are connected by a zigzag of braid
equivalences.
\end{definition}

\begin{theorem}[{An--Park~\cite[Thm.~1.1]{AnPark2017}}]
\label{thm:an-park}
For any finite simplicial complex \(X\) there is a simple complex
\(X'\) with \(X\equiv_B X'\); in
particular \(\BG_n(X)\cong \BG_n(X')\) for every \(n\geq 1\).
\end{theorem}

For later reference we record an explicit recipe that produces a
simple model \(X'\) from a given \(X\).

\subsubsection{Resolving non-simple vertices.}
\label{sssec:general-resolution}
For a non-simple vertex \(\sfv\) of \(X\), let \(L\) be a
connected component of \(\lk_X(\sfv)\) of dimension at least
\(1\). Separate \(\operatorname{Cone}(L)\) from \(\sfv\),
giving its apex a new name \(\sfv'\), and join \(\sfv\) to
\(\sfv'\) by a new edge \(\sfe\). Denote the resulting complex
by \(X'\), so that \(X'/\sfe=X\). The new apex is simple,
with link \(L\sqcup\{*\}\).

Here is the thickening argument for a general \(L\), extending
the graph-link construction of
An--Park~\cite[Prop.~3.12]{AnPark2017}.
If \(L\) is a ball, the disk-link embedding of
\cite[Prop.~3.11]{AnPark2017} applies directly. Otherwise
thicken the \(L\)-side near \(\sfv'\), without touching the
interior of \(\sfe\), by the cell attachments described in
\cref{prop:reduce-higher-cells,prop:reduce-2-cell} and
illustrated in \cref{fig:branched-thickening}. This enlarges
\(L\) to a ball \(D^k\) and gives a braid equivalence
\(X'\hookrightarrow\widehat X'\), with
\(\lk_{\widehat X'}(\sfv')=D^k\sqcup\{*\}\).
The branch-point condition is checked in the ambient complex:
the extra edge at \(\sfv'\) supplies a direction outside the
cone. On the collapsed side, the other link components at
\(\sfv\) supply the corresponding extra directions.
In particular, when \(L=S^1\), the relevant branch point is
in the ambient space, not in the disc
\(\operatorname{Cone}(L)\) alone. The non-simple hypothesis
ensures that these directions remain after separating \(L\);
the remaining endpoint is not a boundary point.

Make the same cell attachments on the collapsed side. They
give \(X\hookrightarrow\widehat X'/\sfe\), again a braid
equivalence. The disk-link embedding now gives the middle
arrow in
\[
X'\lhook\joinrel\longrightarrow\widehat X'
\lhook\joinrel\longrightarrow\widehat X'/\sfe
\lhook\joinrel\longleftarrow X.
\]
All three arrows are braid equivalences. Locally, the middle
embedding places the resolved cone and its extra edge inside
the thickened cone; it has the effect of the collapse on
homotopy classes. As in
\cite[Props.~3.1,~3.11]{AnPark2017}, its induced map is inverse
to the collapse-lifting map \(q_\sfe^*\).
Since the two thickenings use the same attachments away from
\(\mathring\sfe\), they commute with this lifting construction.
Consequently the original
\(q_\sfe^*\colon\PB_n(X)\to\PB_n(X')\) is an isomorphism,
compatible with the strand maps.

Iterating over all non-simple vertices yields a simple complex
\(X'\) braid-equivalent to \(X\) (uniqueness up to
homeomorphism is the content of
\cref{prop:simple-model-unique} below).

In particular, \(X\) is recovered from \(X'\) by collapsing each
newly created edge \(\sfe\) to a point: \(X = X'/\sfe\). See
\cref{fig:resolve-nonsimple} for the complete resolution of the
non-simple vertex of \cref{fig:non-simple-vertex}, which
requires two steps, one for each positive-dimensional link
component.

\begin{figure}[ht]
\centering
\[
\begin{tikzcd}[row sep=-2pc]
& 
\begin{tikzpicture}[baseline=-.5ex,line width=0.8pt,scale=0.85]
  \draw (0,0) -- (0,-1.05);
  \filldraw[red] (0,-1.05) circle (1.3pt);
  \draw (0,0) -- node[pos=0.4, below, sloped] {\(\sfe_2\)} (0:1.2);
\begin{scope}[rotate=0]
\begin{scope}[yshift=0cm]
\fill[fill=gray!30,opacity=0.5] (0,0.6) circle (0.6);
\draw (0.6,0.6) arc (0:-180:0.6);
\draw[red] (0.6,0.6) arc (0:180:0.6);
\filldraw (0,0) circle (1.6pt);
  \filldraw[red] (0.6,0.6) circle (1.6pt) (-0.6,0.6) circle (1.6pt);
\end{scope}
\end{scope}
\begin{scope}[rotate=-90]
\begin{scope}[yshift=1.2cm]
  \fill[fill=gray!30,opacity=0.5] (0,0) ellipse (1.2 and 0.5);
  \draw[red] (0,0) ellipse (1.2 and 0.5);
\filldraw (0,0) circle (1.6pt) node[above] {\(\sfv_2'\)};
\end{scope}
\end{scope}
  \filldraw (0,0) circle (1.6pt);
  \node[left=2pt] at (0.02,-0.12) {\(\sfv\)};
  \path (0,-1.35) (0,1.35);
\end{tikzpicture}
\ar[rd,"q_{\sfe_2}"]\\
\bar X=\begin{tikzpicture}[baseline=-.5ex,line width=0.8pt,scale=0.85]
  \draw (0,0) -- (0,-1.05);
  \filldraw[red] (0,-1.05) circle (1.3pt);
  \draw (0,0) -- node[midway, above, sloped] {\(\sfe_2\)} (30:1.05) (0,0) -- node[midway, above, sloped] {\(\sfe_1\)} (150:1.05);
\begin{scope}[rotate=60]
\begin{scope}[yshift=1.05cm]
\fill[fill=gray!30,opacity=0.5] (0,0.6) circle (0.6);
\draw (0.6,0.6) arc (0:-180:0.6);
\draw[red] (0.6,0.6) arc (0:180:0.6);
\filldraw (0,0) circle (1.6pt) node[left] {\(\sfv_1'\)};
  \filldraw[red] (0.6,0.6) circle (1.6pt) (-0.6,0.6) circle (1.6pt);
\end{scope}
\end{scope}
\begin{scope}[rotate=-60]
\begin{scope}[yshift=1.05cm]
  \fill[fill=gray!30,opacity=0.5] (0,0) ellipse (1.2 and 0.5);
  \draw[red] (0,0) ellipse (1.2 and 0.5);
\filldraw (0,0) circle (1.6pt) node[above] {\(\sfv_2'\)};
\end{scope}
\end{scope}
  \filldraw (0,0) circle (1.6pt);
  \node[right=2pt] at (0.02,-0.12) {\(\sfv\)};
  \path (0,-1.35) (0,1.35);
\end{tikzpicture}
\ar[ru,"q_{\sfe_1}"] \ar[rd, "q_{\sfe_2}"]
& & 
\begin{tikzpicture}[baseline=-.5ex,line width=0.8pt,scale=0.85]
  \draw (0,0) -- (0,-1.05);
  \filldraw[red] (0,-1.05) circle (1.3pt);
  \fill[fill=gray!30,opacity=0.5] (0,0) ellipse (1.2 and 0.5);
  \draw[red] (0,0) ellipse (1.2 and 0.5);
\fill[fill=gray!30,opacity=0.5] (0,0.6) circle (0.6);
\draw (0.6,0.6) arc (0:-180:0.6);
\draw[red] (0.6,0.6) arc (0:180:0.6);
  \filldraw[red] (0.6,0.6) circle (1.6pt) (-0.6,0.6) circle (1.6pt);
  \filldraw (0,0) circle (1.6pt);
  \node[right=2pt] at (0.02,-0.12) {\(\sfv\)};
  \path (0,-1.35) (0,1.35);
\end{tikzpicture}=X\\
& 
\begin{tikzpicture}[baseline=-.5ex,line width=0.8pt,scale=0.85]
  \draw (0,0) -- (0,-1.05);
  \filldraw[red] (0,-1.05) circle (1.3pt);
\begin{scope}[rotate=0]
\begin{scope}[yshift=0cm]
  \fill[fill=gray!30,opacity=0.5] (0,0) ellipse (1.2 and 0.5);
  \draw[red] (0,0) ellipse (1.2 and 0.5);
\filldraw (0,0) circle (1.6pt);
\end{scope}
\end{scope}
  \draw (0,0) -- node[near end, above, sloped] {\(\sfe_1\)} (90:1.05);
\begin{scope}[rotate=0]
\begin{scope}[yshift=1.05cm]
\fill[fill=gray!30,opacity=0.5] (0,0.6) circle (0.6);
\draw (0.6,0.6) arc (0:-180:0.6);
\draw[red] (0.6,0.6) arc (0:180:0.6);
  \filldraw[red] (0.6,0.6) circle (1.6pt) (-0.6,0.6) circle (1.6pt);
\filldraw (0,0) circle (1.6pt) node[above] {\(\sfv_1'\)};
\end{scope}
\end{scope}
  \filldraw (0,0) circle (1.6pt);
  \node[right=2pt] at (0.02,-0.12) {\(\sfv\)};
  \path (0,-1.35) (0,1.35);
\end{tikzpicture}
\ar[ru,"q_{\sfe_1}"]
\end{tikzcd}
\]
\caption{Resolving the non-simple vertex \(\sfv\) of the complex
\(X\) of \cref{fig:non-simple-vertex}.
}
\label{fig:resolve-nonsimple}
\end{figure}

\begin{proposition}\label{prop:simple-model-unique}
For any finite connected simplicial complex \(X\), resolving all
non-simple vertices as above yields a simple complex, denoted
\(\bar X\), that is unique up to homeomorphism.
\end{proposition}

\begin{proof}
For a vertex \(\sfv\), let \(r(\sfv)\) be the number of
positive-dimensional link components and \(s(\sfv)\) the number
of isolated link points. Writing \(r=r(\sfv)\) and
\(s=s(\sfv)\), define
\[
c(\sfv)\coloneqq
\begin{cases}
0,&r=0\text{ or }(r=1\text{ and }s\leq1),\\
1,&(r,s)=(2,0),\\
r,&\text{otherwise},
\end{cases}
\qquad
c(X)\coloneqq\sum_{\sfv\in X^{(0)}}c(\sfv).
\]
The zero-weight cases are exactly the simple vertices of
\cref{def:simple-complex}. Resolving a non-simple vertex changes
its pair \((r,s)\) to \((r-1,s+1)\); the new apex has connected
link together with one isolated point and is simple. If
\((r,s)=(2,0)\), the remaining vertex is simple after one step.
In every other non-simple case \(r+s\geq3\), so each resolution
reduces the weight by one. Thus \(c(X)\) decreases by exactly
one and the procedure terminates.

The local outcome is independent of the choices. At a vertex
with \(r+s\geq3\), every positive-dimensional link component
must be separated: any remaining one, together with the
isolated directions, would still give a non-simple vertex.
The result is therefore a central vertex with one new edge to
each separated cone and with the original isolated directions
unchanged. When \((r,s)=(2,0)\), either choice produces the same
two cones joined by a single edge, up to homeomorphism.
Resolutions at distinct original vertices can be carried out
in disjoint small neighbourhoods. These local descriptions
therefore determine the resulting simple complex up to
homeomorphism.
\end{proof}

\subsubsection{Reducing higher-dimensional cells.}
The resolution above lowers the link complexity of \(X\) but not
its dimension. A second, independent reduction lowers the
dimension; it rests on the following result of An--Park, to the
effect that capping off a sphere by a cell of dimension at least
three leaves the braid groups unchanged.

\begin{proposition}[{An--Park~\cite[Prop.~3.6]{AnPark2017}}]
\label{prop:reduce-higher-cells}
Let \(k\geq 3\), and let \(Y\) be a finite simplicial complex with
\(Y\not\cong S^2\). If
\[
X = Y\cup_\varphi D^k
\]
is obtained from \(Y\) by attaching a \(k\)-cell along an embedding
\(\varphi\colon\partial D^k = S^{k-1}\hookrightarrow Y\), then the
inclusion \(Y\hookrightarrow X\) induces an isomorphism \(\BG_n(Y)\cong\BG_n(X)\) for every \(n\geq 1\).
\end{proposition}

The hypothesis \(Y\not\cong S^2\) excludes the ordinary capping-off
of a \(2\)-sphere by a \(3\)-ball, which is genuinely exceptional:
there the braid groups do change, \(\BG_n(S^2)\) being infinite for
\(n\geq 4\) while \(\BG_n(D^3)\) is finite.

Read in the reverse direction,
\cref{prop:reduce-higher-cells} says that a cell of
dimension \(\geq 3\) whose attaching sphere is embedded in the rest
of the complex may be deleted without affecting the braid groups.
Deleting such cells one at a time, from the top dimension
downwards, reduces \(X\) to its \(2\)-skeleton.

\begin{remark}\label{rem:reduce-to-dim-2}
If we additionally want a simple complex of dimension at most
\(2\) with the same braid groups as the given \(X\), we may subdivide
\(X\) if necessary so that its \(2\)-skeleton is not homeomorphic to
\(S^2\), and then remove every cell of dimension \(\geq 3\). By
iterating \cref{prop:reduce-higher-cells}, this
induces an isomorphism \(\BG_n(X)\cong\BG_n(X^{(2)})\) for every
\(n\), replacing \(X\) by its \(2\)-skeleton \(X^{(2)}\)
(An--Park~\cite[Cor.~3.8]{AnPark2017}).

Note, however, that taking the \(2\)-skeleton depends on the
chosen simplicial structure on \(X\): even if we start from two
homeomorphic complexes \(X_1\) and \(X_2\), their respective simple
models of dimension at most \(2\), \(X_1'\) and \(X_2'\), need not be
homeomorphic to each other. The uniqueness statement of
\cref{prop:simple-model-unique} therefore breaks down
once one further restricts to simple models of dimension
\(\leq 2\).
\end{remark}

\subsubsection{Reducing \(2\)-cells.}
A \(2\)-cell can sometimes be removed as well, under a condition on
the \emph{branch set} \(\operatorname{br}(Y)\) of \(Y\) -- the set of
points \(y\in Y\) whose link is homeomorphic to neither a sphere
\(S^k\) nor a disc \(D^k\), equivalently the points at which \(Y\)
fails to be locally a manifold (with or without boundary); thus
\(\operatorname{br}(Y)=\emptyset\) exactly when \(Y\) is a manifold.

\begin{proposition}[{An--Park~\cite[Cor.~3.9]{AnPark2017}}]
\label{prop:reduce-2-cell}
Let \(X=Y\cup_\varphi D^2\) be obtained from a finite simplicial
complex \(Y\) by attaching a \(2\)-cell along an embedding
\(\varphi\colon\partial D^2\hookrightarrow Y\). Suppose \(\varphi(\partial D^2)\)
bounds a disc \(D\subseteq Y\) whose interior meets the
branch set, \(\mathring D\cap\operatorname{br}(Y)\neq\emptyset\). Then
the inclusion \(Y\hookrightarrow X\) induces an isomorphism \(\BG_n(Y)\cong\BG_n(X)\) for every \(n\geq 1\).
\end{proposition}

\cref{fig:reduce-2-cell} illustrates
\cref{prop:reduce-2-cell}: capping the disc of
\(Y=D^2\cup_p\mathsf{e}\) along its boundary circle produces
\(X=S^2\cup_p\mathsf{e}\). The disc bounded by the attaching
circle contains the branch point \(p\) in its interior, so the
inclusion is a braid equivalence; indeed
\(\BG_n(S^2\cup_p\mathsf{e})\cong
\BG_n(D^2\cup_p\mathsf{e})\cong\mathbb{S}_n\) by
\cref{ex:disc-interior}.

\begin{figure}[ht]
\centering
\[
\begin{tikzcd}[column sep=4pc]
Y=\begin{tikzpicture}[baseline=-.5ex,line width=0.8pt]
  \fill[gray!30,opacity=0.5] (0,0) ellipse (1.2 and 0.5);
  \draw[red] (0,0) ellipse (1.2 and 0.5);
  \draw (0,0) -- (0,1.1);
  \filldraw (0,1.1) circle (1.3pt);
  \node[right=1pt] at (0.03,0.65) {\(\mathsf{e}\)};
  \filldraw (0,0) circle (1.6pt);
  \node[below right=-1pt] at (0.04,-0.06) {\(p\)};
\end{tikzpicture}\ar[r,"\cup_\varphi D^2"]&
X=\begin{tikzpicture}[baseline=-.5ex,line width=0.8pt]
  \fill[gray!30,opacity=0.5] (0,0) circle (1.0);
  \draw (0,0) circle (1.0);
  \draw[red] (-1,0) arc (180:360:1 and 0.3);
  \draw[red, dashed] (-1,0) arc (180:0:1 and 0.3);
  \draw (0,1.0) -- (0,1.75);
  \filldraw (0,1.75) circle (1.3pt);
  \node[right=1pt] at (0.03,1.42) {\(\mathsf{e}\)};
  \filldraw (0,1.0) circle (1.6pt);
  \node[below right=-1pt] at (0.05,0.98) {\(p\)};
\end{tikzpicture}
\end{tikzcd}
\]
\caption{An instance of \cref{prop:reduce-2-cell}: attaching a
\(2\)-cell to \(Y=D^2\cup_p\mathsf{e}\) along the boundary
circle of the disc (in red; on the right it becomes the equator
of the sphere) yields \(X=S^2\cup_p\mathsf{e}\). The disc
bounded by the attaching circle contains the branch point \(p\)
in its interior, so \(Y\hookrightarrow X\) is a braid
equivalence.}
\label{fig:reduce-2-cell}
\end{figure}

\cref{prop:reduce-higher-cells,prop:reduce-2-cell} together allow a complex to be
thickened to a manifold without changing its braid groups. We
use this to compute the braid groups of a natural class of
non-manifold complexes.

\begin{proposition}\label{prop:branched-surface}
Let \(X\) be a simple complex in which every vertex has valency
\(1\). If \(X\) is not a \(2\)-manifold, then for every \(n\geq 1\) the
inclusion \(F_n X\hookrightarrow X^n\) induces an isomorphism
\[
\PB_n(X,\base)\xrightarrow{\ \cong\ }
\prod_{i=1}^n\pi_1(X,x_i^0).
\]
\end{proposition}

\begin{proof}
The one-dimensional case is an interval and is immediate.
Otherwise, by \cref{rem:reduce-to-dim-2}, replace \(X\) for this
computation by the \(2\)-skeleton \(Y\) of a sufficiently fine
subdivision. The inclusion \(Y\hookrightarrow X\) is a braid
equivalence and induces an isomorphism on \(\pi_1\), compatibly
with the strand maps. Vertex links in \(Y\) remain connected;
if \(\dim X\geq3\), take the subdivision fine enough that \(Y\)
is not a \(2\)-manifold. Thus \(Y\) is a \(2\)-dimensional
complex with connected links which is not a surface.

We recall the cell-by-cell thickening used by
An--Park~\cite[discussion preceding Lemma~4.4]{AnPark2017}.
There is an edge of \(Y\) incident to at least three triangles:
otherwise every connected vertex link would be an interval
or a circle, making \(Y\) a surface. Choose an interior point
\(v\) of such an edge. Two of its pages give an embedded disc
\(D\) with
\(\ v\in\mathring D\cap\operatorname{br}(Y)\).
Attach a new disc \(E\) along \(\partial D\), then fill the
sphere \(D\cup E\) with a \(3\)-cell \(B\).
The first attachment is a braid equivalence by
\cref{prop:reduce-2-cell}, and the second by
\cref{prop:reduce-higher-cells}; together they thicken \(D\)
to a ball which retracts onto \(D\).
See \cref{fig:branched-thickening}.

Continue along the adjacent edges: a boundary patch of the
already thickened region, together with a strip of the next
page, provides the next disc. Its interior crosses the
branch locus where the page meets the thickened region, so
the same cap-and-fill move applies. Iterating over the
\(1\)-skeleton and then the \(2\)-cells gives the
regular-neighbourhood thickening of the cited construction.
Every \(2\)-cell attachment uses a disc meeting the current
branch set in its interior; the remaining attachments have
dimension at least \(3\) and are covered by
\cref{prop:reduce-higher-cells}. The ambient dimension is not
restricted to \(3\).

The result is a closed regular neighbourhood \(N(Y)\) in a
PL manifold of dimension at least \(3\), since the three-page
germ cannot embed in a surface. The inclusion
\(Y\hookrightarrow N(Y)\) is a braid equivalence and \(N(Y)\)
deformation retracts onto \(Y\). General position with respect
to the pairwise diagonals, whose codimension is at least \(3\),
gives the manifold strand-map isomorphism
\cite[Thm.~1]{Birman1969}. Hence
\[
\PB_n(Y)\cong\PB_n(N(Y))
\cong\prod_i\pi_1(N(Y))
\cong\prod_i\pi_1(Y).
\]
These identifications commute with the strand maps.
The comparison \(Y\hookrightarrow X\) therefore gives the
asserted isomorphism for \(X\).
\end{proof}

\begin{figure}[t]
\centering
\begin{subfigure}[t]{0.30\textwidth}
\centering
\begin{tikzpicture}[x=1.12cm,y=1.12cm,line width=.65pt,
  font=\footnotesize]
  \path[use as bounding box] (-1.38,-1.12) rectangle (1.38,1.15);
  \filldraw[fill=gray!18,draw=gray!65]
    (-.86,0)--(-1.08,.94)--(.64,.94)--(.86,0)--cycle;
  \filldraw[fill=blue!12,draw=blue!65!black]
    (0,0) ellipse (1.12 and .34);
  \draw[red!75!black,line width=1.15pt] (-.86,0)--(.86,0);
  \fill[red!75!black] (0,0) circle (1.5pt);
  \node[below=3pt] at (0,0) {\(v\)};
  \node[blue!65!black] at (.72,-.18) {\(D\)};
  \node[gray!75!black] at (-.14,.70) {third page};
  \node[red!75!black] at (0,-.69) {branch locus};
\end{tikzpicture}
\caption{A disc through a branch point.}
\end{subfigure}
\hfill
\begin{subfigure}[t]{0.30\textwidth}
\centering
\begin{tikzpicture}[x=1.12cm,y=1.12cm,line width=.65pt,
  font=\footnotesize]
  \path[use as bounding box] (-1.38,-1.12) rectangle (1.38,1.15);
  \filldraw[fill=gray!18,draw=gray!65]
    (-.86,0)--(-1.08,.94)--(.64,.94)--(.86,0)--cycle;
  \path[fill=green!13] (-1.12,0)
    .. controls (-1.12,-1.20) and (1.12,-1.20) .. (1.12,0)
    arc[start angle=0,end angle=180,x radius=1.12,y radius=.34]
    --cycle;
  \draw[green!45!black] (-1.12,0)
    .. controls (-1.12,-1.20) and (1.12,-1.20) .. (1.12,0);
  \filldraw[fill=blue!12,draw=blue!65!black]
    (0,0) ellipse (1.12 and .34);
  \draw[red!75!black,line width=1.15pt] (-.86,0)--(.86,0);
  \fill[red!75!black] (0,0) circle (1.5pt);
  \node[below=3pt] at (0,0) {\(v\)};
  \node[blue!65!black] at (.72,-.18) {\(D\)};
  \node[green!35!black] at (0,-.57) {\(B\)};
  \node[green!35!black] at (.87,-.77) {\(E\)};
\end{tikzpicture}
\caption{Cap \(\partial D\), then fill \(D\cup E\).}
\end{subfigure}
\hfill
\begin{subfigure}[t]{0.30\textwidth}
\centering
\begin{tikzpicture}[x=1.12cm,y=1.12cm,line width=.65pt,
  font=\footnotesize]
  \path[use as bounding box] (-1.38,-1.12) rectangle (1.52,1.15);
  \filldraw[fill=gray!18,draw=gray!65]
    (-.03,-.43)--(1.35,-.13)--(1.35,.69)--(-.03,.43)
    .. controls (.19,.17) and (.19,-.17) .. cycle;
  \filldraw[fill=green!13,draw=green!45!black]
    (-.62,0) ellipse (.78 and .85);
  \draw[green!35!black,dashed] (-.62,-.85)
    .. controls (-.18,-.45) and (-.18,.45) .. (-.62,.85);
  \node[green!35!black] at (-.84,.38) {\(B\)};
  \filldraw[fill=blue!15,draw=blue!65!black]
    (.31,.02) ellipse (.71 and .23);
  \draw[red!75!black,line width=1.15pt] (-.03,-.43)
    .. controls (.19,-.17) and (.19,.17) .. (-.03,.43);
  \fill[red!75!black] (.135,0) circle (1.5pt);
  \draw[gray!75!black,dashed] (.50,.14)--(1.28,.31);
  \node[above=1pt] at (1.13,.32) {\(e\)};
  \node[blue!65!black] at (.61,-.40) {\(D_e\)};
  \node[red!75!black] at (.45,-.84)
    {\(\mathring D_e\cap\operatorname{br}\ne\varnothing\)};
\end{tikzpicture}
\caption{Repeat along an adjacent edge.}
\end{subfigure}
\caption{The local thickening move in
\cref{prop:branched-surface}. Blue discs meet the red branch
locus in their interiors. Attaching the cap \(E\) and filling
with \(B\) preserves braid groups. A patch of \(\partial B\)
and the next page supply the next disc \(D_e\).
The drawing shows a local \(3\)-cell attachment; the full
regular neighbourhood may have higher dimension.}
\label{fig:branched-thickening}
\end{figure}

\begin{remark}\label{rem:branched-surface-pendant}
The conclusion of \cref{prop:branched-surface} persists when
pendant free edges are suitably attached to \(X\) -- the
thickening \(N(X')\) of the enlarged complex \(X'\) is then
still a manifold of dimension at least \(3\). See
\cref{cor:thick-with-pendants} for a precise statement, with a
proof via \cref{lem:pendant-edge}.
\end{remark}

\subsection{Free parts and (weakly) admissible trees}
\label{ssec:free-parts}

Recall the informal description from
\cref{ssec:main-results}; we now record the formal
definition we will use throughout the remainder of the paper.

\begin{definition}[Free edge and free part]
\label{def:free-part}
Let \(X\) be a finite simplicial complex.
\begin{itemize}[leftmargin=2em]
    \item A \emph{free edge} of \(X\) is a \(1\)-simplex
          contained in no simplex of dimension at least \(2\).
    \item A \emph{joint} of \(X\) is a vertex \(\sfv\in X\) of valency at least \(2\), whose link is of dimension at least \(1\).
\end{itemize}
The \emph{free part} of \(X\) is the subspace
\[
\mathsf{F}_X \coloneqq
\Bigl(\,\bigcup_{\sfe\text{ free edge of } X} \sfe\,\Bigr)
\cup
\Bigl(\,\bigcup_{\sfv\text{ joint of }X} \{\sfv\}\,\Bigr)
\subseteq X,
\]
viewed as an at most one-dimensional subcomplex of \(X\). Its connected
components are themselves finite graphs.

Let \(J_X\) be the set of joints, and put
\[
V_X\coloneqq J_X\cup
\{\,v\in\mathsf F_X:\deg_{\mathsf F_X}(v)\neq2\,\}.
\]
A \emph{maximal free chain} is the closure of an interval
component of \(\mathsf F_X\setminus V_X\), with its two ends
recorded; the ends may coincide, in which case the chain is a
loop. Equivalently, one suppresses the bivalent vertices of
\(\mathsf F_X\) that are not joints. A maximal free chain is
\emph{pendant} if one end has valency \(1\) in \(X\).
Subdividing a free edge does not change these chains. A circle
component containing no point of \(V_X\) is treated separately;
for connected \(X\) this is precisely the case \(X\cong S^1\).
Free edges refer to simplicial cells, whereas maximal free
chains refer to the entire thin arcs between joints, branch
points or free ends. An attached pendant interval may be
subdivided into several free edges.
\end{definition}

\begin{remark}\label{rem:joint-redundant-simple}
If \(X\) is simple (\cref{def:simple-complex}), then
every joint of \(X\) is automatically the endpoint of some free
edge of \(X\), so the joint contribution to \(\mathsf{F}_X\) is
redundant: in this case
\[
\mathsf{F}_X = \bigcup_{\sfe\text{ free edge of } X} \sfe.
\]
The general definition of \(\mathsf{F}_X\) is therefore needed only
for non-simple \(X\).
\end{remark}

\begin{definition}[Thick components and unwrapped thick components]
\label{def:free-components}
The set \(\{\Sigma_1,\dots,\Sigma_m\}\) of \emph{thick components} of \(X\)
consists of the subcomplexes \(\Sigma_i\) of \(X\) obtained by taking
the closure (in \(X\)) of each connected component of the
complement \(X\setminus\mathsf{F}_X\). Each thick component is a
subcomplex of dimension at least \(2\) containing no free edges.

For a thick component \(\Sigma\) and a joint \(\sfv\in\Sigma\), write
\(\val_\Sigma(\sfv)\) for the number of connected components
\(L\subseteq\lk_X(\sfv)\cap\Sigma\) of the link of \(\sfv\) lying in
\(\Sigma\). The \emph{unwrapped thick component}
\(\bar{\Sigma}\) is the simplicial complex obtained from
\(\Sigma\) by separating, at every joint \(\sfv\in\Sigma\), the
\(\val_\Sigma(\sfv)\) link components meeting at \(\sfv\): one introduces
a distinct copy \(\sfv^{(L)}\) of \(\sfv\) for each component
\(L\subseteq\lk_X(\sfv)\cap\Sigma\), and attaches the local cone over
\(L\) to \(\sfv^{(L)}\) rather than to \(\sfv\). The original \(\Sigma\)
is recovered by identifying, for each joint \(\sfv\), all the copies
\(\sfv^{(L)}\) back to \(\sfv\); this identification is a canonical
quotient map
\[
q_\Sigma\colon\bar{\Sigma}\twoheadrightarrow\Sigma,
\]
which is a homeomorphism away from the joints of \(\Sigma\) and is
\(\val_\Sigma(\sfv)\)-to-one over each joint \(\sfv\in\Sigma\). When \(X\)
is simple, every joint \(\sfv\in\Sigma\) has \(\val_\Sigma(\sfv)=1\) --
the link of a simple joint is of the form ``connected
\(\sqcup\) point'' (\cref{def:simple-complex}(ii)), and
only the connected piece lies in \(\Sigma\) -- so \(q_\Sigma\) is a
homeomorphism and \(\bar{\Sigma}\cong\Sigma\).
\end{definition}

\begin{example}\label{ex:free-part-unwrapped}
Consider the complex \(X\) shown in the centre of
\cref{fig:free-part-example}. It has three thick components:
\(\Sigma_1\), a sphere with two points identified at \(\sfw\) --
so that \(\lk_X(\sfw)=S^1\sqcup S^1\sqcup\{\mathrm{pt}\}\), the isolated point corresponding to \(\sff_5\), and \(\sfw\) is a
joint -- carrying one further joint \(\sfs\); a disc \(\Sigma_2\)
with an interior joint \(u\) and boundary joints \(\sfb_1,\sfb_2\);
and a disc \(\Sigma_3\) with boundary joints \(\sfc_1,\sfc_2\). The
free part \(\mathsf{F}_X\) -- drawn in blue, with the joints in
red -- consists of five free edges: \(\sff_1\), joining \(\sfs\) to
\(\sfb_1\); \(\sff_2\), joining \(\sfu\) to \(\sfc_1\); and \(f_3\),
\(\sff_4\), \(\sff_5\), joining \(\sfc_2\), \(\sfb_2\) and \(\sfw\),
respectively, to a common trivalent vertex. This last vertex
belongs to \(\mathsf{F}_X\) but is \emph{not} a joint: its link
in \(X\) is \(0\)-dimensional, consisting of three points.
Around \(X\) are its three unwrapped thick components:
\(\bar{\Sigma}_1\) is a genuine sphere with the three
leaves \(\sfw^{(1)},\sfw^{(2)},s\), and \(q_{\Sigma_1}\) identifies
\(\sfw^{(1)}\) and \(\sfw^{(2)}\) to the joint \(\sfw\), of valency
\(\val_{\Sigma_1}(\sfw)=2\), while \(q_{\Sigma_2}\) and
\(q_{\Sigma_3}\) are homeomorphisms.
\end{example}

\begin{figure}[ht]
\centering
\begin{tikzcd}[column sep=1.5pc, row sep=-3pc]
\bar{\Sigma}_2=\begin{tikzpicture}[line width=0.8pt, scale=0.62,
                    baseline=(current bounding box.center)]
  \fill[gray!30,opacity=0.5] (0,0) ellipse (1.5 and 0.65);
  \draw (0,0) ellipse (1.5 and 0.65);
  \filldraw[red] (0,0) circle (1.8pt);
  \node[above] at (0,0.05) {\(\sfu\)};
  \filldraw[red] (-1.41,-0.22) circle (1.8pt);
  \node[below] at (-1.43,-0.25) {\(\sfb_1\)};
  \filldraw[red] (1.41,-0.22) circle (1.8pt);
  \node[below] at (1.43,-0.25) {\(\sfb_2\)};
\end{tikzpicture}
\arrow[ddr, yshift=.7cm, "q_{\Sigma_2}"]
\\[1pc]&&
\begin{tikzpicture}[line width=0.8pt, scale=0.62,
                    baseline=-.5ex]
  \fill[gray!30,opacity=0.5] (0,0) ellipse (1.5 and 0.65);
  \draw (0,0) ellipse (1.5 and 0.65);
  \filldraw[red] (-1.3,0.33) circle (1.8pt);
  \node[above] at (-1.32,0.36) {\(\sfc_1\)};
  \filldraw[red] (-1.36,-0.27) circle (1.8pt);
  \node[below] at (-1.38,-0.3) {\(\sfc_2\)};
\end{tikzpicture}=\bar{\Sigma}_3
\arrow[ld, "q_{\Sigma_3}"', yshift=.3cm] \\
& X=\begin{tikzpicture}[line width=0.8pt, scale=0.75,
                      baseline=(current bounding box.center)]
  \fill[gray!30,opacity=0.5,even odd rule]
    (0,0) .. controls (-0.9,1.9) and (-3.4,1.4) .. (-3.4,0)
          .. controls (-3.4,-1.4) and (-0.9,-1.9) .. (0,0) -- cycle
    (0,0) .. controls (-1.2,0.85) and (-2.0,0.35) .. (-2.0,0)
          .. controls (-2.0,-0.35) and (-1.2,-0.85) .. (0,0) -- cycle;
  \draw (0,0) .. controls (-0.9,1.9) and (-3.4,1.4) .. (-3.4,0)
              .. controls (-3.4,-1.4) and (-0.9,-1.9) .. (0,0);
  \draw (0,0) .. controls (-1.2,0.85) and (-2.0,0.35) .. (-2.0,0)
              .. controls (-2.0,-0.35) and (-1.2,-0.85) .. (0,0);
\draw (-2,0) arc (0:-180:0.7 and 0.35);
\draw[dashed] (-2,0) arc (0:180:0.7 and 0.35);
  \node at (-2.7,-1.55) {\(\Sigma_1\)};
  \fill[gray!30,opacity=0.5] (-1.8,3.9) ellipse (1.1 and 0.5);
  \draw (-1.8,3.9) ellipse (1.1 and 0.5);
  \node at (-2.45,4.55) {\(\Sigma_2\)};
  \fill[gray!30,opacity=0.5] (2.9,2.6) ellipse (1.1 and 0.5);
  \draw (2.9,2.6) ellipse (1.1 and 0.5);
  \node at (3.35,3.15) {\(\Sigma_3\)};
  \draw[blue] (-1.59,1.22) -- node[black, midway,left] {\(\sff_1\)} (-2.43,3.49);
  \draw[blue] (-1.8,3.9) to[out=60,in=120] node[black, midway, above] {\(\sff_2\)} (1.95,2.85);
  \draw[fill,blue] (0,0) -- node[black, midway, right] {\(\sff_5\)} (0.5,2) circle (1.8pt);
  \draw[blue] (0.5,2) -- node[black, midway, below] {\(\sff_3\)} (1.9,2.39);
  \draw[blue] (-1.17,3.49) -- node[black, midway, below] {\(\sff_4\)} (0.5,2);
  \filldraw[red] (0,0) circle (1.8pt) node[black, right] {\(\sfw\)};
  \filldraw[red] (-1.59,1.22) circle (1.8pt) node[black, above right] {\(\sfs\)};
  \filldraw[red] (-1.8,3.9) circle (1.8pt) node[black, left] {\(\sfu\)};
  \filldraw[red] (-2.43,3.49) circle (1.8pt) node[black, below left] {\(\sfb_1\)};
  \filldraw[red] (-1.17,3.49) circle (1.8pt) node[black,below] {\(\sfb_2\)};
  \filldraw[red] (1.95,2.85) circle (1.8pt) node[black, left] {\(\sfc_1\)};
  \filldraw[red] (1.9,2.39) circle (1.8pt) node[black, below] {\(\sfc_2\)};
\end{tikzpicture} &\\
\bar{\Sigma}_1=\begin{tikzpicture}[line width=0.8pt, scale=0.62,
                    baseline=-.5ex]
  \fill[gray!30,opacity=0.5] (0,0) circle (1.1);
  \draw (0,0) circle (1.1);
  \draw (-1.1,0) arc (180:360:1.1 and 0.33);
  \draw[dashed] (-1.1,0) arc (180:0:1.1 and 0.33);
  \filldraw[red] (0.95,0.55) circle (1.8pt);
  \node[right] at (1.0,0.62) {\(\sfw^{(1)}\)};
  \filldraw[red] (0.95,-0.55) circle (1.8pt);
  \node[right] at (1.0,-0.62) {\(\sfw^{(2)}\)};
  \filldraw[red] (-0.4,1.02) circle (1.8pt);
  \node[above] at (-0.42,1.05) {\(\sfs\)};
\end{tikzpicture}
\arrow[ru, "q_{\Sigma_1}",yshift=-.5cm]
\end{tikzcd}
\caption{The complex \(X\) of \cref{ex:free-part-unwrapped}
(free edges in blue, joints in red) and its unwrapped
thick components.}
\label{fig:free-part-example}
\end{figure}

For a thick component \(\Sigma\), write \(J(\Sigma)\subseteq\Sigma\)
for its set of joints and set
\[
J(\bar{\Sigma}) \coloneqq q_\Sigma^{-1}\bigl(J(\Sigma)\bigr)
\subseteq\bar{\Sigma};
\]
thus \(J(\bar{\Sigma})\) consists of the copies \(\sfv^{(L)}\)
of the joints \(\sfv\in J(\Sigma)\), one for each component
\(L\subseteq\lk_X(\sfv)\cap\Sigma\).

For each thick component \(\Sigma\), choose a tree
\(\Phi_{\bar{\Sigma}}\subseteq\bar{\Sigma}\) that
contains \(J(\bar{\Sigma})\) and all of whose leaves lie in
\(J(\bar{\Sigma})\) -- such a tree exists, since
\(\bar{\Sigma}\) is connected, but is in general not
unique -- and let
\[
\Phi_\Sigma \coloneqq q_\Sigma\bigl(\Phi_{\bar{\Sigma}}\bigr)
\subseteq\Sigma
\]
be its image, a connected graph. If
\(\Phi_{\bar\Sigma}\) has more than one vertex, each joint
\(\sfv\in J(\Sigma)\) has degree at least
\(\val_\Sigma(\sfv)\) in \(\Phi_\Sigma\), since its distinct
copies each have an incident edge before identification. If
\(J(\bar\Sigma)\) is a singleton, take
\(\Phi_{\bar\Sigma}\) to be that point. When
\(J(\Sigma)=\emptyset\) -- so \(\Sigma\) meets no joint -- take
\(\Phi_\Sigma\) to be a single interior vertex
\(\sfv_\Sigma\in\Sigma\). In every other case choose any
vertex \(\sfv_\Sigma\) of \(\Phi_\Sigma\) as a representative;
this choice does not require \(\Phi_\Sigma\) to be a star.

\begin{definition}[Scaffolds and weakly admissible trees]
\label{def:scaffold}
A \emph{scaffold} of \(X\) is any graph
\[
\Phi_X \coloneqq \mathsf{F}_X \cup\bigcup_{\Sigma}\Phi_\Sigma\subset X,
\]
the union of the free part with one choice of \(\Phi_\Sigma\), as
constructed above, for each thick component \(\Sigma\) of \(X\). Its edges are of two kinds:
the \emph{thin edges}, inherited from \(\mathsf{F}_X\), and the
\emph{thick edges}, those lying in some \(\Phi_\Sigma\). We write
\(\mathfrak{G}_X\) for the \emph{set of scaffolds of \(X\)}.

A \emph{weakly admissible tree} of \(X\) is a maximal tree
\(\sfT_X\subseteq\Phi_X\) of some scaffold \(\Phi_X\in\mathfrak{G}_X\)
that contains the free part \(\sfF_X\); we write
\(\mathfrak{T}_X\) for the \emph{set of weakly admissible trees of \(X\)}.
\end{definition}

Every \(\Phi_X\in\mathfrak{G}_X\) contains the free part
\(\mathsf{F}_X\), and -- being a union of subcomplexes
\(\mathsf{F}_X,\Phi_\Sigma\subseteq X\) -- is, after a subdivision
of \(X\) if necessary, a subgraph of the \(1\)-skeleton of \(X\).

Since the thin edges of \(\Phi_X\) are exactly the edges of
\(\mathsf{F}_X\), the edges of \(\Phi_X\) omitted from such a
\(\sfT_X\) are all thick edges. The set \(\mathfrak{T}_X\) is
non-empty whenever \(\mathsf{F}_X\) is a forest -- any forest in a
connected graph extends to a maximal tree -- and is empty when
\(\mathsf{F}_X\) is not a forest, since a tree cannot contain a
cycle of \(\mathsf{F}_X\).

\begin{example}\label{ex:scaffold-example}
We continue with the complex \(X\) of
\cref{ex:free-part-unwrapped} and enumerate the possible trees
\(\Phi_{\bar{\Sigma}}\) on its three unwrapped thick
components: such a tree contains \(J(\bar{\Sigma})\), and
all of its leaves lie in \(J(\bar{\Sigma})\). On
\(\bar{\Sigma}_1\), with
\(J(\bar{\Sigma}_1)=\{\sfw^{(1)},\sfw^{(2)},\sfs\}\), the tree is
either a \emph{tripod} -- three arcs meeting at a trivalent
centre -- or a path in which one of the three points is an
interior vertex; the first row of
\cref{fig:scaffold-example-trees} shows the images in
\(\Sigma_1\) of the path with \(\sfs\) interior (the outer loop),
of the two paths with \(\sfw^{(1)}\), respectively \(\sfw^{(2)}\),
interior, and of a tripod, among others. Likewise on
\(\bar{\Sigma}_2\), with
\(J(\bar{\Sigma}_2)=\{\sfu,\sfb_1,\sfb_2\}\): the second row shows
the three paths, with \(\sfu\), \(\sfb_1\), respectively \(\sfb_2\)
interior, and two (ambient isotopic) tripods. On
\(\bar{\Sigma}_3\), a tree with the two leaves
\(\sfc_1,\sfc_2\) is an embedded arc, unique up to isotopy. A
scaffold of \(X\) combines the free part with one choice of
\(\Phi_\Sigma\) for each thick component; one such scaffold is
shown in \cref{fig:scaffold-example-scaffold}.
\end{example}

\begin{figure}[ht]
\centering
\begin{subfigure}[b]{\textwidth}
\centering
\begin{align*}
\Phi_{\Sigma_1} &=
\vcenter{\hbox{
\begin{tikzpicture}[line width=0.8pt, scale=0.75, baseline=-.5ex, yshift=-2cm]
  \fill[gray!30,opacity=0.5,even odd rule]
    (0,0) .. controls (-0.9,1.9) and (-3.4,1.4) .. (-3.4,0)
          .. controls (-3.4,-1.4) and (-0.9,-1.9) .. (0,0) -- cycle
    (0,0) .. controls (-1.2,0.85) and (-2.0,0.35) .. (-2.0,0)
          .. controls (-2.0,-0.35) and (-1.2,-0.85) .. (0,0) -- cycle;
  \draw (0,0) .. controls (-0.9,1.9) and (-3.4,1.4) .. (-3.4,0)
              .. controls (-3.4,-1.4) and (-0.9,-1.9) .. (0,0);
  \draw (0,0) .. controls (-1.2,0.85) and (-2.0,0.35) .. (-2.0,0)
              .. controls (-2.0,-0.35) and (-1.2,-0.85) .. (0,0);
  \draw (-2,0) arc (0:-180:0.7 and 0.35);
  \draw[dashed] (-2,0) arc (0:180:0.7 and 0.35);
  \draw[line width=2,green!55!black] (0,0) .. controls (-0.9,1.9) and (-3.4,1.4) .. (-3.4,0)
              .. controls (-3.4,-1.4) and (-0.9,-1.9) .. (0,0);
  \filldraw[red] (0,0) circle (1.8pt);
  \filldraw[red] (-1.59,1.22) circle (1.8pt);
\end{tikzpicture}
}}
\vcenter{\hbox{
\begin{tikzpicture}[line width=0.8pt, scale=0.75, baseline=-.5ex, yshift=-2cm]
  \fill[gray!30,opacity=0.5,even odd rule]
    (0,0) .. controls (-0.9,1.9) and (-3.4,1.4) .. (-3.4,0)
          .. controls (-3.4,-1.4) and (-0.9,-1.9) .. (0,0) -- cycle
    (0,0) .. controls (-1.2,0.85) and (-2.0,0.35) .. (-2.0,0)
          .. controls (-2.0,-0.35) and (-1.2,-0.85) .. (0,0) -- cycle;
  \draw (0,0) .. controls (-0.9,1.9) and (-3.4,1.4) .. (-3.4,0)
              .. controls (-3.4,-1.4) and (-0.9,-1.9) .. (0,0);
  \draw (0,0) .. controls (-1.2,0.85) and (-2.0,0.35) .. (-2.0,0)
              .. controls (-2.0,-0.35) and (-1.2,-0.85) .. (0,0);
  \draw (-2,0) arc (0:-180:0.7 and 0.35);
  \draw[dashed] (-2,0) arc (0:180:0.7 and 0.35);
  \draw[line width=2,green!55!black] (-1.59,1.22) to[out=-10,in=115] (0,0);
  \draw[line width=2,green!55!black] (-1.9,0.9) .. controls (-1.0,1.05) and (-0.45,0.6) .. (0,0);
  \draw[line width=2,green!55!black] (-1.9,0.9) .. controls (-3.15,0.75) and (-3.15,-0.75) .. (-1.9,-0.85)
                        .. controls (-0.95,-0.95) and (-0.4,-0.55) .. (0,0);
  \filldraw[red] (0,0) circle (1.8pt);
  \filldraw[red] (-1.59,1.22) circle (1.8pt);
\end{tikzpicture}
}}
\vcenter{\hbox{
\begin{tikzpicture}[line width=0.8pt, scale=0.75, baseline=-.5ex, yshift=-2cm]
  \fill[gray!30,opacity=0.5,even odd rule]
    (0,0) .. controls (-0.9,1.9) and (-3.4,1.4) .. (-3.4,0)
          .. controls (-3.4,-1.4) and (-0.9,-1.9) .. (0,0) -- cycle
    (0,0) .. controls (-1.2,0.85) and (-2.0,0.35) .. (-2.0,0)
          .. controls (-2.0,-0.35) and (-1.2,-0.85) .. (0,0) -- cycle;
  \draw (0,0) .. controls (-0.9,1.9) and (-3.4,1.4) .. (-3.4,0)
              .. controls (-3.4,-1.4) and (-0.9,-1.9) .. (0,0);
  \draw (0,0) .. controls (-1.2,0.85) and (-2.0,0.35) .. (-2.0,0)
              .. controls (-2.0,-0.35) and (-1.2,-0.85) .. (0,0);
  \draw (-2,0) arc (0:-180:0.7 and 0.35);
  \draw[dashed] (-2,0) arc (0:180:0.7 and 0.35);
  \draw[line width=2,green!55!black] (-1.59,1.22) .. controls (-1.35,1.3) and (-3.4,1.4) .. (-3.4,0) .. controls (-3.4,-1.4) and (-0.9,-1.9) .. (0,0);
  \draw[line width=2,green!55!black] (-1.9,0.9) .. controls (-1.0,1.05) and (-0.45,0.6) .. (0,0);
  \draw[line width=2,green!55!black] (-1.9,0.9) .. controls (-3.15,0.75) and (-3.15,-0.75) .. (-1.9,-0.85)
                        .. controls (-0.95,-0.95) and (-0.4,-0.55) .. (0,0);
  \filldraw[red] (0,0) circle (1.8pt);
  \filldraw[red] (-1.59,1.22) circle (1.8pt);
\end{tikzpicture}
}}
\vcenter{\hbox{
\begin{tikzpicture}[line width=0.8pt, scale=0.75, baseline=-.5ex, yshift=-2cm]
  \fill[gray!30,opacity=0.5,even odd rule]
    (0,0) .. controls (-0.9,1.9) and (-3.4,1.4) .. (-3.4,0)
          .. controls (-3.4,-1.4) and (-0.9,-1.9) .. (0,0) -- cycle
    (0,0) .. controls (-1.2,0.85) and (-2.0,0.35) .. (-2.0,0)
          .. controls (-2.0,-0.35) and (-1.2,-0.85) .. (0,0) -- cycle;
  \draw (0,0) .. controls (-0.9,1.9) and (-3.4,1.4) .. (-3.4,0)
              .. controls (-3.4,-1.4) and (-0.9,-1.9) .. (0,0);
  \draw (0,0) .. controls (-1.2,0.85) and (-2.0,0.35) .. (-2.0,0)
              .. controls (-2.0,-0.35) and (-1.2,-0.85) .. (0,0);
  \draw (-2,0) arc (0:-180:0.7 and 0.35);
  \draw[dashed] (-2,0) arc (0:180:0.7 and 0.35);
  \draw[line width=2,green!55!black] (-1.9,0.9) .. controls (-1.75,1.05) .. (-1.59,1.22);
  \draw[line width=2,green!55!black] (-1.9,0.9) .. controls (-1.0,1.05) and (-0.45,0.6) .. (0,0);
  \draw[line width=2,green!55!black] (-1.9,0.9) .. controls (-3.15,0.75) and (-3.15,-0.75) .. (-1.9,-0.85)
                        .. controls (-0.95,-0.95) and (-0.4,-0.55) .. (0,0);
  \filldraw[line width=2,green!55!black] (-1.9,0.9) circle (1.4pt);
  \filldraw[red] (0,0) circle (1.8pt);
  \filldraw[red] (-1.59,1.22) circle (1.8pt);
\end{tikzpicture}
}}
\cdots
\\
\Phi_{\Sigma_2} &=
\vcenter{\hbox{
\begin{tikzpicture}[baseline=-.5ex,line width=0.8pt, scale=0.75, yshift=-3.9cm]
  \fill[gray!30,opacity=0.5] (-1.8,3.9) ellipse (1.1 and 0.5);
  \draw (-1.8,3.9) ellipse (1.1 and 0.5);
\draw[line width=2,green!55!black] (-1.17,3.49) -- (-1.8,3.9) -- (-2.43, 3.49);
  \filldraw[red] (-1.8,3.9) circle (1.8pt);
  \filldraw[red] (-2.43,3.49) circle (1.8pt);
  \filldraw[red] (-1.17,3.49) circle (1.8pt);
\end{tikzpicture}
}}\quad
\vcenter{\hbox{
\begin{tikzpicture}[baseline=-.5ex,line width=0.8pt, scale=0.75, yshift=-3.9cm]
  \fill[gray!30,opacity=0.5] (-1.8,3.9) ellipse (1.1 and 0.5);
  \draw (-1.8,3.9) ellipse (1.1 and 0.5);
\draw[line width=2,green!55!black] (-1.17,3.49) -- (-2.43, 3.49) +(0,0) -- (-1.8,3.9);
  \filldraw[red] (-1.8,3.9) circle (1.8pt);
  \filldraw[red] (-2.43,3.49) circle (1.8pt);
  \filldraw[red] (-1.17,3.49) circle (1.8pt);
\end{tikzpicture}
}}\quad
\vcenter{\hbox{
\begin{tikzpicture}[baseline=-.5ex,line width=0.8pt, scale=0.75, yshift=-3.9cm]
  \fill[gray!30,opacity=0.5] (-1.8,3.9) ellipse (1.1 and 0.5);
  \draw (-1.8,3.9) ellipse (1.1 and 0.5);
\draw[line width=2,green!55!black] (-2.43, 3.49) -- (-1.17,3.49) +(0,0) -- (-1.8,3.9);
  \filldraw[red] (-1.8,3.9) circle (1.8pt);
  \filldraw[red] (-2.43,3.49) circle (1.8pt);
  \filldraw[red] (-1.17,3.49) circle (1.8pt);
\end{tikzpicture}
}}\quad
\vcenter{\hbox{
\begin{tikzpicture}[baseline=-.5ex,line width=0.8pt, scale=0.75, yshift=-3.9cm]
  \fill[gray!30,opacity=0.5] (-1.8,3.9) ellipse (1.1 and 0.5);
  \draw (-1.8,3.9) ellipse (1.1 and 0.5);
  \draw[line width=2,green!55!black] (-1.8,3.6) -- (-1.8,3.9);
  \draw[line width=2,green!55!black] (-1.8,3.6) .. controls (-2.1,3.5) .. (-2.43,3.49);
  \draw[line width=2,green!55!black] (-1.8,3.6) .. controls (-1.5,3.5) .. (-1.17,3.49);
  \filldraw[line width=2,green!55!black] (-1.8,3.6) circle (1.4pt);
  \filldraw[red] (-1.8,3.9) circle (1.8pt);
  \filldraw[red] (-2.43,3.49) circle (1.8pt);
  \filldraw[red] (-1.17,3.49) circle (1.8pt);
\end{tikzpicture}
}}\quad
\vcenter{\hbox{
\begin{tikzpicture}[baseline=-.5ex,line width=0.8pt, scale=0.75, yshift=-3.9cm]
  \fill[gray!30,opacity=0.5] (-1.8,3.9) ellipse (1.1 and 0.5);
  \draw (-1.8,3.9) ellipse (1.1 and 0.5);
  \draw[line width=2,green!55!black] (-1.8,4.2) -- (-1.8,3.9);
  \draw[line width=2,green!55!black] (-1.8,4.2) to[out=180,in=90] (-2.43,3.49);
  \draw[line width=2,green!55!black] (-1.8,4.2) to[out=0,in=90] (-1.17,3.49);
  \filldraw[line width=2,green!55!black] (-1.8,4.2) circle (1.4pt);
  \filldraw[red] (-1.8,3.9) circle (1.8pt);
  \filldraw[red] (-2.43,3.49) circle (1.8pt);
  \filldraw[red] (-1.17,3.49) circle (1.8pt);
\end{tikzpicture}
}}\\
\Phi_{\Sigma_3} &=
\begin{tikzpicture}[baseline=-.5ex,line width=0.8pt, scale=0.75, yshift=-2.6cm]
  \fill[gray!30,opacity=0.5] (2.9,2.6) ellipse (1.1 and 0.5);
  \draw (2.9,2.6) ellipse (1.1 and 0.5);
  \draw[line width=2,green!55!black] (1.95,2.85) to[out=0,in=90] (2.4,2.6) to[out=-90,in=0] (1.9,2.39);
  \filldraw[red] (1.95,2.85) circle (1.8pt);
  \filldraw[red] (1.9,2.39) circle (1.8pt);
\end{tikzpicture}
\end{align*}
\caption{Some possible \(\Phi_{\Sigma_i}\) on the three thick
components, drawn in green.}
\label{fig:scaffold-example-trees}
\end{subfigure}

\medskip

\begin{subfigure}[b]{\textwidth}
\[
\Phi_X=
\vcenter{\hbox{\begin{tikzpicture}[line width=0.8pt, scale=0.75]
  \fill[gray!30,opacity=0.5,even odd rule]
    (0,0) .. controls (-0.9,1.9) and (-3.4,1.4) .. (-3.4,0)
          .. controls (-3.4,-1.4) and (-0.9,-1.9) .. (0,0) -- cycle
    (0,0) .. controls (-1.2,0.85) and (-2.0,0.35) .. (-2.0,0)
          .. controls (-2.0,-0.35) and (-1.2,-0.85) .. (0,0) -- cycle;
  \draw (0,0) .. controls (-0.9,1.9) and (-3.4,1.4) .. (-3.4,0)
              .. controls (-3.4,-1.4) and (-0.9,-1.9) .. (0,0);
  \draw (0,0) .. controls (-1.2,0.85) and (-2.0,0.35) .. (-2.0,0)
              .. controls (-2.0,-0.35) and (-1.2,-0.85) .. (0,0);
  \draw (-2,0) arc (0:-180:0.7 and 0.35);
  \draw[dashed] (-2,0) arc (0:180:0.7 and 0.35);
  \fill[gray!30,opacity=0.5] (-1.8,3.9) ellipse (1.1 and 0.5);
  \draw (-1.8,3.9) ellipse (1.1 and 0.5);
  \fill[gray!30,opacity=0.5] (2.9,2.6) ellipse (1.1 and 0.5);
  \draw (2.9,2.6) ellipse (1.1 and 0.5);
  \draw[green!55!black] (-1.9,0.9) .. controls (-1.75,1.05) .. (-1.59,1.22);
  \draw[green!55!black] (-1.9,0.9) .. controls (-1.0,1.05) and (-0.45,0.6) .. (0,0);
  \draw[green!55!black] (-1.9,0.9) .. controls (-3.15,0.75) and (-3.15,-0.75) .. (-1.9,-0.85)
                        .. controls (-0.95,-0.95) and (-0.4,-0.55) .. (0,0);
  \filldraw[green!55!black] (-1.9,0.9) circle (1.4pt);
  \draw[green!55!black] (-1.8,3.6) -- (-1.8,3.9);
  \draw[green!55!black] (-1.8,3.6) .. controls (-2.1,3.5) .. (-2.43,3.49);
  \draw[green!55!black] (-1.8,3.6) .. controls (-1.5,3.5) .. (-1.17,3.49);
  \filldraw[green!55!black] (-1.8,3.6) circle (1.4pt);
  \draw[green!55!black] (1.95,2.85) to[out=0,in=90] (2.4,2.6) to[out=-90,in=0] (1.9,2.39);
  \draw[blue] (-1.59,1.22) -- (-2.43,3.49);
  \draw[blue] (-1.8,3.9) to[out=60,in=120] (1.95,2.85);
  \draw[fill,blue] (0,0) -- (0.5,2) circle (1.8pt);
  \draw[blue] (0.5,2) -- (1.9,2.39);
  \draw[blue] (-1.17,3.49) -- (0.5,2);
  \filldraw[red] (0,0) circle (1.8pt);
  \filldraw[red] (-1.59,1.22) circle (1.8pt);
  \filldraw[red] (-1.8,3.9) circle (1.8pt);
  \filldraw[red] (-2.43,3.49) circle (1.8pt);
  \filldraw[red] (-1.17,3.49) circle (1.8pt);
  \filldraw[red] (1.95,2.85) circle (1.8pt);
  \filldraw[red] (1.9,2.39) circle (1.8pt);
\end{tikzpicture}
}}
\]
\caption{A scaffold \(\Phi_X\), combining one choice for each
thick component with the free part.}
\label{fig:scaffold-example-scaffold}
\end{subfigure}
\caption{Scaffolds of the complex \(X\) of
\cref{ex:free-part-unwrapped}: thin edges in blue, thick edges
in green, joints in red.}
\label{fig:scaffold-example}
\end{figure}

\begin{remark}
The terminology reflects regular neighbourhoods inside \(X\): a
thin edge is a \(1\)-cell of \(X\) and so has a \(1\)-dimensional
regular neighbourhood, whereas a thick edge lies in a thick
component, of dimension at least \(2\), and so admits a regular
neighbourhood of dimension at least \(2\).
\end{remark}

The scaffold is stable under the resolution of non-simple
vertices, up to contraction of the newly created free edge; we
record this at the level of weakly admissible trees. Resolving a
non-simple vertex \(v\) along a positive-dimensional link component
\(L\) of \(\lk_X(\sfv)\) (\cref{ssec:simple-an-park}) separates
the cone over \(L\) from \(\sfv\), creating a new apex \(v'\) and a
new free edge \(\sfe\in\mathsf{F}_{X'}\) of the resulting complex
\(X'\) with two endpoints \(\sfv,\sfv'\); collapsing \(\sfe\) recovers \(X\). Write
\[
q_\sfe\colon X'\longrightarrow X'/\sfe = X
\]
for this collapse map; it contracts \(e\) to the vertex \(v\) and
restricts to a homeomorphism
\(X'\setminus \sfe\xrightarrow\cong X\setminus\{\sfv\}\). The effect
on the free part is only the adjunction of the single edge \(\sfe\),
so that contracting \(e\) carries a scaffold of \(X'\) back to
one of \(X\); the following lemma makes this precise and records
the resulting correspondence of weakly admissible trees.

\begin{lemma}\label{lem:max-tree-resolution}
Let \(\sfv\) be a non-simple vertex of \(X\), with \(X'\), \(\sfe\), and
\(q_\sfe\colon X'\to X'/\sfe=X\) as above. Then collapsing the thin edge
\(e\) induces a bijection
\[
\frG_{X'}\xrightarrow\cong\frG_X,
\qquad
\Phi_{X'}\longmapsto q_\sfe(\Phi_{X'})=\Phi_{X'}/\sfe,
\]
between the scaffolds of \(X'\) and those of \(X\); writing
\(\Phi_X=\Phi_{X'}/\sfe\) for corresponding scaffolds, it restricts
to a bijection
\[
\frT_{X'}\xrightarrow\cong\frT_X,
\qquad
\sfT_{X'}\longmapsto q_\sfe(\sfT_{X'})=\sfT_{X'}/\sfe,
\]
with inverse \(\sfT_X\mapsto q_\sfe^{-1}(\sfT_X)\). In particular, every
\(\sfT_{X'}\in\frT_{X'}\) contains \(\sfe\).
\end{lemma}

\begin{proof}
The edge \(\sfe\) is a \(1\)-cell of \(X'\) contained in the closure
of no \(2\)-cell, hence a free edge; so \(\sfe\subseteq\sfF_{X'}\)
and \(\sfe\) is a thin edge of every scaffold of \(X'\), and
\(q_\sfe(\sfF_{X'})=\sfF_X\).

We first match scaffolds. Resolving \(\sfv\) along \(L\) separates
only the single copy \(\sfv^{(L)}\) of \(\sfv\) in the thick component
\(\Sigma\) meeting \(L\), turning it into the apex \(\sfv'\); the
unwrapped thick components of \(X\) and \(X'\) therefore coincide,
and we may use the same trees \(\Phi_{\bar{\Sigma}}\) for
both. With this choice the quotient maps satisfy
\(q_\sfe\circ q_{\Sigma'}=q_\Sigma\) -- collapsing \(\sfe\) re-identifies
\(\sfv'=\sfv^{(L)}\) with \(\sfv\) -- so
\(q_\sfe(\Phi_{\Sigma'})=\Phi_\Sigma\) for every thick component, and
\(q_\sfe\) carries the scaffold
\(\Phi_{X'}=\sfF_{X'}\cup\bigcup\Phi_{\Sigma'}\) onto
\(\Phi_X=\sfF_X\cup\bigcup\Phi_\Sigma\), with
\(\Phi_X=\Phi_{X'}/\sfe\). This correspondence
\(\Phi_{X'}\leftrightarrow\Phi_X\) is a bijection between the scaffolds of \(X'\) and those of \(X\).

Now let \(\sfT_{X'}\in\frT_{X'}\) be a weakly admissible tree of such
a \(\Phi_{X'}\). Then \(e\subseteq\sfF_{X'}\subseteq \sfT_{X'}\),
and contracting the spanning-tree edge \(\sfe\) yields a spanning tree
\(q_\sfe(\sfT_{X'})=\sfT_{X'}/\sfe\) of \(\Phi_X\) containing
\(\sfF_X=q_\sfe(\sfF_{X'})\); hence
\(q_\sfe(\sfT_{X'})\in\frT_X\). Conversely, for a weakly admissible
tree \(\sfT_X\in\frT_X\) of \(\Phi_X\), the preimage
\(q_\sfe^{-1}(\sfT_X)\) contains \(\sfe\) -- since \(q_\sfe(\sfe)=\sfv\in \sfT_X\) -- and
is a spanning tree of \(\Phi_{X'}\) containing
\(\sfF_{X'}=q_\sfe^{-1}(\sfF_X)\), hence lies in
\(\frT_{X'}\). The two operations are mutually inverse,
giving the asserted bijection.
\end{proof}

\begin{definition}[Pull-back and push-forward of weakly admissible
trees]\label{def:tree-pullback-pushforward}
In the setting of \cref{lem:max-tree-resolution}, for
\(\sfT_{X'}\in\frT_{X'}\) and \(\sfT_X\in\frT_X\) we call
\[
q_\sfe(\sfT_{X'}) = \sfT_{X'}/\sfe \in\frT_X
\quad\text{the \emph{push-forward} of \(\sfT_{X'}\) along \(q_\sfe\),}
\]
and
\[
q_\sfe^{-1}(\sfT_X) \in\frT_{X'}
\quad\text{the \emph{pull-back} of \(\sfT_X\) along \(q_\sfe\).}
\]
Iterating along the chain of resolutions
\(\bar X\twoheadrightarrow X\) of
\cref{prop:simple-model-unique}, the same terminology
extends to weakly admissible trees on the two ends of the chain: the
\emph{pull-back} of \(\sfT_X\in\frT_X\) is the unique
\(\sfT_{\bar X}\in\frT_{\bar X}\) obtained by iterated
\(q_{\sfe_i}^{-1}\), and the \emph{push-forward} of
\(\sfT_{\bar X}\in\frT_{\bar X}\) is the corresponding
\(\sfT_X\in\frT_X\) obtained by iterated \(q_{\sfe_i}\).
\end{definition}

\begin{definition}[Types of joints; admissible trees]
\label{def:admissible-tree}
Let \(\sfv\) be a joint of \(X\) and let \(L\subseteq\lk_X(\sfv)\) be
a link component of positive dimension, lying in the thick
component \(\Sigma\), so that the pair \((\sfv,L)\) corresponds to
the point \(\sfv^{(L)}\in J(\bar{\Sigma})\) of
\(\Phi_{\bar{\Sigma}}\). The pair \((\sfv,L)\) is of
\emph{type 1} if the unwrapped component \(\bar{\Sigma}\)
is a \(2\)-manifold and \(L\) is an interval -- equivalently,
\(\sfv^{(L)}\) lies on the boundary of the surface
\(\bar{\Sigma}\) -- and of \emph{type 2} otherwise.

A weakly admissible tree \(\sfT\in\frT_X\) is \emph{admissible} if the following holds for every thick component
\(\Sigma\) possessing a type-2 pair: lifting the thick edges of
\(\sfT\) lying in \(\Phi_\Sigma\) along \(q_\Sigma\) to a
subforest
\(\sfT_{\bar{\Sigma}}\subseteq\Phi_{\bar{\Sigma}}\),
every pair \((v,L)\) on \(\Sigma\) is connected within
\(\sfT_{\bar{\Sigma}}\) to a type-2 pair, the pairs being
identified with the points of \(J(\bar{\Sigma})\). On
thick components all of whose pairs are of type 1, no condition
is imposed.
\end{definition}

\begin{figure}[ht]
\centering
\begin{subfigure}[b]{0.48\textwidth}
\centering
\begin{tikzpicture}[line width=0.8pt, scale=1]
  \fill[gray!30,opacity=0.5]
    (-1.5,1.0) -- (0,0) -- (-1.5,-1.0) arc (-90:90:0.3 and 1.0) -- cycle;
  \draw (-1.5,1.0) -- (0,0) -- (-1.5,-1.0);
  \draw (-1.5,0) ellipse (0.3 and 1.0);
  \node at (-1.5,0) {\(D_2\)};
  \fill[gray!30,opacity=0.5]
    (1.5,1.0) -- (0,0) -- (1.5,-1.0) arc (270:90:0.3 and 1.0) -- cycle;
  \draw (1.5,1.0) -- (0,0) -- (1.5,-1.0);
  \draw (1.5,-1.0) arc (-90:90:0.3 and 1.0);
  \draw[dashed] (1.5,-1.0) arc (270:90:0.3 and 1.0);
  \node at (1.5,0) {\(D_1\)};
  \draw[green!55!black] (-1.5,1.0) .. controls (-0.72,0.52) .. (0,0);
  \draw[green!55!black] (1.5,1.0) .. controls (0.72,0.52) .. (0,0);
  \draw[blue] (-1.5,1.0) .. controls (-0.95,2.05) and (0.95,2.05) .. (1.5,1.0);
  \node[above=1pt] at (0,1.82) {\(\sfe\)};
  \filldraw[red] (0,0) circle (1.8pt);
  \node[below=2pt] at (0,-0.05) {\(\sfv\)};
  \filldraw[red] (-1.5,1.0) circle (1.8pt);
  \node[above left=-1pt] at (-1.52,1.02) {\(\sfw_2\)};
  \filldraw[red] (1.5,1.0) circle (1.8pt);
  \node[above right=-1pt] at (1.52,1.02) {\(\sfw_1\)};
\end{tikzpicture}
\caption{The wedge of two discs with a boundary chord:
\((D_1\vee_\sfv D_2)\cup \sfe\).}
\label{fig:wedge-discs-pic}
\end{subfigure}
\hfill
\begin{subfigure}[b]{0.48\textwidth}
\centering
\begin{tikzpicture}[line width=0.8pt, scale=1.25]
  \fill[gray!30,opacity=0.5] (0,0) ellipse (1.2 and 0.5);
  \draw (0,0) ellipse (1.2 and 0.5);
  \node at (-1,0) {\(D\)};
  \draw[green!55!black] (1.2,0) -- (0,0);
  \draw[blue] (0,0) arc (180:0:0.6 and 1) node[midway,above] {\(\sfe\)};
  \filldraw[red] (0,0) circle (1.8pt);
  \node[below left=-2pt] at (-0.02,-0.04) {\(\sfv\)};
  \filldraw[red] (1.2,0) circle (1.8pt);
  \node[right=2pt] at (1.22,-0.02) {\(\sfw\)};
  \path (0,-0.85) (0,1.3);
\end{tikzpicture}
\caption{The disc with a chord: \(D\cup \sfe\).}
\label{fig:disc-chord-pic}
\end{subfigure}
\caption{Two complexes admitting weakly admissible trees but no
admissible tree (\cref{ex:no-admissible-tree}): the thin edge
(blue) and the forced thick edges (green) close a cycle in
every scaffold.}
\label{fig:goldberg-counterexamples}
\end{figure}

\begin{remark}\label{rem:admissible-tree}
If \(\bar{\Sigma}\) is not a \(2\)-manifold, every pair on
\(\Sigma\) is of type \(2\) and the condition on \(\Sigma\)
holds vacuously; the definition genuinely constrains only
surface components carrying both boundary and interior
attachments. If \(X\) is simple, then \(\val_\Sigma(v)=1\) for
every joint, the pairs may be identified with the joints, and
\(q_\Sigma\) is a homeomorphism: a joint is then of type \(1\)
exactly when its thick component is a \(2\)-manifold and its
link is the disjoint union of a point and an interval. Finally,
since a resolution leaves the unwrapped thick components and
the trees \(\Phi_{\bar{\Sigma}}\) unchanged
(\cref{lem:max-tree-resolution}), the pairs, their types and
the forests \(\sfT_{\bar{\Sigma}}\) are all preserved by the
pull-back and push-forward of
\cref{def:tree-pullback-pushforward}: a weakly admissible tree is
admissible if and only if its pull-back to the simple
model is.
\end{remark}

\begin{example}[Weakly admissible but not admissible]
\label{ex:no-admissible-tree}
\cref{fig:goldberg-counterexamples} shows two complexes that
admit weakly admissible trees but no admissible tree.

In \cref{fig:wedge-discs-pic}, \(X=(D_1\vee_\sfv D_2)\cup\sfe\)
consists of two discs glued at interior points -- the non-simple
wedge point \(\sfv\) -- together with a free edge \(\sfe\)
joining boundary points \(\sfw_1\in\partial D_1\) and
\(\sfw_2\in\partial D_2\). Here
\(\sfF_X=\sfe\cup\{\sfv\}\) is a forest, so weakly admissible
trees exist, and the thick components are the two discs
\(D_1,D_2\). The two pairs at \(\sfv\) -- one for each circle of
\(\lk_X(\sfv)=S^1\sqcup S^1\) -- are interior attachments, of
type \(2\), while the pairs at \(\sfw_1,\sfw_2\) are boundary
attachments, of type \(1\). Admissibility would therefore
require thick edges connecting \(\sfw_i\) to \(\sfv\) inside
\(D_i\) for both \(i\); together with the mandatory thin edge
\(\sfe\), these close a cycle of the scaffold, so no weakly
admissible tree is admissible. Note the role of the pairs: the
single point \(\sfv\) participates once for each disc.

In \cref{fig:disc-chord-pic}, \(X=D\cup\sfe\) is a disc together
with a free edge \(\sfe\) joining an interior point \(\sfv\) of
\(D\) to a boundary point \(\sfw\in\partial D\). This complex is
simple, with a type-\(2\) joint \(\sfv\) -- an interior
attachment -- and a type-\(1\) joint \(\sfw\) -- a boundary
attachment -- and \(\sfF_X=\sfe\) is a forest. Every scaffold
consists of \(\sfe\) together with an arc in \(D\) joining
\(\sfw\) to \(\sfv\), and the two close a cycle; a weakly
admissible tree contains the thin edge \(\sfe\), hence cannot
contain the whole arc, and the type-\(1\) joint \(\sfw\) is then
not connected to the type-\(2\) joint \(\sfv\) within
\(\sfT\cap D\). Again no admissible tree exists.
\end{example}

\section{Graph-of-spaces decomposition}
\label{sec:config-decomposition}

Throughout this section, fix \(n\ge 2\) and a point \(x\in X\) of
valency \(k\), and write
\[
\lk_X(x) = \coprod_{j=1}^{k} L^j
\]
for the decomposition of the link of \(x\) into its connected
components.

We call the closure (in \(X\)) of each connected component of
\(X_x\coloneqq X\setminus\{x\}\) an \emph{\(x\)-component} of \(X\), and
denote the \(x\)-components by \(X^1,\dots,X^m\). The number \(m\) of
\(x\)-components is at most \(\val_X(x)=k\): indeed, each link
component \(L^j\) is contained in exactly one \(x\)-component, while
every \(x\)-component approaches \(x\) along at least one link
component.

Fix a basepoint \(\base=(x_1^0,\dots,x_n^0)\in F_n(X_x)\), and recall
that \(F_n(X,\base)\) denotes the connected component of \(F_n(X)\)
containing \(\base\). For each \(\bfx=(x_1,\dots,x_n)\in F_n(X)\)
and each index \(i\in\{1,\dots,n\}\), write
\[
\hat{\bfx}_i
\coloneqq
(x_1,\dots,\hat{x}_i,\dots,x_n)\in F_{n-1}(X)
\]
for the \((n-1)\)-tuple obtained by dropping the \(i\)-th coordinate.
We abbreviate \(\hat{\bfx}^0_i\coloneqq\widehat{(\base)}_i
=(x_1^0,\dots,\hat{x}_i^0,\dots,x_n^0)\).

\subsection{Decomposition at a point and \texorpdfstring{\(x\)}{x}-component
splittings}
\label{ssec:decomposition-splittings}

Let \(U_x\subseteq X\) be the open star of \(x\), and let
\(V_x\subseteq U_x\) be the smaller open star obtained by halving
the radial parameter. Concretely, with respect to the cone
parametrisation \(\lk_X(x)\times[0,1)/\lk_X(x)\times\{0\}\)
of \(U_x\) (in which the apex \([\lk_X(x)\times\{0\}]\)
corresponds to the vertex \(x\)), we have
\begin{align*}
U_x&\cong \lk_X(x)\times[0,1)\big/\lk_X(x)\times\{0\},
&
V_x&\cong \lk_X(x)\times[0,1/2)\big/\lk_X(x)\times\{0\},
\end{align*}
so that the difference \(U_x\setminus \overline{V_x}\) is the disjoint
union
\[
U_x\setminus \overline{V_x}
\cong \lk_X(x)\times(1/2,1)=
\coprod_{j=1}^k L^j\times(1/2,1).
\]

For each \(i\in\{1,\dots,n\}\)
define the \emph{\(i\)-th insertion map}
\[
f_i\colon F_{n-1}(X_x)\longrightarrow F_n(X), 
\]
as
\[
f_i(x_1,\dots,x_{i-1},\hat x_i,x_{i+1},\dots,x_n)
=
(x_1,\dots,x_{i-1},x,x_{i+1},\dots,x_n),
\]
which inserts \(x\) as the \(i\)-th coordinate.

Partitioning configurations by which (if any) coordinate equals
\(x\) then gives the disjoint decomposition
\[
F_n(X)
=
F_n(X_x) \sqcup
\coprod_{i=1}^n f_i\bigl(F_{n-1}(X_x)\bigr),
\]
where the first piece consists of the \(n\)-tuples in \(F_n(X)\)
avoiding \(x\), and the \(i\)-th piece \(f_i(F_{n-1}(X_x))\) consists
of those whose \(i\)-th coordinate equals \(x\). (The pieces are
pairwise disjoint because the coordinates of an element of
\(F_n(X)\) are distinct, so at most one of them can equal \(x\).)

\begin{proposition}\label{prop:valency-decomp}
The ordered configuration space \(F_n(X)\) is homotopy equivalent to the \emph{graph-of-spaces}
\(\cG_n(X,x)\) over the underlying \emph{distribution graph} \(\sfG_n(X,x)\) described as follows:
\begin{enumerate}[label=(\roman*),leftmargin=2em]
    \item the vertex spaces of \(\cG_n(X,x)\) are the connected
          components of \(F_n(X_x)\), together with the connected
          components of \(n\) copies
          \(\{F_{n-1}^i(X_x):i=1,\dots,n\}\) of \(F_{n-1}(X_x)\);
    \item the edge spaces of \(\cG_n(X,x)\) are the connected
          components of \(F_{n-1}^i(X_x)\times L^j\) for each \(1\le i\le n, 1\le j\le k\); each such edge space
          \(C\times L^j\) (with \(C\in\pi_0(F_{n-1}^i(X_x))\)) is
          attached by the two maps
          \[
            \begin{tikzcd}[column sep=large]
              F_n(X_x) &
              C\times L^j
                \arrow[l, "\phi_{ij}"']
                \arrow[r, "\proj_1"]
              & F_{n-1}^i(X_x),
            \end{tikzcd}
          \]
          where \(\proj_1\) is the projection onto the
          first factor and \(\phi_{ij}\) is induced (up to homotopy)
          by inserting each point of \(L_j\) as the \(i\)-th
          coordinate.
\end{enumerate}
In particular the total numbers of vertices and edges of
\(\cG_n(X,x)\) are
\[
\bigl|\pi_0(F_n(X_x))\bigr|+n\cdot\bigl|\pi_0(F_{n-1}(X_x))\bigr|
\quad\text{ and }\quad
nk\cdot \bigl|\pi_0(F_{n-1}(X_x))\bigr|.
\]
\end{proposition}

\begin{proof}
Define
\[
F_n'(X) \subseteq F_n(X)
\]
to be the subspace consisting of those configurations
\((x_1,\dots,x_n)\in F_n(X)\) with \emph{at most one} coordinate
lying in \(U_x\).

\smallskip
\noindent\emph{Step 1: \(F_n(X)\simeq F_n'(X)\).}
Choose a conical neighbourhood extending the open star:
\[
W_x\cong\lk_X(x)\times[0,2)\big/\lk_X(x)\times\{0\},
\qquad U_x=\{t<1\}.
\]
Such a choice is available after shrinking the star and
subdividing. Extend its radial coordinate continuously to
\(t:X\to[0,2]\) by setting \(t=2\) outside \(W_x\).

For a configuration \(\bfx\), let \(a_0(\bfx)\) be the second
smallest of \(t(x_1),\dots,t(x_n)\), and put
\(a=\min\{a_0,1\}\). At most one point equals \(x\), so \(a>0\),
and the order statistic \(a\) depends continuously on \(\bfx\).
For \(0<a\leq1\), define a common radial function
\[
f_a(u)=
\begin{cases}
u/a,&0\leq u\leq a,\\[2pt]
1+(u-a)/(2-a),&a\leq u\leq2.
\end{cases}
\]
It is strictly increasing, fixes \(0\) and \(2\), and is the
identity when \(a=1\). Apply
\(h_{s,a}(u)=(1-s)u+s f_a(u)\), \(0\leq s\leq1\), to every
point in its cone ray, using the same \(a=a(\bfx)\) for all
coordinates, and fix points outside \(W_x\). Since
\(h_{s,a}(2)=2\), the construction is continuous at the outer
boundary of the collar. Distinct points on the same ray remain
distinct because \(h_{s,a}\) is strictly increasing; points on
different rays meet only at the apex, which is fixed and is
occupied by at most one strand. Thus this is a homotopy within
\(F_n(X)\).

At \(s=1\), the second smallest radial coordinate is at least
\(1\), so at most one point lies in \(U_x\). Conversely, for a
configuration in \(F_n'(X)\) we have \(a=1\), and every point is
fixed throughout. This gives a strong deformation retraction
\(F_n(X)\to F_n'(X)\).

\smallskip
\noindent\emph{Step 2: decomposition of \(F_n'(X)\).}
For each \(i\in\{1,\dots,n\}\) define the \(i\)-th
\emph{thickened insertion map}
\[
g_i\colon F_{n-1}\bigl(X\setminus U_x\bigr)
\times U_x \longrightarrow F_n'(X)
\]
by
\[
g_i\bigl((x_1,\dots,\hat{x}_i,\dots,x_n),\,y\bigr)
=
(x_1,\dots,x_{i-1},\,y,\,x_{i+1},\dots,x_n);
\]
this is analogous to \(f_i\) but allows the inserted point \(y\) to
range over the entire neighbourhood \(U_x\). Each \(g_i\) is an open
map, and the images \(g_i(F_{n-1}(X\setminus U_x)\times U_x)\) are
pairwise disjoint as \(i\) varies.

Set
\[
F_n'(X\setminus \overline{V_x})
=
F_n(X\setminus \overline{V_x})\cap F_n'(X)
\subseteq F_n(X),
\]
an open subset of \(F_n'(X)\). It admits the disjoint decomposition
\begin{align*}
F_n'(X\setminus \overline{V_x})
&=
F_n(X\setminus U_x)
\coprod
\left(\coprod_{i=1}^n g_i\bigl(F_{n-1}(X\setminus U_x)
\times (U_x\setminus \overline{V_x})\bigr)\right) \\
&=
F_n(X\setminus U_x)
\coprod
\left(\coprod_{i=1}^n\coprod_{j=1}^k
g_i\bigl(F_{n-1}(X\setminus U_x)\times L^j\times(1/2,1)\bigr)\right).
\end{align*}

On the other hand, \(F_n'(X)\) itself decomposes as a union of
open subsets
\[
F_n'(X)
=
F_n'\bigl(X\setminus \overline{V_x}\bigr)
\cup
\bigcup_{i=1}^n
g_i\bigl(F_{n-1}(X\setminus U_x)\times U_x\bigr),
\]
where the first piece collects the configurations with no
coordinate in \(\overline{V_x}\), and the \(i\)-th \(g_i\)-piece those
whose \(i\)-th coordinate lies in \(U_x\) and whose other coordinates
lie outside \(U_x\). The intersection of the main piece with the
\(i\)-th \(g_i\)-piece is
\begin{align*}
F_n'(X\setminus \overline{V_x})
\cap
g_i\bigl(F_{n-1}(X\setminus U_x)\times U_x\bigr)
&=
g_i\bigl(F_{n-1}(X\setminus U_x)\times (U_x\setminus \overline{V_x})\bigr)\\
&=
\coprod_{j=1}^k g_i\bigl(F_{n-1}(X\setminus U_x)\times L^j\times(1/2,1)\bigr).
\end{align*}

\smallskip
\noindent\emph{Step 3: identifying the pieces up to homotopy.}
It follows directly
that for each \(*=n-1, n\),
\[
F_*'\bigl(X\setminus \overline{V_x}\bigr)
\simeq
F_*\bigl(X\setminus U_x\bigr)
\simeq
F_*(X_x),
\]
and so 
each \(g_i\)-piece \(g_i\bigl(F_{n-1}(X\setminus U_x)\times U_x\bigr)\) is homotopy
equivalent to the image \(f_i(F_{n-1}(X_x))\) of the insertion map
\(f_i\) defined above, and hence to \(F_{n-1}(X_x)\) itself. 
To track the dependence on \(i\), let us
write \(F_{n-1}^i(X_x)\) for the copy of \(F_{n-1}(X_x)\) associated
with the subspace \(f_i(F_{n-1}(X_x))\subseteq F_n(X)\). We then have
\[
g_i\bigl(F_{n-1}(X\setminus U_x)\times U_x\bigr)
\simeq
f_i\bigl(F_{n-1}(X_x)\bigr)
\simeq
F_{n-1}^i(X_x).
\]

Finally, for each \(j\in\{1,\dots,k\}\), we likewise have
\[
g_i\bigl(F_{n-1}(X\setminus U_x)\times L^j\times(1/2,1)\bigr)
\simeq
f_i\bigl(F_{n-1}(X_x)\bigr)\times L^j
\simeq
F_{n-1}^i(X_x)\times L^j.
\]

\smallskip
\noindent\emph{Step 4: assembling the graph-of-spaces.}
We first describe the underlying distribution graph
\(\sfG=\sfG_n(X,x)\). Its vertex set is
\[
V(\sfG) = \pi_0\bigl(F_n(X_x)\bigr)\sqcup
\coprod_{i=1}^{n}\pi_0\bigl(F^i_{n-1}(X_x)\bigr),
\]
with one vertex of \emph{\(F_n\)-type} for each connected
component of \(F_n(X_x)\), and one vertex of
\emph{\(F_{n-1}\)-type}, written \((i,C)\), for each strand
index \(i\) and each component
\(C\in\pi_0\bigl(F^i_{n-1}(X_x)\bigr)\). Its edge set is
\[
E(\sfG) = \bigl\{(i,C,j) :
1\le i\le n,\ C\in\pi_0\bigl(F^i_{n-1}(X_x)\bigr),\
1\le j\le k\bigr\},
\]
with one edge for each connected component \(C\times L^j\) of
the spaces \(F^i_{n-1}(X_x)\times L^j\) -- each link component
\(L^j\) being connected. The edge \((i,C,j)\) joins the
\(F_{n-1}\)-type vertex \((i,C)\) to the \(F_n\)-type vertex
determined by the component of \(F_n(X_x)\) containing the
configurations obtained from those of \(C\) by inserting a
point of \(L_j\times\{3/4\}\) as the \(i\)-th coordinate.

Combining Steps 2 and 3, we see that for each
\((i,j)\in\{1,\dots,n\}\times\{1,\dots,k\}\) there are two natural
maps
\[
\begin{tikzcd}
F_n(X_x) &
F_{n-1}^i(X_x)\times L^j
  \arrow[l, "\phi_{ij}"']
  \arrow[r, "\proj_1"]
& F_{n-1}^i(X_x),
\end{tikzcd}
\]
where \(\proj_1\) is the projection onto the first
factor and \(\phi_{ij}\) is (up to homotopy) the composition
\[
\phi_{ij} \simeq
\Bigl(
F_{n-1}^i(X_x)\times L^j
\xrightarrow{\simeq}
F_{n-1}\bigl(X\setminus \overline{U_x}\bigr)\times \bigl(L_j\times\{3/4\}\bigr)
\xrightarrow{g_i}
F_n(X_x)
\Bigr).
\]
The two pieces of Step 2 together with these gluing maps exhibit
\(F_n(X)\) as a graph of spaces with the vertex spaces and edge
spaces described in the proposition.

The total numbers of vertices and edges are therefore
\[
\bigl|\pi_0(F_n(X_x))\bigr|+n\cdot \bigl|\pi_0(F_{n-1}(X_x))\bigr|\quad\text{and}\quad
n\,k \cdot \bigl|\pi_0(F_{n-1}(X_x))\bigr|.\qedhere
\]
\end{proof}

\begin{remark}\label{rem:forgetful_map}
It is worth noting that the forgetful map
\(q_i\colon F_n(X_x)\to F_{n-1}^i(X_x)\) that drops the \(i\)-th
coordinate satisfies
\[
q_i\circ \phi_{ij} \simeq \proj_1:F^i_{n-1}(X_x)\times L_j\to F_{n-1}(X_x),
\]
and in particular, since \((\proj_1)_\ast\) is a
surjection on \(\pi_1\), so is \((q_i)_\ast\).
\end{remark}

\cref{prop:valency-decomp} is not restricted to the
case where \(F_n(X)\) is connected: the graph-of-spaces
\(\cG_n(X,x)\) it produces is itself allowed to be disconnected,
with its set of connected components in bijection with
\(\pi_0(F_n(X))\). When \(F_n(X)\) is disconnected and one wishes to
work with a single basepoint, one applies
\cref{prop:valency-decomp} to the basepoint component
\(F_n(X,\base)\): this yields a connected graph-of-spaces, which is precisely the connected
subgraph-of-spaces \(\cG_n((X,x),\base)\) of \(\cG_n(X,x)\) over the subgraph \(\sfG_n((X,x),\base)\) of \(\sfG_n(X,x)\) corresponding to the
component of \(\base\) in \(F_n(X)\).

\begin{proposition}\label{prop:pi1G-embedding}
Let \(x\) be as above, and fix \(\base\in F_n(X)\). For each
vertex \(\sfv\) (resp.\ edge \(\sfe\)) of \(\sfG=\sfG_n((X,x),\base)\), let
\(\cG_\sfv\) (resp.\ \(\cG_\sfe\)) denote the corresponding vertex space
(resp.\ edge space) of \(\cG=\cG_n((X,x),\base)\). Then there is
an embedding
\[
\pi_1(\sfG) \hookrightarrow
\pi_1(\cG)
\cong\PB_n(X,\base),
\]
which remains injective after passing to the quotient of
\(\PB_n(X,\base)\) by the normal closure of all the vertex and
edge groups. In other words, the composition
\[
\pi_1(\sfG)\hookrightarrow \PB_n(X,\base)
\twoheadrightarrow
\PB_n(X,\base)\Big/\co{\pi_1(\cG_\sfv),\,\pi_1(\cG_\sfe)
\,:\,\sfv\in V(\sfG),\,\sfe\in E(\sfG)}^{\PB_n(X,\base)}
\]
is still injective, with image isomorphic to
\(\pi_1(\sfG)\).
\end{proposition}

\begin{proof}
Pick a spanning tree \(\sfT\subseteq \sfG\). The graph-of-spaces
description of \cref{prop:valency-decomp}, together with van
Kampen's theorem, presents \(\PB_n(X,\base)\cong\pi_1(\cG)\)
using the vertex-group generators, one stable letter
\(t_\sfe\) for each edge outside \(\sfT\), and the edge
relations identifying the two attaching-map images up to the
corresponding stable letter. No injectivity of the attaching
maps is required here. Vertex and edge groups in the normal
closures below mean their images in \(\pi_1(\cG)\), after
the chosen basepoint transports.

Sending each vertex- and edge-group generator to the identity and
each stable letter \(t_\sfe\) to itself extends to a well-defined
retraction
\[
r\colon \PB_n(X,\base) \twoheadrightarrow
\mathbb{F}\bigl(\{t_\sfe : \sfe\in E(\sfG)\setminus E(\sfT)\}\bigr)
\cong \pi_1(\sfG),
\]
whose kernel is precisely the normal closure of the vertex and
edge groups of \(\cG\), that is,
\[
\ker(r) =
\co{\pi_1(\cG_\sfv),\,\pi_1(\cG_\sfe)\,:\,\sfv\in V(\sfG),\,\sfe\in E(\sfG)}^{\PB_n(X,\base)}.
\]
The section \(\pi_1(\sfG)\hookrightarrow\PB_n(X,\base)\) given by the
stable letters supplies the desired embedding, and the identity
above shows that this embedding remains injective after passing to
the announced quotient.
\end{proof}

\begin{corollary}\label{cor:valency-quotient}
Fix a basepoint \(\base\in F_n(X_x)\), and suppose that both
\(F_n(X_x)\) and \(F_{n-1}(X_x)\) are connected. Then
\[
\PB_n(X,\base)\big/\co{\PB_n(X_x,\base)}^{\PB_n(X,\base)}
\cong
\mathbb{F}_{n(k-1)},
\]
the free group of rank \(n(k-1)\).
\end{corollary}

\begin{proof}
The connectedness hypotheses imply that the graph-of-spaces
\(\cG=\cG_n(X,x)\) over \(\sfG=\sfG_n(X,x)\) of
\cref{prop:valency-decomp} has exactly one
\(F_n(X_x)\)-vertex and one \(F_{n-1}^i(X_x)\)-vertex for each
\(i\in\{1,\dots,n\}\), for a total of \(n+1\) vertices; each of the
\(nk\) edge spaces \(F_{n-1}^i(X_x)\times L^j\) (with
\((i,j)\in\{1,\dots,n\}\times\{1,\dots,k\}\)) contributes a single
edge joining the \(F_n(X_x)\)-vertex to the \(F_{n-1}^i(X_x)\)-vertex.
Hence the distribution graph \(\sfG\) is connected with
Euler characteristic
\[
\chi(\sfG) = (n+1)-nk = 1-n(k-1),
\]
so \(\pi_1(\sfG)\) is free of rank \(n(k-1)\).

For each \(i\), choose an incident edge space
\(F_{n-1}^i(X_x)\times L^j\). Its projection to the
\(F_{n-1}^i(X_x)\)-vertex space is surjective on fundamental
groups. The edge relation identifies this image, up to
conjugation, with its attaching-map image in the
\(F_n(X_x)\)-vertex group. Killing the latter vertex-group
image therefore kills every auxiliary vertex-group image as
well, and hence all edge-group images. In other words,
\[
\co{\PB_n(X_x)}^{\PB_n(X)}=
\co{\pi_1(\cG_\sfv),\,\pi_1(\cG_\sfe)\,:\,\sfv\in V(\sfG),\,
\sfe\in E(\sfG)}^{\PB_n(X)},
\]
and \cref{prop:pi1G-embedding} identifies the quotient
with \(\pi_1(\sfG)\cong\mathbb{F}_{n(k-1)}\).
\end{proof}

\begin{remark}\label{rem:cor28-hypotheses}
The hypotheses of \cref{cor:valency-quotient} are
satisfied whenever \(X\not\cong S^1\) and \(X_x\) is path-connected.
Indeed, these two assumptions together force the star graph
\(\sfS_3\) to embed in \(X_x\); \cref{lem:rho-surj}
(applied to \(X_x\) in place of \(X\)) then implies that \(F_m(X_x)\) is
connected for every \(m\geq 1\). In particular \(F_n(X_x)\) and
\(F_{n-1}(X_x)\) are both connected, which forces the graph-of-spaces
\(\cG_n(X,x)\) of \cref{prop:valency-decomp} to
be connected, and hence \(F_n(X)\simeq\cG_n(X,x)\) is
connected as well.
\end{remark}

When the chosen point \(x\) is an \emph{interior point of a free
edge}, the valency is \(k=2\), and the presentation underlying
\cref{cor:valency-quotient} exhibits \(\PB_n(X)\) as a
quotient of a free product.

\begin{corollary}\label{cor:free-edge-splitting}
Suppose \(x\) is an interior point of a free edge \(\sfe\) of \(X\), and
that both \(F_n(X_x)\) and \(F_{n-1}(X_x)\) are connected. Then there
is a surjection
\[
\PB_n(X_x,\base)\,\ast\,\mathbb{F}_n
\twoheadrightarrow
\PB_n(X,\base)
\]
of \(\PB_n(X_x,\base)\ast\mathbb{F}_n\) -- the free product of
\(\PB_n(X_x,\base)\) with a free group of rank \(n\) -- onto
\(\PB_n(X,\base)\). Moreover this surjection is injective on each free
factor: both \(\PB_n(X_x,\base)\) and \(\mathbb{F}_n\) embed into
\(\PB_n(X,\base)\).
\end{corollary}

\begin{proof}
Since \(x\) lies in the interior of an edge \(e\) it has valency
\(\val_X(x)=2\), and \(\lk_X(x)=L^1\coprod L^2\) consists of the two
directions of \(\sfe\) at \(x\). As \(F_n(X_x)\) and \(F_{n-1}(X_x)\) are
connected, the graph-of-spaces \(\cG_n(X,x)\) of
\cref{prop:valency-decomp} has a single
\(F_n(X_x)\)-vertex \(\sfv_0\), a single vertex \(\sfv_i\) modelled on
\(F_{n-1}^i(X_x)\) for each \(i\in\{1,\dots,n\}\), and, for each \(i\),
exactly two edges \(\sfe_{i,1},\sfe_{i,2}\) joining \(\sfv_0\) to \(\sfv_i\) -- one for
each link point \(L^1,L^2\). The distribution graph is therefore
\(n\) bigons sharing the central vertex \(\sfv_0\), with first Betti number
\(n\). In the associated graph of groups the vertex group at \(\sfv_0\) is
\(\PB_n(X_x,\base)\), the vertex group at each \(\sfv_i\) is
\(\PB_{n-1}(X_x)\), and each edge group is \(\PB_{n-1}(X_x)\), included
into \(\sfv_i\) by the identity (the projection, \(L^j\) being a single
point) and into \(\sfv_0\) by the insertion-induced map
\((\phi_{ij})_\ast\).

Collapsing the spanning tree \(\sfT=\bigcup_i \sfe_{i,1}\) absorbs each
\(\sfv_i\)-group into its image \((\phi_{i1})_\ast(\PB_{n-1}(X_x))\) inside
\(\PB_n(X_x,\base)\), and the \(n\) off-tree edges \(\sfe_{i,2}\) supply
stable letters \(t_1,\dots,t_n\); Bass--Serre theory then gives the
presentation
\[
\PB_n(X,\base)=\bigl\langle\,\PB_n(X_x,\base),\,
t_1,\dots,t_n
\bigm|
t_i\,(\phi_{i1})_\ast(h)\,t_i^{-1}=(\phi_{i2})_\ast(h),
h\in\PB_{n-1}(X_x)\,\bigr\rangle .
\]
The free product \(\PB_n(X_x,\base)\ast\mathbb{F}_n\), with
\(\mathbb{F}_n=\langle t_1,\dots,t_n\rangle\), has exactly the same
generators and only the obvious relations of its two free factors.
Mapping each generator to the like-named generator of the
presentation above therefore defines a homomorphism
\(\PB_n(X_x,\base)\ast\mathbb{F}_n\to\PB_n(X,\base)\); it hits every
generator of \(\PB_n(X,\base)\) and is consequently surjective. 

Finally, each free factor embeds. The presentation above exhibits
\(\PB_n(X,\base)\) as a multiple HNN extension with base group
\(\PB_n(X_x,\base)\) and stable letters \(t_1,\dots,t_n\); by Britton's
lemma~\cite{Britton1963} the base group \(\PB_n(X_x,\base)\) injects. For the other
factor, sending \(\PB_n(X_x,\base)\) to \(1\) and each \(t_i\) to itself
respects every defining relation
\(t_i\,(\phi_{i1})_\ast(h)\,t_i^{-1}=(\phi_{i2})_\ast(h)\) -- both sides
become trivial -- and hence defines a retraction
\(\PB_n(X,\base)\twoheadrightarrow\mathbb{F}_n\) that is the identity on
\(\langle t_1,\dots,t_n\rangle\); therefore \(\mathbb{F}_n\) injects as
well. Thus the surjection of the corollary restricts to an embedding on
each free factor.
\end{proof}

We now drop the assumption that \(X_x\) is connected, in which case
\(F_n(X)\) itself may be disconnected. Let \(X'\) be an \(x\)-component
of \(X\) (so \(X'=X\) when \(X_x\) is connected), write
\(X'_x\coloneqq X'\setminus\{x\}\) for the corresponding open
component of \(X_x\), and fix a basepoint \(\base\in F_n(X'_x)\).

\begin{proposition}\label{prop:disconnected-graph-of-spaces}
Let \(X'\) be an \(x\)-component of \(X\), fix a basepoint
\(\base\in F_n(X'_x)\), and let
\(\lk_{X'}(x)=\coprod_{j=1}^{k'} L^j\) be the decomposition of the
link of \(x\) in \(X'\) into its connected components. Then:
\begin{enumerate}[label=(\roman*),leftmargin=2em]
    \item \cref{prop:valency-decomp}, applied to \(X'\)
          at the vertex \(x\), exhibits \(F_n(X',\base)\) as a
          \emph{connected} graph-of-spaces
          \(\cG'=\cG_n((X',x),\base)\) with underlying distribution graph
          \(\sfG'=\sfG_n((X',x),\base)\). Its vertex spaces are the
          components of \(F_n(X'_x)\) and of
          \(F_{n-1}^i(X'_x)\), with edge spaces the components of
          \(F_{n-1}^i(X'_x)\times L^j\), retaining precisely those
          vertices and edges in the connected component of the
          total graph of spaces containing \(\base\).
          In particular, several components of \(F_n(X'_x)\)
          may occur in this one connected total space. The natural
          inclusion \(X'\hookrightarrow X\) realises
          \(\cG'\) as a \emph{subgraph-of-spaces}
          of \(\cG=\cG_n((X,x),\base)\), and correspondingly
          \(\sfG'\) as a subgraph of \(\sfG=\sfG_n((X,x),\base)\).
    \item There is an isomorphism
          \[
          \PB_n(X,\base)\big/\co{\PB_n(X',\base)}^{\PB_n(X,\base)}
          \cong
          \pi_1\bigl(\cG/\cG'\bigr),
          \]
          where
          \(\cG/\cG'\)
          denotes the graph-of-spaces obtained from
          \(\cG\) by collapsing the
          subgraph-of-spaces \(\cG'\) to a
          single point.
    \item In particular, the fundamental group
          \(\pi_1\bigl(\sfG/\sfG'\bigr)\) of the
          quotient distribution graph embeds into the quotient
          \(\PB_n(X,\base)\big/\co{\PB_n(X',\base)}^{\PB_n(X,\base)}\).
\end{enumerate}
\end{proposition}

\begin{proof}
(i) The thickened insertion maps \(g_i\) from the proof of
\cref{prop:valency-decomp} restrict to \(X'\):
configurations in
\(F_n'(X',\base)\coloneqq F_n'(X)\cap F_n(X',\base)\) are exactly
those whose coordinates all lie in \(X'\), with at most one of
them in \(U_x\). Repeating Steps~1--4 of the proof of
\cref{prop:valency-decomp} with \(X'\) in place of \(X\)
therefore produces the graph-of-spaces \(\cG'\) with the
listed vertex and edge spaces, and the inclusion
\(F_n(X',\base)\hookrightarrow F_n(X,\base)\) realises
\(\cG'\) as a sub-graph-of-spaces of \(\cG\).
Connectedness of \(\cG'\) follows from the homotopy
equivalence \(F_n(X',\base)\simeq\cG'\) together with the
fact that \(F_n(X',\base)\) is path connected by definition.

(ii) Since \(\pi_1(\cG')=\PB_n(X',\base)\), collapsing
\(\cG'\) to a point inside \(\cG\) kills exactly the
normal closure of \(\PB_n(X',\base)\) inside
\(\pi_1(\cG)=\PB_n(X,\base)\). By van Kampen, the
fundamental group of the resulting quotient graph-of-spaces is
\(\pi_1(\cG/\cG')\), which yields the announced
isomorphism.

(iii) The retraction argument of
\cref{prop:pi1G-embedding}, applied to the quotient
graph-of-spaces \(\cG/\cG'\) (whose distribution graph
is \(\sfG/\sfG'\)), exhibits \(\pi_1(G/G')\) as a retract of
\(\pi_1(\cG/\cG')\). Combined with the isomorphism
in~(ii), this embeds \(\pi_1(\sfG/\sfG')\) into
\(\PB_n(X,\base)/\co{\PB_n(X',\base)}^{\PB_n(X,\base)}\).
\end{proof}

When the link of \(x\) is simple enough, the subgraphs-of-spaces of
\cref{prop:disconnected-graph-of-spaces}(i) embed on the level of
fundamental groups as well.

\begin{lemma}\label{lem:x-component-embedding}
Suppose that every connected component of \(\lk_X(x)\) is simply
connected -- for instance, \(X\) is a graph, or \(x\) is an
interior point of a free edge, so that the components are single
points. Then every attaching map of \(\cG=\cG_n((X,x),\base)\) is
\(\pi_1\)-injective, and \(\PB_n(X,\base)\) is the fundamental
group of the associated graph of groups. Consequently every
connected subgraph-of-spaces of \(\cG\) is \(\pi_1\)-embedded; in
particular, for every \(x\)-component \(X'\) of \(X\) with
\(\base\in F_n(X'_x)\), the inclusion induces an embedding
\[
\PB_n(X',\base)\hookrightarrow\PB_n(X,\base).
\]
\end{lemma}

\begin{proof}
An edge space of \(\cG\) is a component \(C\times L^j\) of some
\(F^i_{n-1}(X_x)\times L^j\). As \(L^j\) is simply connected,
\(\pi_1(C\times L^j)\cong\pi_1(C)\) and the attaching map
\(\proj_1\) induces an isomorphism on fundamental groups. For the
other attaching map \(\phi_{ij}\), \cref{rem:forgetful_map} gives
\(q_i\circ\phi_{ij}\simeq\proj_1\), so \((\phi_{ij})_\ast\) is
split injective. Hence every edge group injects into its two
adjacent vertex groups, and \(\cG\) carries the structure of a
graph of groups with \(\pi_1(\cG)=\PB_n(X,\base)\). By the
Bass--Serre normal form, every reduced loop of a connected
subgraph-of-spaces remains reduced in \(\cG\), so its fundamental
group embeds into \(\PB_n(X,\base)\); applied to
\(\cG'=\cG_n((X',x),\base)\) from
\cref{prop:disconnected-graph-of-spaces}(i), this yields the
displayed embedding.
\end{proof}

We now record two consequences for a free edge whose removed
interior point \emph{separates} \(X\). Here the edge \(\sfe\) is a
bridge, so \(X_x\) falls into two \(x\)-components meeting only at
\(x\): one side embeds in the quotient by the other
(\cref{cor:other-x-component-embeds}), and the two sides generate
-- not by themselves, but together with the free group of loops of
the distribution graph, realised by \emph{shuffle} braids in place of
the wrapping generators of \cref{cor:free-edge-splitting} -- with
each side and the loop group embedding
(\cref{cor:free-edge-disconnected}).

\begin{corollary}\label{cor:other-x-component-embeds}
Suppose \(x\) is an interior point of a free edge \(\sfe\) of \(X\)
with \(X_x\) disconnected, so that \(\val_X(x)=2\) and \(X\) has
exactly two \(x\)-components \(X^1,X^2\), meeting only at \(x\).
Choose basepoints \(\bfx^1\in F_n(X^1_x)\) and
\(\bfx^2\in F_n(X^2_x)\), and fix a path \(\gamma\) from
\(\bfx^1\) to \(\bfx^2\) in \(F_n(X)\). Then the inclusion
\(X^2\hookrightarrow X\), combined with the change of basepoint
along \(\gamma\), induces an embedding
\[
\PB_n(X^2,\bfx^2) \hookrightarrow
\PB_n(X,\bfx^1)\big/\co{\PB_n(X^1,\bfx^1)}^{\PB_n(X,\bfx^1)}.
\]
\end{corollary}

\begin{proof}
Since \(x\) is interior to an edge, \(\val_X(x)=2\) and
\(\lk_X(x)=L^1\sqcup L^2\) is a pair of points; as \(X_x\) is
disconnected, the two \(x\)-components are \(X^1\supseteq L^1\) and
\(X^2\supseteq L^2\), meeting only at \(x\).

Write \(\cG=\cG_n((X,x),\bfx^2)\) for the graph-of-spaces of
\cref{prop:valency-decomp}, based for convenience at \(\bfx^2\); the
change of basepoint along \(\gamma\) identifies
\(\PB_n(X,\bfx^1)\cong\PB_n(X,\bfx^2)\) and carries
\(\co{\PB_n(X^1,\bfx^1)}^{\PB_n(X,\bfx^1)}\) onto the corresponding
normal closure
\(N\coloneqq\co{\PB_n(X^1,\bfx^1)}^{\PB_n(X,\bfx^2)}\), so it
suffices to prove the statement at \(\bfx^2\). By
\cref{lem:x-component-embedding} -- the link components
\(L^1,L^2\) are single points -- applied to the \(x\)-component
\(X^2\), the inclusion-induced map
\(\iota^2_\ast\colon\PB_n(X^2,\bfx^2)\to\PB_n(X,\bfx^2)\) is
injective; it therefore remains to show
\(\im\iota^2_\ast\cap N=1\). We do so by producing a homomorphism
\(\varphi\) out of \(\PB_n(X,\bfx^2)\) with
\(N\subseteq\ker\varphi\) and \(\varphi\circ\iota^2_\ast\)
injective: given these, any
\(g=\iota^2_\ast(h)\in\im\iota^2_\ast\cap N\) satisfies
\((\varphi\circ\iota^2_\ast)(h)=\varphi(g)=1\), whence \(h=1\) and
\(g=1\).

\smallskip
\noindent\emph{The stabilisation telescope.}
For a subset \(S\subseteq\{1,\dots,n\}\) let \(F_S(X^2_x)\) denote
the configuration space of the strands labelled by \(S\) in
\(X^2_x\), so that \(F_{\{1,\dots,n\}}(X^2_x)=F_n(X^2_x)\) and
\(F_\varnothing(X^2_x)\) is a single point. Let
\(\varepsilon=e\cap X^2_x\) be the half-open free arc of \(X^2_x\)
ending at \(x\), and fix a point \(p\in\varepsilon\) close to
\(x\). For \(i\notin S\) define the \emph{leaf stabilisation}
\[
\sigma_i\colon F_S(X^2_x)\longrightarrow F_{S\cup\{i\}}(X^2_x)
\]
by pushing a configuration off the initial segment of
\(\varepsilon\) between \(x\) and \(p\) (sliding the finitely many
strands on \(\varepsilon\) outward, as in Step~1 of the proof of
\cref{prop:valency-decomp}) and then inserting the strand \(i\) at
\(p\). Dropping the \(i\)-th coordinate defines a forgetful map
\(q_i\colon F_{S\cup\{i\}}(X^2_x)\to F_S(X^2_x)\) with
\(q_i\circ\sigma_i\simeq\mathrm{id}\)
(cf.\ \cref{rem:forgetful_map}); in particular each
\((\sigma_i)_\ast\) is split injective on fundamental groups.
Assemble these maps into a graph-of-spaces \(\cS\): the vertex
spaces are the connected components of \(F_S(X^2_x)\) for all
\(S\subseteq\{1,\dots,n\}\), and for each \(i\notin S\) and each
component \(C\) of \(F_S(X^2_x)\) there is an edge space \(C\),
attached to \(C\) by the identity and to the component of
\(F_{S\cup\{i\}}(X^2_x)\) containing \(\sigma_i(C)\) by
\(\sigma_i\). Both attaching maps of every edge are
\(\pi_1\)-injective, so each component of \(\cS\) has the
fundamental group of a graph of groups, and connected
subgraphs-of-spaces of \(\cS\) are \(\pi_1\)-embedded by the same
Bass--Serre normal-form argument as in
\cref{lem:x-component-embedding}.

\smallskip
\noindent\emph{The forgetting map \(\Psi\colon\cG\to\cS\).}
Since \(X_x=X^1_x\sqcup X^2_x\), every configuration in
\(F_n(X_x)\) or \(F^i_{n-1}(X_x)\) splits its strands into the two
sides, with locally constant label sets. Map a component of
\(F_n(X_x)\) whose \(X^2\)-side labels form \(S\) to the
corresponding component of \(F_S(X^2_x)\) by forgetting the
\(X^1\)-side strands, and a component of \(F^i_{n-1}(X_x)\) with
\(X^2\)-side labels \(S\) to \(F_S(X^2_x)\) by forgetting the
\(X^1\)-side strands together with the strand \(i\). For an edge of
\(\cG\) with \(j=1\) the two attaching maps agree after forgetting
(the strand \(i\) is forgotten on both ends), so \(\Psi\) collapses
its cylinder into the common target vertex space. An edge with
\(j=2\), attached by \(\proj_1\) and \(\phi_{i2}\), is carried to
the corresponding \(\sigma_i\)-cylinder of \(\cS\): on the
\(F^i_{n-1}\)-end the two prescriptions agree outright, while on
the \(F_n\)-end forgetting the \(X^1\)-side strands of
\(\phi_{i2}(-)\) leaves the strand \(i\) on \(\varepsilon\) near
\(x\), which agrees with \(\sigma_i\circ(\text{forget})\) up to the
homotopy sliding \(i\) outward to \(p\). These squares commute up
to the indicated homotopies, so the prescriptions assemble into a
map of graphs-of-spaces \(\Psi\colon\cG\to\cS\), inducing
\[
\varphi\coloneqq\Psi_\ast\colon\PB_n(X,\bfx^2)\longrightarrow
\pi_1\bigl(\cS,\Psi(\bfx^2)\bigr).
\]

\smallskip
\noindent\emph{\(\varphi\) kills \(N\).}
The connected subgraph-of-spaces
\(\cG^1=\cG_n((X^1,x),\bfx^1)\subseteq\cG\) provided by
\cref{prop:disconnected-graph-of-spaces}(i) has all its vertex
spaces with \(X^2\)-side label set \(\varnothing\), and only
\(j=1\) edges; hence \(\Psi(\cG^1)\) is the single point
\(F_\varnothing(X^2_x)\), and \(\varphi\) vanishes on the image of
\(\PB_n(X^1,\bfx^1)=\pi_1(\cG^1)\) in \(\PB_n(X,\bfx^2)\)
(transported into the basepoint along \(\gamma\), which only
changes the image by a conjugation). As \(\ker\varphi\) is normal,
\(N\subseteq\ker\varphi\).

\smallskip
\noindent\emph{\(\varphi\circ\iota^2_\ast\) is injective.}
On \(\cG^2=\cG_n((X^2,x),\bfx^2)\subseteq\cG\) the map \(\Psi\)
restricts to an isomorphism onto a subgraph-of-spaces of \(\cS\):
the vertex spaces of \(\cG^2\) are components of \(F_n(X^2_x)\) and
of the \(F^i_{n-1}(X^2_x)\), which \(\Psi\) maps by the identity
(there are no \(X^1\)-side strands to forget), and the \(j=2\)
edges of \(\cG^2\) are carried bijectively to the corresponding
\(\sigma_i\)-edges, with the same attaching maps up to homotopy.
Here the free-edge hypothesis is essential: the link component
\(L^2\) is a single point, so each edge space of \(\cG^2\) is a
copy of an \(F^i_{n-1}(X^2_x)\)-component, exactly as in \(\cS\).
By the Bass--Serre normal form for subgraphs-of-groups,
\(\varphi\circ\iota^2_\ast=\pi_1\bigl(\Psi|_{\cG^2}\bigr)\) is
injective. (For a general vertex \(x\) the link components need not
be points; the edge spaces of \(\cG^2\) then carry \(L^2\)-factors
that \(\Psi\) collapses, the comparison with \(\cS\) breaks down,
and the embedding can fail.)

This completes the proof.
\end{proof}

\cref{cor:other-x-component-embeds} shows that \(\PB_n(X^2)\)
survives -- embedded -- in the quotient of \(\PB_n(X,\base)\) by the
normal closure of \(\PB_n(X^1)\). It is tempting to conclude that
the two sides together account for all of \(\PB_n(X,\base)\); they
do not. In general even the quotient by both sides,
\[
\PB_n(X,\base)\big/
\co{\PB_n(X^1),\,\PB_n(X^2)}^{\PB_n(X,\base)},
\]
is non-trivial: braids that \emph{shuffle} strands between the two
sides survive. In fact this quotient is computed exactly -- it is
the free group of loops of the distribution graph
(\cref{cor:free-edge-both-quotient}).

The simplest such surviving shuffle lives on the letter
\(\sfH\) itself:

\begin{example}[The letter \(\sfH\)]
\label{ex:two-tripods-shuffle}
Let \(\sfH\) be the tree with two trivalent vertices \(\sfv\)
and \(\sfv'\), joined by a middle bar \(\sfe\), and carrying
two further legs \(\sfe_1,\sfe_2\) at \(\sfv\) and
\(\sfe_3,\sfe_4\) at \(\sfv'\) (\cref{fig:H-shuffle}). Let
\(x\) be an interior point of \(\sfe\); then the two
\(x\)-components \(X^1,X^2\) of \(\sfH\) are each a tripod with
one prolonged leg, and \(X^i_x\) is homotopy equivalent to a
tripod. Take \(n=2\).

Each \(F_2(X^i_x)\) is connected, with
\(\PB_2(X^i_x)\cong\Z\) generated by the square of the hexagonal
exchange move of two strands on a tripod. In the decomposition
\(\cG_2((\sfH,x),\base)\), the two \emph{corner} vertices -- both
strands on one side -- carry the groups \(\Z\); the two
\emph{mixed} vertices -- one strand on each side -- carry
\(\pi_1(X^1_x\times X^2_x)=1\); the four \(F^i_1\)-type vertices
and all eight edge spaces have contractible factors and carry
trivial groups. The distribution graph
\(\sfG=\sfG_2((\sfH,x),\base)\) is an octagon, alternating
\(F_2\)- and \(F^i_1\)-type vertices, so
\[
\PB_2(\sfH,\base)\cong\Z\ast\Z\ast\Z ,
\]
free of rank three: the two corner copies of \(\Z\), and one
stable letter \(t\) for the octagon. In particular the images of
\(\PB_2(X^1)\) and \(\PB_2(X^2)\) generate only a free factor of
rank two, and
\[
\PB_2(\sfH,\base)\big/\co{\PB_2(X^1),\,\PB_2(X^2)}\cong\Z\neq1 .
\]
The surviving generator \(t\) is realised by the \emph{shuffle
braid} \(\gamma_\sfH\) of \cref{fig:H-shuffle}. Place the base
configuration \(\base=(x^0_1,x^0_2)\) on the middle bar, with
\(x^0_1\) closer to \(\sfv\), and perform the six numbered
moves: first the second strand travels from \(x^0_2\) onto
\(\sfe_4\), and then the first strand from \(x^0_1\) onto
\(\sfe_3\); next the second strand travels from \(\sfe_4\)
across the bar onto \(\sfe_2\), and then the first strand from
\(\sfe_3\) across the bar onto \(\sfe_1\); finally the second
strand returns from \(\sfe_2\) to \(x^0_2\), and then the first
strand from \(\sfe_1\) to \(x^0_1\). The two strands cross
\(\sfe\) twice, in opposite orders, and the braid traverses the
octagon once; the resulting pure braid is a well-defined element
of \(\PB_2(\sfH,\base)\) lying in neither side's image. (The rank
\(1\) of \(\pi_1(\sfG)\) here agrees with the general formula
\(1+2^{n-1}(n-2)\) of \cref{cor:free-edge-loop-rank} below.)
\end{example}

\begin{figure}[ht]
\centering
\begin{tikzpicture}[line width=0.8pt, scale=1.05]
\draw (0,0) -- (5,0);
\node[below=2pt] at (0.8,-0.02) {\(\sfe\)};
\draw (0,0) -- (-1.1,0.75)  node[left]  {\(\sfe_1\)};
\draw (0,0) -- (-1.1,-0.75) node[left]  {\(\sfe_2\)};
\draw (5,0) -- (6.1,0.75)   node[right] {\(\sfe_3\)};
\draw (5,0) -- (6.1,-0.75)  node[right] {\(\sfe_4\)};
\filldraw (0,0) circle (1.6pt) node[below=3pt] {\(\sfv\)};
\filldraw (5,0) circle (1.6pt) node[below=3pt] {\(\sfv'\)};
\filldraw (1.7,0) circle (1.4pt) node[above=1.5pt] {\(x^0_1\)};
\filldraw (3.1,0) circle (1.4pt) node[above=1.5pt] {\(x^0_2\)};
\draw[->, red!80!black, thick]
  (3.1,-0.07) to[bend right=15]
  node[pos=0.5, circle, draw, fill=white, inner sep=0.7pt]{\tiny 1}
  (5.28,-0.26);
\draw[->, red!80!black, thick]
  (5.38,-0.36) to[bend left=22]
  node[pos=0.5, circle, draw, fill=white, inner sep=0.7pt]{\tiny 3}
  (-0.36,-0.32);
\draw[->, red!80!black, thick]
  (-0.26,-0.22) to[bend right=12]
  node[pos=0.45, circle, draw, fill=white, inner sep=0.7pt]{\tiny 5}
  (3.1,-0.09);
\draw[->, blue!70!black, thick]
  (1.7,0.07) to[bend left=15]
  node[pos=0.62, circle, draw, fill=white, inner sep=0.7pt]{\tiny 2}
  (5.28,0.26);
\draw[->, blue!70!black, thick]
  (5.38,0.36) to[bend right=22]
  node[pos=0.5, circle, draw, fill=white, inner sep=0.7pt]{\tiny 4}
  (-0.36,0.32);
\draw[->, blue!70!black, thick]
  (-0.26,0.22) to[bend left=12]
  node[pos=0.45, circle, draw, fill=white, inner sep=0.7pt]{\tiny 6}
  (1.7,0.09);
\end{tikzpicture}
\caption{The tree \(\sfH\) and the shuffle braid
\(\gamma_\sfH\) for \(n=2\). The second strand (moves
\(1,3,5\)) travels
\(x^0_2\to \sfe_4\to \sfe_2\to x^0_2\), and the first strand
(moves \(2,4,6\)) travels
\(x^0_1\to \sfe_3\to \sfe_1\to x^0_1\), the six moves performed
in the numbered order.}
\label{fig:H-shuffle}
\end{figure}

Such shuffles exist in great generality: all that is needed is
an embedded copy of \(\sfH\) inside \(X\) whose middle bar
contains a free edge. They will be the basic tool of
\cref{sec:goldberg-trivial}.

Let \(\alpha\subseteq X\) be an embedded path containing a
free edge \(\sfe\) of \(X\), with endpoints \(\sfv\) and \(\sfv'\),
and suppose that each of \(\sfv\) and \(\sfv'\) either has valency at
least \(3\) in \(X\) or lies in a thick component of
\(X\). Then one can choose two embedded arcs
\(\sfe_1,\sfe_2\subseteq X\) issuing from \(v\), meeting only at
\(v\) and avoiding \(\alpha\): along two link components other
than that of \(\alpha\), if \(\val_{X}(\sfv)\geq3\); and
inside the cone over a link component of positive dimension --
which offers infinitely many directions -- if \(\sfv\) lies in a
thick component. Similarly there are two such arcs \(\sfe_3,\sfe_4\)
at \(\sfv'\). Taking all four arcs short, the union
\[
\sfe_1\cup \sfe_2\cup\alpha\cup \sfe_3\cup \sfe_4
\subseteq X
\]
is an embedded copy of the tree \(\sfH\) of
\cref{ex:two-tripods-shuffle}, with essential vertices \(\sfv\)
and \(\sfv'\) and with \(\alpha\) in place of the middle bar.

Place a base configuration with \(x^0_1,x^0_2\) on
\(\mathring{\sfe}\), the point \(x^0_1\) closer to \(\sfv\)
along \(\alpha\), and with \(x^0_3,\dots,x^0_n\) parked on
\(\sfe_4\). Let \(\gamma_\sfH\in\PB_n(\sfH,\base)\) be the
braid which keeps the strands \(3,\dots,n\) parked on
\(\sfe_4\) and moves the first two through the six numbered
moves of \cref{ex:two-tripods-shuffle}
(\cref{fig:H-shuffle}), with \(\alpha\) in place of the middle
bar. The braid \(\gamma_\sfH\) is pure and supported in
\(\sfH\).

\begin{lemma}\label{lem:H-shuffle}
The braid \(\gamma_\sfH\) is non-trivial in \(\PB_n(\sfH,\base)\), and
its image in \(\PB_n(X,\base)\) is non-trivial as well;
moreover the image lies in \(\ker\iota_\ast\).
\end{lemma}

\begin{proof}
The last assertion is immediate: \(\gamma_\sfH\) is supported in the
contractible tree \(\sfH\), so all its strand classes vanish.

For the non-triviality, fix a generic point \(x\in\mathring{\sfe}\)
between \(x^0_2\) and \(\sfv'\) along \(\alpha\), apply
\cref{prop:valency-decomp} to
\(X\) at \(x\), and let
\(c_\ast\colon\PB_n(X,\base)\to
\pi_1\bigl(\sfG_n((X,x),\base)\bigr)\) be the collapse of
\cref{prop:pi1G-embedding}. Throughout \(\gamma_\sfH\), the strands
\(3,\dots,n\) rest on \(\sfe_4\), on the \(\sfv'\)-side of \(x\), while
the strands \(2\) and \(1\) cross \(x\) exactly four times, once
in each of the moves \(1\)--\(4\) of \cref{fig:H-shuffle}, while
the moves \(5\) and \(6\) take place entirely on the
\(\sfv\)-side of \(x\): in move~\(1\) the second strand crosses
\(x\) towards \(\sfv'\); in move~\(2\) the first strand crosses
\(x\) towards \(\sfv'\); in move~\(3\) the second strand crosses
\(x\) back towards \(\sfv\); and in move~\(4\) the first strand
crosses \(x\) back towards \(\sfv\)
-- the two tripods of \(\sfH\) are exactly what allows each of
the two strands to step aside for the other on either side of
\(x\). If \(X_x\) is connected, the distribution graph
\(\sfG_n((X,x),\base)\) consists, as in
\cref{cor:free-edge-splitting}, of \(n\) bigons sharing the
single \(F_n\)-type vertex; writing \(t_k\) for the loop of the
\(k\)-th bigon -- the trajectory of a braid whose \(k\)-th strand
crosses \(x\) and returns -- the crossing sequence above gives
\[
c_\ast(\gamma_\sfH)=t_2\,t_1\,t_2^{-1}\,t_1^{-1}\neq1
\]
in the free group \(\pi_1(\sfG_n)\) of rank \(n\geq2\). If
\(X_x\) is disconnected, the trajectory of \(\gamma_\sfH\) is
instead an octagonal edge loop, alternating between the four
\(F_n\)-type vertices given by the four distributions of the
strands \(1,2\) over the two sides of \(x\) -- all other strands
staying on the \(\sfv'\)-side -- and the four \(F_{n-1}\)-type
vertices realising the transfers between them; its eight vertices
are pairwise distinct, so the trajectory is an embedded cycle of
\(\sfG_n\), and again \(c_\ast(\gamma_\sfH)\neq1\). In both cases the
image of \(\gamma_\sfH\) in \(\PB_n(X,\base)\) is non-trivial,
and therefore so is \(\gamma_\sfH\in\PB_n(\sfH,\base)\) itself.
\end{proof}

\begin{corollary}\label{cor:free-edge-both-quotient}
Suppose \(x\) is an interior point of a free edge \(\sfe\) of \(X\)
with \(X_x\) disconnected, and let \(X^1,X^2\) be the two
\(x\)-components of \(X\). Then the collapse map of
\cref{prop:pi1G-embedding} induces an isomorphism
\[
\PB_n(X,\base)\big/
\co{\PB_n(X^1),\,\PB_n(X^2)}^{\PB_n(X,\base)}
\xrightarrow{\cong}
\pi_1\bigl(\sfG_n((X,x),\base)\bigr):
\]
killing both sides leaves exactly the free group of loops of the
distribution graph.
\end{corollary}

\begin{proof}
Since \(x\) is interior to an edge it has valency
\(\val_X(x)=2\), with \(\lk_X(x)=L^1\sqcup L^2\) a pair of points.
Each \(x\)-component contains \(x\) and approaches it along at
least one of \(L^1,L^2\), so the number \(m\) of \(x\)-components
satisfies \(m\le2\); as \(X_x\) is disconnected, \(m=2\). Label
the components so that \(L^1\subset X^1\) and \(L^2\subset X^2\);
then \(X^1_x\) and \(X^2_x\) are the two connected components of
\(X_x\), the two pieces meet only at \(x\), and \(x\) is a
univalent vertex of each \(X^i\) -- the tip of a free half-edge of
\(\sfe\). Fix, for \(i=1,2\), a path \(\gamma^i\) in \(F_n(X)\) from
\(\base\) to a point \(\bfx^i\in F_n(X^i_x)\), and write
\(\iota^i_\ast\colon\PB_n(X^i,\bfx^i)\to\PB_n(X,\base)\) for the
map induced by the inclusion \(X^i\hookrightarrow X\) followed by
the change of basepoint along \(\gamma^i\).

Apply \cref{prop:valency-decomp} to \(X\) at \(x\) and examine the
resulting graph-of-spaces \(\cG_n((X,x),\base)\), with underlying distribution
graph \(\sfG=\sfG_n((X,x),\base)\). Because \(X_x=X^1_x\coprod
X^2_x\) is a disjoint union, every component of \(F_n(X_x)\) --
and likewise of each \(F_{n-1}^i(X_x)\) -- splits as a product
\[
F_{n_1}\bigl(X^1_x\bigr)\times F_{n_2}\bigl(X^2_x\bigr),
\qquad n_1+n_2=n\text{ or }n-1,
\]
indexed by how the strands are distributed between the two
components; the corresponding vertex group is the direct product
\[
\PB_{n_1}\bigl(X^1_x\bigr)\times\PB_{n_2}\bigl(X^2_x\bigr).
\]
Moreover, since each link point \(L^j\) is a single point, the
attaching map carried by an edge space is the \emph{inclusion}
\(\PB_{n-1}(X_x)\hookrightarrow\PB_n(X_x)\) that inserts a strand
at the fixed point \(L^j\) adjacent to \(x\); geometrically it
slides one strand across \(x\) along the free edge \(\sfe\),
transferring it between the two components \(X^1_x\) and
\(X^2_x\).

By van Kampen, \(\PB_n(X,\base)\) is generated by the vertex
groups of \(\cG_n((X,x),\base)\), each conjugated to the basepoint
along its path in a fixed spanning tree \(\sfT\subseteq\sfG\),
together with one stable letter \(t_\sff\) for each edge
\(\sff\notin E(\sfT)\), subject to the edge relations. The collapse map
\(c\) of \cref{prop:pi1G-embedding} crushes every vertex space to
a point and sends \(t_\sff\) to the corresponding generator of
\(\pi_1(\sfG)\); killing the vertex groups renders all edge
relations trivial, so
\[
\ker c_\ast =
\co{\,\text{tree-based vertex groups}\,}^{\PB_n(X,\base)} .
\]
Furthermore \(\im\iota^i_\ast\subseteq\ker c_\ast\): as \(x\) is a
univalent vertex of \(X^i\), the link \(\lk_{X^i}(x)\) is a single
point, so in the subgraph \(\sfG^i\) underlying
\(\cG^i=\cG_n((X^i,x),\bfx^i)\)
(\cref{prop:disconnected-graph-of-spaces}(i)) every
\(F^i_{n-1}\)-type vertex is a leaf, whence \(\sfG^i\) is a tree
and \(c_\ast\) vanishes on \(\pi_1(\cG^i)=\PB_n(X^i,\bfx^i)\).
Consequently \(c_\ast\) descends to a surjection
\[
Q\coloneqq\PB_n(X,\base)\big/N
\twoheadrightarrow\pi_1(\sfG),
\qquad
N\coloneqq\co{\PB_n(X^1),\,\PB_n(X^2)}^{\PB_n(X,\base)},
\]
split by the section \(s\) of \cref{prop:pi1G-embedding}. It
remains to show injectivity, i.e.\ that every tree-based vertex
group is contained in \(N\).

Choose the spanning tree \(\sfT\) with the following two
properties. First, for every \(F_n\)-type vertex \(v\) whose
configurations carry \(n_2\geq1\) strands on the \(X^2\)-side,
select one incident long edge -- through an \(F^i_{n-1}\)-type
vertex parking one chosen \(X^2\)-side strand at \(x\) -- that
transfers this strand to the \(X^1\)-side, and put both of its
half-edges into \(\sfT\); since each such transfer strictly
decreases \(n_2\), and each vertex selects only one, these
\emph{slide edges} form a forest. Second, include spanning trees
of \(\sfG^1\) and \(\sfG^2\); as the slide edges only enter these
subgraphs at their endpoints of type \((n,0)\), the union remains
a forest, and we complete it to a spanning tree \(\sfT\). We also
take the paths \(\gamma^1,\gamma^2\) (and the section \(s\)) to be
based along \(\sfT\).

Since \(\sfT\) is spanning, every \(F^i_{n-1}\)-type vertex has an
incident tree edge, whose relation identifies its tree-based group
with a subgroup of a tree-based \(F_n\)-type vertex group; so it
suffices to treat the latter. A vertex-group element factors as
\[
(g_1,g_2)=(g_1,1)\cdot(1,g_2),\qquad
g_1\in\PB_{n_1}\bigl(X^1_x\bigr),
g_2\in\PB_{n_2}\bigl(X^2_x\bigr),
\]
and we treat the two factors separately, by induction on \(n_2\)
(resp.\ \(n_1\)). For \((g_1,1)\) with \(n_2\geq1\): the parked
strand of the slide edge at \(v\) is stationary in \((g_1,1)\), so
\((g_1,1)\) lies in the image of the edge group of the slide edge;
as both halves of the slide edge belong to \(\sfT\), the two edge
relations carry the tree-based \((g_1,1)\) to the tree-based
element \((g_1',1)\) of the vertex group at the other end -- the
same braid with the transferred strand stationary on the
\(X^1\)-side -- with \(n_2\) decreased by one. After \(n_2\) steps
we reach a vertex of type \((n,0)\), i.e.\ a vertex group of
\(\cG^1\); since \(\sfT\) restricts to a spanning tree of
\(\sfG^1\) and \(\gamma^1\) runs along \(\sfT\), the tree-based
image lies in \(\im\iota^1_\ast\subseteq N\). For \((1,g_2)\) with
\(n_1\geq1\): choose any long edge at \(v\) parking an
\(X^1\)-side strand at \(x\) and transferring it to the
\(X^2\)-side. Its two halves need not lie in \(\sfT\), but the
corresponding edge relations then hold up to conjugation by the
stable letters of the missing halves; hence the tree-based
\((1,g_2)\) is a conjugate, by a word in the stable letters, of
the tree-based \((1,g_2')\) at the other end, with \(n_1\)
decreased by one. After \(n_1\) steps, the tree-based \((1,g_2)\)
is a conjugate of an element of \(\im\iota^2_\ast\) by a word in
the stable letters; as \(N\) is normal, \((1,g_2)\in N\). Hence
every tree-based vertex group lies in \(N\), completing the proof.
(For later use we record what the induction showed: every
tree-based vertex-group element is a product of conjugates of
elements of \(\im\iota^1_\ast\cup\im\iota^2_\ast\) by words in
the stable letters, hence by elements of \(\im s\).)
\end{proof}

\begin{corollary}\label{cor:shuffle-survives-quotient}
In the situation of \cref{cor:free-edge-both-quotient}, suppose
in addition that \(X\) contains an embedded tree \(\sfH\), as in
the construction preceding \cref{lem:H-shuffle}, whose path
\(\alpha\) contains the free edge \(\sfe\). Then the
image of the shuffle braid \(\gamma_\sfH\) under the isomorphism
of \cref{cor:free-edge-both-quotient} is the class of the
embedded octagonal loop appearing in the proof of
\cref{lem:H-shuffle}; in particular
\[
\gamma_\sfH\notin
\co{\PB_n(X^1),\,\PB_n(X^2)}^{\PB_n(X,\base)}:
\]
the shuffle braid survives the quotient by both sides, for every
\(n\geq2\).
\end{corollary}

\begin{proof}
Take the interior point of \(\sfe\) used in the proof of
\cref{lem:H-shuffle} as the point \(x\). By
\cref{cor:free-edge-both-quotient} the quotient map is
identified with the collapse \(c_\ast\), and the proof of
\cref{lem:H-shuffle} computes \(c_\ast(\gamma_\sfH)\) to be the
class of an embedded -- hence essential -- cycle of
\(\sfG_n((X,x),\base)\).
\end{proof}

\begin{corollary}\label{cor:free-edge-disconnected}
Suppose \(x\) is an interior point of a free edge \(e\) of \(X\) and that
\(X_x\) is disconnected. Then \(X\) has exactly two \(x\)-components
\(X^1\) and \(X^2\), meeting only at \(x\) (so \(X=X^1\cup X^2\) with
\(X^1\cap X^2=\{x\}\)), and the inclusions \(X^1,X^2\hookrightarrow X\),
together with the section
\(s\colon\pi_1\bigl(\sfG_n((X,x),\base)\bigr)\to\PB_n(X,\base)\) of
\cref{prop:pi1G-embedding}, induce a surjection
\[
\Psi\colon\PB_n(X^1)\,\ast\,\PB_n(X^2)\,\ast\,
\pi_1\bigl(\sfG_n((X,x),\base)\bigr)
\twoheadrightarrow\PB_n(X,\base)
\]
which is injective on each free factor: \(\PB_n(X^1)\),
\(\PB_n(X^2)\) and the free group
\(\pi_1\bigl(\sfG_n((X,x),\base)\bigr)\) all embed into
\(\PB_n(X,\base)\). The third factor is essential: in general
\(\PB_n(X,\base)\) is \emph{not} generated by the two sides alone
(\cref{ex:two-tripods-shuffle} for \(n=2\), and
\cref{cor:shuffle-survives-quotient} for every \(n\geq2\)).
\end{corollary}

\begin{proof}
Retain the notation of the proof of
\cref{cor:free-edge-both-quotient}; in particular, for \(i=1,2\),
the path \(\gamma^i\) runs from
\(\base\) to a point \(\bfx^i\in F_n(X^i_x)\). 
The
inclusion \(X^i\hookrightarrow X\) induces
\(\iota^i_\ast\colon\PB_n(X^i,\bfx^i)\to\PB_n(X,\base)\) followed by the change of basepoints along \(\gamma^i\).
By
the universal property of the free product these assemble, together
with the section \(s\) of \cref{prop:pi1G-embedding}, into a
single homomorphism
\[
\Psi\colon\PB_n(X^1,\bfx^1)\ast\PB_n(X^2,\bfx^2)\ast
\pi_1\bigl(\sfG_n((X,x),\base)\bigr)\longrightarrow
\PB_n(X,\base) .
\]

It remains to show that \(\Psi\) is surjective. By van Kampen (as
recalled in the proof of \cref{cor:free-edge-both-quotient}),
\(\PB_n(X,\base)\) is generated by the tree-based vertex groups of
\(\cG_n((X,x),\base)\) together with the stable letters \(t_\sff\),
\(\sff\notin E(\sfT)\); and every \(t_\sff\) lies in \(\im s\), since \(s\)
sends the loop of \(\sfG\) determined by \(\sff\) (closed up with
\(\sfT\)-paths) to \(t_\sff\). By the inductive argument in the proof
of \cref{cor:free-edge-both-quotient}, every tree-based
vertex-group element is a product of conjugates of elements of
\(\im\iota^1_\ast\cup\im\iota^2_\ast\) by elements of
\(\im s\). Hence every generator of
\(\PB_n(X,\base)\) lies in the subgroup
\(\langle\im\iota^1_\ast,\im\iota^2_\ast,\im s\rangle\), and
\(\Psi\) is surjective.

Finally, each free factor embeds. Since \(X^1\) and \(X^2\) are
precisely the two \(x\)-components of \(X\), and the link
components of \(x\) are single points,
\cref{lem:x-component-embedding}, applied to each of them
in turn (with basepoint \(\bfx^i\in F_n(X^i_x)\)), shows that the
inclusion-induced map
\(\PB_n(X^i,\bfx^i)\to\PB_n(X,\bfx^i)\) is injective; composing with
the change-of-basepoint isomorphism
\(\PB_n(X,\bfx^i)\cong\PB_n(X,\base)\) along \(\gamma^i\) shows that
each \(\iota^i_\ast\) is injective. The third factor embeds by
\cref{prop:pi1G-embedding}: the section \(s\) is split injective.
Thus \(\Psi\) restricts to an embedding on each of the three free
factors. The necessity of the third factor is witnessed by
\cref{ex:two-tripods-shuffle}.
\end{proof}

Under a mild non-degeneracy hypothesis the third factor can be
computed exactly.

\begin{corollary}\label{cor:free-edge-loop-rank}
In the situation of \cref{cor:free-edge-disconnected}, suppose
moreover that neither \(X^1\) nor \(X^2\) is homeomorphic to an
interval. Then the distribution graph \(\sfG_n((X,x),\base)\) is
connected, with \(2^n+n\,2^{n-1}\) vertices and \(2n\cdot2^{n-1}\)
edges, and the third free factor of
\cref{cor:free-edge-disconnected} -- equivalently, by
\cref{cor:free-edge-both-quotient}, the quotient of
\(\PB_n(X,\base)\) by the normal closure of both sides -- is free
of rank
\[
1-\chi\bigl(\sfG_n((X,x),\base)\bigr)=1+2^{n-1}(n-2).
\]
\end{corollary}

\begin{proof}
Each \(X^i\) contains \(x\) as a univalent vertex -- the tip of the
free half-edge of \(e\) -- so \(X^i\not\cong S^1\); being connected
and homeomorphic to neither an interval nor \(S^1\), it contains a
tripod \(\sfS_3\), which persists in \(X^i_x\). Hence
\(F_m(X^i_x)\) is connected for every \(m\geq0\): any two
configurations with the same labels are joined by sliding the
strands one at a time, using the tripod to let strands pass one
another (cf.~\cite{Abrams2000}). Consequently the connected
components of \(F_n(X_x)\) are exactly the \(2^n\) distributions of
the \(n\) labels over the two sides, and likewise each
\(F^i_{n-1}(X_x)\) has \(2^{n-1}\) components. The vertex and edge
count of \cref{prop:valency-decomp}, with \(k=2\) link points,
gives
\[
\#V = 2^n+n\,2^{n-1},
\qquad
\#E = 2n\cdot 2^{n-1},
\]
whence \(\chi=2^n-n\,2^{n-1}=2^{n-1}(2-n)\). The graph is
connected, since any distribution is reached from any other by
transferring strands across \(x\) one at a time; in particular
\(\sfG_n((X,x),\base)=\sfG_n(X,x)\). A connected graph has free
fundamental group of rank \(1-\chi\), which is the stated value.
\end{proof}

\subsection{Local graph embeddings}
\label{ssec:local-graph-embeddings}

Beyond the free-edge situation, we expect the picture at a general
point \(x\) to be governed by the star spanned by the
\(x\)-components. Suppose \(X\) has \(m\geq 2\) \(x\)-components
\(X^1,\dots,X^m\). Every \(x\)-component contains at least one
connected component of \(\lk_X(x)\), so we may pick link points
\(p^i\in\lk_X(x)\cap X^i\) for \(i=1,\dots,m\), and let
\(a^i\subseteq X^i\) be a short embedded arc from \(x\) through
\(p^i\), extending slightly past \(p^i\). As the \(p^i\) lie in
distinct \(x\)-components, these arcs pairwise meet only at \(x\),
and their union
\[
\sfS_m\coloneqq a^1\cup\cdots\cup a^m\subseteq X
\]
is an embedded star with central vertex \(x\) and one leaf running
into each \(x\)-component; the induced link map
\(\lk_{\sfS_m}(x)\to\lk_X(x)\) is injective, sending the \(i\)-th
leaf direction to \(p^i\).

\begin{proposition}\label{prop:star-embedding}
Let \(x\in X\) have \(m\geq 2\) \(x\)-components, and let
\(\sfS_m\subseteq X\) be the embedded star constructed above. For
any basepoint \(\bfy^0\in F_n(\sfS_m)\), identified with \(\base\)
along a path in \(F_n(X)\), the inclusion
\(\sfS_m\hookrightarrow X\) induces an embedding
\[
\PB_n(\sfS_m,\bfy^0)\hookrightarrow\PB_n(X,\base),
\]
and in fact an embedding into the quotient by all of the sides:
\[
\PB_n(\sfS_m,\bfy^0)\hookrightarrow
\PB_n(X,\base)\big/
\co{\PB_n(X^1),\dots,\PB_n(X^m)}^{\PB_n(X,\base)},
\]
where each \(\PB_n(X^i,\bfx^i)\) is transported into the basepoint
along any path (the normal closure is independent of these
choices).
\end{proposition}

\begin{proof}
Write \(\cG_\sfS=\cG_n((\sfS_m,x),\bfy^0)\) and
\(\sfG_\sfS=\sfG_n((\sfS_m,x),\bfy^0)\) for the graph-of-spaces
decomposition of \(F_n(\sfS_m,\bfy^0)\) at \(x\) provided by
\cref{prop:valency-decomp}, and \(\cG=\cG_n((X,x),\bfy^0)\),
\(\sfG=\sfG_n((X,x),\bfy^0)\) for that of \(F_n(X,\bfy^0)\). We
prove that the inclusion induces an embedding
\(\PB_n(\sfS_m,\bfy^0)\hookrightarrow\PB_n(X,\bfy^0)\); the change
of basepoint along the chosen path then yields the statement.

\smallskip
\noindent\emph{Step 1: \(\PB_n(\sfS_m,\bfy^0)\cong\pi_1(\sfG_\sfS)\).}
The punctured star \((\sfS_m)_x\) is a disjoint union of \(m\)
half-open arcs, so every connected component of
\(F_n((\sfS_m)_x)\) and of the \(F^i_{n-1}((\sfS_m)_x)\) is a
product of configuration spaces of arcs with linearly ordered
strands, hence contractible; the link of \(x\) in \(\sfS_m\)
consists of the \(m\) leaf directions, so the edge spaces are
contractible as well. The collapse map
\(c_\sfS\colon\cG_\sfS\to\sfG_\sfS\), sending each vertex space to its
vertex and each edge cylinder onto its edge (as in the proof of
\cref{prop:pi1G-embedding}), is therefore a homotopy equivalence,
and \(\PB_n(\sfS_m,\bfy^0)=\pi_1(\cG_\sfS)\cong\pi_1(\sfG_\sfS)\).

\smallskip
\noindent\emph{Step 2: the induced map of distribution graphs.}
Both decompositions are produced by cutting along the same
neighbourhoods \(V_x\subseteq U_x\) of \(x\) (Steps~1--2 of the
proof of \cref{prop:valency-decomp}), so the inclusion
\(F_n(\sfS_m)\hookrightarrow F_n(X)\) respects the pieces and
induces a map of graphs-of-spaces \(\cG_\sfS\to\cG\) over a graph map
\(h\colon\sfG_\sfS\to\sfG\): a vertex -- a component of
\(F_n((\sfS_m)_x)\) or of some \(F^i_{n-1}((\sfS_m)_x)\) -- is sent
to the component of \(F_n(X_x)\), resp.\ \(F^i_{n-1}(X_x)\),
containing it, and the edge of \(\sfG_\sfS\) indexed by
\((i,C_\sfS,a)\) -- where \(C_\sfS\in\pi_0(F^i_{n-1}((\sfS_m)_x))\) and
\(a\in\{1,\dots,m\}\) is a leaf direction -- is sent to the edge of
\(\sfG\) indexed by \((i,C_X,L^{j(a)})\), where \(C_X\supseteq C_\sfS\)
and \(L^{j(a)}\subseteq\lk_X(x)\) is the link component containing
the leaf point \(p^a\). In particular the square
\[
\begin{tikzcd}
\cG_\sfS \arrow[r] \arrow[d,"c_\sfS"'] & \cG \arrow[d,"c"] \\
\sfG_\sfS \arrow[r,"h"] & \sfG
\end{tikzcd}
\]
commutes up to homotopy, where \(c\colon\cG\to\sfG\) is the
collapse map of the \(X\)-side decomposition.

\smallskip
\noindent\emph{Step 3: \(h\) is an immersion.}
We claim that \(h\) is locally injective, i.e.\ injective on the
set of edges incident to each vertex of \(\sfG_\sfS\). At a vertex of
type \(F^i_{n-1}\), say \((i,C_\sfS)\), the incident edges are
\((i,C_\sfS,a)\) for \(a=1,\dots,m\); their images
\((i,C_X,L^{j(a)})\) are pairwise distinct because the leaves of
\(\sfS_m\) run into pairwise distinct \(x\)-components of \(X\),
and distinct \(x\)-components contain disjoint collections of link
components, so \(L^{j(1)},\dots,L^{j(m)}\) are pairwise distinct.
At a vertex \(\sfv\) of type \(F_n\), an edge \((i,C_\sfS,a)\) is
incident to \(\sfv\) only if the strand \(i\) is the innermost strand
on the arm \(a\) in the configurations of \(\sfv\), and then \(C_\sfS\)
is determined by \(\sfv\) and \(i\) (delete the strand \(i\)); hence
the incident edges are indexed by the non-empty arms \(a\) of \(\sfv\)
together with their innermost strands \(i_a\), and for distinct
arms both coordinates \(i\) and \(L^{j(a)}\) of the images differ.
In either case \(h\) is injective on incident edges.

\smallskip
\noindent\emph{Step 4: conclusion.}
A locally injective map of graphs sends reduced edge paths to
reduced edge paths -- a backtrack in the image at some vertex would
exhibit two distinct incident edges with the same image -- and a
reduced loop in a graph is essential; hence
\(h_\ast\colon\pi_1(\sfG_\sfS)\to\pi_1(\sfG)\) is injective
(cf.\ Stallings~\cite{Stallings1983}). By the square of Step~2,
\[
c_\ast\circ\iota_\ast
=h_\ast\circ(c_\sfS)_\ast\colon
\pi_1(\cG_\sfS)\longrightarrow\pi_1(\sfG),
\]
where
\(\iota_\ast\colon\PB_n(\sfS_m,\bfy^0)=\pi_1(\cG_\sfS)\to
\pi_1(\cG)=\PB_n(X,\bfy^0)\) is induced by the inclusion. The
right-hand side is injective by Steps~1 and~3, hence so is
\(\iota_\ast\).

\smallskip
\noindent\emph{Step 5: descending to the quotient by all sides.}
Write
\(N\coloneqq\co{\PB_n(X^1),\dots,\PB_n(X^m)}^{\PB_n(X,\base)}\),
with each \(\PB_n(X^i,\bfx^i)\) transported into the basepoint
along some path; \(N\) is independent of these choices. For each
\(i\), let \(\sfG^i\subseteq\sfG\) be the subgraph underlying
\(\cG^i=\cG_n((X^i,x),\bfx^i)\)
(\cref{prop:disconnected-graph-of-spaces}(i)), and let
\(\bar\sfG\) be the graph obtained from \(\sfG\) by collapsing each
\(\sfG^i\) to a point, so that
\(\pi_1(\sfG)\to\pi_1(\bar\sfG)\) has kernel
\(\co{\pi_1(\sfG^1),\dots,\pi_1(\sfG^m)}\). Since the collapse map
\(c\) carries \(\cG^i\) into \(\sfG^i\), the composite
\[
\bar\varphi\colon\PB_n(X,\base)
\xrightarrow{c_\ast}\pi_1(\sfG)
\longrightarrow\pi_1(\bar\sfG)
\]
kills \(N\), hence descends to
\(\PB_n(X,\base)/N\). It therefore suffices to prove that
\(\bar\varphi\circ\iota_\ast\) is injective.

Let \(M\) be the set of edges of \(\sfG_\sfS\) whose \(h\)-image lies
inside some \(\sfG^i\). The vertices of \(\sfG_\sfS\) mapping into
\(\bigcup_i\sfG^i\) are those all of whose strands lie on a single
arm \(a^i\); such an \(F_n\)-type vertex, with the strands in the
order \(\tau\) along \(a^i\), is incident to exactly one edge of
\(\sfG_\sfS\) -- the one inserting the innermost strand \(\tau_1\)
from the direction of \(a^i\) -- and such an \(F^j_{n-1}\)-type
vertex carries exactly one edge towards \(a^i\) as well. Hence
\(M\) is a matching, and collapsing its edges does not change the
fundamental group: \(\pi_1(\sfG_\sfS)\cong\pi_1(\sfG_\sfS/M)\), and
\(h\) induces a graph map \(\bar h\colon\sfG_\sfS/M\to\bar\sfG\).

The map \(\bar h\) is again an immersion. At a vertex of
\(\sfG_\sfS/M\) that is not a collapsed edge, the star and the images
of its edges are as in Step~3; none of these edges is collapsed in
\(\bar\sfG\) (each has an endpoint carrying strands on at least
two arms, so its image is not internal to any \(\sfG^i\)), and
distinct edges of \(\sfG\) not internal to the \(\sfG^i\) remain
distinct in \(\bar\sfG\). At a collapsed vertex, arising from a
matching edge with \(F^j_{n-1}\)-type end \((j,\tau')\) on the arm
\(a^i\), the star consists of the \(m-1\) edges inserting the
strand \(j\) into the arms \(a^{i'}\) with \(i'\neq i\); their
images \((j,C_X,L^{j(i')})\) are pairwise distinct because the
link components \(L^{j(i')}\) are. Hence \(\bar h\) is locally
injective, so it carries reduced loops to reduced loops and
\(\bar h_\ast\colon\pi_1(\sfG_\sfS/M)\to\pi_1(\bar\sfG)\) is
injective. Combining with Step~1 and the collapse square of
Step~2,
\(\bar\varphi\circ\iota_\ast
=\bar h_\ast\circ(\pi_1(\sfG_\sfS)\cong\pi_1(\sfG_\sfS/M))
\circ(c_\sfS)_\ast\) is injective, and therefore the composite of
\(\iota_\ast\) with the projection
\(\PB_n(X,\base)\twoheadrightarrow\PB_n(X,\base)/N\) is injective
as well, proving the second embedding.

Note that the argument uses neither the \(\pi_1\)-injectivity of
the attaching maps of \(\cG\) nor any hypothesis on the link
components of \(x\): on the \(X\)-side, only the existence of the
collapse map \(c\) is needed.
\end{proof}

\begin{remark}\label{rem:star-link-components}
The two assertions of \cref{prop:star-embedding} use different
parts of the hypothesis. The first embedding,
\(\PB_n(\sfS_m,\bfy^0)\hookrightarrow\PB_n(X,\base)\), is
established by Steps~1--4, which use only the pairwise
distinctness of the link components entered by the leaves; it
therefore holds verbatim for \emph{any} embedded star at \(x\)
with at least two leaves entering pairwise distinct connected
components of \(\lk_X(x)\), regardless of how these link
components are distributed among the \(x\)-components. The
second embedding -- into the quotient by all sides -- is
established by Step~5, which makes essential use of one leaf in
each \(x\)-component.
\end{remark}

When \(m=2\) and \(\val_X(x)\geq3\), the star of
\cref{prop:star-embedding} does not see the extra link directions
on a single side; the appropriate probe is a \emph{lollipop}.
Suppose \(X\) has exactly two \(x\)-components \(X^1,X^2\) and
\(\val_X(x)\geq3\), so that one of them, say \(X^1\), contains two
distinct connected components \(L^1_1\neq L^1_2\) of \(\lk_X(x)\).
Pick link points \(p^1_1\in L^1_1\), \(p^1_2\in L^1_2\) and
\(p^2\in\lk_{X^2}(x)\). Since \(X^1_x\) is connected, the two short
arcs from \(x\) through \(p^1_1\) and \(p^1_2\) may be joined by an
embedded arc in \(X^1_x\), producing an embedded cycle
\(\mathsf{C}\subseteq X^1\) through \(x\); together with a short
pendant arc \(\mathsf{e}\subseteq X^2\) through \(p^2\), this
yields an embedded lollipop
\(\Lollipop=\mathsf{C}\cup\mathsf{e}\subseteq X\)
(\cref{fig:lollipop}) with trivalent vertex \(x\).

\begin{proposition}\label{prop:lollipop-embedding}
With \(\Lollipop\subseteq X\) as above and any basepoint
\(\bfy^0\in F_n(\Lollipop)\), identified with \(\base\) along a
path in \(F_n(X)\), the inclusion \(\Lollipop\hookrightarrow X\)
induces an embedding
\[
\PB_n(\Lollipop,\bfy^0)\hookrightarrow\PB_n(X,\base),
\]
and in fact embeddings into the quotients
\[
\PB_n(X,\base)\big/\co{\PB_n(X^2,\bfx^2)}^{\PB_n(X,\base)}
\quad\text{and}\quad
\PB_n(X,\base)\big/
\co{\PB_n(X^1_x),\,\PB_n(X^2_x)}^{\PB_n(X,\base)},
\]
by the pendant side and by both \emph{punctured} sides
respectively, all subgroups being transported into the basepoint
along arbitrary paths.
\end{proposition}
\begin{remark}
No such statement holds for the quotient
by the full cycle side \(X^1\): braids rotating all strands
around \(\mathsf{C}\subseteq X^1\) are supported in \(X^1\) and
die in that quotient. The point of the punctured sides is
precisely that braids supported in \(X^i_x\) never pass through
\(x\).
\end{remark}

\begin{proof}
Steps~1--4 of the proof of \cref{prop:star-embedding} apply with
the following adjustments. The punctured lollipop
\(\Lollipop_x=(\mathsf{C}\setminus\{x\})\sqcup
(\mathsf{e}\setminus\{x\})\) is a disjoint union of two arcs, so
every vertex space and every edge space of
\(\cG_\Lollipop\coloneqq\cG_n((\Lollipop,x),\bfy^0)\) is
contractible and
\(\PB_n(\Lollipop,\bfy^0)\cong\pi_1(\sfG_\Lollipop)\), where
\(\sfG_\Lollipop\) is the distribution graph. The link of \(x\) in
\(\Lollipop\) consists of three directions: the two ends
\(d^1_1,d^1_2\) of \(\mathsf{C}\), entering \(\lk_X(x)\) at
\(p^1_1,p^1_2\), and the end \(d^2\) of \(\mathsf{e}\), entering
at \(p^2\). As in Step~2, the two decompositions are cut along the
same neighbourhoods of \(x\), so the inclusion induces a map of
graphs-of-spaces over a graph map
\(h\colon\sfG_\Lollipop\to\sfG=\sfG_n((X,x),\bfy^0)\), commuting
with the collapse maps up to homotopy.

It remains to check that \(h\) is an immersion (Step~3). At an
\(F^i_{n-1}\)-type vertex the three incident edges, one for each
direction, have images \((i,C_X,L)\) with
\(L\in\{L^1_1,L^1_2,L^2\}\), where \(L^2\subseteq\lk_{X^2}(x)\) is
the link component containing \(p^2\); these three link components
are pairwise distinct -- \(L^1_1\neq L^1_2\) by the choice of
\(p^1_1,p^1_2\), and \(L^2\) lies in the other \(x\)-component --
so the images are pairwise distinct. At an \(F_n\)-type vertex the
incident edges are indexed by the innermost strand of
\(\mathsf{e}\setminus\{x\}\) with the direction \(d^2\), and the
two extreme strands of the arc \(\mathsf{C}\setminus\{x\}\) with
the directions \(d^1_1\) and \(d^1_2\). Two distinct such edges
either involve distinct strands, or -- when
\(\mathsf{C}\setminus\{x\}\) carries a single strand, extreme on
both sides -- the same strand with the two distinct directions
\(d^1_1,d^1_2\); in either case the images \((i,C_X,L)\) differ in
the strand index or in the link component. Hence \(h\) is locally
injective, and Step~4 concludes: \(h_\ast\) is injective, and so
is \(\iota_\ast\colon\PB_n(\Lollipop,\bfy^0)\to\PB_n(X,\base)\).

\smallskip
For the quotient statement we follow Step~5 of the proof of
\cref{prop:star-embedding}, with only the side \(X^2\) collapsed.
Let \(\sfG^2\subseteq\sfG\) be the subgraph underlying
\(\cG^2=\cG_n((X^2,x),\bfx^2)\), let \(\bar\sfG=\sfG/\sfG^2\), and
let \(M\) be the set of edges of \(\sfG_\Lollipop\) whose image
lies inside \(\sfG^2\). A vertex of \(\sfG_\Lollipop\) mapping
into \(\sfG^2\) has all its strands on the arc
\(\mathsf{e}\setminus\{x\}\); such an \(F_n\)-type vertex is a
leaf of \(\sfG_\Lollipop\) (its only incident edge inserts the
innermost strand of \(\mathsf{e}\)), and such an
\(F^j_{n-1}\)-type vertex carries exactly one edge in the
direction \(d^2\). Hence \(M\) is a matching and
\(\pi_1(\sfG_\Lollipop)\cong\pi_1(\sfG_\Lollipop/M)\). The
induced map \(\bar h\colon\sfG_\Lollipop/M\to\bar\sfG\) is again
an immersion: at a collapsed vertex the star consists of the two
edges of the \(F^j_{n-1}\)-type end in the directions
\(d^1_1,d^1_2\), whose images differ in the link components
\(L^1_1\neq L^1_2\); at the remaining vertices the star is as
before, and no edge outside \(M\) is collapsed in \(\bar\sfG\)
(each has an endpoint carrying a strand on the
\(\mathsf{C}\)-side). As in Step~5 of \cref{prop:star-embedding},
reduced loops survive, and the composite of \(\iota_\ast\) with
the projection onto
\(\PB_n(X,\base)/\co{\PB_n(X^2,\bfx^2)}^{\PB_n(X,\base)}\) is
injective.

\smallskip
The quotient by the punctured sides is simpler still. A braid
supported in \(X^i_x\) keeps all strands away from \(x\)
throughout, so after the retraction of Step~1 of the proof of
\cref{prop:valency-decomp} its loop is contained in a single
vertex space of \(\cG_n((X,x),\base)\); hence \(c_\ast\) kills the
image of \(\PB_n(X^i_x)\) (transported to the basepoint), and
\(\co{\PB_n(X^1_x),\,\PB_n(X^2_x)}^{\PB_n(X,\base)}
\subseteq\ker c_\ast\). Since \(c_\ast\circ\iota_\ast\) is
injective by Steps~1--4, so is the composite of \(\iota_\ast\)
with the quotient by any normal subgroup contained in
\(\ker c_\ast\); in particular the second displayed embedding
holds.
\end{proof}

For two strands, a trivalent point can be thickened to a disc
without changing the pure braid group.

Let \(x\in X\) be a point admitting a regular neighbourhood
\(N\cong\sfS_3\), a tripod with centre \(x\) and arms
\(\sfe^1,\sfe^2,\sfe^3\). Let \(I\subseteq N\) denote the interval through
\(x\) formed by the halves of \(\sfe^1\) and \(\sfe^2\) adjacent to
\(x\), and let
\[
X_D\coloneqq X\cup_I D
\]
be obtained by gluing a \(2\)-disc \(D\) to \(X\) along an arc of
\(\partial D\) identified with \(I\); up to homeomorphism,
\(X_D\) is the complex obtained from \(X\) by replacing \(x\) with
a small disc whose boundary carries the three arms. 

\begin{lemma}\label{lem:tripod-to-disc}
For
\(n\leq2\), the inclusion \(X\hookrightarrow X_D\) induces an
isomorphism
\[
\PB_n(X)\xrightarrow{\cong}\PB_n(X_D).
\]
\end{lemma}

\begin{proof}
For \(n=1\) the claim is immediate: \(D\) collapses onto \(I\), so
the inclusion is a homotopy equivalence. Let \(n=2\); we induct on
the number of non-pendant maximal free chains of \(X\).

If \(X\) has no non-pendant maximal free chain, then in particular the
maximal free chains containing the three arms \(\sfe^i\) are
pendant, ending in leaves; as \(X\) is connected and every exit
from \(x\) runs through these dead-ending arms, \(X=\sfS_3\), and
\(X_D\) is a disc
with three pendant edges attached along its boundary. Both groups
are then infinite cyclic (\cref{ex:S3,ex:disc}), and the inclusion
carries the square of the hexagonal exchange to a full twist
generating the target, so
\(\Z\cong\PB_2(X)\xrightarrow{\cong}\PB_2(X_D)\cong\Z\).

Suppose now that \(X\) has \(k+1\geq1\) non-pendant maximal free chains.
Choose a non-pendant maximal free chain \(\sfe\) and an interior point
\(y\in\mathring{\sfe}\) away from \(N\). Use compact truncations
at the cut when applying the induction hypothesis, as in
\cref{conv:subcomplex}, and compare the
decompositions of \(F_n(X)\) and \(F_n(X_D)\) at \(y\)
(\cref{prop:valency-decomp}). Suppose first that \(X_y\) is
connected; then so is \((X_D)_y\), and the inclusion induces, for
\(m=1,2\), commutative squares
\[
\begin{tikzcd}
\PB_m(X_y) \ar[r] \ar[d] & \PB_m(X) \ar[d] \\
\PB_m\bigl((X_D)_y\bigr) \ar[r] & \PB_m(X_D)
\end{tikzcd}
\]
of vertex groups. Puncturing \(\sfe\) turns it into two pendant
half-edges, so \(X_y\) has at most \(k\) non-pendant maximal free chains,
while
the regular neighbourhood of \(x\) is unchanged; the induction
hypothesis gives \(\PB_2(X_y)\cong\PB_2((X_D)_y)\), and
\(\PB_1(X_y)\cong\PB_1((X_D)_y)\) holds always. The pendant-edge
stabilisations occurring in the two decompositions are compatible
with the inclusion \(X\hookrightarrow X_D\), so the inclusion
identifies the two graph-of-groups structures, and
\(\PB_2(X)\cong\PB_2(X_D)\).

If instead \(X_y=X^1_y\sqcup X^2_y\) is disconnected, then so is
\((X_D)_y=(X_D^1)_y\sqcup(X_D^2)_y\), where we label the pieces so
that \(x\in X^1\) and the disc \(D\) lies on the side of
\(X_D^1\). The vertex groups of the two decompositions are the
groups \(\PB_m(X^1_y)\) and \(\PB_m(X^2_y)\) for \(m=1,2\),
together with the mixed products
\(\PB_1(X^1_y)\times\PB_1(X^2_y)\), and correspondingly for
\(X_D\); here \(\PB_m(X^2_y)\cong\PB_m((X_D^2)_y)\) on the nose,
and \(\PB_m(X^1_y)\cong\PB_m((X_D^1)_y)\) by the induction
hypothesis. Hence the inclusion identifies the graph-of-groups
structures of \(\PB_n(X)\) and \(\PB_n(X_D)\), and induces an
isomorphism as claimed.
\end{proof}

For arbitrarily many strands, a star with at least three legs
still carries the entire braid group of a disc:

\begin{lemma}\label{lem:star-to-disc}
Let \(\sfS_m\subseteq D^2\) be an embedded star with
\(m\geq3\) legs. Then the inclusion induces a surjection
\[
\PB_n(\sfS_m,\bfy^0)\;\twoheadrightarrow\;\PB_n(D^2,\bfy^0)
\]
for every \(n\geq1\) and every base configuration
\(\bfy^0\in F_n(\sfS_m)\).
\end{lemma}

\begin{proof}
Since \(\sfS_3\subseteq\sfS_m\), the configuration space
\(F_n(\sfS_m)\) is connected (\cref{lem:rho-surj}), so the
choice of base configuration is immaterial; place
\(\bfy^0=(y_1,\dots,y_n)\) on the first leg \(\ell_1\), in this
order from the centre outwards. With respect to such a linearly
ordered configuration, \(\BG_n(D^2,[\bfy^0])\) is generated by
the half-twists \(\sigma_1,\dots,\sigma_{n-1}\), where
\(\sigma_i\) exchanges \(y_i\) and \(y_{i+1}\) along the
subsegment of \(\ell_1\) joining them.

We first show that each \(\sigma_i\) lies in the image of
\(\BG_n(\sfS_m,[\bfy^0])\). Let \(g_i\) be the following braid
of the star: slide \(y_1,\dots,y_{i-1}\), in order, through the
centre and far out onto the second leg \(\ell_2\); exchange
\(y_i\) and \(y_{i+1}\) through the inner portions of
\(\ell_2\) and \(\ell_3\); then return \(y_1,\dots,y_{i-1}\) to
their original slots. As an element of the fundamental group,
\(g_i\) is the exchange loop conjugated by the parking path,
so its image in \(\BG_n(D^2,[\bfy^0])\) is the half-twist
about the image of the segment \([y_i,y_{i+1}]\subseteq\ell_1\)
under the parking isotopy of the disc. That isotopy is
supported near the parking track -- the portion of \(\ell_1\)
between the centre and \(y_{i-1}\), the centre, and \(\ell_2\)
-- which is disjoint from the segment \([y_i,y_{i+1}]\); hence
the segment is fixed and the image of \(g_i\) is exactly
\(\sigma_i^{\pm1}\). Consequently
\(\BG_n(\sfS_m,[\bfy^0])\to\BG_n(D^2,[\bfy^0])\) is
surjective.

Finally, let \(\beta\in\PB_n(D^2,\bfy^0)\) and choose
\(g\in\BG_n(\sfS_m,[\bfy^0])\) mapping to \(\beta\). The
underlying permutation of a braid agrees with that of its
image, so \(g\) is pure, and
\(\PB_n(\sfS_m,\bfy^0)\to\PB_n(D^2,\bfy^0)\) is surjective.
\end{proof}

Finally, we record a variant in which the quotient is taken by a
subcomplex that avoids one of the star's directions.

Let \(L^1,L^2,L^3\) be pairwise distinct connected components of
\(\lk_X(x)\), let \(\sfS_3\subseteq X\) be an embedded star at
\(x\) with leaves \(a^1,a^2,a^3\) through points \(p^j\in L^j\),
and let \(Z\subseteq X\) be a connected subcomplex with
\(Z\cap A_{L^1}=\varnothing\), where
\(A_{L^1}\subseteq U_x\setminus\{x\}\) denotes the cone direction
over \(L^1\); fix a basepoint \(\bfz^0\in F_n(Z)\) identified with
\(\base\) along a path in \(F_n(X)\). 

Suppose that
\begin{enumerate}[label=(\alph*),leftmargin=2em]
\item the \(x\)-components containing \(L^2\) and \(L^3\) are
      distinct -- the \(x\)-component of \(L^1\) may coincide with
      either -- or
\item \(Z\) meets the \(x\)-component \(X^1\) containing \(L^1\)
      only in \(\{x\}\); in this case no condition is imposed on
      the \(x\)-components of \(L^2,L^3\), and
      \(Z\cap X^1=\{x\}\) already implies
      \(Z\cap A_{L^1}=\varnothing\).
\end{enumerate}

\begin{proposition}\label{prop:star-avoiding-embedding}
For any basepoint
\(\bfy^0\in F_n(\sfS_3)\) identified with \(\base\) along a path in
\(F_n(X)\), the inclusion induces an embedding
\[
\PB_n(\sfS_3,\bfy^0)\hookrightarrow
\PB_n(X,\base)\big/
\co{\,\im\bigl(\PB_n(Z,\bfz^0)\to\PB_n(X,\base)\bigr)}^{\PB_n(X,\base)}.
\]
\end{proposition}

\begin{proof}
Steps~1--4 of the proof of \cref{prop:star-embedding} apply to
\(\sfS_3\) verbatim -- the three leaves lie in pairwise distinct
link components -- and give
\(\PB_n(\sfS_3,\bfy^0)\cong\pi_1(\sfG_\sfS)\) for
\(\sfG_\sfS=\sfG_n((\sfS_3,x),\bfy^0)\), together with the immersion
\(h\colon\sfG_\sfS\to\sfG=\sfG_n((X,x),\base)\) commuting with the
collapse maps up to homotopy.

\smallskip
\noindent\emph{The collapsed subgraph.}
Let \(\sfG_Z\subseteq\sfG\) be the image of the distribution graph
of the decomposition of \(F_n(Z,\bfz^0)\) at \(x\) under the map
induced by \(Z\hookrightarrow X\) (a component of \(F_n(Z_x)\) is
carried into the component of \(F_n(X_x)\) containing it, and a
crossing of \(x\) inside \(Z\) into the edge of \(\sfG\) over the
link component of \(\lk_X(x)\) containing the corresponding
direction; when \(x\notin Z\), the decomposition is trivial and
\(\sfG_Z\) is a single vertex). It is a connected subgraph, and
the \(c_\ast\)-image of every element of
\(\im\bigl(\PB_n(Z,\bfz^0)\to\PB_n(X,\base)\bigr)\) is represented
by a loop inside \(\sfG_Z\). Hence the composite
\[
\PB_n(X,\base)\xrightarrow{c_\ast}\pi_1(\sfG)
\longrightarrow\pi_1(\sfG/\sfG_Z)
\]
kills the normal closure of that image and descends to the
quotient of the statement. Crucially, since
\(Z\cap A_{L^1}=\varnothing\), a strand of a \(Z\)-configuration
can never cross \(x\) in the direction of \(L^1\), so \(\sfG_Z\)
contains \emph{no edge over \(L^1\)}.

\smallskip
\noindent\emph{The source blob is a forest.}
Let \(M_Z\subseteq\sfG_\sfS\) consist of all vertices and edges whose
\(h\)-image lies in \(\sfG_Z\). By the above, \(M_Z\) contains no
edge in the direction of the leaf \(a^1\). Within the subgraph of
\(\sfG_\sfS\) spanned by the edges of the other two directions, the
strands on the arm \(a^1\) are frozen -- they can neither leave
\(a^1\) nor be joined by another strand, as either move uses an
\(a^1\)-direction edge -- while the remaining strands move along
the arc \(a^2\cup\{x\}\cup a^3\). For each frozen configuration
this subgraph is a copy of the distribution graph of the
decomposition of a configuration space of an arc, whose components
are trees (all vertex and edge spaces are contractible, and the
components of the configuration space of an arc are contractible).
Hence \(M_Z\) is a forest, and collapsing its components leaves
\(\pi_1(\sfG_\sfS)\cong\pi_1(\sfG_\sfS/M_Z)\).

\smallskip
\noindent\emph{The induced map is an immersion.}
Consider \(\bar h\colon\sfG_\sfS/M_Z\to\sfG/\sfG_Z\). An edge of
\(\sfG_\sfS\) outside \(M_Z\) has its image outside \(\sfG_Z\), and
distinct edges of \(\sfG\) not internal to \(\sfG_Z\) stay
distinct after the collapse; at the vertices of \(\sfG_\sfS/M_Z\)
that are not collapsed, local injectivity is the statement of
Step~3. Assume first hypothesis (a); let \(\sfT\) be a component of
\(M_Z\) and let
\(\sfe\neq \sfe'\) be edges at the collapsed vertex \([\sfT]\) with
\(\bar h(\sfe)=\bar h(\sfe')\). Equal images force the same strand index
\(i\), the same leaf direction (the three link components being
pairwise distinct), and the same ambient component data. Along
\(\sfT\), the population of the arm \(a^1\) and the order of its
strands are constant, and the linear order of the remaining
strands along the arc \(a^2\cup\{x\}\cup a^3\) is invariant; and
since \(a^2\) and \(a^3\) run into distinct \(x\)-components, the
common ambient data of \(\sfe\) and \(\sfe'\) also determines the
distribution of the moving strands between the two arms -- if
\(a^1\) shares its \(x\)-component with one of them, the count of
moving strands on that side differs from the ambient count only by
the population of \(a^1\), which is constant along \(\sfT\). These invariants determine
the endpoint states of \(\sfe\) and \(\sfe'\) inside \(\sfT\) completely
(an \(\sfS_3\)-state is recovered from its arm distribution, the
orders along the arms, and the frozen \(a^1\)-data; for a mixed
pair note that the \(F_n\)-type end of an edge is determined by
its \(F^i_{n-1}\)-type end together with its direction), whence
\(\sfe=\sfe'\). Thus \(\bar h\) is locally injective, reduced loops
survive, \(\bar h_\ast\) is injective, and the composite of
\(\iota_\ast\) with the projection to the quotient of the
statement is injective.

\smallskip
\noindent\emph{Case (b).}
Suppose now \(Z\cap X^1=\{x\}\). Then a configuration of \(Z\)
carries no strand on the \(X^1\)-side, so no vertex of \(\sfG_Z\)
involves a strand in \(X^1_x\); consequently \(M_Z\) contains no
vertex with a strand on the arm \(a^1\), and \(M_Z\) lies in the
subgraph of \(\sfG_\sfS\) spanned by the states with all strands on
the arc \(a^2\cup\{x\}\cup a^3\) -- a forest as before, with no
\(a^1\)-direction edges. For the local injectivity of \(\bar h\)
at a collapsed component \(\sfT\), the distribution argument of case
(a) is replaced by a sharper use of the order invariant. For an
\(F^i_{n-1}\)-type vertex of \(\sfT\), the strand \(i\) occupies the
slot of \(x\) in the linear order along the arc, so the spectators
below \(i\) are exactly those on \(a^2\) and those above \(i\) are
on \(a^3\): the state is determined by \(i\) and the linear order,
which is constant along \(\sfT\). For the \(F_n\)-type end of an
edge \((i,C_\sfS,d)\) -- with \(C_\sfS\in\pi_0(F^i_{n-1}((\sfS_3)_x))\)
and \(a^d\) the arm direction, as in the proof of
\cref{prop:star-embedding} -- the strand \(i\) is innermost on
the arm \(a^d\), so no spectator lies between \(i\) and \(x\), and the
spectators' distribution is again read off from their position
relative to \(i\) in the linear order. Hence two distinct edges at
\([\sfT]\) with the same strand index and direction would share their
endpoint states, and \(\bar h\) is locally injective; the proof
concludes as in case (a).
\end{proof}

\subsection{Attaching a pendant.}
We close the section by recording how the pure braid group changes
when a \emph{pendant} is attached to \(X\); this extends
An--Park~\cite[Prop.~3.1]{AnPark2017}.

Let \(x\) be a vertex of \(X\) whose link \(L=\lk_X(x)\) is
\emph{connected}, i.e.\ \(\val_X(x)=1\). Attaching a
new edge \(\sfe\) -- a \(1\)-simplex -- to \(X\) at \(x\), with one
endpoint identified with \(x\) and the other a new leaf vertex,
produces a complex \(X'\supseteq X\) in which \(\sfe\) is a pendant
free edge, \(\val_{X'}(x)=2\), and
\(\lk_{X'}(x)=L\sqcup\{\ell_\sfe\}\). We write
\(j\colon X\hookrightarrow X'\) for the inclusion. 

\begin{lemma}\label{lem:pendant-edge}
With \(X'\) as above -- in particular \(L\) connected -- the
inclusion induces a surjection
\[
j_\ast\colon\PB_n(X,\base)\twoheadrightarrow
\PB_n(X',\base),
\]
whose kernel is normally generated by the pure braids in which
\(n-1\) strands are stationary and the remaining strand traverses
a loop through \(x\) along the link \(L\).
\end{lemma}

\begin{proof}
Apply \cref{prop:valency-decomp} to \(X'\) at \(x\) (of valency
\(2\), with \(\lk_{X'}(x)=L\sqcup\{\ell_\sfe\}\)): it
exhibits \(F_n(X')\) as the graph-of-spaces \(\cG_n(X',x)\), whose
vertex spaces are the components of \(F_n(X'_x)\) and of the
\(F^i_{n-1}(X'_x)\), and whose edge spaces are the components of
\(F^i_{n-1}(X'_x)\times L\) or \(F^i_{n-1}(X'_x)\times \ell_\sfe\).

As \(\sfe\) is a pendant at \(x\), removing \(x\) disconnects it:
\[
X'_x = X_x\sqcup\sfe_x.
\]
Hence \(\PB_m(\sfe_x)\cong\PB_m(\sfe)=1\) for every \(m\) (the configuration space of an
interval is a disjoint union of contractible cells), and every
component of \(F_n(X'_x)\) -- and of each \(F^i_{n-1}(X'_x)\) --
is a product
\[
F_m(X_x)\times F_{n-m}(\sfe_x)\simeq F_m(X_x)
\qquad(0\le m\le n),
\]
indexed by the number \(m\) of strands lying in \(X_x\). Thus the
vertex groups are the \(\PB_m(X_x)\), and the edge spaces carrying
non-trivial attaching data are exactly those over the link points
of \(L\),
\[
F^i_{m-1}(X_x)\times L \longrightarrow F_m(X_x),
\]
inserting the \(i\)-th strand at \(x\) along \(L\).

\emph{Surjectivity.} The locus where no strand lies in \(\sfe\) (the
\(m=n\) part) is precisely \(\cG_n(X,x)\), the graph-of-spaces of
\cref{prop:valency-decomp} for \(X\) at \(x\); the inclusion
\(j\) realises it as a sub-graph-of-spaces of \(\cG_n(X',x)\),
and \(j_\ast\) is the induced map on \(\pi_1\). Recall that
\(\PB_n(X')=\pi_1(\cG_n(X',x))\) is generated by its vertex groups
and by the stable letters for edges outside a spanning tree of the distribution
graph \(\sfG_n(X',x)\).

At the level of distribution graphs, \(\sfG_n(X',x)\) is a
\emph{forest}: this is where the connectedness of \(L\) enters.
Indeed, every \(F^i_{n-1}\)-type vertex has exactly two incident
edges (one over \(L\), one over \(\ell_\sfe\)), and writing
\(c_m=\bigl|\pi_0(F_m(X_x))\bigr|\), the vertex and edge count of
\cref{prop:valency-decomp} gives
\[
\chi\bigl(\sfG_n(X',x)\bigr)
=\sum_{m=0}^{n}\frac{n!}{m!}\,c_m
-n\sum_{m=0}^{n-1}\frac{(n-1)!}{m!}\,c_m
=c_n .
\]
On the other hand, unparking moves -- sliding the strands parked
in \(\sfe\) back into \(X_x\) one at a time -- join every vertex to
the \(m=n\) locus, and since the door region over the connected
link \(L\) lies in a single component of \(X_x\), such an
excursion returns a configuration to the component of
\(F_n(X_x)\) it started from; hence \(\sfG_n(X',x)\) has exactly
\(c_n\) components, so its first Betti number vanishes. In
particular there are no stable letters at all, and \(\PB_n(X')\)
is generated by the vertex groups alone. For the vertex
groups, each \(\PB_m(X_x)\) lies in \(\im j_\ast\) as well: sliding
the \(n-m\) strands parked in \(e\) back across \(x\) into \(X_x\)
carries the braid into the \(m=n\) vertex space \(F_n(X_x)\),
realising it through \(\PB_n(X)\). As the vertex groups and stable
letters together generate \(\PB_n(X')\), we conclude
\(\im j_\ast=\PB_n(X')\).

\emph{The kernel.} The pendant edges extending \(\sfG_n(X,x)\)
to \(\sfG_n(X',x)\) carry edge spaces \(F^i_{m-1}(X_x)\times L\)
\((1\le m\le n)\), attached at their free end to the vertex space
\(F^i_{m-1}(X_x)\) by the projection onto the first factor; on
fundamental groups this is
\[
\PB_{m-1}(X_x)\times\pi_1(L)\twoheadrightarrow\PB_{m-1}(X_x),
\]
with kernel \(\pi_1(L)\) -- the braids in which the \(m-1\) strands
of \(F^i_{m-1}(X_x)\) stay fixed, away from the region \(A_L\) of
\(U_x\setminus\{x\}\) lying over \(L\), while a single strand
traverses the link loop \(L\). Since this attaching map is
surjective, attaching the pendant edge amounts, by van Kampen, to
the pushout that absorbs the leaf vertex group into the adjacent
\(\PB_m(X_x)=\pi_1(F_m(X_x))\) and kills the image of its kernel
\(\pi_1(L)\) under the insertion
\(F^i_{m-1}(X_x)\times L\to F_m(X_x)\). Thus, inside
\(\PB_m(X_x)\), it trivialises the braids in which one strand runs
along \(L\) while the other \(m-1\) sit stationary outside
\(A_L\).

Carrying out these collapses at all the pendant edges is precisely
the passage from \(\PB_n(X)\) to \(\PB_n(X')\). Transcribed inside
\(\PB_n(X)\) -- placing the remaining \(n-m\) strands stationary in
\(A_L\), near \(x\) -- the trivialised elements are the pure braids
in which one strand traverses \(L\), \(n-m\) strands sit stationary
in \(A_L\), and the other \(m-1\) sit stationary outside \(A_L\).
These normally generate \(\ker j_\ast\).
\end{proof}

\begin{remark}\label{rem:pendant-valency}
The connectedness of \(L\) cannot be dropped. Take \(X\) to be a
tripod and \(x\) an interior point of one of its legs, so that
\(\val_X(x)=2\), and let \(X'\) be the result of attaching a
pendant edge at \(x\) -- a tree with two trivalent vertices,
homeomorphic to the one of \cref{ex:two-tripods-shuffle}. Then
\(\PB_2(X)\cong\Z\), while \(\PB_2(X')\) is free of rank three, so
\(j_\ast\) cannot be surjective. The failure is precisely the
appearance of new loops in the distribution graph: when
\(\val_X(x)\geq2\), a strand parked in \(e\) lets the remaining
strands cross \(x\) between different link directions in the two
possible orders, creating shuffle cycles as in
\cref{ex:two-tripods-shuffle}; a corrected statement would require
a further free factor \(\pi_1\bigl(\sfG_n((X',x),\base)\bigr)\),
as in \cref{cor:free-edge-disconnected}.
\end{remark}

We record the promised precise form of
\cref{rem:branched-surface-pendant}: attaching pendant edges to
a complex covered by \cref{prop:branched-surface} -- or to a
\(2\)-manifold, provided some pendant edge lands in its interior
-- leaves the strand map an isomorphism. The model case is a
single thick component together with its pendant edges.

\begin{corollary}\label{cor:thick-with-pendants}
Let \(X'\) be obtained from a connected simple complex \(X\),
all of whose vertices have valency \(1\), by attaching finitely
many pendant free edges at pairwise distinct points of \(X\). Suppose that either
\(X\) is not a \(2\)-manifold, or \(X\) is a compact
\(2\)-manifold and some pendant edge is attached at an interior
point of \(X\). Then the inclusion
\(F_nX'\hookrightarrow X'^{\,n}\) induces an isomorphism
\[
\PB_n(X',\base)\xrightarrow{\ \cong\ }
\prod_{i=1}^n\pi_1(X',x^0_i).
\]
\end{corollary}

\begin{proof}
Attach the pendant edges one at a time. Since the attachment
points are pairwise distinct, each still has its original
connected link when its edge is attached, so at each
step \cref{lem:pendant-edge} presents the pure braid group of
the enlarged complex as the quotient of the previous one by the
normal closure of the braids whose single moving strand
traverses the link loop at the attachment point. The strand
classes of these braids contract through the star of the
attachment point, so the braids lie in \(\ker\iota_\ast\).

Suppose first that \(X\) is not a \(2\)-manifold. Then
\(\iota_\ast\) is an isomorphism for \(X\) by
\cref{prop:branched-surface}, so at each step the kernel
generators are trivial, the quotient map is an isomorphism, and
\(\iota_\ast\) remains an isomorphism throughout.

Suppose instead that \(X\) is a compact \(2\)-manifold, and
attach first a pendant edge lying at an interior point \(p\) of
\(X\). The kernel generators of \cref{lem:pendant-edge} -- a
single strand traversing the link circle of \(p\), encircling
the strands parked beside \(p\) -- lie in \(\ker\iota_\ast\),
and those encircling exactly one parked strand are conjugates of
the standard generators of the disc braid group; together they
normally generate
\(\co{\im\PB_n(D^2,\base)}=\ker\iota_\ast\)
(\cref{thm:goldberg}; for \(S^2\) and \(\RP^2\) see
\cref{rem:s2-rp2}). Hence
\[
\PB_n(X\cup_p \sfe,\base)\cong
\PB_n(X,\base)\big/\ker\iota_\ast\cong
\prod_{i=1}^n\pi_1(X,x^0_i),
\]
by the surjectivity of \(\iota_\ast\), and \(\iota_\ast\) is an
isomorphism for \(X\cup_p \sfe\). The remaining pendant edges are
then absorbed exactly as in the previous case.
\end{proof}

\section{Isolating the ingredients of the Goldberg condition}
\label{sec:isolate}

We first recall the two homomorphisms whose interplay shapes the
entire discussion. The strand map
\[
\iota_\ast\colon \PB_n(X,\base)\longrightarrow
\prod_{i=1}^n \pi_1(X, x_i^0)
\]
is induced by the inclusion \(F_n X\hookrightarrow X^n\) followed
by the \(i\)-th coordinate projection.
For a connected subspace \(X_0\subseteq X\) with \(\base\in F_n(X_0)\), the
inclusion \(X_0\hookrightarrow X\) induces
\[
j_\ast\colon \PB_n(X_0,\base)\longrightarrow\PB_n(X,\base),
\]
and the composition \(\iota_\ast\circ j_\ast\) records the strand
classes of pure braids that originate in \(X_0\); in particular,
\(\im j_\ast\subseteq\ker\iota_\ast\) whenever \(\pi_1(X_0)\to\pi_1(X)\)
is the zero map.

\subsection{Strengths of the Goldberg condition}
\label{ssec:goldberg-strengths}

We begin by formalising the terminology informally introduced in
\cref{ssec:kernel-question}.

\begin{definition}[(Weakly) Goldbergness]\label{def:goldberg-strengths}
Let \(X\) be a finite connected simplicial complex of dimension at
least \(1\). We say:
\begin{itemize}[leftmargin=2em]
    \item \(X\) is \emph{weakly Goldberg} if there exists a
          contractible subspace \(X_0\subseteq X\) -- a finite
          simplicial subcomplex, after subdivision
          (\cref{conv:subcomplex}) -- with \(\base\in F_n(X_0)\) such
          that \(\ker \iota_\ast=\co{\im j_\ast}\);
    \item \(X\) is \emph{Goldberg} if there exists such
          an \(X_0\) for which, additionally,
          \(j_\ast\colon\PB_n(X_0)\to\PB_n(X)\) is an embedding.
\end{itemize}
In particular, \(X\) is \emph{Goldberg-trivial}, or \emph{trivially
Goldberg}, if \(\ker\iota_\ast\) is trivial.
\end{definition}

\begin{lemma}\label{lem:gt-implies-goldberg}
If \(X\) is Goldberg-trivial, then \(X\) is Goldberg.
\end{lemma}

\begin{proof}
Choose an embedded arc \(\alpha\subseteq X\) and then choose
pairwise distinct basepoints \(x_1^0,\dots,x_n^0\) on it.
After subdivision, the arc is a subcomplex and these points
are vertices. As for the interval
(\cref{subsubsection:simply-connected}),
\(\PB_n(\alpha,\base) = 1\). Taking \(X_0 = \alpha\), the induced map
\(j_\ast\colon\PB_n(\alpha) = 1\to\PB_n(X)\) is trivially injective
and has trivial image; the Goldberg-trivial hypothesis gives
\(\ker\iota_\ast = 1 = \im j_\ast\). Hence \(\alpha\) witnesses that
\(X\) is Goldberg.
\end{proof}

The gap between weakly Goldberg and Goldberg is genuine, and
the sharpest witnesses to this are the two exceptional closed
surfaces. For \(X\cong S^2\) or \(\RP^2\) the kernel of the
strand map is described exactly as in Goldberg's theorem, yet --
with a single exception at \(n=2\) -- it is not realised by an
\emph{embedded} copy of the Artin pure braid group.

\begin{theorem}[The exceptional surfaces \(S^2\) and \(\RP^2\)]
\label{thm:s2-rp2-exceptional}
Fix \(n\geq 2\), and let \(X\) be homeomorphic to \(S^2\) or
\(\RP^2\). Then \(X\) is weakly Goldberg. Moreover, \(X\) is
Goldberg if and only if \(X\cong S^2\) and \(n=2\), in which
case \(\PB_2(S^2,\base)\) is trivial and \(X\) is even
Goldberg-trivial.
\end{theorem}

\begin{proof}
After a sufficient subdivision, fix an embedded closed disc
\(D^2\subseteq X\) containing all of the basepoints
\(x_1^0,\dots,x_n^0\); this is possible since the \(n\) basepoints
are finitely many points of the surface \(X\). The disc inclusion
induces
\[
j_\ast^{D}\colon \PB_n(D^2,\base)\longrightarrow\PB_n(X,\base),
\]
where \(\PB_n(D^2,\base)\) is the classical Artin pure braid group.

We verify the kernel description directly, including these
exceptional surfaces. Write
\(\Delta_{ij}=\{(z_1,\ldots,z_n)\in X^n:z_i=z_j\}\), so that
\(F_n(X)=X^n\setminus\bigcup_{i<j}\Delta_{ij}\).
Each pairwise diagonal has real codimension \(2\); strata
with at least two independent collision conditions have
codimension at least \(4\). General position for based loops
therefore gives the surjectivity of \(\iota_*\).

Let \([\beta]\in\ker\iota_*\). The loop \(\beta\) bounds a
disc in \(X^n\). Using a smooth model of the surface, perturb
this filling relative to its boundary to avoid the
higher-collision strata and to meet the simple pairwise
diagonals transversely. There are finitely many intersection
points. Remove small disjoint discs around their preimages.
The remaining punctured disc maps into \(F_n(X)\), so its
boundary relation expresses \([\beta]\) as a product of
conjugates of local collision meridians and their inverses.
Here a meridian is the boundary of a small normal disc to a
simple diagonal; geometrically, two nearby labelled points
perform a full twist while all other points remain fixed.

For each pair \(i,j\), the stratum on which precisely these
two points coincide is connected: identifying the coincident
pair identifies it with \(F_{n-1}(X)\).
Transport a small normal meridian along this stratum to a
collision configuration with all points in
\(\operatorname{int}D^2\). The transported loop lies in
\(F_n(D^2)\). Choosing its basepoint path inside \(F_n(D^2)\)
shows that the original based meridian is conjugate in
\(\PB_n(X,\base)\) to an element of \(\im j_*^D\), possibly
with its orientation reversed. This also covers the
nonorientable surface \(\RP^2\).
Consequently
\(\ker\iota_*\subseteq\co{\im j_*^D}\).
The reverse inclusion follows from the contractibility of
\(D^2\), since every strand of a disc braid is null-homotopic
in \(X\). Thus
\[
\ker\iota_*=\co{\im j_*^D}.
\]
The contractible disc \(X_0=D^2\) is therefore a weak
Goldberg witness.

The pure braid group \(\PB_2(S^2,\base)\) is trivial
\cite{VanBuskirk1966}, so \(\ker\iota_\ast=1\): \(X\) is
Goldberg-trivial, hence Goldberg by \cref{lem:gt-implies-goldberg}.
(In particular the disc witness above has
\(\im j_\ast^{D}=1=\ker\iota_\ast\), the full twist dying in
\(\PB_2(S^2,\base)\).) In all remaining cases -- \(X\cong S^2\)
with \(n\geq3\), or \(X\cong\RP^2\) with \(n\geq2\) -- we show
that \(X\) is not Goldberg.

Let \(\Delta^2\in\PB_n(D^2,\base)\) be the full twist, the
positive generator of the (infinite cyclic) centre of the Artin
pure braid group, and put
\(\delta:=j_\ast^{D}(\Delta^2)\in\PB_n(X,\base)\). While
\(\Delta^2\) has infinite order, its image \(\delta\) has
\emph{finite} order: the non-trivial torsion of \(\PB_n(X)\)
occurring precisely for \(X\cong S^2,\RP^2\) is exactly the
degeneration recorded in \cref{rem:s2-rp2}
\cite{VanBuskirk1966}. Moreover \(\delta\neq 1\) in the cases at hand: for
\(X\cong S^2\) and \(n\geq3\) the element \(\delta\) generates
the order-two centre of \(\PB_n(S^2,\base)\), and for
\(X\cong\RP^2\) and \(n\geq2\) it is the unique element of
order two in \(\PB_n(\RP^2)\)
\cite[Prop.~15]{GoncalvesGuaschi2004}. Finally, since \(D^2\) is simply
connected, every strand of a pure braid in \(\PB_n(D^2)\) is a
null-homotopic loop in \(X\), so
\(j_\ast^{D}\bigl(\PB_n(D^2)\bigr)\subseteq\ker\iota_\ast\); in
particular \(\delta\in\ker\iota_\ast\).

Suppose, for a contradiction, that \(X\) is Goldberg in one of
the remaining cases, witnessed by a contractible subcomplex
\(X_0\subseteq X\): thus \(\co{\im j_\ast}=\ker\iota_\ast\) and
\(j_\ast\colon\PB_n(X_0,\base)\to\PB_n(X,\base)\) is injective.
Since \(X_0\) is contractible while the closed surface \(X\) is
not, \(X_0\) is a \emph{proper} subcomplex. (It is precisely to
exclude the degenerate witness \(X_0=X\) that contractibility,
rather than mere connectedness, is required of a Goldberg
witness; see \cref{rem:simpler-X0}.) Choose a point
\(p\notin X_0\) together with a small open disc
\(\mathring N(p)\) around it disjoint from \(X_0\); then
\(X_0\subseteq\Sigma_p\coloneqq X\setminus\mathring N(p)\), a
compact surface with boundary -- a disc if \(X\cong S^2\), a
M\"obius band if \(X\cong\RP^2\) -- so that
\(\Sigma_p\not\cong S^2,\RP^2\). Throughout, base
configurations are transported along paths; this changes the
maps below only by conjugations and affects no injectivity or
order statement. We distinguish three cases according to the
shape of the witness.

\emph{Case 1: \(\dim X_0=2\).} Then \(X_0\) contains a closed
\(2\)-simplex, an embedded disc \(D_0\). The inclusion
\(D_0\subseteq\Sigma_p\) induces an injective map
\(\PB_n(D_0)\to\PB_n(\Sigma_p)\)
(\cref{thm:goldberg,rem:goldberg-compact}), and it factors
through \(\PB_n(X_0)\); hence \(\PB_n(D_0)\to\PB_n(X_0)\) is
injective, and composing with \(j_\ast\), so is the
inclusion-induced map \(\PB_n(D_0)\to\PB_n(X,\base)\). But this
map carries the full twist of \(D_0\), an element of infinite
order, to a conjugate of \(\delta\) -- any two embedded discs
in \(X\) being ambient isotopic -- which has finite order: a
contradiction.

\emph{Case 2: \(\dim X_0=1\), and some vertex \(\sfv\) of
\(X_0\) has valency \(m\geq3\).} Being contractible, \(X_0\) is
a tree, and the star of \(\sfv\) in \(X_0\) is an embedded
\(\sfS_m\). Applying \cref{prop:star-embedding} to the complex
\(X_0\) at the point \(\sfv\), which has \(m\geq2\)
\(\sfv\)-components, yields the injectivity of
\(\PB_n(\sfS_m)\to\PB_n(X_0)\); composing with \(j_\ast\), the
inclusion-induced map \(\PB_n(\sfS_m)\to\PB_n(X,\base)\) is
injective. Now choose a disc neighbourhood
\(N(\sfv)\cong D^2\) of \(\sfv\) in the surface \(X\)
containing \(\sfS_m\); the inclusions
\(\sfS_m\subseteq X_0\subseteq X\) and
\(\sfS_m\subseteq N(\sfv)\subseteq X\) induce the same map on
braid groups, so the composite
\(\PB_n(\sfS_m)\to\PB_n(N(\sfv))\to\PB_n(X,\base)\) is
injective. By \cref{lem:star-to-disc} the first map is
surjective; pick \(\tilde\delta\in\PB_n(\sfS_m)\) mapping to
the full twist of the disc \(N(\sfv)\). Then \(\tilde\delta\)
has infinite order, since its image in \(\PB_n(N(\sfv))\)
does, while its image in \(\PB_n(X,\base)\) is a conjugate of
\(\delta\), of finite order: a contradiction.

\emph{Case 3: \(\dim X_0=1\), and every vertex of \(X_0\)
has valency at most \(2\).} Then the tree \(X_0\) is an arc and
\(\PB_n(X_0,\base)=1\)
(\cref{subsubsection:simply-connected}). The witness property
then reads \(\ker\iota_\ast=\co{\,1\,}=1\), contradicting
\(\delta\neq1\) in the cases at hand.

In all three cases we have reached a contradiction; therefore
\(X\) is not Goldberg.
\end{proof}

\begin{remark}\label{rem:s2-rp2-exceptional}
\cref{thm:s2-rp2-exceptional} isolates the precise way in
which \(S^2\) and \(\RP^2\) escape Goldberg's exact sequence. The
kernel description \(\co{\im j_\ast}=\ker\iota_\ast\) (exactness at
\(\PB_n(X)\)) and the surjectivity of \(\iota_\ast\) both persist;
what fails, outside the exceptional case \(X\cong S^2\) with
\(n=2\), is exactness at the left. The obstruction is the
torsion of \(\ker\iota_\ast\): as the proof shows, a witness
that is \(2\)-dimensional, or branches at a vertex, carries a
braid of infinite order whose image in \(\PB_n(X)\) is
conjugate to the full twist \(\delta\), of finite order, while
a witness without branching has trivial braid group altogether.
Note also that the braid group of the punctured surface
\(X\setminus\{p\}\) -- through which the inclusion of any
witness avoiding \(p\) factors -- is torsion-free: its
configuration space is aspherical of finite dimension by the
Fadell--Neuwirth fibration \cite{FadellNeuwirth1962}.
\end{remark}

\subsection{The surjectivity dichotomy}
\label{ssec:surj-dichotomy}

The surjectivity of \(\iota_\ast\) is tangential
to the kernel question but worth settling here in full, because
its failure is rare and because the failing case is Goldberg-trivial
(\cref{prop:s1-case}). The key local input is the
lollipop graph.

\begin{lemma}[Lollipop graph]\label{lem:lollipop}
Let \(\Lollipop = \mathsf{C}\cup \mathsf{e}\) be the lollipop graph
(\cref{fig:lollipop}). Then for every \(n\geq 1\) the strand
map
\[
\iota_\ast \colon \PB_n(\Lollipop)\longrightarrow
\prod_{i=1}^n\pi_1(\Lollipop,x_i^0) \cong \Z^n
\]
is surjective.
\end{lemma}

\begin{proof}
\(\Lollipop\) deformation retracts onto \(\mathsf{C}\), so
\(\pi_1(\Lollipop)\cong\Z\) and the target of \(\iota_\ast\) is \(\Z^n\).
To realise the standard \(i\)-th generator \((0,\dots,0,1,0,\dots,0)\)
by a pure braid, park the \(n-1\) strands
\(\{1,\dots,n\}\setminus\{i\}\) at distinct points along the edge
\(\mathsf{e}\) (subdivide \(\mathsf{e}\) if needed), and let strand
\(i\) traverse \(\mathsf{C}\) once and return. The other strands stay
fixed, so the resulting pure braid maps to
\((0,\dots,0,1,0,\dots,0)\) under \(\iota_\ast\). Since these
elements generate \(\Z^n\), \(\iota_\ast\) is surjective.
\end{proof}

Combining \cref{lem:lollipop} with the lollipop embedding
result of \cref{lem:lollipop-embed}, we obtain the general
dichotomy:

\begin{theorem}[Surjectivity of \(\iota_\ast\)]\label{prop:iota-surj}
Let \(X\) be a finite connected simplicial complex of dimension at
least \(1\), and fix \(n\ge2\). Then the
strand map \(\iota_\ast\) is surjective if and only if \(X\not\cong S^1\).
\end{theorem}

\begin{proof}
For \(X\cong S^1\), the strand map is the diagonal embedding
\(\Z\hookrightarrow\Z^n\) and is not surjective when \(n\geq2\)
(\cref{prop:s1-case}). If \(X\) is simply connected, its target
is trivial. Assume henceforth that \(X\) is neither simply
connected nor homeomorphic to \(S^1\).

Then \(X\) contains an embedded tripod. After subdivision,
arrange that this tripod is a subgraph of the \(1\)-skeleton
\(\Gamma=X^{(1)}\) and that every \(x_i^0\) is a vertex. Choose
a maximal tree \(T\subseteq\Gamma\) containing the tripod. Thus
\(\base\in F_n(T)\), and \(F_n(T)\) is connected by
\cref{lem:rho-surj}.

For each edge \(\sfe\) of \(\Gamma\) outside \(T\), let
\(C_\sfe\) be the embedded cycle consisting of \(\sfe\) and the
unique path in \(T\) between its endpoints. Since \(T\) contains
a tripod, it is not contained in \(C_\sfe\). Choose an arc
\(a_\sfe\subseteq T\) meeting \(C_\sfe\) in just one endpoint
\(q_\sfe\). Then \(C_\sfe\cup a_\sfe\) is a lollipop whose
stick lies in \(T\).

Fix a strand \(i\). Choose a configuration \(\bfy\in F_n(T)\)
with \(y_i=q_\sfe\) and all other points on
\(a_\sfe\setminus\{q_\sfe\}\), and choose a path
\(\eta\) in \(F_n(T)\) from \(\base\) to \(\bfy\). At \(\bfy\),
let \(\beta\) move only the \(i\)-th strand once around
\(C_\sfe\), keeping the others parked on the stick. The
resulting based pure braid \(\eta\beta\eta^{-1}\) has strand
map
\[
\iota_\ast(\eta\beta\eta^{-1})
=(1,\dots,1,g_{\sfe,i},1,\dots,1),
\qquad
g_{\sfe,i}=[\,\eta_i C_\sfe\eta_i^{-1}\,],
\]
where \(\eta_i\) is the \(i\)-th coordinate path of \(\eta\).
Since \(\eta_i\) lies in the tree \(T\), it is homotopic
relative to its endpoints to the unique tree path from
\(x_i^0\) to \(q_\sfe\). Consequently the \(g_{\sfe,i}\), as
\(\sfe\) ranges over the edges outside \(T\), are the images of
the standard maximal-tree generators of
\(\pi_1(\Gamma,x_i^0)\). They generate \(\pi_1(X,x_i^0)\),
because \(\Gamma\hookrightarrow X\) is surjective on
fundamental groups. We have therefore realised a generating
set in each coordinate, with all other coordinates trivial,
which proves surjectivity.
\end{proof}

\begin{corollary}\label{cor:not-gt-implies-surj}
If \(X\) is not Goldberg-trivial, then \(\iota_\ast\) is surjective.
\end{corollary}

\begin{proof}
By \cref{prop:s1-case}, \(S^1\) is Goldberg-trivial. So
if \(X\) is not Goldberg-trivial, then \(X\not\cong S^1\), and
\(\iota_\ast\) is surjective by \cref{prop:iota-surj}.
\end{proof}

The surjectivity dichotomy characterizes the strand-map
containment \(\im j_\ast\subseteq\ker\iota_\ast\) underlying weak
Goldbergness:

\begin{lemma}\label{lem:weak-goldberg-pi1-zero}
Let \(X\) satisfy \cref{rem:standing}, and let \(X_0\subseteq X\) be
a connected subspace with \(\base\in F_n(X_0)\). Writing
\(j_\ast\colon\PB_n(X_0,\base)\to\PB_n(X,\base)\) for the
inclusion-induced map, one has
\[
\im j_\ast\subseteq\ker\iota_\ast^X
\]
if and only if the inclusion-induced map
\[
(j_{X_0})_\ast\colon\pi_1(X_0, x_i^0)\longrightarrow
\pi_1(X, x_i^0)
\]
is the zero map for some -- equivalently, every -- \(i\); that is,
if and only if every loop in \(X_0\) is null-homotopic in \(X\).
\end{lemma}

\begin{proof}
Naturality of the strand map with respect to
\(X_0\hookrightarrow X\) provides the commutative square
\[
\begin{tikzcd}[column sep=huge]
\PB_n(X_0,\base)
    \arrow[r, "\iota_\ast^{X_0}"]
    \arrow[d, "j_\ast"']
& \displaystyle\prod_{i=1}^{n}\pi_1(X_0,x_i^0)
    \arrow[d, "i_\ast^n"] \\
\PB_n(X,\base)
    \arrow[r, "\iota_\ast^X"']
& \displaystyle\prod_{i=1}^{n}\pi_1(X,x_i^0),
\end{tikzcd}
\]
in which \(i_\ast^n\) is the coordinatewise product of the
inclusion-induced maps \((j_{X_0})_\ast\).

\noindent\((\Leftarrow)\) If \((j_{X_0})_\ast\) is the zero map for
every \(i\), then \(i_\ast^n=1\), so commutativity gives
\(\iota_\ast^X\circ j_\ast = i_\ast^n\circ\iota_\ast^{X_0}=1\); that
is, \(\im j_\ast\subseteq\ker\iota_\ast^X\).

\noindent\((\Rightarrow)\) Suppose
\(\im j_\ast\subseteq\ker\iota_\ast^X\), i.e.\
\(\iota_\ast^X\circ j_\ast = 1\). Commutativity of the square then
forces \(i_\ast^n\) to vanish on the image of \(\iota_\ast^{X_0}\).
Apply \cref{prop:iota-surj} to \(X_0\):
\begin{itemize}[leftmargin=2em]
    \item If \(X_0\not\cong S^1\), then \(\iota_\ast^{X_0}\) is surjective
          onto \(\prod_{i=1}^{n}\pi_1(X_0,x_i^0)\), so \(i_\ast^n\) vanishes on
          the whole product. In particular, the \(i\)-th coordinate
          \((j_{X_0})_\ast\colon\pi_1(X_0,x_i^0)\to\pi_1(X,x_i^0)\) is the
          zero map for every \(i\).
    \item If \(X_0\cong S^1\), write \(c_i\in\pi_1(X_0,x_i^0)\cong\Z\) for
          the generator (any choice of compatible orientations). The
          ``global rotation'' pure braid that drags every strand once
          around \(X_0\) has strand-map image
          \(\iota_\ast^{X_0}(\tau) = (c_1,c_2,\dots,c_n)\), so this
          element lies in \(\im\iota_\ast^{X_0}\). The previous paragraph
          gives \(i_\ast^n(c_1,\dots,c_n)=1\), i.e.\
          \((j_{X_0})_\ast(c_i) = 1\) in \(\pi_1(X,x_i^0)\) for every \(i\).
          Since \(c_i\) generates \(\pi_1(X_0,x_i^0)\), the inclusion-induced
          map vanishes.
\end{itemize}
In either case \((j_{X_0})_\ast\) is the zero map; as \(X_0\) is
connected, this holds for one \(i\) if and only if it holds for
every \(i\).
\end{proof}

The kernel equality defining a witness is upward closed among
simply connected subspaces.

\begin{proposition}[Upward closure of witnesses]
\label{prop:witness-upward}
Suppose \(X_0\subseteq X\) satisfies
\(\co{\im j_\ast^{X_0}}=\ker\iota_\ast^X\) with \(\base\in F_n(X_0)\) --
for instance, a witness to the weak Goldbergness of \(X\). If
\(X_0\subseteq X_0'\subseteq X\) and \(X_0'\) is simply connected,
then \(X_0'\) satisfies the same equality,
\[
\co{\im j_\ast^{X_0'}}=\ker\iota_\ast^X.
\]
If, in addition, \(X_0'\) is contractible, then it is a weak
Goldberg witness of \(X\).
\end{proposition}

\begin{proof}
The inclusion \(X_0\hookrightarrow X\) factors through \(X_0'\), so
\(\im j_\ast^{X_0}\subseteq\im j_\ast^{X_0'}\); hence
\[
\ker\iota_\ast^X=\co{\im j_\ast^{X_0}}\subseteq\co{\im j_\ast^{X_0'}}.
\]
For the reverse inclusion, note that \(\base\in F_n(X_0)\subseteq
F_n(X_0')\) and that \(X_0'\) is connected; since \(X_0'\) is simply
connected, the inclusion-induced map
\((j_{X_0'})_\ast\colon\pi_1(X_0')\to\pi_1(X)\) is the zero map, so
\cref{lem:weak-goldberg-pi1-zero} gives
\(\im j_\ast^{X_0'}\subseteq\ker\iota_\ast^X\). The two inclusions
yield \(\co{\im j_\ast^{X_0'}}=\ker\iota_\ast^X\).
\end{proof}

\section{Lower bounds: necessary conditions on \texorpdfstring{\(X_0\)}{X0}}
\label{sec:lower-bound}

We now record some necessary conditions that any weakly Goldberg
witness \(X_0\subseteq X\) must satisfy. These give a lower bound
on which subspaces of \(X\) can possibly serve as a witness for
Goldbergness, and will constrain the candidate choices of \(X_0\)
in later sections.

Fixing a vertex \(x\) of valency at least \(2\), we address two
questions about how the witness must meet \(x\):
\begin{enumerate}[label=\textup{(\roman*)},leftmargin=2.4em]
  \item \emph{When is \(x\) itself forced into the witness,
        \(x\in X_0\)?}
  \item \emph{Once \(x\in X_0\), from which directions of the link
        \(\lk_X(x)\) must \(X_0\) approach \(x\)?}
\end{enumerate}
Both answers depend on the number \(m\) of \(x\)-components and on
the valency \(\val_X(x)\), and the subsections below treat the
several regimes in turn. Assembled over the vertices of an embedded
cycle, they yield the obstruction of
\cref{thm:F-cycle-obstruction}: a non-simply-connected free part
prevents weak Goldbergness.

\begin{assumption}
Throughout the rest of this section, we assume the standing
assumption (\cref{rem:standing}), suppose that \(X\) is
weakly Goldberg with witness \(X_0\subseteq X\)
(\cref{def:goldberg-strengths}), and fix a point
\(x\in X\) of valency \(k\coloneqq\val_X(x)\geq 2\).
Let \(m\) denote the number of \(x\)-components \(X^1,\dots,X^m\) of \(X\).
\end{assumption}

For a connected component \(L\subseteq\lk_X(x)\), write \(A_L\) for
the component of the punctured open star \(U_x\setminus\{x\}\) that
approaches \(x\) along \(L\) (\cref{sec:config-decomposition}). In
these terms question~(ii) above asks for which \(L\) one has
\(X_0\cap A_L\neq\varnothing\) -- equivalently, when the inclusion
\(\lk_{X_0}(x)\hookrightarrow\lk_X(x)\) is surjective on \(\pi_0\).

Most of the propositions below conclude by
isolating a non-trivial subgroup of \(\PB_n(\sfS_3)\) or
\(\PB_n(X^i)\) that embeds in
\(\PB_n(X)/\co{\PB_n(X^+)}^{\PB_n(X)}\) for some
\(X^+\supseteq X_0\) and then forcing it to be trivial by contractibility
or by the strand-map naturality. The following lemma records this
common step.

\begin{lemma}[Kernel--quotient reduction]\label{lem:kernel-quotient-reduction}
Let \(X^+\subseteq X\) be a connected
subspace with \(X_0\subseteq X^+\), let \(j^+\colon X^+\to X\) be the canonical inclusion, and set
\[
Q_{X^+}\coloneqq
\PB_n(X,\base)\big/\co{\im j^+_*}^{\PB_n(X,\base)}.
\]
Let \(Y\subseteq X\) be a connected subspace with basepoint
\(\bfy^0\in F_n(Y)\) identified with \(\base\) along a path in
\(F_n(X)\), and let \(H\subseteq\PB_n(Y,\bfy^0)\) be a subgroup.
Then the natural composition
\[
H\hookrightarrow\PB_n(Y,\bfy^0)
\xrightarrow{j_\ast^Y}\PB_n(X,\base)
\twoheadrightarrow Q_{X^+}
\]
factors as
\[
H\xrightarrow{\iota_\ast^Y|_H}
\prod_{\ell=1}^{n}\pi_1(Y,y_\ell^0)
\xrightarrow{i_\ast^n}
\prod_{\ell=1}^{n}\pi_1(X,x_\ell^0)
\twoheadrightarrow Q_{X^+}.
\]
Consequently, if \(H\to Q_{X^+}\) is injective then so is
\(\iota_\ast^Y|_H\colon H\to\prod_\ell\pi_1(Y,y_\ell^0)\). In particular,
if \(Y\) is contractible, then \(H=1\).
\end{lemma}

\begin{proof}
The weakly Goldberg equality gives
\(\ker\iota_\ast^X = \langle\im j_\ast\rangle^{\PB_n(X,\base)}\). The inclusion \(X_0\subseteq X^+\)
yields \(\ker\iota_\ast^X=\langle \im j_\ast\rangle^{\PB_n(X,\base)} \subseteq \co{\im j^+_*}^{\PB_n(X,\base)}\),
so the projection \(\PB_n(X,\base)\twoheadrightarrow Q_{X^+}\) factors
through \(\PB_n(X,\base)/\ker\iota_\ast^X\). The latter is isomorphic to
\(\prod_\ell\pi_1(X,x_\ell^0)\) via \(\iota_\ast^X\), which is surjective
by \cref{prop:iota-surj} since \(X\not\cong S^1\).

The naturality square of the strand map under
\(j^Y\colon Y\hookrightarrow X\),
combined with the above factorisation of
\(\PB_n(X,\base)\twoheadrightarrow Q_{X^+}\) through \(\iota_\ast^X\) 
yields the displayed factorisation of \(H\to Q_{X^+}\).
\[
\begin{tikzcd}[column sep=huge]
H\ar[r,hookrightarrow] & \PB_n(Y,\bfy^0)
   \arrow[r, "\iota_\ast^Y"]
   \arrow[d, "j_\ast^Y"']
& \prod_{\ell=1}^{n}\pi_1(Y,y_\ell^0)
   \arrow[d, "i_\ast^n"] \\
&\PB_n(X,\base) \arrow[r, "\iota_\ast^X"', two heads]
& \prod_{\ell=1}^{n}\pi_1(X,x_\ell^0) \ar[r,two heads] & Q_{X^+},
\end{tikzcd}
\]

If \(H\to Q_{X^+}\) is injective, the first arrow \(\iota_\ast^Y|_H\) in
the factorisation must be injective. If moreover \(Y\) is contractible,
then \(\prod_\ell\pi_1(Y,y_\ell^0)=1\) and injectivity of
\(\iota_\ast^Y|_H\) forces \(H=1\).
\end{proof}

\subsection{Vertices with only one \texorpdfstring{\(x\)}{x}-component}
Here \(m=1\), i.e.\ \(X_x\) is connected; we settle both questions
at such a vertex.

\begin{proposition}\label{prop:X0-contains-free-points}
If \(m=1\), i.e., \(X_x\) is connected, then \(x\in X_0\).
\end{proposition}

\begin{proof}
Suppose, for contradiction, that \(x\notin X_0\). Then
\(X_0\subseteq X_x\), so the inclusion \(j\colon X_0\hookrightarrow X\)
factors as \(j=j^x\circ j^0\) through the two embeddings
\(j^0\colon X_0\hookrightarrow X_x\) and
\(j^x\colon X_x\hookrightarrow X\). Hence
\(\im j_\ast\subseteq\im j^x_\ast\) inside \(\PB_n(X,\base)\), the
normal closures satisfy
\(\co{\im j_\ast}\subseteq\co{\im j^x_\ast}\), and the identity of
\(\PB_n(X,\base)\) descends to a surjection
\[
\PB_n(X,\base)\big/\co{\im j_\ast}
\twoheadrightarrow
\PB_n(X,\base)\big/\co{\im j^x_\ast}
=
\PB_n(X,\base)\big/\co{\PB_n(X_x,\base)} .
\]
We identify the two sides. On the left, the weakly Goldberg
hypothesis gives \(\co{\im j_\ast}=\ker\iota_\ast\), while
\(\iota_\ast\) is surjective by \cref{prop:iota-surj}; hence
the left-hand side is \(\prod_{i=1}^n\pi_1(X,x_i^0)\). On the right,
\(X_x\) is connected, so by \cref{rem:cor28-hypotheses} the
connectivity hypotheses of \cref{cor:valency-quotient}
hold and the right-hand side is the free group
\(\mathbb{F}_{n(k-1)}\). Altogether we obtain a surjection
\[
\Phi\colon\prod_{i=1}^n\pi_1(X,x_i^0)
\twoheadrightarrow\mathbb{F}_{n(k-1)} .
\]

Next we factor \(\Phi\) coordinatewise. Since \(X_0\) is contractible,
the \(\pi_1\)-level analogue of the surjection above reads, in each
coordinate,
\[
\pi_1(X,x_i^0)
\twoheadrightarrow
\pi_1(X,x_i^0)\big/
\co{\im\bigl(\pi_1(X_x,x_i^0)\to\pi_1(X,x_i^0)\bigr)}
\cong\mathbb{F}_{k-1},
\]
the isomorphism being van Kampen's theorem for
\(X=(X\setminus\mathrm{st}(x))\cup\mathrm{Cone}(\lk_X(x))\), glued
along the \(k\) components of the link; write
\(p\colon\pi_1(X,x_i^0)\twoheadrightarrow\mathbb{F}_{k-1}\) for this
quotient map. The restriction of \(\Phi\) to the \(i\)-th factor
factors through \(p\): it kills the image of \(\pi_1(X_x,x_i^0)\),
since a class \(\bar h\) coming from \(h\in\pi_1(X_x,x_i^0)\) admits a
lift \(\beta_h\in\PB_n(X_x,\base)\) with
\(\iota_\ast(\beta_h)=(1,\dots,\bar h,\dots,1)\) -- apply
\cref{prop:iota-surj} to the finite subcomplex
\(X\setminus\mathrm{st}(x)\simeq X_x\), not homeomorphic to \(S^1\) as
it contains \(\sfS_3\) by \cref{rem:cor28-hypotheses} --
and \(\beta_h\) dies in
\(\PB_n(X,\base)\big/\co{\PB_n(X_x,\base)}\); kernels being normal,
the entire normal closure dies. Since the images of the \(n\)
factors commute elementwise, the universal property of the direct
product assembles these factorisations into a homomorphism
\[
\Theta\colon\prod_{i=1}^n\mathbb{F}_{k-1}
\longrightarrow\mathbb{F}_{n(k-1)}
\]
with \(\Phi=\Theta\circ\prod_ip\); in particular \(\Theta\) is
surjective.

No such surjection exists. The images \(A_1,\dots,A_n\) of the \(n\)
direct factors under \(\Theta\) commute with one another
elementwise. If at least two of the \(A_i\) are non-trivial then,
since the centraliser of a non-trivial element of a free group is
infinite cyclic, all the non-trivial \(A_i\) lie in a common
infinite cyclic subgroup, so
\(\im\Theta=\langle A_1,\dots,A_n\rangle\) is cyclic and cannot be
the free group \(\mathbb{F}_{n(k-1)}\) of rank \(n(k-1)\ge2\). If at
most one \(A_i\) is non-trivial, then \(\im\Theta=A_i\) is generated
by \(k-1\) elements, whereas \(\mathbb{F}_{n(k-1)}\) cannot be
generated by fewer than \(n(k-1)>k-1\) elements. In either case
\(\Theta\) fails to be surjective -- a contradiction. Hence
\(x\in X_0\).
\end{proof}

When \(m=1\) the witness moreover enters \(x\) from every link
direction.

\begin{proposition}\label{prop:X0-meets-link-m1}
If \(m=1\), then the inclusion
\(\lk_{X_0}(x)\hookrightarrow\lk_X(x)\) is surjective on \(\pi_0\);
equivalently, \(X_0\cap A_L\neq\varnothing\) for every connected
component \(L\subseteq\lk_X(x)\).
\end{proposition}

\begin{proof}
By \cref{prop:X0-contains-free-points} we have \(x\in X_0\).
Suppose, for contradiction, that \(X_0\cap A_L=\varnothing\)
for some connected component \(L\subseteq\lk_X(x)\).

We use the following observation. Let \(\sfe\) be a free edge of a
finite complex \(Y\) incident to a vertex \(u\) of valency at least
\(2\), and suppose the punctured space \(Y_u=Y\setminus\{u\}\) is
connected. Then \(\sfe\) lies on a cycle of \(Y\): its far endpoint
rejoins \(u\) through \(Y_u\), for otherwise that endpoint would be
a leaf and \(\sfe\setminus\{u\}\) a connected component of \(Y_u\) on
its own. Consequently no interior point \(y\) of \(\sfe\) separates
\(Y\), so each such \(y\) has a single \(y\)-component; if \(Y\) is
weakly Goldberg with witness \(Y_0\),
\cref{prop:X0-contains-free-points} -- applied with \(y\) in place
of \(x\) -- then places \(y\) in \(Y_0\), and letting \(y\) vary we
conclude that the interior of \(\sfe\) should be contained in \(Y_0\).

\emph{Case 1: \(L\) is \(0\)-dimensional.} The cone over \(L\)
is then a free edge \(\sfe\) of \(X\) incident to \(x\) and running
into \(A_L\). Since \(X_x\) is connected (\(m=1\)) and
\(\val_X(x)\geq 2\), the observation gives \(\sfe\subseteq X_0\); but
then \(X_0\) meets \(A_L\), a contradiction.

\emph{Case 2: \(\dim L\geq 1\).} Resolve \(x\) along \(L\)
(\cref{ssec:simple-an-park}): this separates the cone over \(L\)
from \(x\), producing a complex \(X'\) with a new apex \(\sfv'\) and a
new free edge \(\sff\) joining \(x\) to \(\sfv'\), and \(X\) is recovered
by collapsing \(\sff\). If \(x\) is non-simple, the construction
in \cref{sssec:general-resolution} gives a braid equivalence
\(X\equiv_B X'\). Otherwise \(x\) is a simple joint with link
\(L\sqcup\{y\}\), and this operation merely lengthens its incident
free edge; the same braid equivalence follows by reparametrising
that edge, with the corresponding basepoint transport. Thus the
collapse \(q_\sff\colon X'\to X'/\sff=X\) and two inclusions \(X_0\to X'\) and \(X_0\to X\) induce a commutative square of
pure braid groups and strand maps whose vertical maps
are isomorphisms. 
\[
\begin{tikzcd}
\PB_n(X_0) \ar[r,"j_*"] \ar[d,equal] & \PB_n(X) \ar[d,"(q_\sff)^*"]\\ %
\PB_n(X_0) \ar[r,"j'_*"] & \PB_n(X') %
\end{tikzcd}
\]
As \(X_0\) misses \(A_L\), it is untouched by the resolution and includes into both \(X\) and
\(X'\) compatibly with \(q_\sff\); transporting the witness equality
\(\co{\im j_\ast}=\ker\iota_\ast\) across the square shows that
\(X_0\) is a weakly Goldberg witness for \(X'\) as well.

Now \(\sff\) is a free edge of \(X'\) at \(x\). Since \(q_\sff\)
restricts to a homeomorphism
\(X'\setminus \sff\xrightarrow\cong X\setminus\{x\}\) and the
latter is connected (\(m=1\)), the punctured space
\(X'_x=X'\setminus\{x\}\) is connected as well. The observation,
applied to \(\sff\) in \(X'\), gives \(\sff\subseteq X_0\); but \(\sff\) runs
into the \(L\)-direction, which \(X_0\) misses -- a contradiction.

In either case \(X_0\cap A_L=\varnothing\) is untenable, so
\(X_0\) meets \(A_L\) for every component \(L\subseteq\lk_X(x)\).
\end{proof}

\subsection{Vertices of valency at least three}

We turn to vertices of valency \(\val_X(x)\geq 3\). The case
\(m=1\) has already been settled, at every valency, by
\cref{prop:X0-contains-free-points,prop:X0-meets-link-m1}, so here
we may assume \(m\geq 2\) and treat the two regimes \(m\geq 3\) and
\(m=2\); in either, a tripod or lollipop embeds at \(x\) and drives
the argument.

\begin{proposition}\label{prop:X0-contains-valency-3}
If \(m\geq 3\), then \(x\in X_0\).
\end{proposition}

\begin{proof}
Suppose, for contradiction, that \(x\notin X_0\). Then
\(X_0\subseteq X_x=\coprod_{i=1}^m X^i_x\), where \(\{X^i\}\) denotes
the (finite) collection of all \(x\)-components of \(X\). Since
\(X_0\) is connected and contains \(\base\), it lies in a single
component of \(X_x\), say \(X^+\), so that %
\(X_0\subseteq X^+\) and \(\base\in F_n(X^+_x)\).

Let \(\sfS_m\subseteq X\) be the embedded star of
\cref{prop:star-embedding}, with one leaf running into each of the
\(m\) \(x\)-components through link points
\(p^i\in\lk_X(x)\cap X^i\), and fix a basepoint
\(\bfy^0\in F_n(\sfS_m)\) identified with \(\base\) along a path in
\(F_n(X)\). By \cref{prop:star-embedding},
\[
\PB_n(\sfS_m,\bfy^0)\hookrightarrow
\PB_n(X,\base)\big/
\co{\PB_n(X^1),\dots,\PB_n(X^m)}^{\PB_n(X,\base)}.
\]
Since \(X^+\) is one of the \(x\)-components, the target is a
further quotient of
\[Q_{X^+}\coloneqq
\PB_n(X,\base)/\co{\PB_n(X^+,\base)}^{\PB_n(X,\base)},\]
and the composite
\(\PB_n(\sfS_m,\bfy^0)\to Q_{X^+}\to
\PB_n(X,\base)/\co{\PB_n(X^1),\dots,\PB_n(X^m)}\) being injective
forces
\[
\PB_n(\sfS_m,\bfy^0)\hookrightarrow Q_{X^+}.
\]

The existence of three distinct \(x\)-components at \(x\) forces
\(X\not\cong S^1\) (any \(x\)-component of \(S^1\) is again
\(S^1\)). Apply \cref{lem:kernel-quotient-reduction} with
\(Y=\sfS_m\) (a tree, hence contractible) and
\(H=\PB_n(\sfS_m,\bfy^0)\): the embedding into \(Q_{X^+}\) just
established implies \(\PB_n(\sfS_m)=1\). But
\(\PB_n(\sfS_m)\) is non-trivial for \(n\geq2\) and \(m\geq3\): the
pure braid in which two strands perform the square of the hexagonal exchange of
\cref{ex:S3} on three of the arms, while the remaining strands
stay parked near the tips of their arms, is non-trivial, as one
sees by forgetting the parked strands. Contradiction.
\end{proof}

\begin{proposition}\label{prop:X0-contains-m2-val3}
Suppose \(m=2\) and \(\val_X(x)\geq 3\). Then \(x\in X_0\).
\end{proposition}

\begin{proof}
Suppose, for contradiction, that \(x\notin X_0\). Then
\(X_0\subseteq X_x = X^1_x\coprod X^2_x\), and since \(X_0\) is
connected and contains \(\base\) it lies in a single
\(x\)-component, say \(X^+\) as before; since \(x\notin X_0\), in
fact \(X_0\subseteq X^+_x\).

Since \(\val_X(x)\geq 3\), one of \(\val_{X^1}(x)\) and
\(\val_{X^2}(x)\), say \(\val_{X^1}(x)\), is at least \(2\). We
pick three points in \(\lk_X(x)\), with \(p^1_1,p^1_2\) chosen in two
distinct connected components of \(\lk_{X^1}(x)\) -- possible since
\(\val_{X^1}(x)\geq2\), and so that the cycle of \(\Lollipop\) can be
closed up inside that \(x\)-component -- and \(p^2\) in \(\lk_{X^2}(x)\). The lollipop graph
\(\Lollipop=\mathsf{C}\cup\mathsf{\sfe}\)
(\cref{fig:lollipop}) then admits an embedding
\(\Lollipop\hookrightarrow X\) with the trivalent vertex of
\(\Lollipop\) mapped to \(x\), the two ends of \(\mathsf{C}\) attached
along \(p^1_1,p^1_2\), and the parking edge \(\mathsf{\sfe}\) along \(p^2\);
the induced link map \(\lk_{\Lollipop}(x)\to\lk_X(x)\) is then an
injection. The canonical inclusion
\(\sfS_3\hookrightarrow\Lollipop\), sending the central
vertex of \(\sfS_3\) to the trivalent vertex of \(\Lollipop\)
and extending each leaf a short distance into the adjacent edge
of \(\Lollipop\), supplies a composition
\[
\sfS_3 \hookrightarrow \Lollipop \hookrightarrow X.
\]

Let us fix a basepoint \(\bfy^0\in F_n(\sfS_3)\subset F_n(\Lollipop)\).
Because \(\Lollipop_x\) is a disjoint
union of two open arcs, every vertex space and every edge space
of \(\cG_n(\Lollipop,x)\) is contractible, hence
\[
\PB_n(\Lollipop)
=\pi_1\bigl(\cG_n(\Lollipop,x)\bigr)
=\pi_1\bigl(\sfG_n(\Lollipop,x)\bigr)
\]
is a free group, of rank \((n-1)\,n!+1\)
(\cref{ex:one-essential}). Likewise
\(\PB_n(\sfS_3,\bfy^0)=\pi_1(\sfG_n(\sfS_3,x))\) is
a free group; the vertex/edge counts of the graph-of-spaces of
the tripod (via \cref{prop:valency-decomp}) give a rank of
\(1 + n!\,(n+1)(n-2)/2\).

The kernel--quotient reduction
(\cref{lem:kernel-quotient-reduction}) requires a
\emph{contractible} probe, which \(\Lollipop\) is not; we
therefore pass to the tripod via the homomorphism
\[
\phi\colon
\PB_n(\sfS_3,\bfy^0)\longrightarrow
\PB_n(\Lollipop,\bfy^0)
\]
induced by the canonical inclusion
\(\sfS_3\hookrightarrow\Lollipop\). In fact \(\phi\) is
\emph{injective}: \(\sfS_3\) is an embedded star at \(x\) inside
\(\Lollipop\) whose three leaves enter the three -- pairwise
distinct -- link components of \(\lk_\Lollipop(x)\), so
\cref{rem:star-link-components} applies with \(\Lollipop\) in the
role of the ambient complex. We may therefore take
\[
H \coloneqq \PB_n(\sfS_3,\bfy^0)
\]
itself; for \(n=2\), for instance, \(H\cong\Z\) is generated by
the square of the hexagonal exchange, whose \(\phi\)-image drags one strand
partway around \(\mathsf{C}\).

\noindent\emph{Claim.}\ The composition
\[
H=\PB_n(\sfS_3,\bfy^0)\xrightarrow{\phi}
\PB_n(\Lollipop,\bfy^0)\longrightarrow
\PB_n(X,\base)\twoheadrightarrow
\PB_n(X,\base)\big/\co{\PB_n(X^+_x)}^{\PB_n(X,\base)}
=: Q_{X^+_x}
\]
is injective.

\smallskip
Indeed, by the punctured-sides clause of
\cref{prop:lollipop-embedding}, \(\PB_n(\Lollipop,\bfy^0)\) embeds
into the quotient
\(\PB_n(X,\base)/\co{\PB_n(X^1_x),\,\PB_n(X^2_x)}^{\PB_n(X,\base)}\);
as \(X^+_x\in\{X^1_x,X^2_x\}\), the latter is a further quotient
of \(Q_{X^+_x}\), and the injectivity of the composite forces
\(\PB_n(\Lollipop,\bfy^0)\hookrightarrow Q_{X^+_x}\). The
composition of the claim equals \(\phi\) followed by this
embedding; as \(\phi\) is injective, the composition is
injective, proving the claim.

\smallskip
With the claim in hand, \(\val_X(x)\geq 3\) forces \(X\not\cong S^1\),
so \cref{lem:kernel-quotient-reduction} applies with the connected
subspace \(X^+_x\supseteq X_0\) in the role of \(X^+\),
\(Y=\sfS_3\) (contractible), and the given \(H\); it gives
\(H=1\).
However, \(H=\PB_n(\sfS_3,\bfy^0)\) is non-trivial for
\(n\geq 2\), and
this contradicts
\(H=1\).
\end{proof}

The next proposition strengthens
\cref{prop:X0-contains-valency-3,prop:X0-contains-m2-val3} in the common regime
\(\val_X(x)\geq 3\) and \(m\geq 2\): not only does \(X_0\) contain \(x\), it
also enters \(x\) from every link direction.

\begin{proposition}\label{prop:X0-meets-link-components}
Suppose \(\val_X(x)\geq 3\) and \(m\geq 2\). Then 
the inclusion \(\lk_{X_0}(x)\hookrightarrow\lk_X(x)\) is surjective on
\(\pi_0\).
\end{proposition}

\begin{proof}
Under the present hypotheses,
\cref{prop:X0-contains-valency-3,prop:X0-contains-m2-val3} already give \(x\in X_0\), so
\(\lk_{X_0}(x)\) is non-empty. We must show \(X_0\cap A_L\neq\varnothing\)
for every component \(L\subseteq\lk_X(x)\).

Suppose, for contradiction, that there exists a link component
\(L^1\subseteq\lk_X(x)\) with \(X_0\cap A_{L^1}=\emptyset\). Write
\(X^1\) for the \(x\)-component containing \(L^1\) and fix
\(p^1\in L^1\). 

If there exist two link components
\(L^2,L^3\subseteq\lk_X(x)\setminus L^1\) whose \(x\)-components
are distinct, then
pick link points \(p^2\in L^2\)
and \(p^3\in L^3\), and form the canonical star embedding
\(\sfS_3\hookrightarrow X\) with central vertex \(x\) and three
short leaves entering \(\lk_X(x)\) at \(p^1,p^2,p^3\); the induced
link map is the injection
\(\{p^1,p^2,p^3\}\hookrightarrow\lk_X(x)\). The witness \(X_0\) is
a connected subcomplex with \(X_0\cap A_{L^1}=\varnothing\) by the
contradiction hypothesis, and the leaves through \(p^2\) and
\(p^3\) run into distinct \(x\)-components, so
\cref{prop:star-avoiding-embedding}, applied with \(Z=X_0\),
yields an embedding
\[
H\coloneqq\PB_n(\sfS_3,\bfy^0)\hookrightarrow
\PB_n(X,\base)\big/\co{\PB_n(X_0,\base)}^{\PB_n(X,\base)}
=: Q_{X_0}.
\]
Since \(\val_X(x)\geq 3\) forces \(X\not\cong S^1\),
\cref{lem:kernel-quotient-reduction} applies with \(X^+=X_0\),
\(Y=\sfS_3\) (contractible), and this \(H\); it yields
\(\PB_n(\sfS_3)=1\), contradicting the non-triviality of the
tripod braid group for \(n\geq 2\).

It remains to treat the case in which every link component of
\(\lk_X(x)\) other than \(L^1\) lies in a single
\(x\)-component. That component must be distinct from \(X^1\)
(otherwise \(m=1\)); in particular no link component other than
\(L^1\) lies in \(X^1\), so \(m=2\), \(\val_{X^1}(x)=1\), and
\(\val_{X^2}(x)\geq 2\) for the other component \(X^2\). (The
labelling is thus forced: were \(\val_{X^1}(x)\geq2\), a second
link component in \(X^1\) together with one in \(X^2\) would
lie in distinct \(x\)-components, and the previous case would
apply.) We claim
that then \(X_0\cap X^1=\{x\}\): a component of
\(X_0\setminus\{x\}\) lying in \(X^1_x\) would have \(x\) in its
closure (as \(X_0\) is connected), and since \(L^1\) is the only
link component of \(x\) in \(X^1\), it would meet the cone
direction \(A_{L^1}\), contradicting
\(X_0\cap A_{L^1}=\varnothing\). Pick two distinct link components
\(L^2,L^3\subseteq\lk_{X^2}(x)\) -- possible since
\(\val_{X^2}(x)\geq2\) -- form the star \(\sfS_3\) through
\(p^1,p^2,p^3\) as before, and apply case (b) of
\cref{prop:star-avoiding-embedding} with \(Z=X_0\): it yields an
embedding \(\PB_n(\sfS_3,\bfy^0)\hookrightarrow Q_{X_0}\), and the
kernel--quotient reduction concludes exactly as above. This
completes the proof.
\end{proof}

\subsection{Bivalent vertices with exactly two \texorpdfstring{\(x\)}{x}-components}

The remaining regime is precisely that of
\cref{cor:other-x-component-embeds}: \(x\) is an interior point of a
free edge of \(X\) with \(X_x\) disconnected, so that
\(\val_X(x)=2\), the link \(\lk_X(x)\) is a pair of points, and \(X\)
has exactly two \(x\)-components \(X^1,X^2\) meeting only at \(x\) (in
particular \(m=2\)). A mild non-degeneracy hypothesis on \(X^1,X^2\)
then settles both questions at once.

\begin{proposition}\label{prop:X0-contains-nontrivial-junction}
Suppose \(x\) is an interior point of a free edge of \(X\) with
\(X_x\) disconnected -- the hypothesis of
\cref{cor:free-edge-both-quotient}, so that \(\val_X(x)=2\) and \(X\)
has exactly two \(x\)-components \(X^1,X^2\) (whence \(m=2\)) -- and
that neither \(X^1\) nor \(X^2\) is homeomorphic to an interval.
Then
\(X_0\cap X^1_x\neq\varnothing\neq X_0\cap X^2_x\). In particular, \(x\in X_0\) and the inclusion \(\lk_{X_0}(x)\hookrightarrow \lk_X(x)\) is surjective on \(\pi_0\).
\end{proposition}

\begin{proof}
Since \(X^1_x\) and \(X^2_x\) are the two connected components
of \(X_x\) and \(X_0\) is connected, once \(X_0\) meets both of
them it must contain the separating vertex \(x\) and enter it
from both link directions; thus it suffices to prove that
\(X_0\cap X^i_x\neq\varnothing\) for \(i=1,2\).

We first build an embedded copy of \(\sfH\) around \(x\). For
\(i=1,2\), follow the maximal chain of free edges issuing from
\(x\) into \(X^i\), and let \(\sfv_i\) be its terminal point. If
\(\sfv_i\) had valency \(1\), the chain would meet the rest of
\(X^i\) only at \(x\); as \(X^i_x\) is connected, this would
force \(X^i\) to be the chain itself, an interval -- excluded by
hypothesis. Hence, by maximality of the chain, each \(\sfv_i\)
either has valency at least \(3\) in \(X\) or is a joint, lying
in a thick component. The union
\(\alpha\subseteq X\) of the two chains is therefore an embedded
path through \(x\), containing the free edge at \(x\), whose
endpoints satisfy the hypothesis of the construction preceding
\cref{lem:H-shuffle}; let \(\sfH\subseteq X\) and
\(\gamma_\sfH\) be the resulting tree and shuffle braid, with
the base configuration transported to \(\base\) -- as we may,
kernels and normal closures being insensitive to this choice.

Suppose now, for contradiction, that
\(X_0\cap X^2_x=\varnothing\). Then \(X_0\subseteq X^1\), so
\[
\ker\iota_\ast=\co{\im j_\ast}
\subseteq\co{\PB_n(X^1,\base)}^{\PB_n(X,\base)}
\subseteq N\coloneqq
\co{\PB_n(X^1),\,\PB_n(X^2)}^{\PB_n(X,\base)} .
\]
But the image of \(\gamma_\sfH\) in \(\PB_n(X,\base)\) lies in
\(\ker\iota_\ast\) by \cref{lem:H-shuffle}, while
\(\gamma_\sfH\notin N\) by
\cref{cor:shuffle-survives-quotient} -- a contradiction.
Therefore \(X_0\cap X^2_x\neq\varnothing\), and symmetrically
\(X_0\cap X^1_x\neq\varnothing\), as desired.
\end{proof}

The next proposition treats the remaining regime, that of
\emph{joints}: points of valency \(2\) whose link has a
component of positive dimension. There the witness is forced not
only through the point, but also into the thick part.

\begin{proposition}[Joints]\label{prop:X0-meets-link-joint}
Suppose \(\val_X(x)=2\) and
\(\lk_X(x)=L\sqcup\{p\}\) with \(\dim L\geq1\) -- that is,
\(x\) is a joint -- and suppose that no \(x\)-component of
\(X\) is homeomorphic to an interval. Then \(x\in X_0\) and the
inclusion \(\lk_{X_0}(x)\hookrightarrow\lk_X(x)\) is surjective
on \(\pi_0\).
\end{proposition}

\begin{proof}
Write \(\sfe_p\) for the edge of \(X\) at \(x\) in the direction of
\(p\) -- a free edge, the link component \(\{p\}\) being a
point -- and \(A_L\) for the component of the punctured open
star at \(x\) on the side of \(L\). Let \(y\) be any interior
point of \(\sfe_p\). If \(X_y\) is connected, then
\cref{prop:X0-contains-free-points,prop:X0-meets-link-m1} place
\(y\) in \(X_0\) together with both germs of \(\sfe_p\); if
\(X_y\) is disconnected, the same follows from
\cref{prop:X0-contains-nontrivial-junction}, whose hypotheses
hold at \(y\): the \(y\)-component on the far side of \(\sfe_p\) is
homeomorphic to the \(x\)-component containing it, hence not an
interval, while the \(y\)-component containing \(x\) contains
the cone over \(L\), of dimension at least \(2\), hence an embedded tripod.
Letting \(y\) vary, the whole edge \(\sfe_p\) lies in \(X_0\); in
particular \(x\in X_0\), approached along the \(p\)-germ.

It remains to show that \(X_0\) approaches \(x\) along \(L\).
Suppose not, and fix an interior point \(y\in\mathring \sfe_p\).
The portion of \(X_0\) beyond \(y\) is then exactly the
terminal segment \([y,x]\) of \(\sfe_p\), with nothing of \(X_0\)
attached to it except at \(y\): the only directions out of
\(x\) are back along \(\sfe_p\) and into \(A_L\), which \(X_0\)
avoids, and the interior points of \(\sfe_p\) have valency \(2\).
Choose \(z\in\mathring\sfe_p\) on the other side of \(y\)
from \(x\). Compress the longer terminal segment \([z,x]\)
strictly monotonically onto a non-degenerate subsegment of
\([z,y)\), fixing \(z\) and extending by the identity outside
\([z,x]\). Linear interpolation gives an isotopy of embeddings
\(X_0\hookrightarrow X\) whose final image misses \(y\)
(\cref{fig:interval-retraction}). An isotopy of embeddings changes
the image of \(\PB_n(X_0)\) in \(\PB_n(X,\base)\) only by a
conjugation, so the resulting subcomplex \(X_0'\) is again a
weak Goldberg witness; but \(y\notin X_0'\), contradicting the
propositions quoted in the previous paragraph, which force
\(y\in X_0'\). Hence \(X_0\) meets \(A_L\) and, entering \(x\)
from both link components, satisfies the stated
\(\pi_0\)-surjectivity.
\end{proof}

\begin{figure}[ht]
\centering
\begin{tikzpicture}[line width=0.8pt, scale=1.1]
  \fill[gray!30,opacity=0.5]
    (0,0) -- (1.5,0.75) .. controls (1.78,0.28) and (1.78,-0.28) .. (1.5,-0.75) -- cycle;
  \draw (0,0) -- (1.5,0.75) .. controls (1.78,0.28) and (1.78,-0.28) .. (1.5,-0.75) -- (0,0);
  \node at (1.05,0) {\small\(\operatorname{Cone}(L)\)};
  \draw (0,0) -- (-2.55,0);
  \node[above=1pt] at (-0.6,0.03) {\(\sfe_p\)};
  \fill[gray!60] (-2.9,0) ellipse (0.42 and 0.28);
  \draw (-2.9,0) ellipse (0.42 and 0.28);
  \node[below=2pt] at (-2.9,-0.28) {\(X_0\)};
  \draw[line width=2.2pt, gray!70] (-2.55,0) -- (0,0);
  \draw[line width=2.2pt, gray!95!black] (-1.3,0) -- (0,0);
  \draw[->, red!75!black, thick] (-0.1,-0.25) -- node[midway,below,black] {\text{isotopy}} (-1.65,-0.25);
  \filldraw (-2.05,0) circle (1.6pt);
  \node[above=2pt] at (-2.05,-0.04) {\(z\)};
  \filldraw (-1.3,0) circle (1.6pt);
  \node[above=2pt] at (-1.3,-0.04) {\(y\)};
  \filldraw[red] (0,0) circle (1.8pt);
  \node[above=2pt] at (0,-0.05) {\(x\)};
\end{tikzpicture}
\caption{The contradiction in the proof of
\cref{prop:X0-meets-link-joint}: if the witness avoided the
side of \(L\), its portion beyond \(y\) would be exactly the
terminal segment \([y,x]\) of \(\sfe_p\) (dark). Compressing
the longer segment \([z,x]\) into \([z,y)\), while fixing
the witness outside \([z,x]\), removes \(y\) from its image
(arrow), contradicting the propositions forcing \(y\in X_0\).}
\label{fig:interval-retraction}
\end{figure}

\subsection{(Weak) Goldbergness under resolution}
\label{ssec:goldberg-resolution}

The conditions above are stated for a fixed complex; we close the
discussion of necessary conditions by recording how the (weak)
Goldberg property of a subcomplex behaves under a single An--Park
resolution. This is the mechanism by which the analysis is
transported between a complex and its simple model.

Throughout this subsection, let \(x\) be a non-simple vertex
of \(X\), and let \(X'\) be obtained by resolving one
positive-dimensional connected link component at \(x\), as in
\cref{sssec:general-resolution}. Thus
\(X'\) carries a new free edge \(\sfe\) whose endpoints we denote by the
original vertex \(x\) and the new apex \(x'\). The collapse
\[
q_\sfe\colon X'\longrightarrow X'/\sfe = X
\]
recovers \(X\), re-identifying \(x'\) with \(x\). For a subcomplex
\(X_0\subseteq X\) write
\(X_0'\coloneqq q_\sfe^{-1}(X_0)\subseteq X'\), so that
\(q_\sfe(X_0')=X_0'/\sfe=X_0\). The resolution is a braid equivalence
\(X\equiv_B X'\) by the construction in \cref{sssec:general-resolution}, so the
collapse induces an isomorphism
\[
q_\sfe^\ast\colon\PB_n(X)\xrightarrow{\cong}\PB_n(X');
\]
moreover \(q_\sfe\) collapses the contractible edge \(\sfe\), hence is a
homotopy equivalence, and the strand maps of \(X\) and \(X'\)
correspond under \(q_\sfe^\ast\) and the induced isomorphism on
\(\prod_i\pi_1\). The same collapse restricts to
\(q_\sfe\colon X_0'\to X_0\) and induces a homomorphism
\[
q_\sfe^\ast\colon\PB_n(X_0)\longrightarrow\PB_n(X_0').
\]
On the subcomplex, however, this map need not be surjective
(\cite[Prop.~3.1]{AnPark2017}): resolving the single vertex of
\(X_0\) at \(x\) need not itself be a braid equivalence.

\begin{proposition}[(Weak) Goldberg witnesses under resolution]
\label{prop:goldberg-resolution}
With the notation above, suppose that
\(q_\sfe^\ast\colon\PB_n(X_0)\to\PB_n(X_0')\) is surjective. Then:
\begin{enumerate}[label=\textup{(\arabic*)},leftmargin=2.4em]
  \item \(X_0\) is a weakly Goldberg witness for \(X\) if and only if
        \(X_0'\) is a weakly Goldberg witness for \(X'\);
  \item if \(X_0\) is a Goldberg witness for \(X\), then \(X_0'\) is a
        Goldberg witness for \(X'\);
  \item if \(X_0'\) is a Goldberg witness for \(X'\),
        \(\val_{X_0'}(x)\ge 2\), and \(\val_{X_0'}(x')=2\), then \(X_0\) is a Goldberg
        witness for \(X\).
\end{enumerate}
\end{proposition}

\begin{proof}
Write \(\phi:=q_\sfe^\ast\colon\PB_n(X_0)\to\PB_n(X_0')\) (surjective by
hypothesis), \(\Phi:=q_\sfe^\ast\colon\PB_n(X)\to\PB_n(X')\) (an
isomorphism), and \(j_\ast\colon\PB_n(X_0)\to\PB_n(X)\),
\(j'_\ast\colon\PB_n(X_0')\to\PB_n(X')\) for the inclusion-induced
maps. Compatibility of the collapse with the inclusions
\(X_0\hookrightarrow X\) and \(X_0'\hookrightarrow X'\) gives a
commutative square
\[
\begin{tikzcd}
\PB_n(X_0)\ar[r,"j_\ast"]\ar[d,"\phi"',two heads]
  & \PB_n(X)\ar[d,"\Phi"',"\cong"]\\
\PB_n(X_0')\ar[r,"j_\ast'"'] & \PB_n(X').
\end{tikzcd}
\tag{$\ast$}
\]
Since \(\phi\) is surjective and \(\Phi\) is an isomorphism,
\((\ast)\) carries normal closures to normal closures:
\begin{equation}
\Phi\bigl(\co{\im j_\ast}\bigr)
=\co{\Phi(\im j_\ast)}
=\co{j'_\ast(\im\phi)}
=\co{\im j'_\ast}.
\label{eq:Phi}
\end{equation}
Moreover \(q_\sfe\) collapses the contractible edge \(\sfe\), hence is a
homotopy equivalence; thus \(X_0\) is contractible if and only if
\(X_0'\) is, the basepoint condition transfers along \(q_\sfe\), and the
strand maps correspond under \(\Phi\) and the induced isomorphism on
\(\prod_i\pi_1\), so that \(\Phi(\ker\iota_\ast)=\ker\iota'_\ast\).

Combining \eqref{eq:Phi} with
\(\Phi(\ker\iota_\ast)=\ker\iota'_\ast\),
\[
\co{\im j_\ast}=\ker\iota_\ast
\Longleftrightarrow
\co{\im j'_\ast}=\Phi\bigl(\co{\im j_\ast}\bigr)
=\Phi(\ker\iota_\ast)=\ker\iota'_\ast .
\]
As contractibility and the basepoint condition transfer along
\(q_\sfe\), \(X_0\) is a weakly Goldberg witness for \(X\) if and only if
\(X_0'\) is a weakly Goldberg witness for \(X'\).

Suppose \(X_0\) is a Goldberg witness, so that \(j_\ast\)
is injective. From \(j'_\ast\circ\phi=\Phi\circ j_\ast\) and the
injectivity of \(\Phi\), the composite \(\Phi\circ j_\ast\) is
injective, hence so is \(\phi\); being also surjective, \(\phi\) is an
isomorphism. Then \(j'_\ast=\Phi\circ j_\ast\circ\phi^{-1}\) is
injective. Together with~(1) this shows \(X_0'\) is a Goldberg witness
for \(X'\).

Suppose \(\val_{X_0'}(x)\geq2\) and \(\val_{X_0'}(x')=2\).
The new apex \(x'\) is bivalent, while \(x\) is not a boundary
point of \(X_0'\), so collapsing \(\sfe\) is a
braid equivalence -- a homeomorphism when the relevant link is a
single point, and otherwise a genuine An--Park resolution of \(X_0\)
at \(x\) (\cref{sssec:general-resolution}, extending
\cite[Props.~3.11--3.12]{AnPark2017}) -- and \(\phi\) is an
isomorphism. Hence \(j_\ast=\Phi^{-1}\circ j'_\ast\circ\phi\) is
injective whenever \(j'_\ast\) is, and by~(1) the kernel equality
transfers. Therefore, if \(X_0'\) is a Goldberg witness for \(X'\),
then \(X_0\) is a Goldberg witness for \(X\).
\end{proof}

The proposition feeds into the following stability statement: weak
Goldbergness, and Goldbergness, are inherited by a resolution.

\begin{theorem}\label{thm:resolution-preserves-goldberg}
If \(X\) is weakly Goldberg \textup{(}resp.\ Goldberg\textup{)}, then
\(X'\) is weakly Goldberg \textup{(}resp.\ Goldberg\textup{)}.
\end{theorem}

\begin{proof}
Let \(X_0\subseteq X\) be a weakly Goldberg (resp.\ Goldberg) witness;
we produce a corresponding witness in \(X'\).

\emph{Case \(x\notin X_0\).} Put \(X_0':=q_\sfe^{-1}(X_0)\). As \(X_0\)
avoids the resolved vertex, \(q_\sfe\) restricts to a homeomorphism
\(X_0'\xrightarrow{\cong}X_0\), so
\(q_\sfe^\ast\colon\PB_n(X_0)\to\PB_n(X_0')\) is an isomorphism, in
particular surjective; parts (1) and (2) of
\cref{prop:goldberg-resolution} make \(X_0'\) a weakly Goldberg
(resp.\ Goldberg) witness for \(X'\).

\emph{Case \(x\in X_0\) and \(\val_X(x)\geq 3\).} Put
\(X_0':=q_\sfe^{-1}(X_0)\). Since \(x\in X_0\), the inclusion
\(\lk_{X_0}(x)\hookrightarrow\lk_X(x)\) is surjective on \(\pi_0\) --
by \cref{prop:X0-meets-link-m1} when \(m=1\), and by
\cref{prop:X0-meets-link-components} when \(m\geq 2\) -- so \(X_0\)
meets \(x\) from every link direction. Consequently both endpoints
\(x\) and \(x'\) of \(\sfe\) have valency at least \(2\) in \(X_0'\), and
\(q_\sfe^\ast\) is surjective by An--Park~\cite[Prop.~3.1]{AnPark2017};
parts (1) and (2) of \cref{prop:goldberg-resolution} again apply.

\emph{Case \(x\in X_0\) and \(\val_X(x)=2\).} Put
\(X_0':=q_\sfe^{-1}(X_0)\). If \(\val_{X_0'}(x)\geq 2\) and
\(\val_{X_0'}(x')\geq 2\), then \(q_\sfe^\ast\) is surjective by
\cite[Prop.~3.1]{AnPark2017} and we conclude as before. Otherwise
\(\val_{X_0'}(x)=1\) or \(\val_{X_0'}(x')=1\), and then \(X_0\) embeds
into \(X'\) as a subcomplex. Taking \(X_0'\) to be this embedded copy
of \(X_0\), both vertical maps in the square \((\ast)\) of
\cref{prop:goldberg-resolution} -- the identity
\(\PB_n(X_0)\to\PB_n(X_0')\) and \(\Phi\colon\PB_n(X)\to\PB_n(X')\) --
are isomorphisms, so the same argument makes \(X_0'\) a weakly
Goldberg (resp.\ Goldberg) witness for \(X'\).

In every case \(X'\) admits a weakly Goldberg (resp.\ Goldberg)
witness; hence \(X'\) is weakly Goldberg (resp.\ Goldberg).
\end{proof}

The converse holds unconditionally for weak Goldbergness, and for
Goldbergness under a mild valency hypothesis on the witness.

\begin{theorem}\label{thm:resolution-reflects-goldberg}
If \(X'\) is weakly Goldberg, then \(X\) is weakly Goldberg. If
moreover \(X'\) is Goldberg and admits a Goldberg witness \(X_0'\) with
\(\val_{X_0'}(x)\ge 2\) and \(\val_{X_0'}(x')=2\), then \(X\) is Goldberg.
\end{theorem}

\begin{proof}
Let \(X_0'\subseteq X'\) be a weak Goldberg witness.
We first show that it already contains \(\sfe\) and meets
both endpoints in at least two link components.

Write \(\lk_X(x)=L\sqcup R\), where \(L\) is the component
being resolved. Since \(x\) is non-simple, \(R\) either has
a positive-dimensional component or consists of at least
two isolated points. For \(y\in\mathring\sfe\), the side
towards \(x'\) contains the cone over \(L\), hence a
\(2\)-simplex. The side towards \(x\) either contains another
positive-dimensional cone, or contains a tripod formed by
the germ towards \(y\) and two directions in \(R\).
Thus, if \(X'_y\) is disconnected, neither \(y\)-component
is an interval. Applying
\cref{prop:X0-contains-free-points} when \(X'_y\) is connected,
and \cref{prop:X0-contains-nontrivial-junction} otherwise,
gives \(y\in X_0'\). Since \(X_0'\) is a closed subcomplex,
\(\sfe\subseteq X_0'\).

At \(x'\), the link is \(L\sqcup\{x\}\). No \(x'\)-component
is an interval, by the same cone/tripod observation.
Hence \cref{prop:X0-meets-link-joint} forces the witness to
meet both link components, giving
\(\val_{X_0'}(x')\geq2\).
At \(x\), if \(\val_{X'}(x)=2\), then \(R\) is connected
and positive-dimensional, and the same joint proposition
applies: each side contains a positive-dimensional cone.
If \(\val_{X'}(x)\geq3\), use
\cref{prop:X0-meets-link-m1} when the cut at \(x\) is
connected and \cref{prop:X0-meets-link-components} otherwise.
In every case \(\val_{X_0'}(x)\geq2\).

Set \(X_0=q_\sfe(X_0')=X_0'/\sfe\). The containment
\(\sfe\subseteq X_0'\) gives \(q_\sfe^{-1}(X_0)=X_0'\),
and \(X_0\) is contractible because collapsing the
contractible subcomplex \(\sfe\) is a homotopy equivalence.
Both endpoints have witness valency at least \(2\), so
An--Park~\cite[Prop.~3.1]{AnPark2017} makes
\(q_\sfe^*\colon\PB_n(X_0)\to\PB_n(X_0')\) surjective.
By \cref{prop:goldberg-resolution}(1), \(X_0\) is a weak
Goldberg witness for \(X\).

Finally, if \(X_0'\) is a Goldberg witness satisfying the
stated valency condition \(\val_{X_0'}(x')=2\), the same
unchanged witness contains \(\sfe\), and
\cref{prop:goldberg-resolution}(3) makes \(X_0\) a Goldberg
witness for \(X\). This proves the strong assertion as well.
\end{proof}

\begin{remark}\label{rem:resolution-goldberg-equivalence}
The valency hypothesis in the second assertion of
\cref{thm:resolution-reflects-goldberg} will turn out to be
harmless: whenever \(X'\) is Goldberg, one can always choose a
Goldberg witness \(X_0'\) satisfying \(\val_{X_0'}(x)\geq2\) and
\(\val_{X_0'}(x')=2\). Granting this,
\cref{thm:resolution-preserves-goldberg} and
\cref{thm:resolution-reflects-goldberg} combine into an
equivalence: \(X\) is Goldberg if and only if its resolution
\(X'\) is Goldberg. Thus, like weak Goldbergness, Goldbergness
itself is detected on the simple model. This is carried out below
-- \cref{lem:witness-apex-valency} controls
\(\val_{X_0'}(x')\), and the discussion following it culminates in
\cref{thm:resolution-goldberg-equivalence}.
\end{remark}

We record how a Goldberg witness can meet the link of a point of
valency at most two. The hypotheses below describe the local
picture at any point \(x\) of a thick component \(\bar\Sigma\) of
a simple complex \(\bar X\) with \(\val_{\bar X}(x)\leq2\) -- if
\(\val_{\bar X}(x)=2\), then \(x\) is a joint, and simplicity
forces the second of the two forms displayed
(\cref{def:simple-complex}) -- and also the local picture at the
apex \(x'\) of a resolution, with \(L=L'\).

\begin{lemma}\label{lem:witness-apex-valency}
Let \(\bar X\) be Goldberg, with Goldberg witness \(\bar X_0\),
and let \(x\in\bar X_0\) be a point of \(\bar X\) with
\(\val_{\bar X}(x)\leq2\) whose link is, apart from at most one
isolated point, connected and of positive dimension:
\[
\lk_{\bar X}(x)=L
\qquad\text{or}\qquad
\lk_{\bar X}(x)=L\sqcup\{y\},
\]
with \(L\) connected and \(\dim L\geq1\) and, in the second
form, \(y\) the direction of a free edge \(\sfe\) of \(\bar X\) at
\(x\), with \(\sfe\subseteq\bar X_0\). If
\(\val_{\bar X_0}(x)\geq2\), then
\[
\val_{\bar X_0}(x)=2,
\]
with a single exceptional case: \(n=2\),
\(\val_{\bar X_0}(x)=3\), and \(L\) an interval or -- in the
first form only -- a circle.
\end{lemma}

\begin{proof}
Write \(N(x)\) and \(N_0(x)\) for regular neighbourhoods of
\(x\) in \(\bar X\) and in \(\bar X_0\) respectively, chosen with
\(N_0(x)\subseteq N(x)\), and set
\(m\coloneqq\val_{\bar X_0}(x)\geq2\). In the second form
\(\sfe\subseteq\bar X_0\) gives \(y\in\lk_{\bar X_0}(x)\), and we
decompose
\[
\lk_{\bar X_0}(x)=L^{1}\sqcup\cdots\sqcup
L^{m-1}\sqcup\{y\},
\qquad L^{j}\subseteq L;
\]
in the first form we decompose
\(\lk_{\bar X_0}(x)=L^{1}\sqcup\cdots\sqcup L^{m}\) with
\(L^{j}\subseteq L\). Suppose, for contradiction, that
\(m\geq3\).

First apply \cref{prop:star-embedding}, in the generality of
\cref{rem:star-link-components}, to \(\bar X_0\) at \(x\): there
is an embedded star \(\sfS_m\subseteq\bar X_0\), with one leaf
through each of the \(m\) link components of
\(\lk_{\bar X_0}(x)\), such that
\(\PB_n(\sfS_m)\to\PB_n(\bar X_0)\) is injective; we may take
\(\sfS_m\subseteq N_0(x)\). Since \(\bar X_0\) is a
\emph{Goldberg} witness, \(\PB_n(\bar X_0)\to\PB_n(\bar X)\) is
injective as well, and these inclusions factor as
\[
\begin{tikzcd}
& & \bar X_0 \ar[rd,hookrightarrow]\\
\sfS_m\ar[r,hookrightarrow] & N_0(x) \ar[ru,hookrightarrow]\ar[rd,hookrightarrow] & & \bar X\\
& & N(x)\ar[ru,hookrightarrow]
\end{tikzcd}
\]
The composite \(\PB_n(\sfS_m)\to\PB_n(\bar X)\) is therefore
injective, and since it factors through \(\PB_n(N(x))\), the map
\(\PB_n(\sfS_m)\to\PB_n(N(x))\) must be injective.

We first show that \(\PB_n(N(x))=1\) whenever \(L\) is
neither an interval nor a circle. Put \(C=\operatorname{Cone}(L)\).
This is contractible, and all of its vertex links are connected:
the apex has link \(L\), while the link of a vertex \(v\in L\)
is the cone over \(\lk_L(v)\). Thus \(C\) is simple with every
vertex of valency one. Since its apex link \(L\) is neither
an interval nor a circle, \(C\) is not a \(2\)-manifold.
Consequently \cref{prop:branched-surface} gives
\[
\PB_n(C)\cong\prod_i\pi_1(C)=1.
\]
In the first form \(N(x)=C\). In the second form \(N(x)\) is
obtained from \(C\) by attaching a pendant half-edge at its
apex; the link there is connected, so \cref{lem:pendant-edge}
gives a surjection
\[
1=\PB_n(C)\twoheadrightarrow\PB_n(N(x)),
\]
and again \(\PB_n(N(x))=1\). This includes every case with
\(\dim L\geq2\), as well as the branched-tree and other
non-manifold graph links.

If \(L\) is a circle and we are in the second form, then
\(N(x)\) is a disc with a pendant edge attached at an interior
point, so \(\PB_n(N(x))=1\) by \cref{ex:disc-interior}.
In all these cases the injectivity already proved would force
\(\PB_n(\sfS_m)=1\), contrary to \(m\geq3\) and \(n\geq2\).

In the remaining cases \(N(x)\) is a disc, possibly with a
pendant edge at a boundary point, and
\(\PB_n(N(x))\cong\PB_n(D^2)\): if \(L\) is an interval, \(x\)
is a boundary point of the disc and, in the second form, the
edge \(\sfe\) is attached there (\cref{ex:disc}); if \(L\) is a
circle -- possible in the first form only, the other cases
having been disposed of above -- then \(x\) is an interior point
of the honest disc \(N(x)\). For \(m\geq3\)
the map \(\PB_n(\sfS_m)\to\PB_n(D^2)\) has non-trivial kernel
unless \(m=3\) and \(n=2\). Indeed, for \(n=2\) and \(m\geq4\),
the group \(\PB_2(\sfS_m)\) is free of rank at least two -- by
the rank formula of \cref{ex:one-essential} with \(\ell=0\),
\(\BG_2(\sfS_m)\) is free of rank \((m-1)(m-2)/2\geq3\), and
Nielsen--Schreier gives \(\PB_2(\sfS_m)\) rank
\((m-1)(m-2)-1\geq5\) -- and free groups of rank \(\geq2\) do not
embed into \(\PB_2(D^2)\cong\Z\). For \(n\geq3\), the map is
\emph{surjective}: the Artin generators \(\sigma_i\) of
\(\BG_n(D^2)\) are realised by star braids -- line the strands
up along one arm and exchange the two innermost ones through two
of the other arms -- so \(\BG_n(\sfS_m)\to\BG_n(D^2)\) is onto,
and comparing the kernels of the projections to the symmetric
group, so is \(\PB_n(\sfS_m)\to\PB_n(D^2)\). Were the latter
also injective, it would be an isomorphism, making
\(\PB_n(D^2)\) a free group; but for \(n\geq3\) it contains
\(\PB_3(D^2)\cong\mathbb{F}_2\times\Z\), hence a free abelian
subgroup of rank two, which a free group does not contain.
Injectivity therefore forces \(m=3\) and \(n=2\); and this case
arises only when \(L\) is an interval or -- in the first form
only -- a circle. This
completes the proof.
\end{proof}

The forcing argument in the proof of
\cref{thm:resolution-reflects-goldberg} already gives
\(\sfe\subseteq X_0'\) and witness valency at least \(2\)
at both endpoints. Applying the lemma with \(\bar X=X'\) and \(x=x'\) the apex of
the resolution -- the second form, with \(L=L'\) and \(y\) the
direction of \(\sfe\) -- we find that \(\val_{X_0'}(x')=2\), except
possibly when \(n=2\), \(\val_{X_0'}(x')=3\) and \(L'\) is an
interval.

We can now remove the exceptional case. Suppose \(n=2\), \(m=3\)
and \(L'\) is an interval, and write
\(\lk_{X_0'}(x')=L'^{\,1}\sqcup L'^{\,2}\sqcup\{y\}\) as before.
If \(L'^{\,1}\) or \(L'^{\,2}\) is an interval rather than a
point, resolve \(X_0'\) at \(x'\) along that component; since the
cone over the interval \(L'\) is two-dimensional, the resolution
can be realised \emph{inside} a regular neighbourhood \(N'\) of
\(x'\) in \(X'\), producing a subcomplex
\(X_0''\subseteq X'\) together with a braid equivalence
\(\PB_n(X_0')\cong\PB_n(X_0'')\)
(\cite[Props.~3.11--3.12]{AnPark2017}) compatible with the inclusions
into \(X'\); hence \(X_0''\) is again a Goldberg witness. Without
loss of generality we may therefore assume that the regular
neighbourhood \(N_0'\) of \(x'\) in \(X_0'\) is
precisely a tripod \(\sfS_3\), i.e.\ that \(L'^{\,1}\) and
\(L'^{\,2}\) are points.

We now enlarge the witness inside \(X'\); throughout, \(X'\)
itself is left untouched and only \(X_0'\) is modified. Let
\(I\subseteq X_0'\) be the interval through \(x'\) formed by the
halves, adjacent to \(x'\), of the two arms of \(N_0'\)
corresponding to \(L'^{\,1}\) and \(L'^{\,2}\). Since the cone
over the interval \(L'\) is two-dimensional, we may choose a
\(2\)-disc \(D\) inside \(N'\subseteq X'\), part of whose boundary
runs along \(I\), and set
\[
X_0''\coloneqq X_0'\cup D\subseteq X'.
\]
The complex \(X_0''\) is precisely the construction of
\cref{lem:tripod-to-disc} applied to \(X_0'\) at the tripod point
\(x'\), so the inclusion \(X_0'\hookrightarrow X_0''\) induces an
isomorphism \(\PB_2(X_0')\cong\PB_2(X_0'')\); in particular the
two witness maps have the same image in \(\PB_2(X')\), and
\(X_0''\) is again a Goldberg witness. Moreover \(\lk_{X_0''}(x')\) now
consists of a single interval -- the link of \(x'\) in the
half-disc attached along \(I\), joining the directions of
\(L'^{\,1}\) and \(L'^{\,2}\) -- together with \(\{y\}\); hence
\(\val_{X_0''}(x')=2\). In other words, in the exceptional case as
well, a Goldberg witness with \(\val(x')=2\) can always be
arranged.

Combining everything, we obtain the equivalence promised in
\cref{rem:resolution-goldberg-equivalence}.

\begin{theorem}\label{thm:resolution-goldberg-equivalence}
\(X\) is weakly Goldberg (resp.\ Goldberg) if and only if its
resolution \(X'\) is.
\end{theorem}

\begin{proof}
For weak Goldbergness, and for the forward implication in either
strength, this is
\cref{thm:resolution-preserves-goldberg,thm:resolution-reflects-goldberg}.
Suppose then that \(X'\) is Goldberg, with Goldberg witness
\(X_0'\). The forcing argument in
\cref{thm:resolution-reflects-goldberg} gives
\(\sfe\subseteq X_0'\), \(\val_{X_0'}(x)\geq2\), and
\(\val_{X_0'}(x')=m\geq2\), without changing the witness. By \cref{lem:witness-apex-valency},
either \(m=2\), or \(n=2\), \(m=3\) and \(L'\) is an interval; in
the latter case the discussion above replaces \(X_0'\) by a
Goldberg witness with \(\val(x')=2\), leaving the situation at
\(x\) unchanged. In either case the hypothesis of
\cref{thm:resolution-reflects-goldberg} is met, and \(X\) is
Goldberg.
\end{proof}

\subsection{Interval pieces of weak witnesses}
\label{ssec:witness-intervals}

In this final subsection we leave the fixed point \(x\) of the
standing assumption and record global structural results on
(weak) Goldberg witnesses of the simple model \(\bar X\): a
direct crossing-record obstruction showing that a weak witness
meets the thick components in no interval pieces. Here \(\bar X\) denotes the simple model of \(X\)
(\cref{prop:simple-model-unique}), and for each thick component
\(\bar\Sigma\) of \(\bar X\) we let \(N(\bar\Sigma)\) denote its
regular neighbourhood in \(\bar X\): the union of \(\bar\Sigma\)
with one pendant edge attached at each of the finitely many
points at which \(\bar\Sigma\) is joined to the remainder of
\(\bar X\). Choose these neighbourhoods small and pairwise
disjoint, with their boundary points in the interiors of free
edges. Together with the closures of the components of the
remaining thin graph, they form a decomposition of \(\bar X\)
whose pieces meet only in finitely many points. After subdivision
all pieces are subcomplexes. The thick neighbourhoods alone
need not cover \(\bar X\).
Given a subcomplex \(\bar X_0\subseteq\bar X\) and a thick
component \(\bar\Sigma\), we write
\(\bar\Sigma_0^1,\dots,\bar\Sigma_0^m\) -- or \(\bar\Sigma_0^j\)
with a running index -- for the connected components of
\(\bar X_0\cap N(\bar\Sigma)\).

\subsubsection{An algebraic preparation}
For later use in \cref{sec:simple-model-necessity}, we record
the following elementary computation with free and direct products.

\begin{lemma}[Coordinatewise free products]
\label{lem:coordinatewise-kernel}
Let \(G\) and \(H\) be groups and \(n\geq1\).
\begin{enumerate}[label=\textup{(\roman*)},leftmargin=2.4em]
\item The kernel of the canonical surjection
      \(\bigl(\prod_{\ell=1}^nG\bigr)\ast
      \bigl(\prod_{\ell=1}^nH\bigr)
      \twoheadrightarrow\prod_{\ell=1}^n(G\ast H)\) is normally
      generated by the commutators \([g_i,h_j]\) with
      \(i\neq j\), where \(g_i\) and \(h_j\) run through the
      \(i\)-th coordinate factor \(G\) and the \(j\)-th
      coordinate factor \(H\), respectively.
\item The kernel of the canonical surjection
      \(\bigl(\prod_{\ell=1}^nG\bigr)\ast F_n
      \twoheadrightarrow\prod_{\ell=1}^n(G\ast\Z)\), sending
      the \(j\)-th letter \(\beta_j\) of \(F_n\) to the
      generator of the \(\Z\)-factor of the \(j\)-th
      coordinate, is normally generated by the commutators
      \([g_i,\beta_j]\) and \([\beta_i,\beta_j]\) with
      \(i\neq j\).
\end{enumerate}
\end{lemma}

\begin{proof}
In either case the displayed commutators map to commutators of
elements supported in distinct coordinates, hence lie in the
kernel. Conversely, let \(Q\) denote the quotient of the source
by their normal closure, and let \(C_j\leq Q\) be the subgroup
generated by the \(j\)-th coordinate factors -- by \(G_j\) and
\(H_j\) in case (i), by \(G_j\) and \(\beta_j\) in case (ii).
Any generator of \(C_j\) commutes in \(Q\) with any generator
of \(C_k\) for \(j\neq k\): coordinates of one and the same
direct product commute already in the source, and the remaining
pairs commute by the imposed relations. The assignments
\(g\mapsto g_j\), \(h\mapsto h_j\) -- respectively
\(z\mapsto\beta_j\) -- define a homomorphism
\(G\ast H\to C_j\), respectively \(G\ast\Z\to C_j\); since the
\(C_j\) commute pairwise and generate \(Q\), these assemble
into a surjection \(\prod_\ell(G\ast H)\twoheadrightarrow Q\),
respectively \(\prod_\ell(G\ast\Z)\twoheadrightarrow Q\). Its
composite with the canonical surjection of the statement is the
identity on each factor, hence the identity; so both
surjections are isomorphisms, and the kernel is as claimed.
\end{proof}

\subsubsection{Interval pieces}

\begin{theorem}[Weak witnesses have no interval pieces]
\label{thm:no-interval-pieces}
Suppose \(\bar X\) is not Goldberg-trivial, and let
\(\bar X_0\) be a weak Goldberg witness of \(\bar X\). Then no
intersection component \(\bar\Sigma_0^j\) is homeomorphic to an
interval.
\end{theorem}

\begin{proof}
Suppose, for contradiction, that some \(\bar\Sigma_0^j\) is an
interval. Call a point of \(\bar X_0\) a \emph{branch point} if
the witness offers at least three directions there, or a link
component of dimension at least \(1\); a \emph{terminal path}
is a maximal embedded path in \(\bar X_0\) starting at a free
end of \(\bar X_0\) and containing no branch point in its
interior. Since \(\bar X\) is not Goldberg-trivial, the witness
\(\bar X_0\) is not homeomorphic to a point or an interval --
the components of the configuration spaces of an interval are
contractible, so such a witness would give
\(\ker\iota_\ast=\co{\im j_\ast}=1\) -- and consequently every
terminal path ends at a branch point of \(\bar X_0\).

\emph{Step 1: \(\bar\Sigma_0^j\) lies on a terminal path.}
Suppose first that \(\bar\Sigma_0^j\) is contained in a
terminal path \(P_0\), with free end \(t_0\) and endpoint a
branch point \(b_0\). We first locate a point
\(p\in P_0\setminus\{b_0\}\) that every weak Goldberg
witness must contain, distinguishing the final approach to
\(b_0\).

If \(P_0\) approaches \(b_0\) along a thin edge \(\sfe^*\) of
\(\bar X\), then \(\sfe^*\) is essential: the side of \(b_0\)
contains the embedded tripod formed at \(b_0\) by the germ of
\(\sfe^*\) together with two further directions of \(\bar X_0\),
while the opposite side contains \(\bar\Sigma\), through which
\(P_0\) has passed. Hence the interior points of \(\sfe^*\) lie in
every weak Goldberg witness
(\cref{prop:X0-contains-free-points,prop:X0-contains-nontrivial-junction}),
so choose \(p\) on this final thin segment, distinct from
\(b_0\).

If instead \(P_0\) approaches \(b_0\) inside the thick part of
some component \(\bar\Sigma''\), there are two possibilities.
Either \(P_0\) has remained in \(N(\bar\Sigma)\) ever since
\(\bar\Sigma_0^j\) -- but then \(b_0\) belongs to the same
intersection component \(\bar\Sigma_0^j\), which therefore
branches at \(b_0\), contradicting its being an interval. Or
\(P_0\) last enters \(N(\bar\Sigma'')\) at a joint \(\sfw''\)
interior to \(P_0\); the two sides of \(\sfw''\) in \(\bar X\)
contain \(\bar\Sigma''\) and \(\bar\Sigma\) respectively,
neither an interval, so \cref{prop:X0-meets-link-joint} places
\(\sfw''\) in every weak Goldberg witness; take \(p=\sfw''\).

In either remaining case, parameterise \(P_0\) by \([0,1]\),
with \(b_0\) at \(0\). Choose \(0<\varepsilon<1\) so that
the subarc \([0,\varepsilon]\) avoids \(p\), and use the
embeddings \(s\mapsto(1-t+t\varepsilon)s\), \(0\leq t\leq1\),
on this terminal path, extending by the identity on the rest
of the witness. This is an isotopy of embeddings, since the
terminal path meets the rest only at \(b_0\). Its final image
is a weak Goldberg witness missing \(p\): moving the base
configuration along the isotopy changes the braid-group image
only by basepoint transport and conjugation. This contradicts
the choice of \(p\).

\emph{Step 2: both prolongations branch.} Otherwise, prolonged
through the witness beyond either endpoint, \(\bar\Sigma_0^j\)
reaches a branch point before terminating. Thus this interval
piece traverses \(N(\bar\Sigma)\): it enters the thick
component at a joint \(\sfw\), leaves it at a distinct joint
\(\sfw'\), and continues along thin edges on both sides.
Write \(W=\bar X_0\). Choose points \(x,y\) in those thin
edges, close to \(\sfw,\sfw'\), respectively. They lie in
the interior of an embedded interval \(J\subset W\) whose
interior has no branch point of \(W\), and which contains
\(\bar\Sigma_0^j\). Orient both thin edges in the direction
along \(J\) from \(x\) to \(y\). We will distinguish their
crossing records without making any assumption on the
connectedness of either ambient cut.

Let \(\sfv\) be the first branch point reached by prolonging
\(J\) to the left. The path from \(\sfv\) to \(\sfw\)
lies in \(W\), passes through \(x\), and avoids \(y\).
Choose two short legs in \(W\) at \(\sfv\), off this
path, and two short legs in \(\bar\Sigma\) at \(\sfw\).
The latter are available in the positive-dimensional link
component of \(\bar\Sigma\), whereas the path approaches
\(\sfw\) along its thin edge. These choices give an embedded
tree \(\sfH\subseteq\bar X\), with the path as its middle
bar, which avoids \(y\). It need not lie in \(W\).
Put the first two base points on the two legs at \(\sfv\),
outside \(J\). Park the remaining points in \(W\) beyond
the portion of one of these legs that will be used for the
two moving strands. All base points are outside \(J\),
and the parked points are outside the moving support.

\emph{Crossing records.} For \(p=x,y\), record the successive
crossings of \(p\) by strands \(1,2\) of a generic braid:
write \(a^{\pm1}\) for strand \(1\), and \(b^{\pm1}\)
for strand \(2\), according to the crossing direction;
ignore the other strands. This defines a homomorphism
\[
c_p:\PB_n(\bar X,\base)\longrightarrow F(a,b).
\]
Indeed, in a generic homotopy a tangency creates or removes
two successive inverse crossings of the same strand.
Crossings of the two strands cannot exchange order, since
that would put both strands at \(p\) simultaneously.
Here the fact that \(p\) is an interior point of a thin
edge provides the local interval coordinate. Equivalently,
this is the crossing-word map of the distribution-graph
decomposition at \(p\).

Set \(\mathcal K=\ker\iota_*^{\bar X}\). If
\(\beta\in\mathcal K\), the exponent sums of both
\(a\) and \(b\) in \(c_p(\beta)\) vanish: signed
crossing number at a thin edge is homotopy invariant for
each individual strand, and each strand of \(\beta\)
is null-homotopic. To retain the order of the crossings,
use the integer Heisenberg group
\[
\mathcal H=\mathbb Z^3,\qquad
(r,s,t)(r',s',t')=(r+r',s+s',t+t'+rs').
\]
Map \(a\) to \(u=(1,0,0)\) and \(b\) to \(v=(0,1,0)\).
Their commutator \(uvu^{-1}v^{-1}=z=(0,0,1)\) is central
and has infinite order. A word with both exponent sums
zero maps uniquely to a power of \(z\). Hence
\[
c_p(\beta)\longmapsto z^{\mu_p(\beta)}
\quad(\beta\in\mathcal K)
\]
defines a homomorphism \(\mu_p:\mathcal K\to\mathbb Z\).
Centrality also gives
\(\mu_p(g\beta g^{-1})=\mu_p(\beta)\) for every
\(g\in\PB_n(\bar X,\base)\). In particular
\(\Delta=\mu_x-\mu_y\) is a homomorphism on \(\mathcal K\)
invariant under conjugation by the entire ambient braid group.

\emph{Witness braids have equal records.} There is an isotopy
\(\phi_t:W\to W\), supported in the interior of \(J\),
with \(\phi_0=\mathrm{id}\) and \(\phi_1(x)=y\): use an
increasing isotopy of the interval, fixing its endpoints,
and the identity on the rest of \(W\). It fixes the base
configuration. For a braid \(\eta\) in \(W\), the
\(y\)-crossing record of \(\phi_1\eta\) equals the
\(x\)-crossing record of \(\eta\), with the chosen
orientations. Since \(\phi_t\eta\) is a based homotopy,
\[
c_x(j_*\eta)=c_y(j_*\eta),\qquad
j:W\hookrightarrow\bar X.
\]
Contractibility of \(W\) gives \(\im j_*\subseteq\mathcal K\),
so \(\Delta\) vanishes on \(\im j_*\). Additivity and
ambient conjugation invariance now imply
\[
\Delta\bigl(\co{\im j_*}^{\PB_n(\bar X,\base)}\bigr)=0.
\]

\emph{An ambient shuffle has unequal records.} In \(\sfH\),
move strand \(1\) across the middle bar to one right leg,
then strand \(2\) to the other right leg, then return
strand \(1\), and finally strand \(2\), each to its
original left leg. These four sequential moves are
collision-free, with all other strands stationary. They
give a pure braid \(\gamma\) supported in the contractible
tree \(\sfH\), so \(\gamma\in\mathcal K\). Its crossing
records are
\[
c_x(\gamma)=aba^{-1}b^{-1},\qquad c_y(\gamma)=1,
\quad\text{and hence}\quad \Delta(\gamma)=1.
\]
Thus \(\gamma\notin\co{\im j_*}\), contrary to the weak
witness equality \(\mathcal K=\co{\im j_*}\). Together with
Step~1, this excludes every interval piece.
\end{proof}

\section{Goldberg-trivial simple complexes}
\label{sec:goldberg-trivial}

Before analysing witnesses any further, we pause to determine
exactly which simple complexes are Goldberg-trivial
(\cref{def:goldberg-strengths}): those with
\(\ker\iota_\ast=1\), which are Goldberg for trivial reasons
(\cref{lem:gt-implies-goldberg}) and which will later play the role of
degenerate local pieces. Throughout this section we keep the
standing assumption (\cref{rem:standing}). The guiding principle
is \cref{rem:birman-high-dim}: manifolds of dimension at least
three are Goldberg-trivial, and the classification below
(\cref{thm:gt-classification}) says that a complex is
Goldberg-trivial precisely when it is secretly three-dimensional.
The basic tool is the shuffle braid \(\gamma\) of
\cref{lem:H-shuffle}.

\begin{lemma}\label{lem:val2-free-edge}
The following hold:
\begin{enumerate}[label=\textup{(\alph*)}, leftmargin=2em]
    \item If some maximal free chain of \(\bar X\) has both ends
          of valency at least \(2\), then \(\bar X\) is
          not Goldberg-trivial.
    \item If \(\bar X\) has at least two thick components, or
          if some scaffold
          \(\Phi_{\bar X}\in\frG_{\bar X}\) of \(\bar X\) is not
          a tree, then \(\bar X\) contains such a maximal free chain.
\end{enumerate}
\end{lemma}

\begin{proof}
(a) First suppose the maximal free chain \(\sfe\) has
distinct ends. An end of valency \(2\) is a joint, since
otherwise the chain would extend past it. Every other end has
valency at least \(3\). The discussion preceding
\cref{lem:H-shuffle} therefore applies with \(\alpha=\sfe\). We thus obtain the embedded tree
\(\sfH\subseteq\bar X\) together with the braid \(\gamma\); by
\cref{lem:H-shuffle}, \(\gamma\) is a non-trivial element of
\(\PB_n(\bar X,\base)\) lying in \(\ker\iota_\ast\). Hence
\(\ker\iota_\ast\neq1\). If the two ends coincide at
\(\sfv\), the loop supplies two isolated link directions there.
Unless \(\bar X\cong S^1\), excluded by the standing
assumption, \(\sfv\) has at least one further link component.
A tripod through three distinct link components is
\(\pi_1\)-embedded on configuration spaces by
\cref{rem:star-link-components}; its non-trivial pure braids
have trivial strand classes. This proves (a) for loops as well.

(b) Suppose first that \(\bar X\) has at least two thick
components. As \(\bar X\) is connected, there is an embedded
path joining points of two distinct thick components. Away from
the thick components the path runs in the free part, and it
traverses each maximal free chain it visits from one endpoint to
the other; let \(\sfe\) be the first maximal free chain so
traversed. An embedded path cannot continue past a univalent
vertex, and this one does not terminate in the interior of the
free part; hence neither endpoint of \(\sfe\) is univalent, that
is, both endpoints of \(\sfe\) have valency at least \(2\).

Suppose instead that some scaffold \(\Phi_{\bar X}\) is not a
tree; the choices entering \(\Phi_{\bar X}\) do not matter here,
since contracting the trees \(\Phi_\Sigma\) shows that the first
Betti number of \(\Phi_{\bar X}\) is independent of them. Then
\(\Phi_{\bar X}\) contains an embedded cycle \(C\), and since
each \(\Phi_\Sigma\) is a tree, \(C\) traverses at least one
thin edge. Let \(P\subseteq C\) be a maximal subpath running in
the free part -- its endpoints are then joints, lying on thick
components, possibly one and the same -- or \(P\coloneqq C\) if
\(C\) avoids the thick components altogether. As before, \(P\)
traverses some maximal free chain of \(\bar X\) completely, and
neither endpoint of that edge is univalent: an interior vertex
of \(P\) carries two directions of \(P\), and the endpoints of
\(P\), if any, are joints. Either way \(\bar X\) contains a maximal free
chain with both endpoints of valency at least \(2\).
\end{proof}

We can now state the classification using pendant maximal free
chains, or equivalently entire pendant intervals, independently
of their subdivision.

\begin{theorem}[Classification of Goldberg-trivial
complexes]\label{thm:gt-classification}
Let \(X\) be a complex as in \cref{rem:standing}. Then
\(X\) is
Goldberg-trivial if and only if \(X\) is simple and all of the
following hold:
\begin{enumerate}[label=\textup{(\roman*)}, leftmargin=2em]
    \item no point of \(X\) has valency at least \(3\);
    \item \(X\) is obtained from a single thick component
          \(\Sigma\) by attaching finitely many intervals, each
          identified with \(\Sigma\) at exactly one endpoint;
    \item either \(\Sigma\) is not a \(2\)-manifold, or
          \(\Sigma\) is a \(2\)-manifold and at least one
          of these intervals is attached at an interior
          point of \(\Sigma\).
\end{enumerate}
\end{theorem}

A typical example is shown in \cref{fig:gt-example}.

\begin{remark}
The exclusion of the sphere (\cref{rem:standing}) matters
here: for \(n=2\) one has
\(\PB_2(S^2)=1\), so \(S^2\) is Goldberg-trivial, while the
criterion of \cref{thm:gt-classification} fails for it; for
\(n\geq3\) the group \(\PB_n(S^2)\) is non-trivial while
\(\pi_1(S^2)=1\), so \(S^2\) is not Goldberg-trivial. Compare
\cref{rem:s2-rp2}.
\end{remark}

\begin{remark}
Simpleness cannot be dropped from the criterion. The
\emph{pinched sphere} -- a sphere with two distinct points
identified -- satisfies (i)--(iii): it has no free edges, a
single thick component, and the pinch point is not a
\(2\)-manifold point. Yet it is not Goldberg-trivial: its simple
model is a sphere with a free edge joining two distinct points,
to which \cref{lem:val2-free-edge}(a) applies.
\end{remark}

\begin{figure}[ht]
\centering
\begin{tikzpicture}[line width=0.8pt]
  \fill[gray!30,opacity=0.5] (0,0) ellipse (1.2 and 0.5);
  \draw (0,0) ellipse (1.2 and 0.5);
  \node at (-0.7,-0.22) {\(\Sigma\)};
  \draw[blue] (0,0) -- (0,1);
  \filldraw (0,1) circle (1pt);
  \draw[blue] (1.2,0) -- (2,0);
  \filldraw (2,0) circle (1pt);
  \filldraw[red] (0,0) circle (1pt);
  \filldraw[red] (1.2,0) circle (1pt);
\end{tikzpicture}
\caption{A typical Goldberg-trivial complex
(\cref{thm:gt-classification})
}
\label{fig:gt-example}
\end{figure}

\begin{proof}[Proof of \cref{thm:gt-classification}]
Assume that \(X\) is simple and that
(i)--(iii) hold, and write \(\bar X=X\) and
\(\bar\Sigma=\Sigma\). By (i) and (ii), \(\bar X\) is the thick
component \(\bar\Sigma\) with finitely many pendant edges
attached at joints. In either case of (iii), \cref{cor:thick-with-pendants}
applies, with \(X=\bar\Sigma\) and \(X'=\bar X\), and shows that
the inclusion \(F_n\bar X\hookrightarrow\bar X^{\,n}\) induces
an isomorphism
\(\iota_\ast\colon\PB_n(\bar X,\base)\xrightarrow{\cong}
\prod_{i}\pi_1(\bar X,x^0_i)\); in particular
\(\ker\iota_\ast=1\), and \(\bar X\) is Goldberg-trivial.

Conversely, suppose \(X\) is Goldberg-trivial.
First, no point \(x\) of \(X\) has \(\val_{X}(x)\geq3\) -- this
is (i), and it requires no simplicity: otherwise
\cref{prop:star-embedding}, in the generality of
\cref{rem:star-link-components}, would embed \(\PB_n(\sfS_3)\)
into \(\PB_n(X,\base)\) along a star through three distinct
link components of \(x\); the image consists of braids supported
in the contractible star, hence lies in \(\ker\iota_\ast\), and
\(\PB_n(\sfS_3)\neq1\) for \(n\geq2\)
(\cref{ex:one-essential}) -- contradicting Goldberg-triviality.

Next, \(X\) is simple. By the above and
\cref{def:simple-complex}, a non-simple vertex \(\sfv\) would have
\(\val_X(\sfv)=2\) with both link components of positive dimension.
Suppose such a vertex exists, and pass to the simple model
\(\bar X\) of \(X\) (\cref{ssec:simple-an-park}); the resolution
braid equivalences commute with the strand maps -- each collapse
is a homotopy equivalence -- so \(\bar X\) is Goldberg-trivial
along with \(X\). Resolving \(\sfv\) introduces a free edge, and
the maximal free chain of \(\bar X\) containing it has both
endpoints of valency at least \(2\): each endpoint is either the
apex of a resolution -- a joint -- or a vertex retaining at
least two link directions. By \cref{lem:val2-free-edge}(a),
\(\ker\iota_\ast^{\bar X}\neq1\) -- a contradiction. Hence
\(X\) is simple, and we write \(\bar X=X\) from now on.

By \cref{lem:val2-free-edge}(a), every maximal free chain of \(\bar X\)
is pendant, and by \cref{lem:val2-free-edge}(b),
\(\bar X\) has at most one thick component. If \(\bar X\) had no thick
component at all, it would be a graph without vertices of
valency at least \(3\), hence an interval or a circle, both
excluded by \cref{rem:standing}. So (ii) holds.

Finally, suppose (iii) failed: \(\bar\Sigma\) is a
\(2\)-manifold and every pendant edge of \(\bar X\) is attached
along \(\partial\bar\Sigma\). Absorbing the pendant edges into
the surface, as in \cref{ex:disc}, the inclusion
\(\bar\Sigma\hookrightarrow\bar X\) induces an isomorphism
\(\PB_n(\bar\Sigma)\cong\PB_n(\bar X)\) compatible with the
strand maps, so we may assume that \(\bar X=\bar\Sigma\) is a
compact surface. If \(\bar\Sigma\not\cong S^2,\RP^2\), then
\cref{thm:goldberg} (in the compact form of
\cref{rem:goldberg-compact}) identifies
\(\ker\iota_\ast=\co{\im j_\ast}\) with \(j_\ast\) injective on
\(\PB_n(D^2)\neq1\); in particular \(\ker\iota_\ast\neq1\). If
\(\bar\Sigma\cong\RP^2\), the kernel is again non-trivial: the
disc braids have non-trivial image in \(\PB_n(\RP^2)\)
\cite{VanBuskirk1966}. And \(\bar\Sigma\cong S^2\) would force
\(\bar X=S^2\) -- a sphere has no boundary to carry pendant
edges -- which is excluded by \cref{rem:standing}. In every case
\(\ker\iota_\ast\neq1\), a contradiction. So (iii) holds,
completing the proof.
\end{proof}

\section{Characterisation of weak Goldbergness}
\label{sec:simple-model-necessity}

By
\cref{prop:goldberg-resolution,thm:resolution-preserves-goldberg,thm:resolution-reflects-goldberg},
and in the Goldberg case by
\cref{thm:resolution-goldberg-equivalence}, both weak Goldbergness
and Goldbergness are detected on the simple model: a complex
\(X\) is weakly Goldberg (resp.\ Goldberg) if and only if its
simple model \(\bar X\) is. This section establishes the
following characterisation of weak Goldbergness.

\begin{theorem}[Characterisation of weak Goldbergness]
\label{thm:weak-goldberg-char}
A complex \(X\) as in \cref{rem:standing} is weakly Goldberg if
and only if its free part \(\mathsf{F}_{X}\) is a forest --
equivalently, if and only if every connected component of
\(\mathsf{F}_{X}\) is simply connected.
\end{theorem}

The proof occupies the remainder of the section, in two halves.
That weak Goldbergness forces \(\mathsf{F}_{X}\) to be a forest
is the cycle obstruction of \cref{ssec:cycle-obstruction}
(\cref{cor:cycle-in-X0,thm:F-cycle-obstruction}). Conversely, when
\(\mathsf{F}_{X}\) is a forest, we construct in
\cref{ssec:witness-construction} an explicit contractible
candidate witness \(\bar X_0\subseteq\bar X\); the verification
that it is indeed a weak Goldberg witness is carried out in
\cref{ssec:candidate-is-witness}.

Throughout this section we keep the standing assumption
(\cref{rem:standing}); in
\cref{ssec:goldberg-witness-structure} we suppose moreover --
unless stronger hypotheses are stated -- that the simple complex
\(\bar X\) is weakly Goldberg, with weak Goldberg witness
\(\bar X_0\subseteq\bar X\).

We recall the bookkeeping already set up at the end of
\cref{ssec:free-parts}. The scaffolds
\(\Phi_X\in\mathfrak{G}_X\) and the weakly admissible maximal trees
\(\mathfrak{T}_X\) were introduced in \cref{def:scaffold};
\(\mathfrak{T}_X\) is non-empty precisely when \(\mathsf{F}_X\)
is a forest. A single resolution exchanges weakly admissible trees,
\(\mathfrak{T}_{X'}\cong\mathfrak{T}_X\)
(\cref{lem:max-tree-resolution}), and iterating to the simple
model identifies \(\mathfrak{T}_X\cong\mathfrak{T}_{\bar X}\);
for \(\sfT_X\in\mathfrak{T}_X\) we write
\(\sfT_{\bar X}\in\mathfrak{T}_{\bar X}\) for its pull-back
(\cref{def:tree-pullback-pushforward}). Since each resolution
attaches a whisker to the free part, \(\mathsf{F}_X\) is a
forest if and only if \(\mathsf{F}_{\bar X}\) is, so the
condition of \cref{thm:weak-goldberg-char} may be checked on
either side. The candidate witness of
\cref{ssec:witness-construction} will accordingly be defined in
the simple model, containing \(\sfT_{\bar X}\), and transported to
a subspace \(X_0\supseteq \sfT_X\) of \(X\) along the chain of
collapses (\cref{lem:witness-resolution}).

\subsection{The cycle obstruction}
\label{ssec:cycle-obstruction}

The three propositions on \(x\in X_0\)
(\crefrange{prop:X0-contains-free-points}{prop:X0-contains-m2-val3})
combine into a single geometric obstruction: every embedded cycle in
the free part of \(X\) is forced into the witness \(X_0\).

\begin{proposition}\label{cor:cycle-in-X0}
Let \(X\) be weakly Goldberg, with weak Goldberg witness
\(X_0\subseteq X\), and let \(C\subseteq\mathsf{F}_X\) be an
embedded simple closed curve. Then \(C\subseteq X_0\).
\end{proposition}

\begin{proof}
We show that every \(x\in C\) lies in \(X_0\) by applying one of
\cref{prop:X0-contains-free-points,prop:X0-contains-valency-3,prop:X0-contains-m2-val3}.

Fix \(x\in C\). Because \(C\) is an embedded simple closed curve,
exactly two of the link components of \(x\) in \(X\) are realised by the
two cycle directions of \(C\) at \(x\); call the corresponding link
points \(p_1,p_2\in\lk_{X}(x)\). In particular \(\val_{X}(x)\geq 2\), so \(x\) satisfies the
hypotheses of the standing setup of \cref{sec:lower-bound}, and
the number \(m\geq 1\) of \(x\)-components of \(X\) is defined.

\smallskip
\noindent\emph{Case~1: \(\val_{X}(x)=2\).}\ Then \(\lk_{X}(x)=\{p_1,p_2\}\),
and both link points lie on \(C\). The complement \(C\setminus\{x\}\) is
a connected arc in \(X\setminus\{x\}\) joining \(p_1\) to \(p_2\), so \(p_1\)
and \(p_2\) lie in the same connected component of \(X\setminus\{x\}\).
Since every connected component of \(X\setminus\{x\}\) touches the link
of \(x\), this forces \(X_x\) to be connected, i.e.\ \(m=1\).
\cref{prop:X0-contains-free-points} gives \(x\in X_0\).

\smallskip
\noindent\emph{Case~2: \(\val_{X}(x)\geq 3\).}\ We split on \(m\):
\begin{itemize}[leftmargin=2em]
  \item If \(m=1\), \cref{prop:X0-contains-free-points}
        gives \(x\in X_0\).
  \item If \(m\geq 3\), \cref{prop:X0-contains-valency-3}
        gives \(x\in X_0\).
  \item If \(m=2\), then since \(\val_{X}(x)\geq 3\),
        \cref{prop:X0-contains-m2-val3} gives \(x\in X_0\).
\end{itemize}

In every case \(x\in X_0\). Since \(x\in C\) was arbitrary, \(C\subseteq X_0\).
\end{proof}

\begin{theorem}\label{thm:F-cycle-obstruction}
Let \(X\) be a finite connected simplicial complex of dimension at
least \(1\) satisfying the standing assumption
(\cref{rem:standing}), and suppose that the free part
\(\mathsf{F}_X\) has a connected component which is not simply
connected. Then \(X\) is not weakly Goldberg.
\end{theorem}

\begin{proof}
Suppose, for contradiction, that \(X\) is weakly Goldberg. By
\cref{prop:goldberg-resolution,thm:resolution-preserves-goldberg,thm:resolution-reflects-goldberg}
the simple model \(\bar X\) is then weakly Goldberg as well; fix a
witness \(\bar X_0\subseteq\bar X\) (\cref{def:goldberg-strengths}),
which is contractible and hence simply connected.

Passing to the simple model preserves the free part up to homotopy: the
collapse \(\bar X\to X\) carries \(\mathsf{F}_{\bar X}\) onto
\(\mathsf{F}_X\) and is a homotopy equivalence, so \(\mathsf{F}_X\) and
\(\mathsf{F}_{\bar X}\) are homotopy equivalent component by component
(\cref{lem:max-tree-resolution}). The non-simply-connected component of
\(\mathsf{F}_X\) therefore corresponds to a non-simply-connected
component \(\mathsf{F}_0\) of \(\mathsf{F}_{\bar X}\). Being a finite
graph, \(\mathsf{F}_0\) has \(\pi_1\) free of positive rank, so it
contains an embedded simple closed curve \(C\subseteq\mathsf{F}_{\bar X}\).

By \cref{cor:cycle-in-X0}, \(C\subseteq\bar X_0\); choosing a path in
\(\bar X_0\) from \(x_1^0\) to a basepoint on \(C\) makes \([C]\) an
element of \(\pi_1(\bar X_0)\). As \(\bar X_0\) is simply connected,
\([C]=1\) there, and a fortiori
\begin{equation}\label{eq:C-trivial-in-X}
[C] = 1 \qquad\text{in } \pi_1(\bar X,x_1^0).
\end{equation}
On the other hand, \(\mathsf{F}_{\bar X}\) is a subcomplex of the scaffold \(\Phi_{\bar X}\) (\cref{def:scaffold}); the inclusion of a
subgraph is \(\pi_1\)-injective on each connected component, and
\(\Phi_{\bar X}\hookrightarrow\bar X\) is \(\pi_1\)-injective, so the
composite \(\mathsf{F}_0\hookrightarrow\bar X\) is \(\pi_1\)-injective.
Since \([C]\) is a non-trivial element of the free group
\(\pi_1(\mathsf{F}_0)\), this forces \([C]\neq 1\) in \(\pi_1(\bar X)\),
contradicting \eqref{eq:C-trivial-in-X}.
\end{proof}

\subsection{Constructing the witness \texorpdfstring{\(X_0\)}{X0}}
\label{ssec:witness-construction}

Throughout this subsection we assume that every connected
component of the free part \(\mathsf{F}_X\) is simply connected --
equivalently, that \(\mathsf{F}_X\) is a forest. By
\cref{thm:F-cycle-obstruction} this is a \emph{necessary}
condition for weak Goldbergness; under it we construct an
explicit contractible candidate witness \(X_0\subseteq X\),
assembled from a scaffold \(\Phi_X\in\mathfrak{G}_X\) (a
subcomplex of \(X\) containing \(\mathsf{F}_X\),
\cref{def:scaffold}).

By the standing hypothesis \(\mathsf{F}_X\) is a forest, so the
set of weakly admissible trees \(\mathfrak{T}_X\)
(\cref{def:scaffold}) is non-empty; we fix a
maximal tree \(\sfT_X\in\mathfrak{T}_X\) of a scaffold
\(\Phi_X\in\mathfrak{G}_X\). Recall that \(\sfT_X\) then contains all
thin edges, so the edges of \(\Phi_X\) omitted from \(\sfT_X\) are
all thick edges. We regard \(\sfT_X\subseteq\Phi_X\) as a
subcomplex of \(X\) (\cref{def:scaffold}).

\begin{definition}[The candidate witness \(X_0\)]
\label{def:candidate-witness}
Given \(X\) and a maximal tree \(\sfT_X\in\mathfrak{T}_X\), let
\(\bar X\) denote the unique simple complex obtained from \(X\) by
iteratively resolving non-simple vertices
(\cref{prop:simple-model-unique}), and let
\(\sfT_{\bar X}\in\mathfrak{T}_{\bar X}\) be the pull-back of \(\sfT_X\)
to \(\bar X\)
(\cref{def:tree-pullback-pushforward}). Define a
subspace \(\bar X_0\subseteq\bar X\) by
\[
\bar X_0 \coloneqq \sfT_{\bar X}\cup
\bigcup_{\sfv}\,N_{\bar X}(\sfv),
\]
the union of the embedded maximal tree
\(\sfT_{\bar X}\subseteq\bar X\) with the closed regular
neighbourhoods, taken in \(\bar X\), of all the vertices \(\sfv\)
of the scaffold \(\Phi_{\bar X}\). The
\emph{candidate witness} for the weak Goldbergness of \(X\) is
the \emph{push-forward} of \(\bar X_0\) along the chain of
collapses
\(q_{\sfe_1}\circ\cdots\circ q_{\sfe_c}\colon\bar X\twoheadrightarrow X\):
\[
X_0 \coloneqq \bigl(q_{\sfe_1}\circ\cdots\circ q_{\sfe_c}\bigr)(\bar X_0)
\subseteq X.
\]
Choose \(\base\in F_n(X_0)\); after subdivision its coordinates
may be taken to be vertices. \cref{fig:candidate-witness}
illustrates the construction.
\end{definition}

\begin{figure}[ht]
\centering
\begin{tikzpicture}[line width=0.8pt, scale=1.15]
  \fill[gray!30,opacity=0.5] (0,0) ellipse (1.3 and 0.55);
  \draw (0,0) ellipse (1.3 and 0.55);
  \node at (0,0.32) {\(\bar\Sigma\)};
  \begin{scope}
    \clip (0,0) ellipse (1.3 and 0.55);
    \fill[orange!45] (-1.3,0) circle (0.4);
    \fill[orange!45] (1.3,0) circle (0.4);
  \end{scope}
  \draw[line width=4pt, orange!45] (-1.3,0) -- (-1.62,0.19);
  \draw[line width=4pt, orange!45] (1.3,0) -- (1.62,0.19);
  \draw[line width=4pt, orange!45] (-2.0,0.42) -- (-2.3,0.6);
  \draw[line width=4pt, orange!45] (2.0,0.42) -- (2.3,0.6);
  \draw[line width=4pt, orange!45] (-1.3,0) .. controls (-1.0,-0.06) .. (-0.92,-0.08);
  \draw[line width=4pt, orange!45] (1.3,0) .. controls (1.0,-0.06) .. (0.92,-0.08);
  \draw[blue] (-1.3,0) -- (-2.3,0.6);
  \draw[blue] (1.3,0) -- (2.3,0.6);
  \draw[green!55!black] (-1.3,0) .. controls (0,-0.25) .. (1.3,0);
  \filldraw[red] (-1.3,0) circle (1.8pt);
  \node[below left=-2pt] at (-1.32,-0.02) {\(\sfw_1\)};
  \filldraw[red] (1.3,0) circle (1.8pt);
  \node[below right=-2pt] at (1.32,-0.02) {\(\sfw_2\)};
  \filldraw (-2.3,0.6) circle (1.5pt);
  \filldraw (2.3,0.6) circle (1.5pt);
\end{tikzpicture}
\caption{The candidate witness of
\cref{def:candidate-witness} for a simple complex \(\bar X\):
a disc \(\bar\Sigma\) with boundary joints \(\sfw_1,\sfw_2\)
and two free edges. The weakly admissible tree \(\sfT_{\bar X}\)
consists of the two thin edges (blue) and one thick edge
(green); \(\bar X_0\) is the union of \(\sfT_{\bar X}\) with the
closed regular neighbourhoods of the vertices of the scaffold
(orange).}
\label{fig:candidate-witness}
\end{figure}

\begin{remark}\label{rem:vSigma-redundant}
In the degenerate case where \(\sfT_{\bar X}\) has no edges --
equivalently, \(\Phi_{\bar X}=\{[\bar X]\}\) is the single vertex,
so \(\mathsf{F}_{\bar X}=\emptyset\) and \(\bar X\) itself is the
unique thick component -- the embedded tree is the single
representative vertex \(\sfv_{\bar X}\), and the definition gives
\(\bar X_0=N_{\bar X}(\sfv_{\bar X})\), the closed star of
\(\sfv_{\bar X}\).
\end{remark}

The candidate witness is contractible: each vertex
neighbourhood \(N_{\bar X}(\sfv)\) is a cone with apex \(\sfv\),
meeting the tree \(\sfT_{\bar X}\) -- and, taken small enough,
nothing else of \(\bar X_0\) -- in a contractible neighbourhood
of \(\sfv\) inside the tree; collapsing the cones one at a time
therefore deformation retracts \(\bar X_0\) onto the embedded
tree \(\sfT_{\bar X}\), which is contractible. The chain of collapses
\(q_{\sfe_1}\circ\cdots\circ q_{\sfe_c}\) is a homotopy equivalence
(each \(q_{\sfe_i}\) contracts the contractible free edge \(\sfe_i\)), so
\(X_0\) inherits the contractibility of \(\bar X_0\). Thus \(X_0\) is
a legitimate witness in the sense of
\cref{def:goldberg-strengths}: a contractible subspace
of \(X\) carrying the basepoint.

\begin{remark}\label{rem:X0-meets-free-component}
For a thick component \(\bar\Sigma\), consider the components of
the subforest \(\sfT_{\bar X}\cap\Phi_{\bar\Sigma}\), including
isolated vertices. Each such component gives an intersection
piece of \(\bar X_0\cap\bar\Sigma\) by adjoining the cone
neighbourhoods in \(\bar\Sigma\) of its scaffold vertices. The
neighbourhoods can be chosen small enough not to join distinct
components of the subforest. In particular every joint belongs
to an intersection piece, even if no thick edge at that joint
lies in \(\sfT_{\bar X}\). This description applies to an
arbitrary scaffold tree; it does not assume a star centred at
\(\sfv_{\bar\Sigma}\). Pushing forward along the resolution
collapses gives the corresponding description in \(X\).
\end{remark}

The candidate witness behaves cleanly under a single resolution
step: \(X_0'\) is precisely an An--Park resolution of \(X_0\) at the
vertex \(\sfv\).

\begin{lemma}\label{lem:witness-resolution}
With the setup of \cref{lem:max-tree-resolution}, let
\(\sfT_{X'}:=q_\sfe^{-1}(\sfT_X)\in\mathfrak{T}_{X'}\) be the pull-back of
\(\sfT_X\in\mathfrak{T}_X\)
(\cref{def:tree-pullback-pushforward}), and let
\(X_0,X_0'\) be the corresponding candidate witnesses of
\cref{def:candidate-witness}. Then
\(\sfe\subseteq \sfT_{X'}\subseteq X_0'\), with
\[
X_0' = q_\sfe^{-1}(X_0),
\qquad
X_0 = q_\sfe(X_0') = X_0'/\sfe,
\]
and \(X_0'\) is the An--Park resolution
(\cref{ssec:simple-an-park}) of \(X_0\) at the vertex \(\sfv\):
the link component \(L\cap X_0\) at \(\sfv\) is separated from \(\sfv\) and
reattached at the new apex \(\sfv'\) via the new free edge \(\sfe\), and
\(q_\sfe\) restricts to a homeomorphism
\(X_0'/\sfe\xrightarrow{\cong}X_0\).
\end{lemma}

\begin{proof}
The new free edge \(\sfe\) lies in \(\mathsf{F}_{X'}\), hence in
\(\sfT_{X'}\) by admissibility, hence in \(X_0'\) by
\cref{def:candidate-witness}. By
\cref{def:candidate-witness}, both \(X_0\) and \(X_0'\)
arise as push-forwards of the same subspace
\(\bar X_0\subseteq\bar X\): the pull-back of \(\sfT_X\) to \(\bar X\)
agrees with the iterated pull-back of \(\sfT_{X'}=q_\sfe^{-1}(\sfT_X)\) to
\(\bar X\), so the simple-model data
\((\bar X,\sfT_{\bar X},\bar X_0)\) entering the definitions of
\(X_0\) and \(X_0'\) coincide. The chain of collapses
\(\bar X\twoheadrightarrow X\) factors as the chain
\(\bar X\twoheadrightarrow X'\) followed by \(q_\sfe\), so
\(X_0=q_\sfe(X_0')\). Since \(q_\sfe\) is injective off
\(\sfe\subseteq X_0'\), we have \(X_0'=q_\sfe^{-1}(X_0)\) and \(q_\sfe\)
restricts to a homeomorphism
\(X_0'/\sfe\xrightarrow{\cong}X_0\). Geometrically,
\(X_0'\setminus\sfe\) maps homeomorphically onto
\(X_0\setminus\{\sfv\}\). The fibre over \(\sfv\) is exactly
the closed edge \(\sfe\); the resolved cone is preserved
away from its apex. This is precisely the An--Park
resolution of \(X_0\) at \(\sfv\) (\cref{ssec:simple-an-park}).
\end{proof}

\subsection{The candidate witness is a weak Goldberg witness}
\label{ssec:candidate-is-witness}

We now show that the candidate witness of
\cref{def:candidate-witness} deserves its name.

\begin{proposition}\label{thm:candidate-witness-weak}
Suppose that the free part \(\mathsf{F}_{\bar X}\) of the simple
complex \(\bar X\) is a forest, and let
\(\bar X_0=\sfT_{\bar X}\cup\bigcup_{\sfv} N_{\bar X}(\sfv)\) be the
candidate witness of \cref{def:candidate-witness}. Then
\(\bar X_0\) is a weak Goldberg witness of \(\bar X\):
\[
\co{\im j_\ast}^{\PB_n(\bar X,\base)}=\ker\iota_\ast .
\]
\end{proposition}

\begin{theorem}\label{cor:candidate-witness-descends}
Let \(X\) be as in \cref{rem:standing} with \(\mathsf{F}_X\) a
forest. Then the push-forward \(X_0\subseteq X\) of
\cref{def:candidate-witness} is a weak Goldberg witness of
\(X\). In particular \(X\) is weakly Goldberg, which completes
the proof of \cref{thm:weak-goldberg-char}.
\end{theorem}

\begin{proof}
Since each resolution attaches a whisker to the free part,
\(\mathsf{F}_{\bar X}\) is a forest along with \(\mathsf{F}_X\),
so the candidate witnesses are defined on both sides of every
collapse in the chain \(\bar X\twoheadrightarrow X\), and by
\cref{lem:witness-resolution} they correspond step by step: the
upper one is the An--Park resolution of the lower one at the
resolved vertex. Since, at each step, the upper candidate witness is the
An--Park resolution of the lower one
(\cref{lem:witness-resolution}), the induced map
\(q_\sfe^\ast\colon\PB_n(X_0)\to\PB_n(X_0')\) is an isomorphism --
a braid equivalence, as in \cref{ssec:simple-an-park} -- and in
particular surjective; \cref{prop:goldberg-resolution}(1)
therefore transports the weak-witness property across each step
of the chain, in both directions. It thus suffices to prove that
the candidate witness at the top of the chain, built from a
maximal tree \(\sfT_{\bar X}\in\mathfrak{T}_{\bar X}\), is a weak
Goldberg witness of \(\bar X\). This is precisely
\cref{thm:candidate-witness-weak}.
\end{proof}

The proof of \cref{thm:candidate-witness-weak} proceeds by
induction on the first Betti number \(b_1(\Phi_{\bar X})\) of
the scaffold -- a number independent of the choices entering
\(\Phi_{\bar X}\), as in the proof of
\cref{lem:val2-free-edge} -- and occupies the remainder of this
subsection. The proof proceeds in Steps 0--3. Step 0 disposes of
the degenerate case in which \(\bar X\) is itself a tree
(\cref{lem:witness-step0}). For \(b_1(\Phi_{\bar X})=0\) -- so
that the maximal tree is the whole scaffold,
\(\sfT_{\bar X}=\Phi_{\bar X}\) -- Step 1 treats the case of a
single thick component (\cref{lem:witness-one-thick}), and
Step 2 the general tree case, by induction on the number of
thick components (\cref{prop:witness-tree-case}); Step 3 is the
inductive step in \(b_1(\Phi_{\bar X})\)
(\cref{prop:witness-betti-step}).

\begin{lemma}[Step 0: \(\bar X\) a tree]
\label{lem:witness-step0}
Suppose that \(\bar X\) is itself a tree. Then
\(\bar X=\Phi_{\bar X}=\sfT_{\bar X}=\bar X_0\), and \(\bar X_0\)
is a Goldberg witness of \(\bar X\); in particular it is a weak
Goldberg witness.
\end{lemma}

\begin{proof}
A tree has no thick components and every edge is free, so
\(\Phi_{\bar X}=\mathsf{F}_{\bar X}=\bar X\), the unique maximal
tree is \(\sfT_{\bar X}=\bar X\), and
\(\bar X_0=\sfT_{\bar X}\cup\bigcup_\sfv N_{\bar X}(\sfv)=\bar X\). As
\(\bar X\) is simply connected, every strand class of every pure
braid vanishes, so \(\ker\iota_\ast=\PB_n(\bar X,\base)\); on
the other hand \(j_\ast\) is the identity. Hence
\(\co{\im j_\ast}=\PB_n(\bar X,\base)=\ker\iota_\ast\), and
\(j_\ast\) is injective.
\end{proof}

\begin{lemma}[Step 1: a single thick component]
\label{lem:witness-one-thick}
Suppose that \(\Phi_{\bar X}\) is a tree, that \(\bar X\) has
exactly one thick component \(\bar\Sigma\), and that every
maximal free chain of \(\bar X\) is pendant -- so that \(\bar X\) is
\(\bar\Sigma\) with finitely many pendant edges attached at
joints. Then, taking \(\sfT_{\bar X}=\Phi_{\bar X}\), the candidate
witness \(\bar X_0\) is a weak Goldberg witness of \(\bar X\).
\end{lemma}

\begin{proof}
Note first that \(\bar X_0\) contains the neighbourhood
\(N_{\bar X}(\sfv_{\bar\Sigma})\) of the representative vertex
\(\sfv_{\bar\Sigma}\in\bar\Sigma\)
(\cref{def:candidate-witness}), and recall that the witness
property is insensitive to moving the base configuration.

Suppose first that \(\bar\Sigma\) is not a \(2\)-manifold, or
that it is a \(2\)-manifold and some pendant edge is attached at
an interior point of \(\bar\Sigma\). By
\cref{cor:thick-with-pendants}, the strand map \(\iota_\ast\) is
then an isomorphism, so \(\ker\iota_\ast=1\). Let \(I\) be an
embedded interval inside
\(N_{\bar X}(\sfv_{\bar\Sigma})\subseteq\bar X_0\); then
\(\co{\im j^{I}_\ast}\subseteq\ker\iota_\ast=1\), with equality,
so \(I\) is a weak Goldberg witness of \(\bar X\). As
\(\bar X_0\) is a contractible subspace containing \(I\),
\cref{prop:witness-upward} promotes the witness property from
\(I\) to \(\bar X_0\).

Suppose instead that \(\bar\Sigma\) is a \(2\)-manifold and
every pendant edge is attached along \(\partial\bar\Sigma\).
Absorbing the pendant edges into the surface, as in
\cref{ex:disc}, the inclusion
\(\bar\Sigma\hookrightarrow\bar X\) induces an isomorphism
\(\PB_n(\bar\Sigma)\cong\PB_n(\bar X)\) compatible with the
strand maps. By Goldberg's theorem (\cref{thm:goldberg}, in the
compact form of \cref{rem:goldberg-compact}; for
\(\bar\Sigma\cong S^2,\RP^2\) the kernel description persists by
\cref{rem:s2-rp2}), an embedded disc \(D\subseteq\bar\Sigma\)
is a weak Goldberg witness of \(\bar\Sigma\), hence of
\(\bar X\): \(\ker\iota_\ast=\co{\im j^{D}_\ast}\) -- for any
small disc in the interior, two such being ambient isotopic. Choosing the
disc inside \(N_{\bar X}(\sfv_{\bar\Sigma})\subseteq\bar X_0\),
\cref{prop:witness-upward} again promotes the witness property
to \(\bar X_0\).
\end{proof}

\begin{proposition}[Step 2: \(\Phi_{\bar X}\) a tree]
\label{prop:witness-tree-case}
Suppose that \(\Phi_{\bar X}\) is a tree, and take
\(\sfT_{\bar X}=\Phi_{\bar X}\). Then the candidate witness
\(\bar X_0\) is a weak Goldberg witness of \(\bar X\).
\end{proposition}

\begin{proof}
We argue by induction on the complexity
\[
c(\bar X)\coloneqq
\#\{\text{thick components of }\bar X\}
+\#\{\text{vertices of }\bar X\text{ of valency}\geq3\} .
\]
By simplicity, a vertex of valency at least \(3\) has
\(0\)-dimensional link, so it is a branch vertex of
\(\mathsf{F}_{\bar X}\), and in particular a vertex of
\(\Phi_{\bar X}\).

\emph{The cases \(c(\bar X)\leq1\).} If \(\bar X\) has no thick
component, then \(\mathsf{F}_{\bar X}=\bar X\) is connected and,
by the standing hypothesis, a forest; so \(\bar X\) is a tree,
and \cref{lem:witness-step0} applies. If \(c(\bar X)=1\) with
one thick component \(\bar\Sigma\), then \(\bar X\) has no
vertex of valency at least \(3\), and every maximal free chain of
\(\bar X\) is pendant: the endpoints of a maximal free
chain are univalent vertices, vertices of valency at least \(3\),
or joints, and a maximal free chain joining two joints of
\(\bar\Sigma\) would close up, through the tree
\(\Phi_{\bar\Sigma}\), to a cycle in \(\Phi_{\bar X}\). Thus
\cref{lem:witness-one-thick} applies.

\emph{The inductive step.} Suppose \(c(\bar X)\geq2\). Call a
vertex of \(\Phi_{\bar X}\) a \emph{site} if it is a
representative vertex \(\sfv_{\bar\Sigma}\) or a vertex of valency
at least \(3\), so that \(\bar X\) carries
\(c(\bar X)\geq2\) sites. The path in the tree \(\Phi_{\bar X}\)
joining two distinct sites traverses at least one thin edge; fix
such a maximal free chain \(\sfe\) of \(\bar X\) and a point
\(x\in\mathring \sfe\). Since \(\Phi_{\bar X}\) is a tree, \(x\)
separates \(\bar X\) -- a path between the two germs of \(\sfe\)
avoiding \(x\) would close up, through the trees
\(\Phi_{\bar\Sigma}\), to a cycle in \(\Phi_{\bar X}\) -- and we
write \(\bar X^1,\bar X^2\) for the two \(x\)-components, each
containing at least one site, so that
\(c(\bar X^i)<c(\bar X)\) for \(i=1,2\).

Likewise
\(x\in\mathring \sfe\subseteq \sfT_{\bar X}\subseteq\bar X_0\)
separates \(\bar X_0\) into the two \(x\)-components
\(\bar X^i_0=\bar X_0\cap\bar X^i\), and \(\bar X^i_0\) is
precisely the candidate witness of \(\bar X^i\) constructed from
the scaffold
\(\Phi_{\bar X^i}=\Phi_{\bar X}\cap\bar X^i=\sfT_{\bar X^i}\): the
point \(x\) becomes a univalent vertex of \(\bar X^i\), whose
star is absorbed into the half-edge of \(\sfe\). Each \(\bar X^i\)
is simple, satisfies the standing assumption -- it contains the
univalent vertex \(x\), so \(\bar X^i\not\cong S^1\) -- and has
\(\Phi_{\bar X^i}\) a tree; the induction hypothesis therefore
makes \(\bar X^i_0\) a weak Goldberg witness of \(\bar X^i\):
\[
\co{\im\PB_n(\bar X^i_0)}^{\PB_n(\bar X^i)}
=\ker\iota^{\bar X^i}_\ast .
\]

Each \(\bar X^i\) contains a site, and each \(\bar X^i_0\)
contains the corresponding neighbourhood
\(N_{\bar X}(\sfv_{\bar\Sigma})\) or the star of a branch vertex of
\(\mathsf{F}_{\bar X}\); so none of
\(\bar X^1,\bar X^2,\bar X^1_0,\bar X^2_0\) is an interval,
their punctured configuration spaces are all connected, and the
underlying distribution graphs of the decompositions of
\(F_n(\bar X)\) and of \(F_n(\bar X_0)\) at \(x\) are
canonically identified, as in the proof of
\cref{cor:free-edge-loop-rank}:
\(\sfG\coloneqq\sfG_n((\bar X,x),\base)\cong
\sfG_n((\bar X_0,x),\base)\). Applying
\cref{cor:free-edge-disconnected} to \(\bar X\) and to
\(\bar X_0\) at \(x\), with the section
\(s\colon\pi_1(\sfG)\to\PB_n(\bar X_0,\base)\) chosen inside
\(\bar X_0\) once and for all, we obtain surjections
\[
\PB_n(\bar X^1)\ast\PB_n(\bar X^2)\ast\pi_1(\sfG)
\twoheadrightarrow\PB_n(\bar X,\base),
\qquad
\PB_n(\bar X^1_0)\ast\PB_n(\bar X^2_0)\ast\pi_1(\sfG)
\twoheadrightarrow\PB_n(\bar X_0,\base),
\]
and the inclusions \(\bar X^i_0\hookrightarrow\bar X^i\) and
\(\bar X_0\hookrightarrow\bar X\) induce maps of the
corresponding factors making the square commute.

Write \(Q\coloneqq\PB_n(\bar X,\base)\big/\co{\im j_\ast}\). In
\(Q\), the \(\pi_1(\sfG)\)-factor dies, its section lying in
\(\bar X_0\), and each \(\co{\im\PB_n(\bar X^i_0)}\) dies as
well; hence, by the induction hypothesis and the surjectivity of
the strand maps (\cref{prop:iota-surj}), the first surjection
descends to
\[
\prod_{\ell=1}^n\pi_1(\bar X^1)\ast
\prod_{\ell=1}^n\pi_1(\bar X^2)
\twoheadrightarrow Q.
\]
On the other hand, \(\iota_\ast\) descends to a surjection
\(Q\twoheadrightarrow\prod_{\ell=1}^n\pi_1(\bar X,x^0_\ell)\)
(\cref{prop:iota-surj} again), and by van~Kampen
\(\pi_1(\bar X)=\pi_1(\bar X^1)\ast\pi_1(\bar X^2)\). The
composite of the two surjections is the canonical map
\[
\prod_{\ell}\pi_1(\bar X^1)\ast\prod_{\ell}\pi_1(\bar X^2)
\longrightarrow
\prod_{\ell}\bigl(\pi_1(\bar X^1)\ast\pi_1(\bar X^2)\bigr),
\]
whose kernel is normally generated by the commutators of an
element supported in the \(i\)-th coordinate of the first factor
and an element supported in the \(j\)-th coordinate of the
second, with \(i\neq j\)
(\cref{lem:coordinatewise-kernel}(i)). Consequently
\(Q\cong\prod_\ell\pi_1(\bar X,x^0_\ell)\) -- whence
\(\co{\im j_\ast}=\ker\iota_\ast\), as desired -- once the
following is proved: \emph{for \(i\neq j\), the image in \(Q\)
of a braid whose \(i\)-th strand traverses a loop in
\(\bar X^1\), the other strands resting, commutes with the image
of a braid whose \(j\)-th strand traverses a loop in
\(\bar X^2\).}

To prove the commutation, put \(A=\bar X^1\),
\(B=\bar X^2\), and \(W=\bar X_0\), and write
\(N=\co{\im j_\ast}\). Fix \(i\neq j\). Choose a short
closed interval about \(x\) in the free edge \(\sfe\subseteq W\),
with endpoints \(p_A\in A\) and \(p_B\in B\). Let \(c\) be
the configuration with strand \(i\) at \(p_A\), strand \(j\)
at \(p_B\), and all remaining strands at distinct points of
\((p_A,x)\). The points \(p_A,p_B\) are the outermost parked
points on their respective sides.

Since \(x\) is a free endpoint of each side, removing the
terminal half-open segment \((p_A,x]\) from \(A\), or
\([x,p_B)\) from \(B\), does not change its fundamental group.
Thus arbitrary classes in \(\pi_1(A,p_A)\) and
\(\pi_1(B,p_B)\) have representatives on the core sides of
\(p_A\) and \(p_B\). Let \(\alpha,\beta\) be the loops in
\(F_n(\bar X,c)\) obtained by moving only strand \(i\),
respectively only strand \(j\), along such representatives.
Their moving strands have disjoint supports, so
\([\alpha,\beta]=1\).

We compare these loops with the all-on-one-side factors above
using explicit paths in \(F_n(W)\). In what follows, path
concatenations are read from left to right. Form an all-\(A\)
configuration \(c_A\) by moving the resting point of strand
\(j\) from \(p_B\) to a vacant point of \((p_A,x)\) closer
to \(x\) than all the other parked points. The path
\(t_A:c_A\to c\) moves only strand \(j\) along the interval.
Its track is disjoint from the moving strand of \(\alpha\).
If \(\alpha_A\) denotes the same \(i\)-strand motion at
\(c_A\), the resulting square of independent motions gives
\[
\alpha_A\simeq t_A\alpha t_A^{-1}.
\]
For the other side, move the \(n-1\) strands in \(A\) from
\(c\) into distinct points of \((x,p_B)\), starting with the
strand nearest \(x\) and filling the available points from
the farthest to the nearest to \(x\). This gives a
collision-free path to an all-\(B\) configuration \(c_B\),
fixing strand \(j\) throughout. Let \(t_B:c_B\to c\) be its
reverse. All its tracks avoid the core side of \(p_B\), so,
for the same \(j\)-strand motion \(\beta_B\) at \(c_B\),
\[
\beta_B\simeq t_B\beta t_B^{-1}.
\]

Use the attaching paths in \(F_n(W)\) for the two side
factors above, and extend them within
\(F_n(W\cap A)\) and \(F_n(W\cap B)\) to paths
\(\lambda_A:\base\to c_A\) and
\(\lambda_B:\base\to c_B\). These extensions exist by the
connectivity already established for the side witnesses.
The induction hypothesis on each side identifies its quotient
by the normal closure of its witness group with the product of
its strand groups. Consequently, after the corresponding
basepoint changes, every element supported in coordinate
\(i\) of the \(A\)-factor is represented in \(Q\) by
\(\lambda_A\alpha_A\lambda_A^{-1}\), and likewise every
element supported in coordinate \(j\) of the \(B\)-factor
by \(\lambda_B\beta_B\lambda_B^{-1}\). Indeed, any two
side braids with the same strand classes differ by the side
kernel, whose image is contained in \(N\).

Now set \(R_A=\lambda_A t_A\) and
\(R_B=\lambda_B t_B\). Both are paths from \(\base\) to
\(c\) in \(F_n(W)\). Hence \(d=R_B R_A^{-1}\) is a
based loop in \(F_n(W,\base)\), and its image belongs to
\(N\). It follows that the two elements under consideration
are represented in \(Q\), respectively, by
\[
R_A\alpha R_A^{-1}
\qquad\text{and}\qquad
R_A\beta R_A^{-1},
\]
because \(R_B\beta R_B^{-1}
=d(R_A\beta R_A^{-1})d^{-1}\) in the fundamental group.
They commute since \([\alpha,\beta]=1\). This proves the
required commutation and completes the induction.
\end{proof}

\begin{proposition}[Step 3: the inductive step in
\(b_1(\Phi_{\bar X})\)]
\label{prop:witness-betti-step}
Let \(b\geq1\), and suppose that the candidate witness is a weak
Goldberg witness for every simple complex as in
\cref{rem:standing} whose scaffolds have first Betti number
less than \(b\). Then the same holds for every simple complex
whose scaffolds have first Betti number \(b\).
\end{proposition}

\begin{proof}
If \(b_1(\Phi_{\bar X})=0\), this is
\cref{prop:witness-tree-case}. So suppose
\(b_1(\Phi_{\bar X})\geq1\), and fix an embedded cycle in
\(\Phi_{\bar X}\). As in the proof of
\cref{lem:val2-free-edge}, the cycle traverses a maximal free
chain \(\sfe\), which has distinct ends since
\(\mathsf F_{\bar X}\) is a forest. Fix \(x\in\mathring \sfe\). The remainder of the cycle joins
the two germs of \(\sfe\), so the cut complex \(\bar X_x\) is
\emph{connected}. For the finite-complex induction and the
scaffold construction below, use its compact cut model and the
compatible witness truncation of \cref{conv:subcomplex}; these
preserve the inclusion-induced braid maps. On the other hand
\(x\in \sfe\subseteq \sfT_{\bar X}\) separates both the maximal tree
and the witness,
\[
\sfT_{\bar X}\setminus\{x\}=\sfT^1_{\bar X}\sqcup \sfT^2_{\bar X},
\qquad
(\bar X_0)_x=\bar X^1_0\sqcup\bar X^2_0 ,
\]
where \(\bar X^i_0\) is the union of \(\sfT^i_{\bar X}\) with the
neighbourhoods \(N_{\bar X}(\sfv)\) of its vertices -- a complex of
exactly the shape of \cref{def:candidate-witness}.

Since \(\bar X_x\) is connected, so is
\(\Phi_{\bar X}\setminus\mathring \sfe\); this graph is the union
of \(\sfT^1_{\bar X}\sqcup \sfT^2_{\bar X}\) with the edges of
\(\Phi_{\bar X}\) omitted from \(\sfT_{\bar X}\) -- all of them
thick -- so some omitted thick edge \(\sff\), lying inside a thick
component \(\bar\Sigma\), joins a vertex \(\sfw\in \sfT^1_{\bar X}\)
to a vertex \(\sfw'\in \sfT^2_{\bar X}\); in particular
\(\bar\Sigma\) meets both sides. Then
\(\sfT_{\bar X_x}\coloneqq \sfT^1_{\bar X}\cup \sfT^2_{\bar X}\cup \sff\) is
a weakly admissible maximal tree of \(\Phi_{\bar X_x}\), and the
associated candidate witness of \(\bar X_x\) is
\[
(\bar X_x)_0=\bar X^1_0\cup\bar X^2_0\cup \sff
=\bigl(\bar X_0\cup \sff\bigr)_x ,
\]
the stars of \(\sfw\), of \(\sfw'\) and of the two new univalent
vertices being already accounted for. Since
\(b_1(\Phi_{\bar X_x})=b_1(\Phi_{\bar X})-1\), the induction
hypothesis applies:
\[
\co{\im\PB_n((\bar X_x)_0)}^{\PB_n(\bar X_x)}
=\ker\iota^{\bar X_x}_\ast .
\]

Two embedded copies of the letter \(\sfH\) enter the argument.
First, connectivity of \(\bar X_x\) forces both endpoints
\(\sfv,\sfv'\) of \(\sfe\) to have valency at least \(2\), so the setup
preceding \cref{lem:H-shuffle} applies with \(\alpha=\sfe\) and
legs chosen inside \(\bar X_0\) -- the other edges of
\(\Phi_{\bar X}\) at \(\sfv,\sfv'\), and the cones over the
positive-dimensional link components, lie in \(\bar X_0\); write
\(\sfH_\sfe\subseteq\bar X_0\) for the image. Second, at the
endpoints \(\sfw,\sfw'\) of \(\sff\) the witness contains the
neighbourhoods \(N_{\bar X}(\sfw)\) and \(N_{\bar X}(\sfw')\), whose
intersections with \(\bar\Sigma\) provide two directions at each
endpoint; the resulting embedding \(\sfH_\sff\) has middle edge
\(\sff\) and is contained both in \((\bar X_x)_0\) and, entirely,
in \(\bar\Sigma\).

Cut \((\bar X_x)_0\) at a point \(y\in\mathring \sff\) -- a free
edge of the one-dimensional part of the witness, though a thick
edge of \(\bar X\) -- so that
\[
(\bar X_x)_0\setminus\{y\}
=\bigl(\bar X^1_0\cup \sff_1\bigr)\sqcup
\bigl(\bar X^2_0\cup \sff_2\bigr),
\]
where \(\sff_1,\sff_2\) are the two components of
\(\sff\setminus\{y\}\). Note that
\(\bar X^1_0\cap \sff_1=[\sfw,z]\) and
\(\bar X^2_0\cap \sff_2=[\sfw',z']\), where \(z\in\lk_{\bar X}(\sfw)\)
and \(z'\in\lk_{\bar X}(\sfw')\) denote the link points of \(\sff\),
the neighbourhoods \(N_{\bar X}(\sfw)\) and \(N_{\bar X}(\sfw')\)
containing the corresponding initial segments of \(\sff\). Thus
\(\bar X^i_0\cup \sff_i\) is obtained from \(\bar X^i_0\) by
attaching an interval at the single point \(z\)
(resp.\ \(z'\)), whose link in \(\bar X^i_0\) is
\emph{connected}: through the triangles of the cone
\(N_{\bar X}(\sfw)\), the point \(z\) is joined to the direction of
\(\sfw\) and to its neighbours in \(\lk_{\bar X}(\sfw)\), so
\(\val_{\bar X^1_0}(z)=1\). Hence \cref{lem:pendant-edge}
applies verbatim and gives surjections
\(\PB_n(\bar X^i_0)\twoheadrightarrow
\PB_n(\bar X^i_0\cup \sff_i)\). Combining these with
\cref{cor:free-edge-disconnected}, applied to \((\bar X_x)_0\)
at \(y\), we obtain a surjection
\[
\PB_n(\bar X^1_0)\ast\PB_n(\bar X^2_0)\ast\Lambda
\twoheadrightarrow\PB_n\bigl((\bar X_x)_0\bigr),
\]
where \(\Lambda\coloneqq\pi_1(\sfG')\) is the free group on the
loops of the distribution graph of the cut at \(y\), with section
realisable inside \(\sfH_\sff\subseteq\bar\Sigma\); in particular
every generator of \(\Lambda\) is realised by a braid supported
in \(\bar\Sigma\).

Combining this surjection with the induction hypothesis and the
surjectivity of the strand map (\cref{prop:iota-surj}) yields
\[
\PB_n(\bar X_x)\big/
\co{\im\PB_n(\bar X^1_0),\,\im\PB_n(\bar X^2_0),\,\im\Lambda}
\cong\prod_{\ell=1}^n\pi_1(\bar X_x).
\]
We claim that \(\im\Lambda\) lies in the normal closure of
\(\im\PB_n(\bar X_0^1)\) in \(\PB_n(\bar X_x)\).
Its section generators are represented, up to basepoint
transport, by braids in the embedded tree
\(\sfH_\sff\subseteq\bar\Sigma\).
Every strand of such a braid is null-homotopic in
\(\bar\Sigma\), since \(\sfH_\sff\) is contractible.

If \(\bar\Sigma\) is not a \(2\)-manifold, all its vertex
links are connected, so \cref{prop:branched-surface} makes its
strand map injective. Each of these braids is therefore
already trivial in \(\PB_n(\bar\Sigma)\), hence in
\(\PB_n(\bar X_x)\).

If \(\bar\Sigma\) is a surface, first push the tree slightly
off the boundary, if necessary, by a collar isotopy of
embeddings of the surface into itself. A sufficiently small
regular neighbourhood of the resulting tree is an embedded
disc \(D\subseteq\operatorname{int}\bar\Sigma\).
Choose another small disc
\(D_1\subseteq N_{\bar X}(\sfw)\cap
\operatorname{int}\bar\Sigma\subseteq\bar X_0^1\).
An isotopy in the surface carries \(D\) to \(D_1\), and the
induced motion of configurations identifies their braid
images up to basepoint transport. After inclusion into
\(\bar X_x\), each section generator is thus conjugate to
the image of a braid in \(D_1\). The comparison with the
fixed base configuration may be made inside \(\bar X_0^1\);
any other choice only changes the ambient conjugating
element. Consequently each generator belongs to the claimed
normal closure. This argument requires no extension of the
surface isotopy to an ambient isotopy of \(\bar X_x\).

In both cases \(\im\Lambda\) is redundant, and
\[
\PB_n(\bar X_x)\big/
\co{\im\PB_n(\bar X^1_0),\,\im\PB_n(\bar X^2_0)}
\cong\prod_{\ell=1}^n\pi_1(\bar X_x).
\]

On the other side, applying \cref{cor:free-edge-disconnected} to
\(\bar X_0\) at \(x\), with the section realised inside
\(\sfH_\sfe\subseteq\bar X_0\), shows that
\[
\co{\im\PB_n(\bar X_0)}^{\PB_n(\bar X,\base)}
=\co{\im\PB_n(\bar X^1_0),\,\im\PB_n(\bar X^2_0),\,
\im\pi_1(\sfG)},
\qquad
\sfG\coloneqq\sfG_n((\bar X_0,x),\base).
\]
Applying \cref{cor:free-edge-splitting} to \(\bar X\) at the
point \(x\) -- the cut \(\bar X_x\) being connected -- gives a
surjection
\(\PB_n(\bar X_x)\ast\mathbb{F}_n\twoheadrightarrow
\PB_n(\bar X,\base)\), with \(\mathbb{F}_n\) free on the bigon
letters \(t_1,\dots,t_n\) at \(x\). Writing
\(Q\coloneqq\PB_n(\bar X,\base)\big/\co{\im\PB_n(\bar X_0)}\)
and passing to quotients, the two displays combine into
surjections
\[
\Bigl(\,\prod_{\ell=1}^n\pi_1(\bar X_x)\Bigr)\ast\mathbb{F}_n
\twoheadrightarrow Q
\twoheadrightarrow\prod_{\ell=1}^n\pi_1(\bar X,x^0_\ell),
\]
the second map induced by \(\iota_\ast\)
(\cref{prop:iota-surj}). By van~Kampen,
\(\pi_1(\bar X)\cong\pi_1(\bar X_x)\ast\mathbb{Z}\), with the
free factor \(\mathbb{Z}\) generated by the new cycle through
\(x\), and the composite surjection above is the canonical map
\[
\Bigl(\,\prod_{\ell}\pi_1(\bar X_x)\Bigr)\ast\mathbb{F}_n
\longrightarrow
\prod_{\ell}\bigl(\pi_1(\bar X_x)\ast\mathbb{Z}\bigr),
\]
whose kernel is normally generated by the commutators
\([\alpha_i,\beta_j]\) and \([\beta_i,\beta_j]\), \(i\neq j\)
(\cref{lem:coordinatewise-kernel}(ii)), where
\(\alpha_i\) denotes a braid whose \(i\)-th
strand traverses a loop of \(\bar X_x\), the other strands
resting, and \(\beta_j\) a braid whose \(j\)-th strand traverses
the new cycle through \(x\), the other strands resting.
Consequently \(Q\cong\prod_\ell\pi_1(\bar X,x^0_\ell)\) --
whence \(\co{\im j_\ast}=\ker\iota_\ast\), completing the
induction -- once the following is proved: \emph{for
\(i\neq j\), the images in \(Q\) of \(\alpha_i\) and
\(\beta_j\), and likewise of \(\beta_i\) and \(\beta_j\),
commute.}

Put \(Y=\bar X_x\), \(W=\bar X_0\), and
\(N=\co{\im\PB_n(W)}\). The preceding removal of
\(\Lambda\), together with the induction hypothesis, has
the following consequence: any loop in \(F_n(Y)\) whose
strand classes are all trivial maps into \(N\). This remains
true at any base configuration, after basepoint transport.
In particular, two paths in \(F_n(Y)\) with the same
endpoints and homotopic individual strand paths become equal
modulo \(N\). In the calculations below paths are composed
from left to right; a letter denoting a strand motion is used
at the configuration reached by the preceding motions.

First consider an old coordinate loop and a new cycle in
distinct coordinates \(i,j\). Choose parking points \(p_-\)
and \(p_+\) immediately on opposite sides of \(x\). Let
\(b_-,b_+\) be configurations which differ only in placing
strand \(j\) at \(p_-\) or \(p_+\). The other points are
away from this small segment. Let
\[
C:b_-\longrightarrow b_+,
\qquad K:b_+\longrightarrow b_-
\]
move only strand \(j\), with \(C\) crossing \(x\) along the
small segment and \(K\) returning through \(Y\). Represent
the desired \(i\)-th coordinate class by a loop avoiding
that segment and all stationary points. The two half-edge
stubs at \(x\) can be shortened, so this imposes no
restriction on the coordinate class. Write \(A_-\) and
\(A_+\) for the same \(i\)-th strand motion at \(b_-\)
and \(b_+\), respectively. Independent motion of these two
strands gives a collision-free square and hence
\(CA_+\simeq A_-C\). Now
\[
g=K A_-K^{-1}A_+^{-1}\in\pi_1(F_n(Y),b_+)
\]
has trivial strand classes: strand \(j\) follows \(K\)
and its reverse, while strand \(i\) follows its loop and
its reverse. Thus \(g\) maps into the transported normal
subgroup \(N\), and
\[
(CK)A_-(CK)^{-1}A_-^{-1}
\simeq CgC^{-1}\in N.
\]
To compare with a common base configuration \(b\), take a
path \(P:b\to b_-\) in \(F_n(Y)\), moving strand \(j\)
to \(p_-\). The other rest positions may be chosen in the
thick part, so such parking paths can avoid them. If \(A_0\)
denotes the same old coordinate loop at \(b\), then
\(P A_-P^{-1}A_0^{-1}\) is a loop in \(F_n(Y)\) with
trivial strand classes. Consequently it is killed in \(Q\),
and the displayed calculation proves the required commutation
for \(\beta_j=PCKP^{-1}\) and \(\alpha_i\).

The old coordinate factors already commute with one another.
Changing the return corridor of a new cycle multiplies its
lift by an old loop in the same coordinate; the comparison
loop avoids \(x\), and is determined in \(Q\) by its
strand classes. These changes preserve the assertion that
\(\beta_i\) and \(\beta_j\) commute, by the mixed
commutations just proved. We need only this invariance of
the relation, not equality of the commutator elements for
different corridors. Common basepoint transports in
\(F_n(Y)\) are allowed as well: cutting the cycle at its
unique crossing of \(x\) compares the connecting paths in
\(F_n(Y)\), where their classes are determined by their
individual strand paths. For a cycle in coordinate \(j\),
the out-and-back paths of the other coordinates cancel,
using the mixed commutations. The remaining change is an
old coordinate-\(j\) factor.

It remains to prove the new-cycle relation with convenient
corridors. Let \(\alpha\) be the unique path in
\(\sfT_{\bar X}\) from \(\sfw\) to \(\sfw'\).
It crosses \(x\), since its endpoints lie in different
components of the cut tree, and it lies in \(W\).
The omitted thick edge \(\sff\) meets this tree only at
its endpoints. Choose a narrow strip along \(\sff\) in
\(\bar\Sigma\), disjoint from \(\alpha\) except at
\(\sfw,\sfw'\). Here only a strip on one side of the
arc is needed, not a surface neighbourhood in the ambient
complex. Such a strip can be constructed after subdivision:
along each edge use a thin sector in an incident triangle;
at each intervening vertex join the sectors along a simple
path in the connected link of \(\bar\Sigma\), and cone
that path. Taking the sectors sufficiently small avoids
the compact portions of \(\alpha\) away from the endpoints.
At either endpoint choose the sector inside its vertex
neighbourhood and away from the single direction of
\(\alpha\). Thus the construction also applies when
\(\bar\Sigma\) is not a surface.

Choose four short legs at the ends of \(\alpha\), two in
\(N_{\bar X}(\sfw)\) and two in
\(N_{\bar X}(\sfw')\), together with two disjoint return
arcs along this strip. More precisely, let \(L_i,L_j\)
lie on the two left legs and \(R_i,R_j\) on the two
right legs. The return arcs \(\delta_i:R_i\to L_i\)
and \(\delta_j:R_j\to L_j\) are chosen as parallel
tracks in the strip, with the legs chosen to match their
endpoints. Their interiors avoid all four legs and
\(\alpha\). The union of the legs with \(\alpha\)
is a tree \(\sfH_\alpha\subseteq W\). Park the other
strands in \(W\) outside these tracks and the portions
of the legs used by the two moving strands; a short
extension of a leg beyond its parking point suffices.

Let \(A_i:L_i\to R_i\) and \(A_j:L_j\to R_j\)
be the motions along their legs and \(\alpha\), and let
\(D_i,D_j\) be the return motions along \(\delta_i,\delta_j\).
Set \(t_i=A_iD_i\), \(t_j=A_jD_j\). Whenever two
motions below are exchanged, their spatial tracks are
disjoint, so the exchange is a homotopy relative to the
endpoint configurations. Specifically, \(D_i\) may be
exchanged with \(A_j,D_j\), and \(D_j\) with \(A_i\).
With the other strand at the position determined by the
preceding factors, these exchanges and backtracking
cancellations give
\begin{align*}
t_it_jt_i^{-1}t_j^{-1}
&=A_iD_iA_jD_jD_i^{-1}A_i^{-1}D_j^{-1}A_j^{-1}\\
&\simeq A_iA_jD_jD_iD_i^{-1}A_i^{-1}D_j^{-1}A_j^{-1}\\
&\simeq A_iA_jA_i^{-1}A_j^{-1}.
\end{align*}
The last expression denotes the loop through the four
successive configurations
\[
(L_i,L_j)\longrightarrow(R_i,L_j)
\longrightarrow(R_i,R_j)\longrightarrow(L_i,R_j)
\longrightarrow(L_i,L_j),
\]
with all other points fixed. Every motion in this loop
lies in \(\sfH_\alpha\subseteq W\); hence it is a
loop in \(F_n(W)\) and is killed in \(Q\). No
identification with a particular shuffle generator is
needed. Thus the new-cycle commutators vanish as well.
This completes the proof.
\end{proof}

\begin{proof}[Proof of \cref{thm:candidate-witness-weak}]
Induct on \(b_1(\Phi_{\bar X})\):
\cref{prop:witness-tree-case} is the base case, and
\cref{prop:witness-betti-step} is the inductive step.
\end{proof}

\section{Characterisation of Goldbergness}
\label{sec:goldberg-classification}

Having characterised \emph{weak} Goldbergness in
\cref{sec:simple-model-necessity}, we now turn to the finer
question of which complexes are Goldberg, collecting the
structural constraints that a (strong) Goldberg witness must
satisfy.

\subsection{Structure of Goldberg witnesses}
\label{ssec:goldberg-witness-structure}

Recall from \cref{ssec:witness-intervals} the regular
neighbourhoods \(N(\bar\Sigma)\) of the thick components and the
remaining thin graph, together with the intersection components
\(\bar\Sigma_0^1,\dots,\bar\Sigma_0^m\) of a subcomplex
\(\bar X_0\subseteq\bar X\). The following necessary condition
constrains a Goldberg witness locally, inside each
\(N(\bar\Sigma)\).

\begin{theorem}\label{thm:witness-local-injectivity}
If \(\bar X_0\) is a Goldberg witness for \(\bar X\), then for every
thick component \(\bar\Sigma\) and every connected component
\(\bar\Sigma_0^i\) of \(\bar X_0\cap N(\bar\Sigma)\), the
inclusion-induced map
\[
\PB_n\bigl(\bar\Sigma_0^i\bigr)\longrightarrow\PB_n\bigl(N(\bar\Sigma)\bigr)
\]
is injective. In particular, if \(N(\bar\Sigma)\) is Goldberg-trivial,
then \(\PB_n(\bar\Sigma_0^i)\) is trivial for every \(i\).
\end{theorem}

\begin{proof}
The inclusions \(\bar\Sigma_0^i\subseteq N(\bar\Sigma)\subseteq\bar X\)
and \(\bar\Sigma_0^i\subseteq\bar X_0\subseteq\bar X\) induce a
commutative square of pure braid groups
\[
\begin{tikzcd}
\PB_n(\bar\Sigma_0^i)\ar[r]\ar[d] & \PB_n(N(\bar\Sigma))\ar[d]\\
\PB_n(\bar X_0)\ar[r,"j_\ast"'] & \PB_n(\bar X).
\end{tikzcd}
\]
Both vertical maps are injective by the vertex-space conclusion
of \cref{lem:x-component-embedding}, applied at the outer tips
\(\sfv_1,\dots,\sfv_r\) of the pendant intervals of
\(N(\bar\Sigma)\). More explicitly, at an interior point
\(v\) of a free edge of a complex \(Z\), each component of
\(F_n(Z\setminus\{v\})\) is a vertex space of the
graph-of-spaces decomposition, and its fundamental group
injects into \(\PB_n(Z)\). This holds whether or not
\(Z\setminus\{v\}\) is connected: it uses the vertex group,
not the closure of a component of the punctured complex.
The same statement holds at a free endpoint, whose link is
a single point.

Puncture successively at the tips \(\sfv_k\), using the compact
collar models of \cref{conv:subcomplex} after each cut. The
component on the \(\bar\Sigma\)-side after all cuts is
\(N(\bar\Sigma)\) with its pendant tips shortened. Restoring
their lengths induces an isomorphism on braid groups, so the
successive vertex-group embeddings give
\(\PB_n(N(\bar\Sigma))\hookrightarrow\PB_n(\bar X)\).
Inside \(\bar X_0\), a tip that belongs to the witness has
link consisting of one or two points. The same procedure
isolates the components \(\bar\Sigma_0^i\), with any incident
terminal collars shortened, and gives
\(\PB_n(\bar\Sigma_0^i)\hookrightarrow\PB_n(\bar X_0)\).
All comparisons commute with inclusion by the common collar
construction.

As \(\bar X_0\) is a Goldberg witness, the bottom map
\(j_\ast\colon\PB_n(\bar X_0)\to\PB_n(\bar X)\) is injective. Hence the
composite \(\PB_n(\bar\Sigma_0^i)\to\PB_n(\bar X_0)\to\PB_n(\bar X)\)
-- the left vertical followed by \(j_\ast\) -- is injective. By
commutativity it coincides with the composite
\(\PB_n(\bar\Sigma_0^i)\to\PB_n(N(\bar\Sigma))\to\PB_n(\bar X)\), the
top map followed by the right vertical, which is therefore injective as
well. In particular its first factor
\(\PB_n(\bar\Sigma_0^i)\to\PB_n(N(\bar\Sigma))\) is injective.

Finally, suppose \(N(\bar\Sigma)\) is Goldberg-trivial, so that
\(\ker\iota_\ast^{N(\bar\Sigma)}=1\). The witness \(\bar X_0\) is
contractible (\cref{def:goldberg-strengths}) and is assembled from
its pieces by gluings along finitely many points, so by van~Kampen
\(\pi_1(\bar\Sigma_0^i)\) is a free factor of
\(\pi_1(\bar X_0)=1\); hence \(\pi_1(\bar\Sigma_0^i)=1\). Consequently the strand map
\(\iota_\ast^{\bar\Sigma_0^i}\colon\PB_n(\bar\Sigma_0^i)\to
\prod_{\ell}\pi_1(\bar\Sigma_0^i)\) is the zero map, and by naturality
of the strand map under \(\bar\Sigma_0^i\hookrightarrow N(\bar\Sigma)\)
the image of \(\PB_n(\bar\Sigma_0^i)\) in \(\PB_n(N(\bar\Sigma))\) lies
in \(\ker\iota_\ast^{N(\bar\Sigma)}=1\). This inclusion-induced map is
also injective by the above, so \(\PB_n(\bar\Sigma_0^i)\) has trivial
image and trivial kernel; hence \(\PB_n(\bar\Sigma_0^i)=1\).
\end{proof}

\subsection{A local-to-global criterion}
\label{ssec:local-to-global}

\cref{thm:witness-local-injectivity} shows that if \(\bar X_0\)
is a Goldberg witness, then each intersection component
\(\bar\Sigma_0^i\) injects, at the level of pure braid groups,
into the neighbourhood \(N(\bar\Sigma)\) of its thick component.
In this subsection we prove the converse: for a \emph{weak}
Goldberg witness this local condition already forces the global
injectivity of \(j_\ast\). The weak witness property enters only
through the containment results of \cref{sec:lower-bound}; the
heart of the matter is a purely local-to-global injectivity
statement, proved by unfolding distribution graphs into finite
covers.

Call a thin edge \(\sfe\) of \(\bar X\) \emph{essential} if, for an
interior point \(x\in\mathring \sfe\), either \(\bar X_x\) is
connected, or neither \(x\)-component of \(\bar X\) is
homeomorphic to an interval; this does not depend on the choice
of \(x\). Pendant thin edges are never essential. By
\cref{prop:X0-contains-free-points,prop:X0-meets-link-m1,prop:X0-contains-nontrivial-junction},
a weak Goldberg witness \(\bar X_0\) contains every essential
thin edge of \(\bar X\), together with both germs at each of its
interior points; and by
\cref{prop:X0-contains-valency-3,prop:X0-contains-m2-val3,prop:X0-meets-link-components},
it contains every point of \(\bar X\) of valency at least \(3\),
together with a germ of every link direction there.

\begin{lemma}[Stabilisation]\label{lem:stabilisation}
Let \(A\subseteq B\) be connected complexes sharing a free leaf collar:
an embedded interval \(J=[0,2]\subseteq A\) whose endpoint
\(0\) is free in \(B\), with
\(J\cap\overline{B\setminus J}=\{2\}\).
Fix \(S\subseteq\{1,\dots,n\}\). Choose the base positions
of the strands in \(S\) outside \(J\), and those of the
remaining strands at distinct points of \([0,1)\).
If \(\PB_n(A)\to\PB_n(B)\) is injective, then so is
\(\PB_S(A)\to\PB_S(B)\).
\end{lemma}

\begin{proof}
On \(J\), use the embeddings
\(c_t(s)=(1-t)s+t(1+s/2)\), and extend them by the identity
outside \(J\). They preserve \(A\), fix the base positions of \(S\),
and \(c_1\) leaves \([0,1)\) empty. Compress a configuration
by \(c_1\), then insert the strands outside \(S\) at distinct
fixed positions in \([0,1)\). This defines stabilisation
homomorphisms for both \(A\) and \(B\), commuting with
inclusion. Forgetting the inserted strands gives
\((c_1)_*=\mathrm{id}\), by the isotopy \(c_t\).
Thus, if \(g\in\PB_S(A)\) dies in \(\PB_S(B)\), its
stabilisation dies in \(\PB_n(B)\), hence in \(\PB_n(A)\)
by hypothesis; forgetting gives \(g=1\).
\end{proof}

\begin{lemma}[Injectivity criterion for morphisms of graphs of
groups]\label{lem:gog-injective-criterion}
Let \(\sfY\subseteq\sfY'\) be a subgraph containing the base
vertex \(\sfv_0\), let \(\cH\) and \(\cG\) be graphs of groups
over \(\sfY\) and \(\sfY'\) respectively, and let
\(\varphi=(\varphi_{\sfv},\varphi_{\sfe})\) be a family of
homomorphisms \(H_{\sfv}\to G_{\sfv}\), \(H_{\sfe}\to G_{\sfe}\)
commuting with the edge inclusions \(\alpha\). Suppose that
every \(\varphi_{\sfv}\) is injective and that, for every
oriented edge \(\sfe\) of \(\sfY\) with terminal vertex
\(\sfv\),
\[
\varphi_{\sfv}^{-1}\bigl(\alpha_{\sfe}(G_{\sfe})\bigr)
=\alpha_{\sfe}(H_{\sfe}).
\]
Then the induced homomorphism
\(\varphi_\ast\colon\pi_1(\cH,\sfv_0)\to\pi_1(\cG,\sfv_0)\) is
injective.
\end{lemma}

\begin{proof}
An element of \(\pi_1(\cH,\sfv_0)\) is represented by a word
\(g_0\sfe_1g_1\cdots \sfe_kg_k\), where \(\sfe_1\cdots \sfe_k\) is an edge
loop at \(\sfv_0\) in \(\sfY\) and each \(g_i\) lies in the
vertex group at the corresponding vertex; the word is
\emph{reduced} if either \(k=0\) and \(g_0\neq1\), or
\(k\geq1\) and no subword \(\sfe_ig_i\sfe_{i+1}\) has
\(\sfe_{i+1}=\bar \sfe_i\) and \(g_i\in\alpha_{\sfe_i}(H_{\sfe_i})\). By the
normal-form theorem for graphs of groups \cite{Serre1980}, a
reduced word represents a non-trivial element. The morphism
\(\varphi\) sends such a word to the word with the same edge
loop -- \(\sfY\) being a subgraph of \(\sfY'\), no
identifications occur among the edges -- and with letters
\(\varphi(g_i)\). If the image word were not reduced, some
\(g_i\) flanked by \(\sfe_i\) and \(\bar \sfe_i\) would satisfy
\(\varphi(g_i)\in\alpha_{\sfe_i}(G_{\sfe_i})\), whence
\(g_i\in\alpha_{\sfe_i}(H_{\sfe_i})\) by the displayed condition,
contradicting reducedness; and if \(k=0\), then
\(\varphi(g_0)\neq1\), since \(\varphi_{\sfv_0}\) is injective.
Hence reduced words map to reduced words, and \(\varphi_\ast\)
is injective.
\end{proof}

\begin{lemma}[Unfolding the distribution graphs]
\label{lem:unfolding}
Let \(Z\) be a connected complex, \(x\) an interior point of a
free edge \(\sfe\) of \(Z\), and \(Y\subseteq Z\) a connected
subcomplex containing a neighbourhood of \(x\) in \(\sfe\). Write
\(\cG\) and \(\cG_0\) for the graph-of-spaces decompositions of
\(F_S(Z)\) and \(F_S(Y)\) at \(x\) (\cref{prop:valency-decomp}),
with underlying distribution graphs \(\sfG\) and \(\sfG_0\). Then:
\begin{enumerate}[label=\textup{(\arabic*)},leftmargin=2.4em]
  \item the natural map \(\sfG_0\to\sfG\) is an immersion,
        i.e.\ locally injective;
  \item it factors as an embedding
        \(\sfG_0\hookrightarrow\tilde\sfG\) followed by a finite
        covering \(p\colon\tilde\sfG\to\sfG\);
  \item writing \(p^\ast\cG\) for the pullback of \(\cG\) along
        \(p\), the total space of \(p^\ast\cG\) is a covering
        space of the model \(T(\cG)\simeq F_S(Z)\). Under the
        induced identification of fundamental groups, the inclusion-induced map
        \(\PB_S(Y)\to\PB_S(Z)\) factors as
        \[
        \PB_S(Y)=\pi_1(\cG_0)
        \xrightarrow{\varphi}
        \pi_1(p^\ast\cG)
        \hookrightarrow
        \PB_S(Z),
        \]
        where \(\varphi\) is a morphism of graphs of groups over
        the embedding \(\sfG_0\hookrightarrow\tilde\sfG\).
\end{enumerate}
\end{lemma}

See \cref{fig:unfolding} for a schematic illustration of the
completion in~(2).

\begin{figure}[ht]
\centering
\begin{tikzpicture}[line width=0.8pt, scale=1]
  \begin{scope}[xshift=-4.6cm]
    \draw[blue, line width=1.2pt] (-60:0.85) -- (0:0.85) -- (60:0.85) -- (120:0.85);
    \filldraw[blue] (-60:0.85) circle (1.5pt);
    \filldraw[blue] (0:0.85) circle (1.5pt);
    \filldraw[blue] (60:0.85) circle (1.5pt);
    \filldraw[blue] (120:0.85) circle (1.5pt);
    \node at (0,-1.25) {\(\sfG_0\)};
  \end{scope}
  \draw[->, thick] (-3.4,0) -- (-2.5,0)
    node[midway, above=1pt] {\scriptsize emb.};
  \begin{scope}
    \draw (0:0.85) \foreach \a in {60,120,180,240,300,360} { -- (\a:0.85) };
    \foreach \a in {0,60,120,180,240,300}
      \filldraw (\a:0.85) circle (1.4pt);
    \draw[blue, line width=1.2pt]
      (120:0.85) -- (60:0.85) -- (0:0.85) -- (-60:0.85);
    \filldraw[blue] (120:0.85) circle (1.5pt);
    \filldraw[blue] (60:0.85) circle (1.5pt);
    \filldraw[blue] (0:0.85) circle (1.5pt);
    \filldraw[blue] (-60:0.85) circle (1.5pt);
    \node at (0,-1.25) {\(\tilde\sfG\)};
  \end{scope}
  \draw[->, thick] (1.3,0) -- (2.2,0) node[midway, above=1pt] {\(p\)};
  \begin{scope}[xshift=3.6cm]
    \draw (90:0.75) -- (210:0.75) -- (330:0.75) -- cycle;
    \foreach \a in {90,210,330} \filldraw (\a:0.75) circle (1.4pt);
    \node at (0,-1.25) {\(\sfG\)};
  \end{scope}
\end{tikzpicture}
\caption{Unfolding an immersion (\cref{lem:unfolding}),
schematically: the graph \(\sfG_0\) immerses into \(\sfG\),
with its endpoints identified in the image; it embeds (blue) into a
finite cover \(\tilde\sfG\xrightarrow{p}\sfG\) completing the
immersion.}
\label{fig:unfolding}
\end{figure}

\begin{proof}
(1) The vertices of \(\sfG_0\) are the components of
\(F_S(Y_x)\) and, for each strand \(i\), the components of the
space of configurations with the \(i\)-th strand at \(x\); the
edges at a vertex of the latter kind correspond to the two
germs of \(\sfe\) at \(x\), while an edge at a vertex \(\sfu\) of
the former kind is determined by its datum \((i,\ell)\) -- the
transferring strand and the germ -- because its other end is
the component of the configurations obtained from those of
\(\sfu\) by moving the \(i\)-th strand to \(x\) from the
\(\ell\)-side. The map \(\sfG_0\to\sfG\) sends a component to
the component containing it and preserves these data; two
distinct edges at a common vertex differ in their data and
hence have distinct images. Thus \(\sfG_0\to\sfG\) is an
immersion.

(2) This is the completion of an immersion of finite graphs to
a finite covering, as in \cite{Stallings1983}: enlarge
\(\sfG_0\) by additional vertices so that all fibres over
vertices of \(\sfG\) acquire a common cardinality, then
complete, over each edge of \(\sfG\), the induced partial
matching between the fibres of its endpoints to a perfect
matching.

(3) The graph-of-spaces structure provides a projection
\(\rho\colon T(\cG)\to\sfG\) of the model total space
onto the distribution graph. The fibre product
\(T(\cG)\times_{\sfG}\tilde\sfG\) is a covering space of
\(T(\cG)\), being the pullback of the covering \(p\), and is
canonically the total space of \(p^\ast\cG\); hence
\(\pi_1(p^\ast\cG)\to\PB_S(Z)\) is injective. The inclusion
\(F_S(Y)\hookrightarrow F_S(Z)\) is a map of graphs of spaces
over \(\sfG_0\to\sfG\); together with the factorisation
of~(2) it lifts to a map \(T(\cG_0)\to T(p^\ast\cG)\) over the
embedding \(\sfG_0\subseteq\tilde\sfG\), whose composite with
the covering projection represents the inclusion under the
homotopy equivalences of the models. Taking fundamental
groups gives the asserted factorisation.
\end{proof}

\begin{proposition}[Local injectivity implies global
injectivity]
\label{thm:local-to-global}
Let \(\bar X_0\) be a weak Goldberg witness for \(\bar X\) such
that, for every thick component \(\bar\Sigma\) and every
connected component \(\bar\Sigma_0^i\) of
\(\bar X_0\cap N(\bar\Sigma)\), the inclusion-induced map
\[
\PB_n\bigl(\bar\Sigma_0^i\bigr)\longrightarrow
\PB_n\bigl(N(\bar\Sigma)\bigr)
\]
is injective. Then
\(j_\ast\colon\PB_n(\bar X_0)\to\PB_n(\bar X)\) is injective;
that is, \(\bar X_0\) is a Goldberg witness.
\end{proposition}

\begin{proof}
Basepoints are chosen in the subcomplexes at hand and
suppressed throughout. Fix an interior point
\(x_\sfe\in\mathring \sfe\) of each essential thin edge \(\sfe\) of
\(\bar X\). For a set \(A\) of such points, write \(\bar X_A\)
for the compact model obtained by cutting \(\bar X\) at every
point of \(A\), using the compatible collar truncations of
\cref{conv:subcomplex}; since \(\bar X_0\) contains every essential thin
edge, the cut applies to \(\bar X_0\) as well, and for a
connected component \(Z\) of \(\bar X_A\) we write
\(Y_Z\subseteq Z\) for the union of the components of the cut
witness contained in \(Z\). We prove the following claim by
induction on the number of essential thin edges of \(\bar X\)
contained in \(Z\).

\emph{Claim. For every component \(Z\) of every iterated cut
\(\bar X_A\), every \(S\subseteq\{1,\dots,n\}\) (requiring \(S=\{1,\dots,n\}\) when \(A=\varnothing\)), and every
partition \(S=\bigsqcup_cS_c\) indexed by the components
\(Y_c\) of \(Y_Z\), the homomorphism
\[
\prod_c\PB_{S_c}(Y_c)\longrightarrow\PB_S(Z),
\]
given on each factor by the inclusion-induced map -- the images
commute, having disjoint supports -- is injective.}

The theorem is the case \(A=\varnothing\), \(Z=\bar X\),
\(S=\{1,\dots,n\}\).

For \(A\neq\varnothing\), every component \(Y_c\) of the
cut witness contains a tip created by a cut: the uncut witness
is connected and contains every point of \(A\), so a component
without such a tip would already be a separate component
before cutting. A sufficiently short interval at that tip is
a common free leaf collar in \(Y_c\) and its ambient cut
component. We will use \cref{lem:stabilisation} only with
these collars. For \(A=\varnothing\) no reduction in the
number of strands is needed.

\emph{Reduction to a single component.} Suppose \((g_c)_c\)
lies in the kernel. Fix an index \(c\) and apply the
homomorphism \(\PB_S(Z)\to\PB_{S_c}(Z)\) forgetting all strands
outside \(S_c\): the factors with index other than \(c\) map to
the identity, so the image of \(g_c\) in \(\PB_{S_c}(Z)\) is
trivial. It therefore suffices to prove the claim for a single
component: a connected \(Y\subseteq Z\) with all the strands of
\(S\) based in \(Y\).

\emph{Inductive step.} Suppose \(Z\) contains an essential thin
edge \(\sfe\) of \(\bar X\), and set \(x=x_\sfe\). Since \(\bar X_0\)
contains \(\sfe\) and \(\sfe\) is uncut in \(Z\), the edge \(\sfe\) lies
in a single component \(Y_\sfe\) of \(Y_Z\).

If \(Y\neq Y_\sfe\), then no braid of \(Y\) crosses \(x\), and
\(Y\) lies in a single connected component \(Z'\) of \(Z_x\);
the map \(\PB_S(Y)\to\PB_S(Z)\) factors as
\(\PB_S(Y)\to\PB_S(Z')\to\PB_S(Z)\). The second map is
injective by \cref{lem:x-component-embedding}, and the first by
the inductive hypothesis applied to the pair \((Z',Y)\): the
component \(Z'\) of \(\bar X_{A\cup\{x\}}\) contains fewer
essential thin edges than \(Z\).

If \(Y=Y_\sfe\), apply \cref{lem:unfolding} to \((Z,Y,x)\): the
map \(\PB_S(Y)\to\PB_S(Z)\) factors through a morphism
\(\varphi\) of graphs of groups over the embedding
\(\sfG_0\hookrightarrow\tilde\sfG\), followed by an injection.
It remains to verify the hypotheses of
\cref{lem:gog-injective-criterion} for \(\varphi\).

\emph{Vertex injectivity.} Consider a vertex \(\sfu\) of
\(\sfG_0\) of \(F_S\)-type. Its vertex group in \(\cG_0\) is
the direct product \(\prod_c\PB_{S_c}(Y'_c)\) over the
components \(Y'_c\) of \(Y_x\), with the strand distribution
recorded by \(\sfu\); the corresponding vertex group of
\(p^\ast\cG\) is the direct product
\(\prod_{Z'}\PB_{S_{Z'}}(Z')\) over the components \(Z'\) of
\(Z_x\), where \(S_{Z'}\) is the union of the \(S_c\) with
\(Y'_c\subseteq Z'\); and \(\varphi_{\sfu}\) is the direct
product, over the \(Z'\), of the maps
\[
\prod_{c\,:\,Y'_c\subseteq Z'}\PB_{S_c}(Y'_c)
\longrightarrow\PB_{S_{Z'}}(Z').
\]
Each of these is injective by the inductive hypothesis: \(Z'\)
is a component of \(\bar X_{A\cup\{x\}}\) with fewer essential
thin edges, and the map is exactly the one of the claim. A
direct product of injective homomorphisms is injective, so
\(\varphi_{\sfu}\) is injective. Vertices of the other type --
one strand deleted, resting at \(x\) -- are handled by the same
argument with \(S\) replaced by the complement of the deleted
strand.

\emph{Coset condition.} Let \(\sfu\) be a vertex of
\(F_S\)-type and \(\sfe\) an adjacent edge of \(\sfG_0\), with
datum \((i,\ell)\). In both \(\cG_0\) and \(p^\ast\cG\), the
edge group of \(\sfe\) includes into the vertex group at
\(\sfu\) as the subgroup of braids in which the \(i\)-th strand
rests at a parking position on the \(\ell\)-germ of \(\sfe\), and
this subgroup is the image of the retraction
\(r\coloneqq\mathrm{st}_i\circ\mathrm{fgt}_i\).
Here stabilisation compresses a common free leaf collar at
the freshly cut tip before inserting the constant strand,
as in \cref{lem:stabilisation}. The cut germ lies in \(Y_x\)
as well as \(Z_x\), so the same compression defines \(r\)
on both vertex groups and commutes with \(\varphi_{\sfu}\). Hence, if \(h\) lies in the
\(\cG_0\)-vertex group and \(\varphi_{\sfu}(h)\) lies in the
edge subgroup, then
\(\varphi_{\sfu}(h)=r(\varphi_{\sfu}(h))
=\varphi_{\sfu}(r(h))\), so \(h=r(h)\) by the injectivity of
\(\varphi_{\sfu}\), and \(h\) lies in the edge subgroup of
\(\cG_0\). At a vertex of the other type the edge inclusions
are isomorphisms onto the vertex groups, and the condition is
vacuous. \cref{lem:gog-injective-criterion} now applies, and
the inductive step is complete.

\emph{Base case.} Suppose \(Z\) contains no essential thin edge
of \(\bar X\). We claim \(Z\) contains at most one
\emph{feature} -- a thick component or a point of valency at
least \(3\). Indeed, two features of \(Z\) are joined by an
embedded path of thin edges of \(\bar X\), and for any interior
point \(x\) of any edge \(\sfe\) on this path, either \(\bar X_x\)
is connected, or each \(x\)-component of \(\bar X\) contains
one of the two features and hence an embedded tripod; in either
case \(\sfe\) is essential -- a contradiction. There are thus
three possibilities.

If \(Z\) contains no feature, then \(Z\) is homeomorphic to an
interval or a point -- a thin cycle is excluded, its edges
being essential, unless \(\bar X\cong S^1\), which is excluded
by \cref{rem:standing} -- and \(\PB_S(Y)\) is trivial, each
component of a configuration space of an interval being
contractible; there is nothing to prove.

If \(Z\) contains a single point \(\sfv\) of valency at least
\(3\) and no thick component, then \(Z\) is a tree, and
straightening its hanging chains exhibits it as a star with
centre \(\sfv\). If \(\sfv\notin Y\), then \(Y\) lies in one leg and
is an interval, and \(\PB_S(Y)=1\) as before. If \(\sfv\in Y\),
then, by the containment results quoted at the beginning of
this subsection, \(Y\) contains a germ of every leg of \(Z\) at
\(\sfv\), so \(Y\) is a star with the same set of legs, each
possibly truncated; the inclusion \(Y\hookrightarrow Z\) is
isotopic to a homeomorphism, stretching each truncated leg over
the full one, and therefore induces an isomorphism
\(\PB_S(Y)\to\PB_S(Z)\).

Finally, suppose \(Z\) contains a single thick component
\(\bar\Sigma\) and no point of valency at least \(3\). Then,
after straightening hanging chains, \(Z\) is homeomorphic to
the regular neighbourhood \(N(\bar\Sigma)\), by a
homeomorphism which is the identity on \(\bar\Sigma\) and
stretches pendant edges. If \(Y\) meets \(\bar\Sigma\), its
image \(\widehat Y\) and the corresponding intersection
component \(\bar\Sigma_0^i\) have the same part in
\(\bar\Sigma\), differing only in the lengths of their
pendant arms. Choose a common subspace \(C\) by shortening
each such arm to a positive length no greater than either
of the two lengths. The inclusions
\(C\hookrightarrow\widehat Y\) and
\(C\hookrightarrow\bar\Sigma_0^i\) induce isomorphisms on
pure braid groups: on each arm, linear compression fixing
its attachment point gives an inverse up to isotopies of
embeddings, extended by the identity elsewhere. Choose the
base configuration in \(C\). The two inclusion maps from
\(C\) to \(N(\bar\Sigma)\) coincide, so injectivity for
\(\bar\Sigma_0^i\) implies injectivity for \(\widehat Y\),
and hence for \(Y\hookrightarrow Z\). Thus
\(\PB_n(Y)\to\PB_n(Z)\) is injective, and
if \(A\neq\varnothing\), the common free leaf collar
identified above and \cref{lem:stabilisation} give the
injectivity of \(\PB_S(Y)\to\PB_S(Z)\).
If \(A=\varnothing\), then \(S=\{1,\dots,n\}\), so the
given injectivity already suffices. If \(Y\) does not meet
\(\bar\Sigma\), it is an interval inside a pendant chain and
\(\PB_S(Y)=1\). This exhausts all cases and completes the
induction, and with it the proof.
\end{proof}

Combining \cref{thm:local-to-global} with
\cref{thm:witness-local-injectivity}, we obtain a
characterisation of Goldbergness relative to weak Goldbergness
by purely local conditions.

\begin{theorem}\label{cor:goldberg-char-local}
A complex \(\bar X\) as in \cref{rem:standing} is Goldberg if
and only if it admits a weak Goldberg witness \(\bar X_0\) such
that \(\PB_n(\bar\Sigma_0^i)\to\PB_n(N(\bar\Sigma))\) is
injective for every thick component \(\bar\Sigma\) and every
intersection component \(\bar\Sigma_0^i\).
\end{theorem}

\begin{proof}
If \(\bar X\) is Goldberg, its Goldberg witness is a weak
Goldberg witness satisfying the condition, by
\cref{thm:witness-local-injectivity}. Conversely, such a
witness is a Goldberg witness by \cref{thm:local-to-global}.
\end{proof}

\subsection{The main characterisation: admissible trees}
\label{ssec:admissible-trees}

We now combine \cref{cor:goldberg-char-local} with the candidate
construction of \cref{sec:simple-model-necessity} into a purely
combinatorial characterisation of Goldbergness.

Recall from \cref{def:admissible-tree} the types of the
pairs \((\sfv,L)\) and the admissible trees. As
\(\bar X\) is simple, the pairs are simply the joints of
\(\bar X\) (\cref{rem:admissible-tree}): a joint on
\(\bar\Sigma\) is of type \(1\) exactly when \(\bar\Sigma\) is
a \(2\)-manifold and its link is the disjoint union of a point
and an interval -- an attachment at a boundary point of a
surface -- and of type \(2\) otherwise.

\begin{theorem}[Goldbergness via admissible trees]
\label{thm:goldberg-char-trees}
Let \(\bar X\) be as in \cref{rem:standing}. Then \(\bar X\)
is Goldberg if
and only if \(\bar X\) admits an admissible tree.
\end{theorem}

\begin{proof}
Suppose first that \(\sfT\in\mathfrak{T}_{\bar X}\) is admissible.

If some thick component carries no joint at all, then, \(\bar X\)
being connected, \(\bar X\) coincides with that component. If
\(\bar X\) is then not a \(2\)-manifold, it is Goldberg-trivial
by \cref{thm:gt-classification}, hence Goldberg
(\cref{lem:gt-implies-goldberg}); if it is a \(2\)-manifold, it is a
compact surface -- other than \(S^2\) and \(\RP^2\), by
\cref{rem:standing} -- and it is Goldberg by
\cref{thm:goldberg} and \cref{rem:goldberg-compact}.
Assume henceforth that every thick component carries a joint.

Call the set of points of a thick component \(\bar\Sigma\)
having no neighbourhood homeomorphic to \(\R^2\) or to a closed
half-plane the \emph{branch locus} of \(\bar\Sigma\). It is
non-empty exactly when \(\bar\Sigma\) is not a \(2\)-manifold,
and it is at least one-dimensional at each of its points.
Indeed, the link there is connected and has positive dimension.
If its dimension is at least \(2\), a ray into a simplex of
dimension at least \(3\) in the local cone consists of branch
points. If the link is one-dimensional, it is a connected
graph other than an interval or a circle, so it has a vertex
of degree at least \(3\); the cone ray over that vertex likewise
consists of branch points.

\emph{Reaching the branch locus.} We first describe a local
modification of the forest \(F=\sfT\cap\bar\Sigma\) in a
non-manifold thick component. The modification fixes all joints,
preserves the partition of the joints among the components of
\(F\), and makes every component meet the branch locus
\(B\). It uses isotopies of the embedded forest and the
addition of pendant arcs; we do not require an ambient isotopy
of \(\bar\Sigma\).

Suppose a component \(F_c\) misses \(B\). It lies in a
component \(M\) of the surface locus \(\bar\Sigma\setminus B\).
The frontier of \(M\) meets \(B\), since \(\bar\Sigma\) is
connected. After subdivision, some frontier edge has a
three-page patch: three embedded half-discs share
their diameter \(I\subset B\), and one page, away from
\(I\), lies in \(M\). To see this, the frontier cannot
consist only of vertices: at such a vertex the directions
belonging to \(M\) would be separated from the remaining
directions in the link, contrary to connectedness of links
in \(\bar\Sigma\). At an interior point of a frontier edge,
either at least three triangles are incident, or a simplex
of dimension at least three is incident; in the latter case
choose two further half-discs inside that simplex. An edge
with only one or two incident triangles and no higher
simplex would instead be a surface point.

Choose a point \(q\in\mathring I\) away from the finitely
many vertices and transverse intersections of the forest.
If the forest contains a segment of \(I\) around \(q\),
push a short interior subarc of that segment into a second
page, fixing its endpoints. This is an isotopy of an arc
across a half-disc, supported away from the forest vertices;
the component still meets \(B\) at the two fixed endpoints.
After shrinking the three-page patch around \(q\),
it is disjoint from the forest. Thus we may assume \(q\notin F\)
and have a small empty patch in the second page.

Choose an embedded PL arc \(\gamma\) from \(F_c\) to \(q\),
whose interior lies in the interior of \(M\), except for its
endpoint \(q\). General position in this surface makes its
intersections with the forest finite and transverse, away
from forest vertices and joints. Read these intersections
from \(q\) towards \(F_c\). At the first intersection \(z\),
replace a short open segment of the forest edge through
\(z\) by a detour: two parallel tracks follow a narrow
strip along \(\gamma[z,q]\), meet \(I\) at two distinct
points on either side of \(q\), and are joined by a short
arc in the second page. The old segment and the detour
bound an embedded disc whose interior is disjoint from
the current forest. Sliding the segment across this disc,
relative to its endpoints, is an isotopy of embeddings.
The affected component now meets \(B\). The arc \(\gamma\)
passes between the two parallel tracks and misses the
joining arc in the second page, so this removes its
intersection at \(z\) without adding another one.

Repeat at the next intersection, using successively narrower
strips and smaller empty patches around \(q\). There are
only finitely many intersections. All replacements remove
segments in \(M\), so they preserve the branch-locus
intersections already obtained. Once \(\gamma\) is clear,
adjoin it as a pendant arc to \(F_c\), if that component
has not already acquired a branch-locus intersection.
No two forest components are joined, and no cycle is
created. Repeating for the finitely many components that
still miss \(B\) proves the asserted modification.

These operations can avoid any prescribed finite set of
forest vertices, except for the fixed initial point of a
new pendant arc. Their supporting discs are disjoint from
the rest of the forest. Consequently, after the finite
construction has been chosen, sufficiently small vertex
neighbourhoods can be kept fixed throughout the isotopies.
The same operations then apply to the union of the forest
with these neighbourhoods. A new pendant arc meets that
union only in its initial point or a short initial segment.

\emph{The witness.} Start with the candidate weak witness
\(W=\sfT\cup\bigcup_\sfv N_{\bar X}(\sfv)\) of
\cref{thm:candidate-witness-weak}, choosing its vertex
neighbourhoods sufficiently small as above. Perform the
modification in each non-manifold thick component, fixing
the thin edges and the surface thick components. Each
isotopy extends to an isotopy of embeddings of the current
witness, fixing its vertex neighbourhoods; it preserves
the weak witness property by basepoint transport and
conjugation. Each pendant-arc addition preserves
contractibility and the weak witness property by
\cref{prop:witness-upward}.

Write \(\sfT^+\) for the final tree and \(W^+\) for the
modified witness. The tree \(\sfT^+\) need not contain the
original \(\sfT\), but it has the same connections between
joints, and every component of its intersection with a
non-manifold thick component meets the branch locus.
Choose a sufficiently small regular neighbourhood
\(\bar X_0=N(\sfT^+)\) containing \(W^+\), with distinct
intersection components remaining disjoint. This is possible
by the choice of the vertex neighbourhoods and the finiteness
of the construction. It is contractible, since it retracts
onto the tree \(\sfT^+\), and \cref{prop:witness-upward}
makes \(\bar X_0\) a weak Goldberg witness.

\emph{Local injectivity.} By \cref{cor:goldberg-char-local} it
remains to show that
\(\PB_n(\bar\Sigma_0^j)\to\PB_n(N(\bar\Sigma))\) is injective
for every thick component \(\bar\Sigma\) and every intersection
component \(\bar\Sigma_0^j\) of \(\bar X_0\cap N(\bar\Sigma)\).
These components are the thickened neighbourhoods of the
connected components \(\sfT^+_c\) of \(\sfT^+\cap\bar\Sigma\),
together with the initial segments of the thin edges attached
at their joints: each is a contractible complex with a single
thick component -- the thickened tree -- no point of valency at
least \(3\), and all maximal free chains pendant. We distinguish three
cases.

If \(\bar\Sigma\) is not a \(2\)-manifold, then \(\sfT^+_c\) meets
the branch locus, so the thick part of \(\bar\Sigma_0^j\) is
not a \(2\)-manifold, and \cref{thm:gt-classification} makes
\(\bar\Sigma_0^j\) Goldberg-trivial. As \(\bar\Sigma_0^j\) is
simply connected, the strand map has trivial target, so
\(\PB_n(\bar\Sigma_0^j)=\ker\iota_\ast=1\), and injectivity
holds trivially.

If \(\bar\Sigma\) is a \(2\)-manifold possessing a type-2
joint, then a type-2 joint of \(\bar\Sigma\) is an interior
attachment: \(\lk_{\bar X}(\sfv)=\{\mathrm{pt}\}\sqcup S^1\). By
strong admissibility the component \(\sfT^+_c=\sfT_c\) contains a
type-2 joint \(\sfv\), and the pendant stub at \(\sfv\) is attached
at an interior point of the thickened disc. By
\cref{thm:gt-classification}(iii), \(\bar\Sigma_0^j\) is again
Goldberg-trivial, and \(\PB_n(\bar\Sigma_0^j)=1\) as before.

If \(\bar\Sigma\) is a \(2\)-manifold without type-2 joints,
then all its attachments are boundary attachments, so
\(\bar\Sigma\) is a compact surface with non-empty boundary --
in particular homeomorphic to neither \(S^2\) nor \(\RP^2\) --
and \(\bar\Sigma_0^j\) is a disc with pendant edges attached
along its boundary. As in the proof of
\cref{lem:witness-one-thick}, absorbing the pendant edges
(\cref{lem:pendant-edge}, \cref{ex:disc}) and applying
Goldberg's theorem (\cref{thm:goldberg},
\cref{rem:goldberg-compact}) to a disc inside the surface
exhibits \(\bar\Sigma_0^j\) as a Goldberg witness of
\(N(\bar\Sigma)\); in particular
\(\PB_n(\bar\Sigma_0^j)\to\PB_n(N(\bar\Sigma))\) is injective.

In all cases the local condition of
\cref{cor:goldberg-char-local} holds, and \(\bar X\) is
Goldberg.

\emph{Conversely}, suppose \(\bar X\) is Goldberg, with Goldberg
witness \(\bar X_0\) -- in particular a weak Goldberg witness.
If some thick component carries no joint, then \(\bar X\)
coincides with it and any weakly admissible tree is vacuously admissible. If \(\bar X\) is Goldberg-trivial, then by
\cref{thm:gt-classification} it has a single thick component
\(\bar\Sigma\) and every maximal free chain is pendant, so every scaffold is the tree \(\Phi_{\bar\Sigma}\cup\mathsf F_{\bar X}\);
the weakly admissible tree \(\sfT=\Phi_{\bar X}\) then satisfies
\(\sfT\cap\bar\Sigma=\Phi_{\bar\Sigma}\), which is connected and
contains every joint, and \(\sfT\) is admissible. Assume
henceforth that \(\bar X\) is not Goldberg-trivial and that
every thick component carries a joint.

\emph{Claim: if a thick component \(\bar\Sigma\) carries a
type-2 joint, then every intersection component
\(\bar\Sigma_0^j\) of \(\bar X_0\cap N(\bar\Sigma)\) containing
a joint contains a type-2 joint.} Suppose not, and let
\(\bar\Sigma_0^j\) be a component all of whose joints are of
type \(1\). Type-1 joints exist only on \(2\)-manifold
components, so \(\bar\Sigma\) is a compact surface, and its
type-2 joints are interior attachments; hence \(N(\bar\Sigma)\)
is a \(2\)-manifold with an interior pendant edge, and
\cref{thm:gt-classification} makes \(N(\bar\Sigma)\)
Goldberg-trivial. Since \(\bar X_0\) is a Goldberg witness,
\cref{thm:witness-local-injectivity} then forces
\(\PB_n(\bar\Sigma_0^j)=1\).

On the other hand, \(\bar\Sigma_0^j\) is not homeomorphic to an
interval (\cref{thm:no-interval-pieces}; \(\bar X\) is not
Goldberg-trivial) nor to a point -- a one-point component would
be all of the connected witness -- so, being connected and
contractible, it contains an embedded tripod. Its
\(\bar\Sigma\)-part is a compact connected contractible
subcomplex of the surface \(\bar\Sigma\), and a regular
neighbourhood of it in \(\bar\Sigma\) is a compact surface
collapsing onto it~\cite[Ch.~3]{RourkeSanderson}. This surface
is contractible, hence a disc \(D\) by the classification of
compact surfaces; the joints of
\(\bar\Sigma_0^j\), being boundary attachments, lie on
\(\partial D\), so that
\(\bar\Sigma_0^j\subseteq D^+\coloneqq
D\cup(\text{pendant arms})\subseteq N(\bar\Sigma)\), the arms
attached along \(\partial D\). By \cref{ex:disc},
\(\PB_n(D^+)\cong\PB_n(D^2)\) is the classical pure braid
group, non-trivial since \(n\geq2\). Moreover the
inclusion-induced map
\(\PB_n(\bar\Sigma_0^j)\to\PB_n(D^+)\) is surjective: for base
points in single file on an embedded tripod of
\(\bar\Sigma_0^j\), each adjacent half-twist \(\sigma_i\) of
\(\BG_n(D^+)\) is realised by the hexagonal exchange of two
consecutive strands on the tripod (\cref{ex:S3}), the remaining
strands resting on its legs; thus
\(\BG_n(\bar\Sigma_0^j)\to\BG_n(D^+)\) is surjective, and its
compatibility with the permutation homomorphisms -- both
surjecting onto \(\mathbb S_n\) -- makes the pure parts
surject as well. Hence \(\PB_n(\bar\Sigma_0^j)\neq1\), a
contradiction. This proves the claim.

\emph{Construction of an admissible tree.} Let \(\sfT_0\)
be the union of all thin edges of \(\bar X\) with, for every
thick component and every intersection component
\(\bar\Sigma_0^j\), an embedded tree in
\(\bar\Sigma_0^j\cap\bar\Sigma\) containing all the joints of
\(\bar\Sigma_0^j\) -- such a tree exists, the
\(\bar\Sigma\)-part of a piece being connected. Then \(\sfT_0\)
contains no embedded cycle: such a cycle would have to leave
the free part, since \(\mathsf F_{\bar X}\) is a forest
(\cref{thm:weak-goldberg-char}), so it would alternate between
piece-trees and thin-edge paths; each thin edge on it is then
essential -- both of its sides contain thick material of the
cycle -- hence contained in \(\bar X_0\), and the entire cycle
would lie in \(\bar X_0\). But an oriented thin edge of this
embedded cycle defines a cellular \(1\)-cocycle on
\(\bar X_0\): it is contained in no \(2\)-simplex, and the
cocycle evaluates to \(1\) on the cycle. This contradicts
\(H_1(\bar X_0;\Z)=0\).

Extend \(\sfT_0\) to a weakly admissible maximal tree \(\sfT\) as follows.
A joint \(\sfw\) not lying in \(\bar X_0\) has, by
\cref{prop:X0-meets-link-joint}, an \(x\)-component of
\(\bar X\) homeomorphic to an interval; this is necessarily the
side of its thin edge, which therefore dead-ends, so \(\sfw\) lies
on no embedded cycle of any scaffold, and an arc in
\(\bar\Sigma\) from \(\sfw\) to any chosen component of \(\sfT_0\)
creates no cycle. On each thick component carrying a type-2
joint, attach every such stray joint by an arc to one of the
piece-trees -- each of which contains a type-2 joint, by the
claim. Finally complete to a maximal tree by further thick
edges, subject only to acyclicity; additional edges merely
merge components. In the resulting weakly admissible tree \(\sfT\), every
joint of a thick component carrying a type-2 joint is connected
within \(\sfT\cap\bar\Sigma\) to a type-2 joint -- through its
piece-tree if it lies in \(\bar X_0\), and through its
attaching arc otherwise. Hence \(\sfT\) is admissible,
and the proof is complete.
\end{proof}

Combining \cref{thm:goldberg-char-trees} with the resolution
machinery, the characterisation extends to arbitrary complexes.

\begin{theorem}[Goldbergness via admissible trees: the
general case]\label{thm:goldberg-char-trees-general}
Let \(X\) be as in \cref{rem:standing}. Then \(X\) is
Goldberg if and only if \(X\) admits an
admissible tree.
\end{theorem}

\begin{proof}
Note first that \(\bar X\) again satisfies the standing
assumption (\cref{rem:standing}). If no resolution occurs,
then \(\bar X=X\). Otherwise a positive-dimensional link
component is resolved, and its cone retains a \(2\)-simplex
throughout the subsequent resolutions; hence \(\bar X\) is
neither an interval nor a circle. The newly introduced free
edges also persist, so \(\bar X\) is neither \(S^2\) nor
\(\RP^2\).

Along the chain of resolutions from \(X\) to its simple model
\(\bar X\) (\cref{prop:simple-model-unique}), the complexes at
the two ends are simultaneously Goldberg or not, by
\cref{thm:resolution-goldberg-equivalence} applied at each
step; and by \cref{rem:admissible-tree} the pull-back and
push-forward of \cref{def:tree-pullback-pushforward} exchange
the admissible trees of \(X\) with those of
\(\bar X\). The statement therefore reduces to
\cref{thm:goldberg-char-trees}.
\end{proof}

\begin{example}[Two discs wedged at a point]
\label{ex:wedge-discs}
Let \(X=(D_1\vee_\sfv D_2)\cup \sfe\) be the wedge of two discs
with a boundary chord of \cref{ex:no-admissible-tree}
(\cref{fig:wedge-discs-pic}). Its free part
\(\sfF_X=\sfe\cup\{\sfv\}\) is a forest, so \(X\) is weakly
Goldberg (\cref{thm:weak-goldberg-char}); but \(X\) admits no
admissible tree (\cref{ex:no-admissible-tree}), so \(X\) is
\emph{not} Goldberg by \cref{thm:goldberg-char-trees-general}.
\end{example}

\begin{example}[A weakly Goldberg complex that is not Goldberg]
\label{ex:disc-chord}
Perhaps the simplest example separating the two notions: the
disc with a chord \(X=D\cup\sfe\) of
\cref{ex:no-admissible-tree} (\cref{fig:disc-chord-pic}). The
complex is simple and \(\sfF_X=\sfe\) is a forest, so \(X\) is
weakly Goldberg (\cref{thm:weak-goldberg-char}); it admits no
admissible tree (\cref{ex:no-admissible-tree}), so \(X\) is not
Goldberg (\cref{thm:goldberg-char-trees}).
\end{example}

\end{document}